\documentclass[11pt, reqno]{amsart}

\usepackage[top=3.75cm, bottom=3cm, left=3cm, right=3cm]{geometry}
\usepackage{mathscinet}
\usepackage{amsmath}
\usepackage{amsthm}
\usepackage{amsfonts}
\usepackage{amssymb}
\usepackage{mathtools}
\usepackage{bm}
\usepackage[colorlinks=true, citecolor=citation, urlcolor=citation, linkcolor=reference, linktocpage = true]{hyperref}
\usepackage{color}
\usepackage[all,cmtip]{xy}
\usepackage{enumitem}
\usepackage{tikz}  
\usepackage{graphicx}
\usepackage{scalerel}
\usepackage{comment}
\usepackage[mathscr]{eucal}
\usepackage{stmaryrd}
\usepackage{appendix}
\usepackage{indentfirst}

\usepackage{tikz-cd}

\definecolor{citation}{rgb}{0,.40,.80}
\definecolor{reference}{rgb}{.80,0,.40}

\numberwithin{equation}{section}

\theoremstyle{plain}
\newtheorem{theorem}{Theorem}[section]
\newtheorem{lemma}[theorem]{Lemma}
\newtheorem{proposition}[theorem]{Proposition}
\newtheorem{corollary}[theorem]{Corollary}
\newtheorem{conjecture}[theorem]{Conjecture}
\newtheorem{question}[theorem]{Question}

\theoremstyle{definition}
\newtheorem{definition}[theorem]{Definition}
\newtheorem{example}[theorem]{Example}
\newtheorem{remark}[theorem]{Remark}
\newtheorem{notation}[theorem]{Notation}
\newtheorem{warning}[theorem]{Warning}
\newtheorem{construction}[theorem]{Construction}
\newtheorem{convention}[theorem]{Convention}

\newtheorem{situation}[theorem]{Situation}

\makeatletter
\newtheoremstyle{italicsname}
 {3pt}
 {3pt}
 {\itshape}
 {}
{\bf}
 {.}
 {.5em}
 {\thmname{#1}\thmnumber{\@ifnotempty{#1}{ }#2}%
 \thmnote{ {\the\thm@notefont(#3)}}}
\makeatother
\theoremstyle{italicsname}
\newtheorem{innerstep}{Step}
\newenvironment{step}[1]
 {\renewcommand\theinnerstep{#1}\innerstep}
 {\endinnerstep}

\setlist[itemize]{topsep=5pt,itemsep=3pt}
\setlist[enumerate]{topsep=5pt,itemsep=3pt}

\newcommand{\st}{\mid} 
\newcommand{\set}[1]{\left\{ \, #1 \, \right\}}

\newcommand{\Db}{\mathrm{D^b}}
\newcommand{\Dperf}{\mathrm{D}_{\mathrm{perf}}}
\newcommand{\Dqc}{\mathrm{D}_{\mathrm{qc}}}
\newcommand{\llangle}{\left \langle}
\newcommand{\rrangle}{\right \rangle}

\DeclareMathOperator{\Hdg}{Hdg}
\DeclareMathOperator{\Sp}{Sp}

\DeclareMathOperator{\colim}{colim}

\newcommand{\cl}{{\mathrm{cl}}}

\newcommand{\Mod}{\mathrm{Mod}}

\newcommand{\Fun}{\mathrm{Fun}}
\newcommand{\Ind}{\mathrm{Ind}}

\newcommand{\Aff}{\mathrm{Aff}}

\newcommand{\Sch}{\mathrm{Sch}}

\newcommand{\Grpd}{\mathrm{Grpd}}

\newcommand{\dSt}{\mathrm{dStk}}
\newcommand{\St}{\mathrm{Stk}}

\newcommand{\one}{\mathbf{1}}

\newcommand{\FM}{\mathrm{FM}}

\newcommand{\Gr}{\mathrm{Gr}}

\DeclareMathOperator{\Pic}{Pic}
\DeclareMathOperator{\cPic}{\mathcal{P}{\it ic}}

\DeclareMathOperator{\Spec}{Spec}

\newcommand{\op}{\mathrm{op}}

\newcommand{\dRing}{\mathrm{dRing}}
\newcommand{\dAff}{\mathrm{dAff}}

\DeclareMathOperator{\coev}{coev}
\DeclareMathOperator{\eva}{ev}

\DeclareMathOperator{\Tr}{Tr}

\newcommand{\cHom}{\mathcal{H}\!{\it om}}
\DeclareMathOperator{\Hom}{Hom}
\DeclareMathOperator{\End}{End}
\DeclareMathOperator{\Ext}{Ext}
\DeclareMathOperator{\Aut}{Aut}
\DeclareMathOperator{\cAut}{\mathcal{A}{\it ut}}

\DeclareMathOperator{\fAut}{\mathfrak{A}{\it ut}}

\DeclareMathOperator{\HH}{HH}
\DeclareMathOperator{\cHH}{\mathcal{HH}}
\DeclareMathOperator{\Map}{Map}

\DeclareMathOperator\coker{coker}

\newcommand{\svee}{\scriptscriptstyle\vee}
\newcommand{\id}{\mathrm{id}}
\newcommand{\pr}{\mathrm{pr}}

\newcommand{\ind}{\mathrm{ind}}

\newcommand{\gl}{\mathrm{gl}}

\newcommand{\Coh}{\mathrm{Coh}}
\newcommand{\Cat}{\mathrm{Cat}}
\newcommand{\PrCat}{\mathrm{PrCat}}

\newcommand{\Ku}{\mathcal{K}u}

\newcommand{\num}{\mathrm{num}}
\DeclareMathOperator{\Knum}{\rK_{\num}}

\DeclareMathOperator{\Stab}{Stab}

\DeclareMathOperator{\Br}{Br}
\DeclareMathOperator{\BrAz}{Br_{Az}}
\DeclareMathOperator{\Div}{Div}

\DeclareMathOperator{\CH}{CH}

\newcommand{\rtop}{\mathrm{top}}

\newcommand{\Ktop}[1][]{\rK_{#1}^{\rtop}}

\newcommand{\vir}{\mathrm{vir}}

\newcommand{\per}{\mathrm{per}}

\DeclareMathOperator{\DT}{DT}
\DeclareMathOperator{\Pf}{Pf}

\newcommand{\class}{u}
\newcommand{\init}{1}
\newcommand{\fin}{0}

\DeclareMathOperator{\NS}{NS}
\newcommand{\bF}{\mathbf{F}}
\DeclareMathOperator{\Spin}{Spin}

\newcommand{\bv}{\mathbf{v}}

\newcommand{\et}{\mathrm{\acute{e}t}}
\newcommand{\tors}{\mathrm{tors}}
\newcommand{\pt}{\mathrm{pt}}
\newcommand{\muk}{\widetilde{\rH}}
\newcommand{\topo}{\mathrm{top}}
\newcommand{\Cliff}{\mathrm{Cl}}

\newcommand{\medwedge}{\textstyle{\bigwedge}}
\newcommand{\ev}{\mathrm{ev}}
\newcommand{\length}{\ell}

\newcommand{\gr}{\mathrm{gr}}

\newcommand{\bmu}{\bm{\mu}}
\newcommand{\SL}{\mathrm{SL}}

\DeclareMathOperator{\cone}{cone}
\DeclareMathOperator{\cofib}{cofib}
\DeclareMathOperator{\fib}{fib}

\DeclareMathOperator{\ch}{{ch}}
\DeclareMathOperator{\rk}{{rk}}
\newcommand{\td}{\mathrm{td}}

\newcommand{\cO}{\mathcal{O}}
\newcommand{\cA}{\mathcal{A}}
\newcommand{\cC}{\mathscr{C}}
\newcommand{\cD}{\mathscr{D}}
\newcommand{\cE}{\mathcal{E}}
\newcommand{\cF}{\mathcal{F}}
\newcommand{\cG}{\mathcal{G}}
\newcommand{\cH}{\mathcal{H}}

\newcommand{\cL}{\mathcal{L}}
\newcommand{\cM}{\mathcal{M}}

\newcommand{\cP}{\mathcal{P}}

\newcommand{\cT}{\mathcal{T}}
\newcommand{\cU}{\mathcal{U}}

\newcommand{\cX}{\mathcal{X}}
\newcommand{\cY}{\mathcal{Y}}

\newcommand{\ccA}{\mathscr{A}}
\newcommand{\ccE}{\mathscr{E}}

\newcommand{\rB}{\mathrm{B}}
\newcommand{\rC}{\mathrm{C}}
\newcommand{\rD}{\mathrm{D}}
\newcommand{\rE}{\mathrm{E}}
\newcommand{\rF}{\mathrm{F}}
\newcommand{\rG}{\mathrm{G}}

\newcommand{\rH}{\mathrm{H}}

\newcommand{\rK}{\mathrm{K}}
\newcommand{\rS}{\mathrm{S}}

\newcommand{\rM}{\mathrm{M}}

\newcommand{\rR}{\mathrm{R}}
\newcommand{\rO}{\mathrm{O}}
\newcommand{\rP}{\mathrm{P}}

\newcommand{\rT}{\mathrm{T}}
\newcommand{\rV}{\mathrm{V}}

\newcommand{\fG}{\mathfrak{G}}

\newcommand{\fm}{\mathfrak{m}}
\newcommand{\fM}{\mathfrak{M}}
\newcommand{\fX}{\mathfrak{X}}

\newcommand{\bC}{\mathbf{C}}

\newcommand{\bG}{\mathbf{G}}
\newcommand{\bH}{\mathbf{H}}

\newcommand{\bZ}{\mathbf{Z}}
\newcommand{\bP}{\mathbf{P}}
\newcommand{\bQ}{\mathbf{Q}}
\newcommand{\bR}{\mathbf{R}}

\usepackage{xpatch}
\makeatletter   
\xpatchcmd{\@tocline}
{\hfil\hbox to\@pnumwidth{\@tocpagenum{#7}}\par}
{\ifnum#1<0\hfill\else\dotfill\fi\hbox to\@pnumwidth{\@tocpagenum{#7}}\par}
{}{}
\makeatother

\makeatletter
\renewcommand\part{%
   \if@noskipsec \leavevmode \fi
   \par
   \addvspace{4ex}%
   \@afterindentfalse
   \secdef\@part\@spart}

\def\@part[#1]#2{%
    \ifnum \c@secnumdepth >\m@ne
      \refstepcounter{part}%
      \addcontentsline{toc}{part}{Part \thepart.\hspace{1em}#1}%
    \else
      \addcontentsline{toc}{part}{#1}%
    \fi
    {\parindent \z@ \raggedright
     \interlinepenalty \@M
     \normalfont
     \ifnum \c@secnumdepth >\m@ne
     \centering 
     \Large\bfseries \partname\nobreakspace\thepart     
       \nobreak. 
     \fi
     \Large \bfseries { #2}%
     \par}%
    \nobreak
    \vskip 3ex
    \@afterheading}
\def\@spart#1{%
    {\parindent \z@ \raggedright
     \interlinepenalty \@M
     \normalfont
     \huge \bfseries #1\par}%
     \nobreak
     \vskip 3ex
     \@afterheading}
\makeatother

\renewcommand{\thepart}{\Roman{part}}

\begin{document}

\title[The period-index conjecture for abelian threefolds and DT theory]{The period-index conjecture for abelian threefolds and Donaldson--Thomas theory} 

\author{James Hotchkiss}
\address{Department of Mathematics, Columbia University, New York, NY 10027 \smallskip}
\email{james.hotchkiss@columbia.edu}

\author{Alexander Perry}
\address{Department of Mathematics, University of Michigan, Ann Arbor, MI 48109 \smallskip}
\email{arper@umich.edu}


\begin{abstract}
We prove the period-index conjecture for unramified Brauer classes on abelian threefolds. 
To do so, 
we develop a theory of reduced Donaldson--Thomas invariants for 3-dimensional Calabi--Yau categories, 
with the feature that the noncommutative variational integral Hodge conjecture holds for classes with nonvanishing invariant. 
The period-index result is then proved by interpreting it as the algebraicity of a Hodge class on the twisted derived category, and specializing within the Hodge locus to an untwisted abelian threefold with nonvanishing invariant. 
As a consequence, we also deduce the integral Hodge conjecture for generically twisted abelian threefolds. 
\end{abstract}

\maketitle

\setcounter{tocdepth}{1}
\tableofcontents


\section{Introduction}
\label{section-intro} 

The goal of this paper is to introduce a new approach to the period-index 
and integral Hodge conjectures based on Donaldson--Thomas theory, 
and to use this perspective to prove the unramified case of the period-index conjecture for abelian threefolds. 
As an important ingredient, we develop a general  
theory of reduced Donaldson--Thomas invariants for $3$-dimensional Calabi--Yau categories.

\subsection{The period-index conjecture} 

Let $K$ be a field. 
The Brauer group $\Br(K)$ is a fundamental invariant of $K$ with numerous applications in geometry and arithmetic.  
The most elementary definition is in terms of central simple algebras over $K$, which are unital, associative $K$-algebras with finite dimension over $K$, center $K$, and no nontrivial two-sided ideals. 
Two such algebras $A$ and $B$ are called Morita equivalent if there exists an isomorphism of matrix algebras $\rM_r(A) \cong \rM_s(B)$ for some positive integers $r$ and $s$. 
Then $\Br(K)$ is the abelian group of central simple algebras over $K$ modulo Morita equivalence, with group operation induced by tensor product over $K$. 
Any central simple algebra over $K$ is Morita equivalent to a central division algebra over $K$ which is unique up to isomorphism, 
so $\Br(K)$ can be thought of as a group classifying central division algebras. 

The complexity of a Brauer class $\alpha \in \Br(K)$ can be measured by two integer invariants. 
The first is the \emph{period} $\per(\alpha)$, equal to the order of $\alpha$ in $\Br(K)$; 
this is a positive integer because $\Br(K)$ is a torsion group.
The second is the \emph{index} $\ind(\alpha)$, equal to the $\sqrt{\dim_K D}$ where $D$ is the unique central division algebra of class $\alpha$; 
this is a positive integer because, more generally, any central simple algebra $A$ satisfies $A \otimes_K \overline{K} \cong \rM_n(\overline{K})$ and hence has square dimension over $K$. 

It is an elementary result that $\per(\alpha) \mid \ind(\alpha)$ (where for integers $a,b$ we use the standard notation $a \mid b$ to mean $a$ divides $b$), and that $\per(\alpha)$ and $\ind(\alpha)$ share the same prime factors. 
The \emph{period-index problem} is to determine an integer $e$ such that $\ind(\alpha) \mid \per(\alpha)^e$. 
This problem plays a central role in the study of Brauer groups and has drawn much attention over the 
past century; see \cite[\S4]{open-problems-Br} or below for a partial survey of results. 
In particular, the following folklore conjecture has emerged. 

\begin{conjecture}[Period-index conjecture] 
\label{conjecture-period-index} 
Let $K$ be a field of transcendence degree $d$ over an algebraically closed field $k$. 
For any $\alpha \in \Br(K)$, we have 
\begin{equation*}
\ind(\alpha) \mid \per(\alpha)^{d-1}. 
\end{equation*} 
\end{conjecture} 

To our knowledge, Conjecture~\ref{conjecture-period-index} was first raised in print by Colliot-Th\'{e}l\`{e}ne \cite{CT-PI} 
(see also \cite{CT-bourbaki, lieblich-period-index}). 
Despite an abundance of work, 
the number of known cases is still quite sparse: 
\begin{itemize}
\item For $d \leq 1$ the conjecture is vacuous, as $\Br(K) = 0$ by an easy argument when $d = 0$ and by Tsen's theorem when $d =1$. 
\item For $d = 2$ the conjecture is true by work of de Jong \cite{dJ-period-index} when $\per(\alpha)$ is prime to the characteristic of $k$, and building on this, by Lieblich \cite{lieblich-period-index} and de Jong--Starr \cite{dJ-starr} in general. 
\item For $d \geq 3$ the conjecture is not known for any field $K$ whatsoever. 
\end{itemize} 
For arbitrary $d$, Matzri \cite{matzri} proved $\ind(\alpha) \mid \per(\alpha)^e$ for an exponent $e$ that can be bounded by a polynomial in $\per(\alpha)$ (depending on $d$); 
however, the bound falls far short of the conjectural one in that $e$ it is not uniform in $\per(\alpha)$ and is very large even for small $\per(\alpha)$. 
Recently, Huybrechts and Mattei \cite{huybrechts-mattei} proved the existence of a bound $\ind(\alpha) \mid \per(\alpha)^{e}$ for all unramified (in the sense described below) classes $\alpha \in \Br(K)$, where $e$ depends on $K$ but not on $\per(\alpha)$. 
Finally, we note that  Conjecture~\ref{conjecture-period-index} is sharp, with examples achieving the bound going back to 1935 \cite{nakayama} (see also \cite{CT-examples, hotchkiss-pi, dJP-pi}).

It is natural to consider Conjecture~\ref{conjecture-period-index} for classes $\alpha \in \Br(K)$ that come from a global model of $K$, 
where more structure is available. 
Namely, by Grothendieck \cite{grothendieck-brauer} there is an extension $\Br(X)$ of the Brauer group to any scheme $X$, 
classifying Azumaya algebras on $X$ up to Morita equivalence. 
If $X$ is a smooth projective variety over $k$ with function field $K$, 
then the restriction map $\Br(X) \to \Br(K)$ is injective, with image independent of 
the chosen model $X$. 
The classes in $\Br(K)$ arising in this way are said to be \emph{unramified}. 
They are the crucial classes to consider for the period-index problem, as work of de Jong--Starr \cite{dJ-starr} implies that for a fixed $d$, Conjecture~\ref{conjecture-period-index} for every $K$ is equivalent to the following a priori weaker conjecture for every $X$: 

\begin{conjecture}[Unramified period-index conjecture] 
Let $X$ be a smooth projective variety of dimension $d$ over an algebraically closed field $k$. 
For any $\alpha \in \Br(X)$, we have 
\begin{equation*}
\ind(\alpha) \mid \per(\alpha)^{d-1} ,
\end{equation*} 
where the period and index are defined as those of the class over the generic point. 
\end{conjecture} 

In this paper, we prove the conjecture for abelian threefolds, under a mild assumption on the characteristic. 

 \begin{theorem}
 \label{theorem-main} 
Let $X$ be an abelian threefold over an algebraically closed field $k$. 
For any $\alpha \in \Br(X)$ with $\per(\alpha)$ prime to the characteristic of $k$, we have   
\begin{equation*}
\ind(\alpha) \mid \per(\alpha)^2. 
\end{equation*} 
 \end{theorem} 

This appears to be the first nontrivial case of the unramified period-index conjecture established in dimension greater than $2$. 
Note that the Brauer group of an abelian threefold $X$ is indeed far from trivial; 
over the complex numbers, $\Br(X) \cong (\bQ/\bZ)^{15-\rho}$ where 
$\rho$ is the Picard number of $X$ (Lemma~\ref{lemma-Br-ses}).  
By an elementary argument one can show that $\ind(\alpha) \mid \per(\alpha)^3$ holds in the situation of the theorem (Lemma~\ref{lem:symbol_length_bounds}), but the improvement to exponent $2$ seems to lie much deeper. 
In fact, in the analytic setting this improvement fails: by \cite{hotchkiss-tori} for a general complex torus of dimension $3$ there exist Brauer classes with $\ind(\alpha) = \per(\alpha)^3$. 

\begin{remark}
We expect the hypothesis that $\per(\alpha)$ is prime to the characteristic of $k$ can be removed in Theorem~\ref{theorem-main}. 
It would be enough to show that when the characteristic of $k$ is positive, then every twisted abelian threefold lifts to characteristic $0$ (see Remark~\ref{remark-lifting-AV}). 
\end{remark} 

Our proof of Theorem~\ref{theorem-main} places it in a larger framework that we explain in the rest of the introduction. 

\subsection{Hodge theory of categories} 
\label{section-intro-hodge-theory}
Let $\cC \subset \Dperf(X)$ be a semiorthogonal component of the category of perfect complexes on a smooth proper complex variety $X$. 
By \cite{IHC-CY2}, there is a canonically associated finitely generated abelian group $\Ktop[0](\cC)$ (equal to $\pi_0$ of Blanc's topological K-theory \cite{blanc}) equipped with a weight $0$ Hodge structure. 
There is a natural map $\rK_0(\cC) \to \Ktop[0](\cC)$ from the Grothendieck group of $\cC$; classes in the image are called \emph{algebraic}, and they lie in the subgroup $\Hdg(\cC, \bZ) \subset \Ktop[0](\cC)$ of integral Hodge classes. 
This motivates the statement of the \emph{integral Hodge conjecture} for $\cC$, which asserts that the map $\rK_0(\cC) \to \Hdg(\cC, \bZ)$ is surjective.
Similarly, the \emph{Hodge conjecture} for $\cC$ asserts surjectivity after tensoring with $\bQ$. 

In the geometric case when $\cC = \Dperf(X)$, the construction $\rK_0(\cC) \to \Hdg(\cC, \bZ)$ rationally recovers the usual cycle class map $\CH^*(X) \otimes \bQ \to \Hdg^*(X, \bQ)$, where the target denotes the group of rational Hodge classes on $X$; in particular, the Hodge conjecture for $\Dperf(X)$ is equivalent to the usual Hodge conjecture for $X$ in all degrees. 
There is also a close relationship between the integral Hodge conjectures for $\Dperf(X)$ and $X$, explained in~\cite{IHC-CY2}. 

The connection to the period-index conjecture arises when 
$\cC = \Dperf(X, \alpha)$ is the derived category of $\alpha$-twisted sheaves for a Brauer class $\alpha \in \Br(X)$,   
in which case 
we write $\Ktop[0](X, \alpha)$, $\Hdg(X, \alpha, \bZ)$, and $\rK_0(X, \alpha)$ for the above invariants applied to $\Dperf(X, \alpha)$. 
Indeed, since the index of $\alpha$ can be computed as the minimal positive rank of an element of $\rK_0(X, \alpha)$ (Lemma~\ref{lemma-ind-as-rk}), 
the period-index conjecture for $\alpha$ may be divided into two steps: 

\begin{step}{1}
\label{PI-step1}
Construct a Hodge class $v \in \Hdg(X, \alpha, \bZ)$ of rank $\per(\alpha)^{\dim X-1}$. 
\end{step} 

\begin{step}{2}
\label{PI-step2}
Show that $v$ is algebraic. 
\end{step} 

In \cite{hotchkiss-pi}, the Hodge structure $\Ktop(X, \alpha)$ was described explicitly when $\alpha$ is topologically trivial and used to solve Step~\ref{PI-step1} for $\per(\alpha)$ prime to $(\dim X - 1)!$. 
In some cases, like when $X$ is an abelian variety, this method refines to a solution to Step~\ref{PI-step1} for all $\alpha$. 

In this paper, we introduce a method to solve Step~\ref{PI-step2} above, or more generally to show that a given Hodge class on a category is algebraic. 
The method is variational in nature, and depends on the notion of an $S$-linear category $\cC$ over a base space $S$ (discussed in \S\ref{section-linear-cats}), which formalizes the idea of a family of categories parameterized by $S$. 
There is a base change operation for such categories, giving rise to a fiber category $\cC_s$ for every $s \in S$. 
When $\cC$ is smooth and proper of geometric origin over $S$, meaning  
it arises as an $S$-linear semiorthogonal component of $\Dperf(X)$ for a smooth proper morphism $X \to S$, 
then by \cite{IHC-CY2} there is a local system $\Ktop[0](\cC/S)$ on $S$ underlying a weight $0$ variation of Hodge structures, with fibers $\Ktop[0](\cC_s)$ for $s \in S(\bC)$.

\begin{conjecture}[Noncommutative variational integral Hodge conjecture]
\label{conjecture-VHC}
Let $\cC$ be a smooth proper $S$-linear category of geometric origin, where $S$ is a complex variety. 
Let $v$ be a section of the local system $\Ktop[0](\cC/S)$ on $S$. 
Assume that there exists a point $0 \in S(\bC)$ such that the fiber $v_0 \in \Ktop[0](\cC_0)$ is algebraic. 
Then $v_s \in \Ktop[0](\cC_s)$ is algebraic for every $s \in S(\bC)$.  
\end{conjecture}

Like the (variational) integral Hodge conjecture for varieties, this statement is false in general, but in keeping with tradition we still refer to it as a ``conjecture''.
Nonetheless, one can hope to prove that Conjecture~\ref{conjecture-VHC} holds in special situations. 

For instance, by \cite[Proposition 8.1]{IHC-CY2} it holds when $v_0$ is not only algebraic, but can be realized as class of an object $E_0 \in \cC_0$ at which the morphism $\cM_{\gl}(\cC/S) \to S$ is smooth, where $\cM_{\gl}(\cC/S)$ is the relative moduli space of gluable objects in $\cC$ (reviewed in \S\ref{section-classical-moduli}). 
As shown in \cite{IHC-CY2}, this criterion can be effectively applied when $\cC$ is a CY$2$ category over $S$, which roughly means that it has the same homological properties as the derived category of a family of Calabi--Yau surfaces 
(see Definition~\ref{definition-CYn} for the precise meaning of a CY$n$ category in general). 
There are two key ingredients: 
\begin{enumerate}
\item \label{smoothness-moduli}
A general version of Mukai's smoothness theorem, asserting that the moduli space 
$s\cM_{\gl}(\cC/S, v)$ of simple gluable objects of class $v$ is smooth over $S$.  
\item \label{nonemptiness-moduli} 
The existence of many simple gluable objects on K3 and abelian surfaces, which follows from nonemptiness of moduli spaces of Bridgeland stable objects. 
\end{enumerate} 
The upshot is that Conjecture~\ref{conjecture-VHC} holds essentially whenever $\cC$ is a CY2 category over $S$ and $\cC_0 \simeq \Db(T)$ for a K3 or abelian surface $T$ \cite[Theorem 1.1]{IHC-CY2}. 
In particular, this leads to a proof of the integral Hodge conjecture for CY2 categories that can be suitably specialized to the derived category of a K3 or abelian surface. 

The above method fails badly in higher dimensions. 
Namely, for CY$n$ categories of dimension $n \geq 3$, 
moduli spaces of objects are rarely smooth and nonemptiness of moduli spaces of stable objects is far from understood, even when $\cC = \Dperf(X)$ is geometric. 
We develop a new approach in the CY3 case, which is roughly to \emph{count} 
the number solutions to $v_0 = [E_0]$ with $E_0 \in \cC_0$ a stable object, 
and to show that Conjecture~\ref{conjecture-VHC} holds when it is nonzero. 
The precise meaning of this count depends on new foundations for Donaldson--Thomas theory. 

\subsection{Donaldson--Thomas theory} 
Let $X$ be a smooth projective complex threefold which is Calabi--Yau in the sense that $\omega_X \cong \cO_X$. 
Let $v \in \Ktop[0](X)$ be a topological class and $H$ a polarization on $X$ such that Gieseker $H$-semistability and $H$-stability coincide for sheaves on $X$ of class $v$. 
The expected dimension of the moduli space $M_H(v)$ of such semistable sheaves at a point $E$ is $\dim \Ext^1(E,E) - \dim \Ext^2(E,E)$, which vanishes by Serre duality and the Calabi--Yau assumption. 
The actual dimension of $M_H(v)$, however, is often much larger. 
To remedy this, Thomas \cite{thomas-DT} showed that $M_H(v)$ carries a symmetric perfect obstruction theory in the sense of Behrend--Fantechi \cite{BF-normal-cone, BF-symmetric}, which gives rise to a cycle of the expected dimension, the \emph{virtual fundamental class} 
\begin{equation*} 
[M_H(v)]^{\vir} \in \CH_0(M_H(v)). 
\end{equation*} 
The degree of this $0$-cycle is the \emph{Donaldson--Thomas (DT)} invariant of $M_H(v)$, which plays a central role in modern enumerative geometry. 
There are also important extensions of these invariants to other settings, for instance to moduli spaces $M_{\sigma}(v)$ of stable objects with respect to a stability condition $\sigma$ on $\Db(X)$ \cite{joyce-song, kontsevich-soibelman}. 

DT invariants are typically studied on Calabi--Yau threefolds which are strict in the sense that $h^1(\cO_X) = 0$. 
Indeed, when $h^1(\cO_X) > 0$ then there are extra symmetries which often force the DT invariants to vanish. 
In such cases, one can instead hope to construct interesting \emph{reduced DT invariants} by suitably modifying the space $M_H(v)$ and its obstruction theory. 
This was done by Gulbrandsen \cite{gulbrandsen} for abelian threefolds and by Oberdieck \cite{obPT} for the product of a K3 surface and elliptic curve; their constructions, however, are ad hoc and depend on the geometry of the specific situation. 

In this paper, we develop a general theory of reduced DT invariants for CY3 categories. 
The main construction is summarized in the following theorem. 

\begin{theorem}
\label{theorem-intro-DT-definition}
Let $\cC$ be a CY3 category of geometric origin over $\bC$.  
Let $\sigma$ be a stability condition on $\cC$ over $\bC$ with respect to a topological Mukai homomorphism. 
Let $v \in \Ktop[0](\cC)$ be a class for which no strictly $\sigma$-semistable objects of class $v$ exist. 
Let $\Aut^0(\cC)$ denote the identity component of the group of autoequivalences of $\cC$.

Then if the quotient stack $M_\sigma(v)/ \Aut^0(\cC)$ is Deligne--Mumford, 
it carries a canonical 
symmetric perfect obstruction theory and hence virtual fundamental class 
\begin{equation*}
[M_\sigma(v)/ \Aut^0(\cC)]^{\vir} \in \CH_0(M_\sigma(v)/ \Aut^0(\cC)). 
\end{equation*} 
If moreover $\Aut^0(\cC)$ is proper, then so is $M_\sigma(v)/ \Aut^0(\cC)$, and thus the virtual class can be integrated to define a \emph{reduced DT invariant} 
\begin{equation*}
\DT_\sigma(v) \coloneqq \int_{[M_\sigma(v)/ \Aut^0(\cC)]^{\vir}} 1 .  
\end{equation*} 
\end{theorem} 

\begin{remark}
We have chosen to denote the reduced DT invariant by $\DT_{\sigma}(v)$, without any adornment, as we believe it is the ``true'' DT invariant in general. 
Indeed, for $\cC = \Dperf(X)$ it recovers the classical DT invariant when $X$ is a strict Calabi--Yau threefold and the previously studied reduced DT invariant when $X$ has an abelian factor. 
\end{remark}

In the formulation of Theorem~\ref{theorem-intro-DT-definition}, we use the notion of a stability condition relative to a base developed in \cite{stability-families}; in particular, this guarantees that a proper moduli space $M_\sigma(v)$ of $\sigma$-semistable objects of class $v$ does indeed exist. 
The Mukai homomorphism of $\sigma$ being topological is always satisfied in practice, and means that its value on any class in $\rK_0(\cC)$ is determined by its image in $\Ktop[0](\cC)$ (Definition~\ref{definition-topological-v}). 
The construction and properties of the group algebraic space of autoequivalences $\Aut(\cC)$ and its identity component $\Aut^0(\cC)$ are discussed in \S\ref{section-autoequivalences}; 
in particular, the assumption that $\Aut^0(\cC)$ is proper is mild, and holds automatically for twisted derived categories of Calabi--Yau varieties (Lemma~\ref{lemma-Aut0-CYn-proper}).

In the body of the text, we prove a more general version of Theorem~\ref{theorem-intro-DT-definition} that applies to suitable quotients of substacks of the moduli stack of all objects in $\cC$ (Theorem~\ref{theorem-obs-theory} and Definition~\ref{definition-DT}). 
In particular, we also obtain an analog of Theorem~\ref{theorem-intro-DT-definition} for moduli spaces of Gieseker semistable sheaves (Definition~\ref{definition-DT-twisted}).  
Our general result also works in families, leading to the constancy of our DT invariants under deformations of $\cC$.  

\begin{theorem}
\label{theorem-intro-DT-defo-invariant}
Let $\cC$ be a CY3 category of geometric origin over a smooth complex variety $S$.  
Let $\sigma$ be a stability condition on $\cC$ over $S$ with respect to a topological Mukai homomorphism. 
Let $v$ be a section of the local system $\Ktop[0](\cC/S)$ such that for every $s \in S(\bC)$, 
the fiber $v_s \in \Ktop[0](\cC_s)$ is a Hodge class and there are no strictly $\sigma_s$-semistable objects in $\cC_s$ of class~$v_s$.  
Let $\Aut^0(\cC/S)$ denote the identity component of the group of $S$-linear autoequivalences of $\cC$.

Then if the quotient stack $M_\sigma(v)/ \Aut^0(\cC/S)$ is Deligne--Mumford, 
it carries a canonical 
symmetric perfect obstruction theory over $S$. 
If moreover $\Aut^0(\cC/S) \to S$ is proper, then the reduced DT invariants 
$\DT_{\sigma_s}(v_s)$ of the fibers are independent of $s \in S(\bC)$. 
\end{theorem} 

We construct the obstruction theory in Theorem~\ref{theorem-intro-DT-defo-invariant} in terms of the ``cohomology'' of a natural $3$-term ``complex of complexes''~\eqref{complex-obs-theory} on the moduli stack of $\sigma$-semistable objects, involving the action of Hochschild cohomology on the endomorphisms of the universal object. 
To do so, we use derived algebraic geometry --- in particular a derived enhancement of the quotient stack $M_\sigma(v)/ \Aut^0(\cC/S)$ --- combined with Pridham's semiregularity results \cite{pridham}. 
Our method should also be useful for constructing reduced obstruction theories in other settings; see \S\ref{section-DT4} for an upcoming example. 

Now we can state our criterion in terms of DT invariants for the validity of the noncommutative variational integral Hodge conjecture. 

\begin{theorem}
\label{theorem-intro-VIHC-sigma}
Let $\cC$ be a CY3 category of geometric origin over a smooth complex variety $S$. 
Let $\sigma$ be a stability condition on $\cC$ over $S$ with respect to a topological Mukai homomorphism. 
Let $v$ be a section of $\Ktop[0](\cC/S)$ whose fibers $v_s \in \Ktop[0](\cC_s)$ are Hodge classes for all $s \in S(\bC)$. 
Assume there exists a point $0 \in S(\bC)$ such that: 
\begin{enumerate}
\item  There do not exist strictly $\sigma_0$-semistable objects of class $v_0$. 
\item $M_{\sigma_0}(v_0)/\Aut^0(\cC_0)$ is Deligne--Mumford. 
\item $\DT_{\sigma_0}(v_0) \neq 0$. 
\end{enumerate} 
Then for every $s \in S(\bC)$ the moduli space $M_{\sigma_s}(v_s)$ is nonempty, and in particular, the class 
$v_s \in \Ktop[0](\cC_s)$ is algebraic. 
\end{theorem} 

Again, in the body of the text we prove a more general version of Theorem~\ref{theorem-intro-VIHC-sigma} that applies to suitable quotients of substacks of the moduli stack of all objects in $\cC$ (Theorem~\ref{theorem-VIHC-general}).
In particular, we also obtain an analog of Theorem~\ref{theorem-intro-VIHC-sigma} for moduli spaces of Gieseker semistable sheaves (Corollary~\ref{corollary-VIHC-Gieseker}). 
The key input for the proof of the above criterion is the deformation invariance of reduced DT invariants from Theorem~\ref{theorem-intro-DT-defo-invariant}. 

\subsection{The period-index conjecture for abelian threefolds} 
\label{section-intro-proof-sketch}
Now we can complete the sketch of our proof of Theorem~\ref{theorem-main}, the  
period-index conjecture for abelian threefolds. 
By a lifting argument, the theorem reduces to the case where the base field is the complex numbers. 
As explained in \S\ref{section-intro-hodge-theory}, 
given any complex abelian threefold $X_1$ and Brauer class $\alpha_1 \in \Br(X_1)$, 
it suffices to construct a Hodge class $v_1 \in \Hdg(X_1, \alpha_1, \bZ)$ of rank $\per(\alpha_1)^2$ such that $v_1$ is algebraic. 
To do so, we construct: a family $X \to S$ of abelian threefolds with a Brauer class $\alpha \in \Br(X)$; a stability condition $\sigma$ on $\cC = \Dperf(X, \alpha)$ over $S$, a section $v$ of $\Ktop[0](\cC/S)$, and a point $0 \in S(\bC)$ satisfying the hypotheses of Theorem~\ref{theorem-intro-VIHC-sigma}; and a point $1 \in S(\bC)$ such that $(X_1, \alpha_1)$ is the fiber of $(X, \alpha)$ over $1$ and $v_1$ has rank $\per(\alpha_1)^2$. 
Then by~Theorem~\ref{theorem-intro-VIHC-sigma} the class $v_1$ is algebraic.  

The construction of the data $(X, \alpha, S, \sigma, v, 0, 1)$ is quite intricate, but the main idea is that by specializing $X_1$ into a suitable Hodge locus in the moduli space of polarized abelian threefolds, 
we are able to choose the family $(X, \alpha) \to S$ so that $\alpha_0 \in \Br(X_0)$ vanishes, and 
thus reduce to computing a classical invariant $\DT_{\sigma_0}(v_0)$ on an abelian threefold $X_0$. 
There is still the issue that $v_0$ is a higher rank class, whose DT invariant is difficult to access directly, but by using results from \cite{OPT} on the behavior of DT invariants under the action of derived equivalences and wall-crossing, 
we are able to further reduce to the case of a curve class. 
For a very careful choice of $v$, we are able to ensure the DT invariant of the resulting curve class is computable and nonzero, using results from \cite{obshen, BOPR}. 
This argument illustrates a powerful, and somewhat surprising, feature of our theory of DT invariants for CY3 categories: it allows one to deduce algebraicity results for a priori inaccessible ``noncommutative'' objects from purely geometric arguments, like curve counting. 

The above summary elided a number of subtleties in the proof of Theorem~\ref{theorem-main}. 
For instance, in order to construct $\sigma$ we generalize the known constructions of stability conditions on varieties, and in particular on abelian threefolds \cite{BMS}, to the twisted setting. 
To ensure that strictly $\sigma_0$-semistable objects of class $v_0$ do not exist, 
we use the deformation theorem for relative stability conditions from \cite{stability-families}; 
in fact, the ability to make such an argument dictated our use of stability conditions, as it is not possible to arrange the analogous condition only using a relative polarization and Gieseker stability. 
For a more detailed discussion of Theorem~\ref{theorem-main} and its proof, we refer to \S\ref{section-main-result-abelian-3folds}. 

\subsection{The integral Hodge conjecture for twisted Calabi--Yau threefolds}
In \S\ref{section-intro-hodge-theory} we explained that for a smooth proper complex variety $X$ and a Brauer class $\alpha \in \Br(X)$, 
the period-index conjecture can be regarded as an instance of the integral Hodge conjecture for $(X, \alpha)$. 
As a complement to the above results, when $X$ is a Calabi--Yau threefold we show that conversely the integral Hodge conjecture for $(X, \alpha)$ is completely controlled by the index of~$\alpha$. 

More precisely, consider the \emph{Voisin group} 
\begin{equation*}
\rV(X,\alpha) = \coker( \rK_0(X, \alpha) \to \Hdg(X, \alpha, \bZ)), 
\end{equation*} 
which measures the failure of the integral Hodge conjecture for $(X, \alpha)$, and 
the \emph{Hodge-theoretic index} 
\begin{equation*}
   \ind_{\Hdg}(\alpha) = \min \set{\rk v \st v \in \Hdg(X, \alpha, \bZ), ~ \rk v > 0 }, 
\end{equation*}  
which is the Hodge-theoretic incarnation of the formula for $\ind(\alpha)$ in terms of the ranks of objects in $\Dperf(X, \alpha)$. 
The Hodge-theoretic index was introduced in \cite{hotchkiss-pi}, 
where it is observed that in general 
$\ind_{\Hdg}(\alpha) \mid \ind(\alpha)$  
and the integral Hodge conjecture for $(X, \alpha)$ implies equality.
When $X$ is a Calabi--Yau threefold, we show that conversely the only possible obstruction to the integral Hodge conjecture for $(X, \alpha)$ is whether $\ind_{\Hdg}(\alpha) = \ind(\alpha)$. 

\begin{theorem}
\label{theorem-voisin-group-twisted-CY3}
    Let $X$ be a complex Calabi--Yau threefold and let $\alpha \in \Br(X)$. 
    Then 
    \begin{equation*}
            \#\rV(X, \alpha) = \frac{\ind(\alpha)}{\ind_{\Hdg}(\alpha)} \cdot 
    \end{equation*} 
\end{theorem}

The proof builds on the fact that the usual integral Hodge conjecture holds for Calabi--Yau threefolds \cite{voisin-IHC, grabowski-IHC, totaro}. 
Since in general 
$\per(\alpha) \mid \ind_{\Hdg}(\alpha) \mid \ind(\alpha)$ \cite[Lemma~5.8]{hotchkiss-pi}, by combining Theorem~\ref{theorem-main} and Theorem~\ref{theorem-voisin-group-twisted-CY3} we deduce the integral Hodge conjecture for many twisted abelian threefolds. 

\begin{corollary}        
\label{corollary-voisin-group-abelian3fold}
    Let $X$ be a complex abelian threefold and let $\alpha \in \Br(X)$. Then
    \[
            \# \rV(X, \alpha) \mid \per(\alpha).
    \]
    Moreover, if $\ind_{\Hdg}(\alpha) = \per(\alpha)^2$, then the integral Hodge conjecture holds for $(X, \alpha)$. 
\end{corollary}
 
When $(X, \alpha)$ is sufficiently generic, the equality $\ind_{\Hdg}(\alpha) = \per(\alpha)^2$ holds (see Lemma~\ref{lem:computing_the_hodge_theoretic_index} and Example~\ref{ex:gabber}); the only remaining case of the integral Hodge conjecture for twisted abelian threefolds is thus along the special loci where $\ind_{\Hdg}(\alpha) = \per(\alpha)$.  
As an interesting consequence of Corollary~\ref{corollary-voisin-group-abelian3fold}, we give an example of a Severi--Brauer variety $P$ for which the classical integral Hodge conjecture fails, but the integral Hodge conjecture for $\Dperf(P)$ holds (Corollary~\ref{cor:existence}).

\subsection{Further directions} 
Our work suggests several further directions. 

\subsubsection{Derived enhancements} 
Many obstruction theories are known to arise from derived algebraic geometry in the following sense: given a derived enhancement $i \colon \cX \to \fX$ of an algebraic stack $\cX$, the induced morphism between cotangent complexes provides an obstruction theory for $\cX$ (Remark~\ref{remark-enhancement-obs}). 
A natural question is whether the obstruction theory from Theorem~\ref{theorem-intro-DT-defo-invariant} arises from a derived enhancement of $M_\sigma(v)/ \Aut^0(\cC/S)$. 

In a sequel to this paper, we will provide a positive answer to this question. 
In fact, our proof of Theorem~\ref{theorem-intro-DT-defo-invariant} already provides some of the ingredients for the construction of the desired derived enhancement (see Remark~\ref{remark-derived-enhancement-MG}). 

\subsubsection{Dimension $4$} 
\label{section-DT4}
It is natural to wonder whether our approach to the period-index conjecture 
can be extended to higher dimensions. 
The main missing ingredient is a suitable version of DT theory. 
In dimension $4$, a theory of DT invariants for Calabi--Yau varieties 
was recently developed by Borisov and Joyce \cite{BJ} from a differential geometric perspective, 
and subsequently by Oh and Thomas \cite{OT} from an algebraic one. 
More recently, Bae, Kool, and Park \cite{BKP} observed that these DT invariants 
typically vanish for classes not supported on curves, and constructed a reduced 
theory to obtain nontrivial invariants.  

In a sequel to this paper, we will extend our theory of reduced DT invariants 
to the setting of CY4 categories; in fact, some of the necessary ingredients are already contained in this paper. 
In particular, we will construct a generalization of the DT theory from \cite{BKP} 
which 
allows a CY4 category $\cC$ in place of a Calabi--Yau fourfold and incorporates an additional reduction when $\Aut^0(\cC)$ is nontrivial. 
This opens a potential route to proving the unramified period-index conjecture for abelian fourfolds, 
along the lines of our proof of Theorem~\ref{theorem-main}.  

\subsubsection{Other threefolds} 
Theorem~\ref{theorem-main} can be regarded as a proof of concept for the use of DT theory in the study of the period-index problem for threefolds. 
In principle, our method could be applied directly to the period-index conjecture for unramified topologically trivial Brauer classes on a threefold $X$ with the following properties: 
\begin{enumerate}
\item $X$ is Calabi--Yau. 
\item $X$ admits stability conditions. 
\item $X$ admits deformations along which any topologically trivial Brauer class can be killed. 
\item $X$ has computable reduced DT invariants for sufficiently many higher rank classes. 
\end{enumerate} 
We have formulated the last two properties loosely, since the precise form in which they are needed could depend subtly on the geometry of the situation, as indicated in \S\ref{section-intro-proof-sketch}. 

The most restrictive condition above is that $X$ is Calabi--Yau. 
Without it, we do not know how to define DT invariants for higher rank classes, so our approach to the period-index conjecture cannot get off the ground. 

\begin{question}
Does there exist a theory of DT invariants for higher rank classes on smooth projective threefolds which are not necessarily Calabi--Yau? 
\end{question} 

More precisely, one might hope for the existence of a natural virtual fundamental class (not necessarily of degree $0$) on moduli spaces of stable objects on any smooth projective threefold $X$, which could then be used to define enumerative invariants. Another possibility is to pass to the local Calabi--Yau fourfold $Y = \mathrm{Tot}(K_X)$, and then to study DT invariants on $Y$.  

\subsubsection{The integral Hodge conjecture for CY3 categories}
\label{section-IHC-CY3}
If $X$ is a complex Calabi--Yau threefold, then the usual integral Hodge conjecture holds for $X$ \cite{voisin-IHC, grabowski-IHC, totaro}. 
It follows that the integral Hodge conjecture holds for $\Dperf(X)$. 
This is the special case of Theorem~\ref{theorem-voisin-group-twisted-CY3} when $\alpha = 0$.
This suggests: 

\begin{question}
\label{question-IHC-CY3}
Does the integral Hodge conjecture hold for any CY3 category of geometric origin over $\bC$? 
\end{question} 

Categories of the form $\Dperf(X, \alpha)$ for a Calabi--Yau threefold $X$ and Brauer class $\alpha \in \Br(X)$ provide an interesting source of test cases for Question~\ref{question-IHC-CY3}. 
As mentioned above, in this paper we reduce the question to the equality $\ind_{\Hdg}(\alpha) = \ind(\alpha)$, and provide a positive answer for generically twisted abelian threefolds. 

Another interesting source of test cases for Question~\ref{question-IHC-CY3} comes from Kuznetsov components of certain Fano varieties \cite{kuznetsov-CY}. 
For these examples, it seems that one can often provide a positive answer by elementary geometric means. 
For instance, in Lemma~\ref{lemma-cubic-sevenfold} we explain that the integral Hodge conjecture holds for the Kuznetsov component of any cubic sevenfold (as well as for the cubic sevenfold itself). 

\subsection{Organization of the paper}
Part~\ref{part-NAG}, consisting of \S\ref{section-twisted-sheaves}-\S\ref{section-CY-cats}, concerns preliminaries on noncommutative algebraic geometry. 
We develop the theory in the correct generality for the rest of the paper, and prove some auxiliary results for which we could not find a suitable reference, like Lemma~\ref{lemma-ch-dual} describing the dual of a Chern character valued in Hochschild homology, or Lemmas~\ref{lemma-hochschild-homology-twisted-variety} and~\ref{lemma-hochschild-cohomology-twisted-variety} computing the Hochschild (co)homology of a Brauer-twisted scheme. 

Part~\ref{part-moduli}, consisting of~\S\ref{section-moduli}-\S\ref{section-moduli-objects-modulo-autoequivalence}, concerns moduli spaces of objects in smooth proper categories.  
In~\S\ref{section-moduli} we recall general existence results for such moduli spaces and their derived enhancements. 
In~\S\ref{section-autoequivalences} we study the structure of the group of autoequivalences of smooth proper categories. 
In~\S\ref{section-moduli-objects-modulo-autoequivalence} we study quotients of  (open subspaces of) the moduli space of objects by (open subgroups of) the group of autoequivalences; in particular, we describe the cotangent complex of a derived version of this quotient. 

Part~\ref{part-stability}, consisting of~\S\ref{section-stability-twisted-sheaves}-\S\ref{section-construction-stability}, concerns notions of stability. 
Except for notation, it is independent from Part~\ref{part-moduli}. 
In~\S\ref{section-stability-twisted-sheaves} we review the theory of stability for twisted sheaves. 
In~\S\ref{section-stability} we discuss the theory of stability conditions on linear categories; along the way, we introduce the notion of a topological Mukai homomorphism and study the action of the identity component of the group of autoequivalences on stability conditions. 
In \S\ref{section-construction-stability} we extend the previously known constructions of stability conditions to the twisted setting. 

Part~\ref{part-DT-theory}, consisting of \S\ref{section-obstruction-theories-and-vfc}-\S\ref{section-VIHC-CY3}, develops reduced Donaldson--Thomas theory for CY3 categories, combining ingredients from Parts~\ref{part-moduli} and~\ref{part-stability}. 
In~\S\ref{section-obstruction-theories-and-vfc} we cover preliminaries on obstruction theories and virtual fundamental classes. 
In~\S\ref{section-obstruction-theory-CY3} we construct a reduced symmetric perfect obstruction theory on a suitable quotient of the moduli space of objects in a CY3 category. 
We use this in \S\ref{section-DT-CY3} to define reduced DT invariants that are preserved under deformations, completing in particular the proofs of Theorems~\ref{theorem-intro-DT-definition} and \ref{theorem-intro-DT-defo-invariant}. 
In \S\ref{section-VIHC-CY3} we deduce the criterion of Theorem~\ref{theorem-intro-VIHC-sigma} for the validity of the noncommutative variational integral Hodge conjecture. 

Part~\ref{part-PI-abelian-3folds}, consisting of \S\ref{section-main-result-abelian-3folds}-\S\ref{section-proof-main-result-abelian-3folds} and Appendix~\ref{appendix-abelian-3folds}, is devoted to the proof of Theorem~\ref{theorem-main}. 
A reader primarily interested in the period-index conjecture may wish to skip directly to Part~\ref{part-PI-abelian-3folds}, 
referring back to the material from Parts~\ref{part-NAG}-\ref{part-DT-theory} as necessary. 
In \S\ref{section-main-result-abelian-3folds} we explain how Theorem~\ref{theorem-main} reduces to the case of complex abelian threefolds, and outline the proof in this case. 
The rest of Part~\ref{part-PI-abelian-3folds} consists of assembling the ingredients to complete this outline. In \S\ref{sec:hodge_classes_on_twisted_abelian_varieties} we discuss the Hodge structure $\Ktop[0](X, \alpha)$ associated to a twisted abelian variety and construct Hodge classes of the rank predicted by the period-index conjecture, compatibly in families. 
In \S\ref{section-discriminant} we discuss a quartic form, known as Igusa's discriminant, on the rational even cohomology of an abelian threefold, which is invariant under autoequivalences and plays an important technical role in our analysis of DT invariants. 
In \S\ref{section-nonvanishing-curve-invariants} we prove the nonvanishing of many curve class DT invariants on abelian threefolds. 
Finally, in \S\ref{section-proof-main-result-abelian-3folds} we prove Theorem~\ref{theorem-main} along the lines sketched in \S\ref{section-intro-proof-sketch} above; 
this relies on some technical auxiliary results about abelian threefolds, gathered in Appendix~\ref{appendix-abelian-3folds}. 

Part~\ref{part-complements}, consisting of \S\ref{section-applications-IHC}-\S\ref{sec:fourier_mukai_partners}, contains several complements to Theorem~\ref{theorem-main}. 
In \S\ref{section-applications-IHC} we discuss applications to the integral Hodge conjecture, and in particular prove Theorem~\ref{theorem-voisin-group-twisted-CY3} concerning the conjecture for twisted Calabi--Yau threefolds. 
In \S\ref{sec:symbol_length}-\S\ref{sec:fourier_mukai_partners} we explain why naive approaches based on symbol length or twisted Fourier--Mukai partners do not suffice to prove Theorem~\ref{theorem-main}. 

\subsection{Conventions} 
\label{conventions} 
A variety over a field $k$ is an integral scheme which is separated and of finite type over $k$. 
A variety is Calabi--Yau if it is smooth, proper, and $\omega_X \cong \cO_X$. 
We follow \cite{stacks-project} for our conventions on algebraic stacks. 

We suppress the analytification of varieties over $\bC$. For instance, given a variety $X$ over $\bC$ and a sheaf of abelian groups $A$, $\rH^*(X, A)$ refers to the sheaf cohomology of $A$ in the analytic topology. Similarly, if $f:X \to S$ is a morphism, then the higher pushforwards $\rR^* f_* A$ are taken in the analytic topology. 

Our conventions on derived algebraic geometry are laid out in \S\ref{section-DAG}; 
in particular, our fiber products are derived (Convention~\ref{convention-derived-fiber-product}). 
All functors between derived categories are also derived by convention; in particular, for a morphism 
$f \colon X \to Y$ we write simply $f_*$ and $f^*$ for derived pushforward and pullback, and we write $\otimes$ for the derived tensor product of complexes.  
For a morphism $\alpha \colon E \to F$ in a stable $\infty$-category $\cC$, we write $\fib(\alpha)$ and $\cofib(\alpha)$ for its fiber and cofiber; in particular, we apply this convention when $\cC$ is a derived category, in which case $\fib(\alpha)$ and $\cofib(\alpha)$ are classically denoted $\cone(\alpha)[-1]$ and $\cone(\alpha)$.  
Given a morphism $X \to Y$ and a complex $E$ on $Y$, we sometimes denote its pullback to $X$ by $E \otimes \cO_X$, to avoid a proliferation of names of morphisms. 

\subsection{Acknowledgements} 
We thank Nick Addington, Arend Bayer, 
Andrei C\u{a}ld\u{a}raru, Yunfeng Jiang, 
Johan de Jong, Bruno Klingler, Max Lieblich, Jacob Lurie, Emanuele Macr\`{i}, Davesh Maulik, Georg Oberdieck, Michel Van den Bergh, and Vadim Vologodsky for helpful discussions related to this work. 

During the preparation of this paper, the first author was partially supported by NSF grant DMS-2052750, and the second author was partially supported by NSF grants DMS-2112747, DMS-2052750, and DMS-2143271, and a Sloan Research Fellowship. 
Part of this work was also completed while the authors were in residence at the Simons Laufer Mathematical Sciences Institute in Spring 2024 under the support of NSF grant DMS-1928930 and Sloan Foundation grant G-2021-16778. 


\newpage 
\part{Noncommutative algebraic geometry}
\label{part-NAG}  

\section{Twisted sheaves} 
\label{section-twisted-sheaves} 

The derived category of twisted sheaves is a motivating example for noncommutative algebraic geometry, which plays a central role in our results on the period-index conjecture later in the paper.   
In this section, we give a concise account of the theory tailored to our applications. 
This material has been extensively developed by Lieblich \cite{lieblich-moduli-twisted}, to which we refer for further details.  

\subsection{Brauer groups} 
For a scheme $X$, the \emph{Brauer group} 
\begin{equation*} 
\Br(X) \coloneqq \rH^2_{\et}(X, \bG_m)_{\tors}
\end{equation*} 
is the torsion subgroup of the \'{e}tale cohomology group $\rH^2_{\et}(X, \bG_m)$. 
In the literature $\Br(X)$ is sometimes called the cohomological Brauer group, to distinguish it from the \emph{Brauer--Azumaya group} $\Br_{\mathrm{Az}}(X)$ defined as the group of Azumaya algebras on $X$ modulo Morita equivalence. 
However, when $X$ admits an ample line bundle, as will be the case for the examples of interest in this paper, the natural map $\Br_{\textrm{Az}}(X) \to \Br(X)$ is an isomorphism \cite{dJ-gabber}. 

The $n$-torsion in the Brauer group can be described by \'{e}tale cohomology with finite coefficients. Namely, if $n$ is an integer invertible on $X$, then there is an exact sequence
\begin{equation}
\label{Br-kummer-seq}
0 \to \Pic(X)/n \to \rH^2_{\et}(X, \bmu_n) \to \Br(X)[n] \to 0, 
\end{equation}
obtained by taking cohomology of the Kummer sequence. 
Taking limits and comparing to usual cohomology, we obtain the following topological description of the Brauer group for complex varieties. 

\begin{lemma}
\label{lemma-Br-ses}
Let $X$ be a complex variety. Then there is an exact sequence 
\begin{equation*}
0 \to \frac{\rH^2(X, \bZ)}{\Pic(X)} \otimes \bQ/\bZ \to \Br(X) \to \rH^3(X, \bZ)_{\tors} \to 0. 
\end{equation*} 
\end{lemma}

\begin{proof}
Using $\bmu_n \cong \bZ/n$ and the comparison with singular cohomology, by taking the colimit of the sequences~\eqref{Br-kummer-seq} we obtain the exact sequence 
\begin{equation*}
0 \to \Pic(X) \otimes \bQ/\bZ \to \rH^2(X, \bQ/\bZ) \to \Br(X) \to 0. 
\end{equation*} 
On the other hand, the exact sequence $0 \to \bZ \to \bQ \to \bQ/\bZ \to 0$ gives an exact sequence on cohomology 
\begin{equation*}
0 \to \rH^2(X, \bZ) \otimes \bQ/\bZ \to \rH^2(X, \bQ/\bZ) \to \rH^3(X, \bZ)_{\tors} \to 0. 
\end{equation*} 
The two sequences fit together in a commutative diagram 
\begin{equation*}
\xymatrix{
0 \ar[r]  & \Pic(X) \otimes \bQ/\bZ \ar[r] \ar[d]_{c_1} & \rH^2(X, \bQ/\bZ) \ar[r] \ar@{=}[d] & \Br(X) \ar[r] \ar[d] & 0 \\ 
0 \ar[r] & \rH^2(X, \bZ) \otimes \bQ/\bZ \ar[r] & \rH^2(X, \bQ/\bZ) \ar[r] & \rH^3(X, \bZ)_{\tors} \ar[r] & 0. 
}
\end{equation*} 
which gives the result. 
\end{proof} 

\begin{definition}
\label{def:topologically_trivial}
    Let $X$ be a complex variety, and let $\alpha \in \Br(X)$. Then $\alpha$ is \emph{topologically trivial} if it lies in the kernel of the morphism to $\rH^3(X, \bZ)_{\tors}$.
\end{definition}

\begin{remark}
When $X$ is a scheme and $\ell$ is a prime invertible on $X$, then there is a similar sequence describing the $\ell$-power torsion in $\Br(X)$: 
\begin{equation*}
0 \to \frac{\rH^2_{\et}(X, \bZ_{\ell}(1))}{\Pic(X) \otimes \bZ_{\ell}} \otimes \bQ_{\ell}/\bZ_{\ell} 
\to \Br(X)[\ell^{\infty}] \to \rH^3_{\et}(X, \bZ_{\ell}(1))_{\tors} \to 0. 
\end{equation*} 
\end{remark}

The following useful observation says that Brauer classes can be killed after passage to a finite cover. 

\begin{lemma}
\label{lemma-finite-cover-kill}
Let $X$ be a smooth projective variety and $\alpha \in \BrAz(X)$. 
Then there exists a smooth projective variety $Y$ and a surjective finite flat morphism $f \colon Y \to X$ such that $f^*(\alpha) = 0$. 
\end{lemma} 

\begin{proof}
Let $p \colon P \to X$ be a Severi--Brauer variety of class $\alpha$. 
Then $p^*(\alpha) = 0$, so it suffices to construct a smooth subvariety $Y \subset P$ such that the map $Y \to X$ is surjective and finite. 
This can be done by taking $Y$ to be a sufficiently generic complete intersection of very ample divisors in $P$ (see \cite[Lemma 3.2.2.1]{lieblich-moduli-twisted}). 
\end{proof} 

\subsection{Categories of twisted sheaves} 
\label{section-categories-of-twisted-sheaves}
Given a scheme $X$ and a class $\alpha \in \rH^2_{\et}(X, \bG_m)$, there are several models for the category of $\alpha$-twisted sheaves, each depending on an auxiliary choice: 
\begin{enumerate}
\item One can choose a $\bG_m$-gerbe $\cX \to X$ of class $\alpha$, 
and consider sheaves on $\cX$ for which the inertial action is given by the standard character \cite{lieblich-moduli-twisted}. 
When $\alpha$ is $n$-torsion, one can instead choose a $\bmu_n$-gerbe $\cX \to X$ whose class maps to $\alpha$ under the map $\rH^2_{\et}(X, \bmu_n) \to \rH^2_{\et}(X, \bG_m)$, and consider sheaves on $\cX$ for which the inertial action is given by the standard character. 
\item One can choose a (hyper)covering and a cocycle $a$ representing $\alpha$, and consider 
sheaves on the cover satisfying an $a$-twisted cocycle condition \cite{caldararu}. 
\item When $\alpha$ is represented by an Azumaya algebra $\cA$ (which is automatic, for instance, if $X$ is a smooth quasi-projective variety), one can consider sheaves of $\cA$-modules. 
\end{enumerate} 
All of these models result in equivalent theories of $\alpha$-twisted sheaves, which are independent of the choices involved (see \cite{lieblich-moduli-twisted}). 
For concreteness, in this paper we adopt the first approach. 

\begin{definition}
Given a $\bG_m$-gerbe or $\bmu_n$-gerbe $\pi \colon \cX \to X$ over a scheme and an integer $k \in \bZ$,  
we define $\Dperf^k(\cX)$ as the full subcategory of $\Dperf(\cX)$ spanned by complexes whose cohomology sheaves have inertial action given by $\chi^k$, where $\chi$ is the standard character of $\bG_m$ or $\bmu_n$. 
We refer to an object of $\Dperf^k(\cX)$ as a \emph{$k$-twisted} perfect complex, or simply a \emph{twisted} perfect complex when $k = 1$. 
We similarly define $\Dqc^k(\cX)$ and, when $X$ is locally noetherian, 
$\Coh^k(\cX)$ and $\mathrm{D}^{\mathrm{b},k}(\cX)$. 

If $\alpha \in \Br(X)$, then we may choose a $\bG_m$-gerbe or $\bmu_n$-gerbe $\pi \colon \cX \to X$ of class $\alpha$, 
and consider the twisted categories defined above. 
By abuse of notation, we shall often suppress the choice of $\cX$ in our notation and write 
simply $\Dperf(X, \alpha)$ for $\Dperf^1(\cX)$, and use similar abbreviations for 
$\Dqc^k(\cX)$, $\Coh^k(\cX)$, and $\mathrm{D}^{\mathrm{b},k}(\cX)$. 
\end{definition} 

\begin{remark}
We will typically use the model for $\alpha$-twisted sheaves in terms of a $\bmu_n$-gerbe $\pi \colon \cX \to X$, instead of a $\bG_m$-gerbe. 
The advantage is that $\cX$ is then a Deligne--Mumford stack. 
Note, however, that while the representation of $\alpha$ as a $\bG_m$-gerbe is unique, 
the representation as a $\bmu_n$-gerbe is not. 
The nonuniqueness is measured by those $\bmu_n$-gerbes $\pi \colon \cX \to X$ with 
class $[\cX] \in \ker( \rH^2_{\et}(X, \bmu_n) \to \rH^2_{\et}(X, \bG_m) )$. 
Such a $\bmu_n$-gerbe is called \emph{essentially trivial}, 
and can be geometrically characterized as follows \cite[Lemma 2.3.4.2]{lieblich-moduli-twisted}. 
\end{remark} 

\begin{lemma}
\label{lemma-essentially-trivial}
A $\bmu_n$-gerbe $\cX \to X$ over a scheme $X$ is essentially trivial if and only if there exists an invertible twisted sheaf on $\cX$. 
\end{lemma} 

Twisted sheaves satisfy the expected functoriality. 
Concretely, let $f \colon X \to Y$ be a morphism of schemes, 
let $\alpha \in \Br(X)$ be a Brauer class represented by a $\bmu_n$-gerbe $\pi \colon \cX \to X$, and form the pullback diagram 
\begin{equation*}
\xymatrix{
\cY \ar[r]^{f'} \ar[d]_{\pi'} & \cX \ar[d]^{\pi} \\ 
Y \ar[r]^{f} & X
}
\end{equation*} 
The class $\beta = f^*(\alpha) \in \Br(Y)$ is represented by the $\bmu_n$-gerbe $\pi' \colon \cY \to Y$. 
The pullback and pushforward functors
\begin{equation*}
f'^* \colon \Dqc(\cX) \to \Dqc(\cY) \quad \text{and} \quad 
f'_* \colon \Dqc(\cY) \to \Dqc(\cX) 
\end{equation*} 
respect $k$-twisted complexes, and hence induce functors between 
$\Dqc(X, \alpha)$ and $\Dqc(Y, \beta)$. 
The functor $f'^*$ restricts to a functor on categories of perfect complexes, 
while $f'_*$ restricts to a functor on bounded derived categories of coherent sheaves when $f$ is proper. 
By a slight abuse of notation, we often simply write $f^*$ or $f_*$ instead of $f'^*$ or $f'_*$ for the pullback or pushforward functors between twisted derived categories. 

\begin{remark}
\label{remark:pullback_killing_brauer_class}
    We shall frequently encounter the following situation: Let $f\colon Y \to X$ be a morphism of schemes, 
    and let $\alpha \in \Br(X)$ be a Brauer class such that $f^* \alpha = 0$. Choose an $f^*\alpha$-twisted line bundle $L$ (on a representative gerbe $\cY$ as above), and consider the functors
    \begin{align*}
        f^*_L &\colon \Dqc(X, \alpha) \to \Dqc(Y), \quad E \mapsto f^*E \otimes L^\vee \\
        f_*^L &\colon \Dqc(Y) \to \Dqc(X, \alpha), \quad F \mapsto f_* (L \otimes F).
    \end{align*}
    Clearly, these functors depend on the choice of $L$, but any two $f^*\alpha$-twisted line bundles differ by an element of $\Pic(Y)$, so the dependence is only up to the usual action of $\Pic(Y)$ on $\Dqc(Y)$. 

    In the literature, these functors are often denoted simply by $f^*$ and $f_*$, eliding the dependence on $L$.
\end{remark}

\subsection{K-theory}
Given a Brauer class $\alpha \in \Br(X)$ on a scheme, we use the notation 
\begin{equation*}
\rK_0(X, \alpha) \coloneqq \rK_0(\Dperf(X,\alpha)) \qquad \text{and} \qquad 
\rG_0(X, \alpha) \coloneqq \rK_0(\Db(X, \alpha)) 
\end{equation*} 
for the Grothendieck groups of the categories of $\alpha$-twisted perfect and bounded coherent complexes. 
When $X$ is connected, taking the rank of a coherent sheaf on our chosen gerbe $\cX \to X$ induces a rank homomorphism 
\begin{equation*}
\rk \colon \rG_0(X, \alpha) \to \bZ, 
\end{equation*} 
in terms of which we can give a global description of the index of $\alpha$. 
Recall that the index of $\alpha$ is defined by restriction to the generic point, i.e. $\ind(\alpha) \coloneqq \ind(\alpha_{k(X)})$ where $k(X)$ denotes the function field of $X$. 

\begin{lemma}
\label{lemma-ind-as-rk}
Let $X$ be an integral noetherian scheme and let $\alpha \in \Br(X)$. 
Then $\ind(\alpha)$ equals the positive generator of the image of the 
rank homomorphism $\rk \colon \rG_0(X, \alpha) \to \bZ$. 
\end{lemma} 

\begin{proof}
Equivalently, we must show that 
\begin{equation*}
\ind(\alpha) = \gcd \set{ \rk(F) \st F \text{ is a coherent $\alpha$-twisted sheaf of positive rank} }. 
\end{equation*} 
When $X = \Spec(K)$ is the spectrum of a field, this holds by \cite[Proposition 3.1.2.1]{lieblich-period-index}. 
The general case follows since any $\alpha_{k(X)}$-twisted coherent sheaf over the generic point can be lifted to a coherent $\alpha$-twisted sheaf by \cite[Lemma 3.1.3.1]{lieblich-period-index}. 
\end{proof}

When $X$ is a proper variety over a field $k$, 
for $E \in \Dperf(X, \alpha)$ and $F \in \Db(X, \alpha)$ the Euler characteristic 
\begin{equation*}
\chi(E,F) = \sum_i (-1)^i \dim_k \Hom(E, F[i]) \in \bZ
\end{equation*} 
induces a pairing 
\begin{equation*}
\chi \colon \rK_0(X, \alpha) \times \rG_0(X, \alpha) \to \bZ. 
\end{equation*} 
Using this, we can define numerical versions of the above Grothendieck groups.  

\begin{definition}
If $X$ is a proper variety and $\alpha \in \Br(X)$, 
then $\rK_{\num}(X, \alpha)$ is the quotient of $\rK_0(X, \alpha)$ by the kernel of $\chi$ on the left, and $\rG_{\num}(X, \alpha)$ is the quotient of $\rG_0(X, \alpha)$ by the kernel of $\chi$ on the right. 
\end{definition} 

\begin{remark} 
Let us record two simple properties. 
\begin{enumerate}
\item When $X$ is Gorenstein, Serre duality shows that the 
natural map $\rK_0(X, \alpha) \to \rG_0(X, \alpha)$ descends to a map 
\begin{equation*}
\rK_{\num}(X, \alpha) \to \rG_{\num}(X, \alpha), 
\end{equation*} 
which is an isomorphism if $X$ is smooth. 

\item $\Knum$ is functorial with respect to pullback and 
$\rG_{\num}$ is functorial with respect to pushforward. 
More precisely, if $f \colon Y \to X$ is a morphism between proper varieties and $\beta = f^*(\alpha)$, then pullback and pushforward induce homomorphisms 
\begin{align*}
f^* & \colon  \Knum(X, \alpha) \to \Knum(Y, \beta)  \\ 
f_* & \colon \rG_{\num}(Y, \beta) \to \rG_{\num}(X, \alpha). 
\end{align*} 
Indeed, this follows easily from the adjunction between pullback and pushforward. 
\end{enumerate}
\end{remark}


\section{Derived algebraic geometry}  
\label{section-DAG}

We will occasionally need the language of derived algebraic geometry, as developed by Lurie \cite{DAG, HA, SAG} and
To\"{e}n--Vezzosi \cite{HAG1, HAG2}; for surveys, see for instance \cite[Part I]{DAG-gaitsgory-1} and \cite{toen-higher-derived, toen-simplicial, toen-DAG}. 
Here we briefly lay out our conventions, and refer to the preceding references for more details. 

\subsection{Derived and higher stacks} 
Let $\dRing$ denote the $\infty$-category of animated rings, defined as the animation (in the sense of \cite[\S5.1]{KS-purity}) of the category of commutative rings, or equivalently as the $\infty$-category obtained from the category of simplicial commutative rings by inverting weak equivalences. 
Let $\dAff$ denote the $\infty$-category of derived affine schemes, defined as the opposite category of $\dRing$. 
Let $\Grpd_{\infty}$ denote the $\infty$-category of $\infty$-groupoids.
A \emph{derived stack} is a functor $X \colon \dAff^{\op} \to \Grpd_{\infty}$ which satisfies descent with respect to the \'{e}tale topology. 
A \emph{derived algebraic stack} is a derived stack $X$ for which there exists (in an appropriate sense) a smooth surjection from a disjoint union of derived affine schemes.\footnote{The complete definition is of a slightly complicated inductive nature; see e.g. \cite[\S5]{toen-simplicial} or \cite[Part I, Chapter 2, \S4]{DAG-gaitsgory-1}. We warn the reader that the terminology is not consistent in the literature; for instance, what we call a derived algebraic stack corresponds to a ``$D^{-}$-stack which is $n$-geometric for some $n$'' in \cite{toen-moduli, HAG2}, and to a ``Artin stack'' in \cite[Part I]{DAG-gaitsgory-1}.} 
Slightly more general still is the notion of a \emph{derived locally algebraic stack} $\fX$, which is a derived stack that can be written as a filtered colimit of derived algebraic stacks $\fX_i$ such that each $\fX_i \to \fX$ is a monomorphism.\footnote{A derived locally algebraic stack is a ``$D^-$-stack which is locally geometric'' in the terminology of \cite{toen-moduli}.} 
A derived locally algebraic stack is called \emph{locally of finite presentation }if each $\fX_i$ can be chosen so; in this case, the $\fX_i \to \fX$ are in fact Zariski open immersions (see the discussion following \cite[Definition 2.17]{toen-moduli}), 
so $\fX$ is a union of open derived algebraic stacks which are locally of finite presentation. 

To relate derived stacks to classical ones, it is useful to consider the intermediate notion of a \emph{higher stack}, which is a functor $X \colon \Aff^{\op} \to \Grpd_{\infty}$ which satisfies descent with respect to the \'{e}tale topology, where $\Aff$ is the category of classical affine schemes. Note that a stack in the classical sense is simply a higher stack which is 1-truncated, i.e. takes values in the subcategory $\Grpd \subset \Grpd_{\infty}$ of $1$-groupoids. 
Similarly to the case of derived algebraic stacks, one can define the notion of higher algebraic stacks and higher locally algebraic stacks; an algebraic stack in the classical sense can then be regarded as a $1$-truncated higher algebraic stack. 

Let $\dSt$ and $\St$ denote the $\infty$-categories of derived stacks and higher stacks. 
Then restriction along the inclusion $\Aff \hookrightarrow \dAff$ induces the functor 
\begin{equation*}
(-)_{\cl} \colon \dSt \to \St
\end{equation*} 
called \emph{classical truncation}. 
For example, for $A \in \dRing$ we have $\Spec(A)_{\cl} = \Spec(\pi_0(A))$. 
The functor $(-)_{\cl}$ admits a fully faithful left adjoint 
\begin{equation*} 
\iota \colon \St \to \dSt 
\end{equation*}
called \emph{derived extension}. 
The classical truncation and derived extension functors preserve the property of being (locally) algebraic. 
By abuse of notation, we often simply write $X$ instead of $\iota(X)$ when we think of a higher stack as a derived stack; in particular, any classical scheme or algebraic stack may be regarded as a derived algebraic stack. 

\begin{remark}[Derived (co)limits of stacks] 
\label{remark-derived-colimits} 
The categories $\St$ and $\dSt$ admit all colimits and limits. 
Classical truncation $(-)_{\cl} \colon \dSt \to \St$ commutes with colimits and limits \cite[Lemma 2.2.4.2]{HAG2}. 
Derived extension commutes with colimits (being a left adjoint), and with pullbacks along flat morphisms of higher algebraic stacks \cite[Proposition 2.2.4.4(3)]{HAG2}, but in general it does not commute with limits. 
\end{remark} 

\begin{convention}[Derived fiber products] 
\label{convention-derived-fiber-product}
Given a diagram $X \to S \leftarrow Y$ of schemes, algebraic stacks, or in general higher stacks, 
we write $X \times_S Y$ for the fiber product taken in the category $\dSt$ of derived stacks. 
In particular, $X \times_S Y$ may no longer be an object of $\St$, even when $X$ and $Y$ are schemes. 
The truncation $(X \times_SY)_{\cl}$ recovers the classical fiber product in view of Remark~\ref{remark-derived-colimits}. 
\end{convention}

\begin{remark}[Derived enhancements]
\label{remark-derived-enhancements} 
Given $X \in \St$, a \emph{derived enhancement} is a derived stack $\fX \in \dSt$ equipped with an equivalence $\fX_{\cl} \simeq X$. 
In this case, by adjunction there is a canonical morphism $X \to \fX$.  
If $\fX$ is a derived algebraic stack, then the morphism $X \to \fX$ is in fact a closed immersion that induces an equivalence between the Zariski sites of $X$ and $\fX$ (see \cite[Proposition 2.2.4.7 and Lemma 2.2.2.10]{HAG2}). 
In this sense, derived algebraic geometry can be thought of as a framework for formally thickening classical algebro-geometric objects in a derived direction. 
\end{remark} 

\subsection{Derived categories} 
For an affine derived scheme $X = \Spec(A)$ corresponding to an animated ring $A \in \dRing$, 
the derived category of quasi-coherent sheaves is defined to be derived category of $A$-modules $\Dqc(X) = \rD(A)$, and the category of perfect complexes is the full subcategory $\Dperf(X) \subset \Dqc(X)$ generated by $A$ under finite colimits and retracts. 
For a general derived stack $X$, the derived category of quasi-coherent sheaves is defined by right Kan extension from the case of derived affine schemes, namely 
\begin{equation*}
\Dqc(X) = \lim_{U \in  \dAff/X} \Dqc(U)
\end{equation*} 
where the limit (taken in the $\infty$-category of stable $\infty$-categories) is over the category of derived affine schemes over $X$. 
The category of perfect complexes is the full subcategory 
\begin{equation*}
\Dperf(X) \subset \Dqc(X) 
\end{equation*} 
consisting of objects whose restriction along any morphism $U \to X$ from a derived affine scheme is perfect; in other words, 
\begin{equation*}
\Dperf(X) = \lim_{U \in \dAff/X} \Dperf(U) . 
\end{equation*} 

The following definition, introduced in \cite{bzfn}, isolates a class of derived stacks with well-behaved derived categories, which contains 
most examples of interest. 

\begin{definition}
\label{definition-perfect-stack} 
A derived algebraic stack $X$ is \emph{perfect} if the following conditions hold: 
\begin{enumerate}
\item The diagonal of $X$ is affine. 
\item \label{das-2} $\Dqc(X)$ is compactly generated. 
\item \label{das-3} The compact and perfect objects of $\Dqc(X)$ coincide. 
\end{enumerate} 
\end{definition} 

The results of \cite{bzfn} show that the category of perfect stacks is stable under fiber products, and that derived pullback and pushforward along morphisms between perfect stacks satisfy the base change and projection formulas. 
To state a useful criterion for perfectness, we use the terminology that a group scheme is \emph{nice} if it is an extension of a finite \'{e}tale tame group scheme by a group scheme of multiplicative type. 

\begin{theorem}[{\cite[Theorem A.3.2]{milnor-squares}}]
A quasi-compact derived algebraic stack with affine diagonal and nice stabilizers is perfect. 
\end{theorem} 

\begin{example} 
\label{example-gerbes-perfect} 
For instance, if $X$ is a quasi-compact scheme with affine diagonal and $\cX \to X$ is a $\bG_m$-gerbe or a $\bmu_n$-gerbe with $n$ invertible on $X$, then $\cX$ is perfect. 
\end{example}


\section{Linear categories}
\label{section-linear-cats} 

This section concerns a formalism for ``noncommutative algebraic geometry'' in terms of linear categories. 
We largely follow \cite{NCHPD}, which is based on Lurie's work \cite{HA}, but we include various auxiliary 
results required later in the paper. 

Fix a perfect derived algebraic stack\footnote{In \cite{NCHPD} the base $S$ is assumed to be a quasi-compact and separated scheme, but everything goes through verbatim for perfect derived algebraic stacks.} $S$, which will serve as the base space for our discussion.
For most purposes in this paper, the case where $S$ is a complex variety suffices. 

\subsection{Small linear categories} 
\label{section-small-linear-cat}
The category $\Dperf(S)$, equipped with the operation of tensor product, may be regarded as a commutative algebra object in the category $\Cat_{\mathrm{st}}$ of small idempotent-complete stable $\infty$-categories. 
An \emph{$S$-linear category} is a $\Dperf(S)$-module object of $\Cat_{\mathrm{st}}$. 
In particular, an $S$-linear category $\cC$ is equipped with an action functor $\cC \times \Dperf(S) \to \cC$ whose action on objects is denoted by $(E, F) \mapsto E \otimes F$.  

The the collection of all $S$-linear categories is organized into an $\infty$-category 
\begin{equation*}
\Cat_S = \Mod_{\Dperf(S)}(\Cat_{\mathrm{st}}),
\end{equation*} 
and admits a symmetric monoidal structure with unit $\Dperf(S)$ and tensor product denoted
\begin{equation*}
\cC \otimes_{\Dperf(S)} \cD. 
\end{equation*}  
A morphism $\cC \to \cD$ in $\Cat_S$ is called an \emph{$S$-linear functor}; the collection of all such form the objects of an $S$-linear category $\Fun_S(\cC,\cD)$, which is the internal mapping object in $\Cat_S$.

For any morphism $T \to S$ of perfect derived algebraic stacks, the tensor product 
\begin{equation}
\label{equation-CT}
\cC_T = \cC \otimes_{\Dperf(S)} \Dperf(T)
\end{equation}
is naturally a $T$-linear category, called the \emph{base change} of $\cC$ along $T \to S$; similarly, 
for any $S$-linear functor $\Phi \colon \cC \to \cD$, by base change we obtain a $T$-linear functor $\Phi_T \colon \cC_T \to \cD_T$. 
There is a natural functor $\cC \to \cC_T$, whose action on objects we often denote by $E \mapsto E_T$. 
For any point $s \in S$, the base change along $\Spec(\kappa(s)) \to S$ gives a $\kappa(s)$-linear category $\cC_s$ called the 
\emph{fiber} of $\cC$ over~$s$. 
In this way, an $S$-linear category can be thought of as a family of categories parameterized by $S$, 
or as a ``noncommutative scheme'' over $S$. 

\begin{example}
\label{example-geometric-lincat}
Let $f \colon X \to S$ be a morphism of perfect derived algebraic stacks. 
Then $\Dperf(X)$ is naturally an $S$-linear category, with the action of $F \in \Dperf(S)$ 
given by $- \otimes f^*F \colon \Dperf(X) \to \Dperf(X)$. 
For any morphism $T \to S$ from a perfect derived algebraic stack $T$, by \cite{bzfn} there is an equivalence 
\begin{equation*}
\Dperf(X)_T \simeq \Dperf(X_T) 
\end{equation*} 
of $T$-linear categories.  
\end{example} 

An $S$-linear category $\cC$ is enriched over $\Dqc(S)$: for objects $E, F \in \cC$, there is a \emph{mapping object} 
\begin{equation*}
\cHom_S(E, F) \in \Dqc(S)
\end{equation*} 
characterized by equivalences 
\begin{equation}
\label{equation-cHom} 
\Map_{\Dqc(S)}(G, \cHom_S(E, F)) \simeq \Map_{\cC}(E \otimes G, F) 
\end{equation} 
for $G \in \Dperf(S)$, where $\Map(-,-)$ denotes the space of maps in an $\infty$-category. 

\begin{example}
In the situation of Example~\ref{example-geometric-lincat}, for $E, F \in \Dperf(X)$ we have 
\begin{equation*}
\cHom_S(E, F) \simeq f_* \cHom_X(E, F) , 
\end{equation*} 
where $\cHom_X(E, F) \in \Dqc(X)$ is the derived sheaf Hom on $X$. 
\end{example} 

\subsection{Semiorthogonal decompositions} 
One of the most fundamental examples of an $S$-linear category is an $S$-linear semiorthogonal 
component $\cC \subset \Dperf(X)$ where $X \to S$ is a morphism of perfect derived algebraic stacks. 
Recall that for $\cC \in \Cat_S$, a semiorthogonal decomposition 
\begin{equation*}
\cC = \langle \cC_1, \dots, \cC_m \rangle
\end{equation*} 
is called \emph{$S$-linear} if each component $\cC_i$ is preserved by the $\Dperf(S)$-action on $\cC$, in which case each $\cC_i$ inherits the structure of an $S$-linear category. 
Base change of linear categories is compatible with semiorthogonal decompositions, i.e. given $T \to S$ there is an induced $T$-linear semiorthogonal decomoposition 
$\cC_T = \langle (\cC_1)_T, \dots, (\cC_m)_T \rangle$. 

\begin{lemma}[{\cite{bernardara-BS, bergh-BS}}]
\label{lemma-D-SB}
Let $X$ be a scheme, let $\alpha \in \Br_{\mathrm{Az}}(X)$, and let $\pi \colon P \to X$ be a Severi--Brauer scheme for $\alpha$ of relative dimension $n$ over $X$. 
Then there is an $X$-linear semiorthogonal decomposition 
\begin{equation*}
\Dperf(P) = \langle \cD_0, \dots, \cD_{n-1} \rangle 
\end{equation*} 
where for each $i$ there is an $X$-linear equivalence $\cD_i \simeq \Dperf(X, \alpha^i)$. 
\end{lemma} 

\begin{lemma}[{\cite[Theorem 5.4]{bergh-BS}}]
\label{lemma-mun-gerbe}
Let $X$ be a scheme, let $\alpha \in \Br(X)[n]$ with $n$ invertible on $X$, and let $\pi \colon \cX \to X$ a $\bmu_n$-gerbe of class $\alpha$. 
Then there is an $X$-linear orthogonal decomposition 
\begin{equation*}
\Dperf(\cX) = \llangle \Dperf^k(\cX) \rrangle_{k \in \bZ/n}. 
\end{equation*} 
\end{lemma} 

Using semiorthogonal decompositions, we can prove the compatibility of tensor products of categories of twisted sheaves with geometric fiber products, by reducing to the case of Example~\ref{example-geometric-lincat}. 
\begin{lemma}
\label{lemma-tensor-twisted-schemes} 
Let $X \to S \leftarrow Y$ be flat morphisms of perfect schemes, 
and let $\alpha \in \Br(X)[m]$ and $\beta \in \Br(Y)[n]$ where $m$ and $n$ are invertible on $S$. 
Then there is an $X \times_S Y$-linear equivalence 
\begin{equation*}
\Dperf(X, \alpha) \otimes_{\Dperf(S)} \Dperf(Y, \beta) \to \Dperf(X \times_S Y, \pr_X^*(\alpha) + \pr_Y^*(\beta)). 
\end{equation*} 
\end{lemma}

\begin{proof}
There is a natural $X \times_S Y$-linear functor 
\begin{equation}
\label{equivalence-twisted-tensor}
\Dperf(X, \alpha) \otimes_{\Dperf(S)} \Dperf(Y, \beta) \to \Dperf(X \times_S Y, \pr_X^*(\alpha) + \pr_Y^*(\beta))
\end{equation} 
induced by $(E, F) \mapsto \pr_X^*(E) \otimes \pr_Y^*(F)$ for $E \in \Dperf(X, \alpha)$ and $F \in \Dperf(Y, \beta)$, 
which we claim is an equivalence. 
Choose $\cX \to X$ a $\bmu_m$-gerbe of class $\alpha$ and $\cY \to Y$ a $\bmu_n$-gerbe of class~$\beta$. 
By Lemma~\ref{lemma-mun-gerbe}, there are orthogonal $S$-linear decompositions 
\begin{equation*}
\Dperf(\cX) = \llangle \Dperf^k(\cX) \rrangle_{k \in \bZ/m} \quad \text{and} \quad 
\Dperf(\cY) = \llangle \Dperf^{\ell}(\cY) \rrangle_{\ell \in \bZ/n}.
\end{equation*} 
Tensoring over $\Dperf(S)$ we obtain an orthogonal decomposition 
\begin{equation*}
\Dperf(\cX \times_S \cY) \simeq \Dperf(\cX) \otimes_{\Dperf(S)} \Dperf(\cY) = \llangle \Dperf^k(\cX) \otimes_{\Dperf(S)} \Dperf^\ell(\cY) \rrangle_{(k,\ell) \in \bZ/m \times \bZ/n}, 
\end{equation*}  
where the first equivalence holds by Example~\ref{example-geometric-lincat}, since 
$\cX$ and $\cY$ are perfect stacks by Example~\ref{example-gerbes-perfect} and our assumption on $m$ and $n$.  
Explicitly, the resulting fully faithful functor  
\begin{equation} 
\label{Phikell}
\Dperf^k(\cX) \otimes_{\Dperf(S)} \Dperf^\ell(\cY) \to \Dperf(\cX \times_S \cY)
\end{equation} 
is induced by 
$(E,F) \mapsto \pr_X^*(E) \otimes \pr_Y^*(F)$ for $E \in \Dperf^{k}(\cX)$ and $F \in \Dperf^{\ell}(\cY)$. 
On the other hand, since $\cX \times_S \cY \to X \times_S Y$ is a $\bmu_{m} \times \bmu_n$-gerbe, 
similar to Lemma~\ref{lemma-mun-gerbe} we also have an orthogonal decomposition 
\begin{equation*}
\Dperf(\cX \times_S \cY) = \llangle 
\Dperf^{(k,\ell)}(\cX \times_S \cY) 
\rrangle_{(k,\ell) \in \bZ/m \times \bZ/n}, 
\end{equation*} 
where $\Dperf^{(k,\ell)}(\cX \times_S \cY)$ denotes the subcategory on which the inertial $\bmu_{m} \times \bmu_{n}$ acts with weight $(k,\ell)$. 
The functor~\eqref{Phikell} has image contained in $\Dperf^{(k,\ell)}(\cX \times_S \cY)$, and hence induces a fully faithful functor 
\begin{equation*} 
\Phi_{k,\ell} \colon \Dperf^k(\cX) \otimes_{\Dperf(S)} \Dperf^\ell(\cY) \to \Dperf^{(k,\ell)}(\cX \times_S \cY). 
\end{equation*} 
Since the image and target of $\Phi_{k,\ell}$ both give orthogonal decompositions of $\Dperf(\cX \times_S \cY)$ as we vary over 
$(k,\ell) \in \bZ/m \times \bZ/n$, the functor $\Phi_{k,\ell}$ must be essentially surjective, and hence an equivalence. 
Finally, we note that there is an equivalence 
\begin{equation*}
\Dperf^{(1,1)}(\cX \times_S \cY) \simeq \Dperf(X \times_S Y, \pr_X^*(\alpha) + \pr_Y^*(\beta))
\end{equation*} 
(see \cite[Proposition 2.1.2.6]{lieblich-moduli}), under which the functor $\Phi_{1,1}$ is identified with~\eqref{equivalence-twisted-tensor}. 
\end{proof} 

\begin{remark}
We have taken unnecessarily strong hypotheses in Lemma~\ref{lemma-tensor-twisted-schemes}, for ease of reference to the literature. 
For instance, the assumption that $X$ and $Y$ are flat over $S$ is only needed to ensure that $X \times_S Y$ is a classical scheme. 
\end{remark} 

\subsection{Smooth, proper, and dualizable categories}

\begin{definition}
\label{definition-smooth-proper} 
Let $\cC$ be an $S$-linear category. 
\begin{enumerate}
\item $\cC$ is \emph{proper (over $S$)} if $\cHom_S(E,F) \in \Dperf(S)$ for all $E, F \in \cC$, and 
\item $\cC$ \emph{smooth (over $S$)} if $\id_{\Ind(\cC)} \in \Fun_{S}(\Ind(\cC), \Ind(\cC))$ is a compact object. 
\end{enumerate} 
\end{definition} 

Here, $\Ind(\cC)$ denotes the presentable $S$-linear category obtained by ind-completion, and $\Fun_S(\Ind(\cC), \Ind(\cC))$ the category of presentable $S$-linear functors (see \S\ref{section-presentable-lin-cat}). 

The above notions are closely related to their usual geometric counterparts. 
For example, if $X \to S$ is a smooth proper morphism of schemes, then $\Dperf(X)$ is smooth and proper over $S$  \cite[Lemma 4.9]{NCHPD}. 
Furthermore, semiorthogonal components of a smooth proper $S$-linear category are smooth and proper \cite[Lemma~4.15]{NCHPD}. 
In particular, this yields the following rich source of smooth proper categories. 

\begin{lemma}
Let $X \to S$ be a smooth proper morphism. 
If $\cC$ is an $S$-linear semiorthogonal component of $\Dperf(X)$, 
then $\cC$ is smooth and proper over $S$. 
\end{lemma} 

\begin{definition}
\label{definition-smooth-proper-geometric}
A smooth proper $S$-linear category $\cC$ 
is \emph{of geometric origin (over $S$)} if there exists a smooth proper morphism $X \to S$ and a fully faithful functor \mbox{$\cC \hookrightarrow \Dperf(X)$} realizing $\cC$ as an $S$-linear semiorthogonal component. 
\end{definition} 

\begin{remark} 
In this paper, the relevance of being of geometric origin is that it ensures $\cC$ has well-behaved Hodge theory when $S$ is a complex variety, in the form of Theorem~\ref{theorem-Ktop} below. 
Nonetheless, we expect that Theorem~\ref{theorem-Ktop} remains true for any smooth proper $S$-linear category. 
In fact, it is an open question of Orlov \cite[Question 4.4]{orlov-gluing} whether there exists any smooth proper category that is not of geometric origin. 
\end{remark} 

\begin{example}
\label{example-twisted-smooth-proper}
If $X \to S$ is a smooth proper morphism of schemes and $\alpha \in \Br_{\mathrm{Az}}(X)$, then by Lemma~\ref{lemma-D-SB} we see that $\Dperf(X, \alpha)$ is smooth and proper of geometric origin over $S$. 
\end{example} 

As in the geometric case, smoothness and properness are stable under base change. 
\begin{lemma}[{\cite[Lemma 4.10]{NCHPD}}] 
\label{lemma-smooth-proper-bc}
Let $\cC$ be an $S$-linear category. 
Let $T \to S$ be a morphism of perfect derived algebraic stacks. 
\begin{enumerate}
\item If $\cC$ is smooth over $S$, then $\cC_T$ is smooth over $T$. 
\item If $\cC$ is proper over $S$, then $\cC_T$ is proper over $T$. 
\end{enumerate}
\end{lemma} 

\begin{remark}
The property of being of geometric origin is also stable under base change, 
as a consequence of base change for semiorthogonal decompositions. 
\end{remark} 

The condition that an $S$-linear category is smooth and proper can be characterized purely categorically in terms of the monoidal structure on $\Cat_S$, via the following general notion. 

\begin{definition} 
\label{definition-dualizable} 
Let $(\ccA, \otimes, \one)$ be a symmetric monoidal $\infty$-category. 
An object $A \in \ccA$ is called \emph{dualizable} if there exist an object $A^{\vee} \in \ccA$ and morphisms 
\begin{equation*}
\coev_A \colon \one \to A \otimes A^{\vee} \qquad \text{and} \qquad 
\eva_A \colon A^{\vee} \otimes A \to \one. 
\end{equation*} 
such that the compositions 
\begin{align*}
 A \xrightarrow{ \,  \coev_A \otimes \id_{A} \,} & A \otimes A^{\vee} \otimes A \xrightarrow{\, \id_{A} \otimes \ev_A \,} A  , \\ 
A^{\vee} \xrightarrow{ \, \id_{A^{\vee}} \otimes \coev_A \,} & A^{\vee} \otimes A \otimes A^{\vee} \xrightarrow{\, \ev_A \otimes \id_{A^{\vee}} \,} A^{\vee} , 
\end{align*}
are equivalent to the identity morphisms of $A$ and $A^{\svee}$. 
\end{definition} 
The \emph{dual} $A^{\vee}$ and the \emph{coevaluation} and \emph{evaluation morphisms} $\coev_A$ and $\ev_A$ are uniquely determined in the homotopy category of $\ccA$. 

\begin{remark}
\label{remark-product-dualizable}
If $A$ and $B$ are dualizable objects of $\ccA$, then so is $A^{\vee}$ (with dual $A$) and $A \otimes B$ (with dual $A^{\vee} \otimes B^{\vee}$). 
\end{remark} 

\begin{lemma}
\label{lemma-smooth-proper-dualizable}
Let $\cC$ be an $S$-linear category. 
Then $\cC$ is smooth and proper over $S$ if and only if it is dualizable as an object of $\Cat_S$, in which case: 
\begin{enumerate}
\item The dual is given by the opposite category $\cC^{\vee} = \cC^{\op}$. 
\item \label{C-times-C}
There is an equivalence 
\begin{equation*}
\FM_{\cC} \colon \cC^{\vee} \otimes_{\Dperf(S)} \cC \xrightarrow{\sim} \Fun_S(\cC, \cC)
\end{equation*}
induced by the 
functor $\cC^{\op} \times \cC \to \Fun_S(\cC, \cC)$, $(E, F) \mapsto  \cHom_S(E,-) \otimes F$. 
\item Under the equivalence~$\FM_{\cC}$ (composed with transposition of the factors), the coevaluation morphism 
\begin{equation*}
\coev_{\cC} \colon \Dperf(S) \to \cC \otimes_{\Dperf(S)} \cC^{\vee}
\end{equation*}  
is the unique $S$-linear functor taking $\cO_S$ to $\id_{\cC}$. 
\item The evaluation morphism 
\begin{equation*}
\ev_{\cC} \colon \cC^{\vee} \otimes_{\Dperf(S)} \cC \to \Dperf(S) 
\end{equation*} 
is induced by the functor $\cHom_S(-,-) \colon \cC^{\op} \times \cC \to \Dperf(S)$. 
\end{enumerate} 
\end{lemma}

\begin{proof}
The first statement is \cite[Lemma 4.8]{NCHPD}, and the rest follows from the discussion there. 
\end{proof} 

\begin{example}
\label{example-duality-data-scheme}
When $\cC = \Dperf(X)$ for $X \to S$ a smooth proper morphism of schemes, 
we can describe~$\FM_{\cC}$, $\coev_{\cC}$, and $\ev_{\cC}$ more explicitly as follows. Namely, Example~\ref{example-geometric-lincat} combined with duality $(-)^{\vee} = \cHom_X(-, \cO_X) \colon \Dperf(X)^{\op} \simeq \Dperf(X)$ gives an equivalence
\begin{equation*}
\Dperf(X)^{\op} \otimes_{\Dperf(S)} \Dperf(X) \simeq \Dperf( X \times_S X), 
\end{equation*}  
under which~$\FM_{\cC}$ is identified with the functor of taking Fourier--Mukai transforms 
\begin{equation*}
\Dperf( X \times_S X) \to \Fun_S(\Dperf(X), \Dperf(X)) , \quad K \mapsto \pr_{2*}(\pr_1^*(-) \otimes K) , 
\end{equation*}
and $\coev_{\cC}$ and $\ev_{\cC}$ are identified with the functors 
\begin{equation*}
\Delta_* \circ p^* \colon \Dperf(S) \to \Dperf( X \times_S X) \quad \text{and} \quad 
p_* \circ \Delta^* \colon \Dperf( X \times_S X) \to \Dperf(S) 
\end{equation*} 
where $p \colon X \times_S X \to S$ is the projection and $\Delta \colon X \to X \times_S X$ is the diagonal. 
\end{example} 

\begin{definition}[Kernel functors]
\label{definition-kernel-functors} 
Let $\cC$ and $\cD$ be $S$-linear categories, with $\cC$ proper over~$S$. 
Consider the functor
\begin{equation*}
\label{C-to-D-kernel}
\FM_{\cC \shortrightarrow \cD} \colon \cC^{\op} \otimes_{\Dperf(S)} \cD \to \Fun_S(\cC, \cD)
\end{equation*}
induced by the functor $\cC^{\op} \times \cD \to \Fun_S(\cC, \cD)$, $(E, F) \mapsto \cHom_S(E,-) \otimes F$. 
We call an object $K \in \cC^{\op} \otimes_{\Dperf(S)} \cD$ a \emph{Fourier--Mukai kernel} and write $\Phi_K = \FM_{\cC\shortrightarrow\cD}(K) \colon \cC \to \cD$ for the associated \emph{Fourier--Mukai functor}. 
\end{definition} 

\begin{remark}
The functor $\Phi_K$ can alternatively be described as the composition 
\begin{equation*}
\Phi_K \colon \cC  
\xrightarrow{ \,  \id_{\cC} \otimes K \,} \cC \otimes_{\Dperf(S)} \cC^{\op} \otimes_{\Dperf(S)} \cD \simeq 
\cC^{\op} \otimes_{\Dperf(S)} \cC \otimes_{\Dperf(S)} \cD 
\xrightarrow{\, \ev_{\cC} \otimes \id_{\cD} \,} \cD
\end{equation*} 
where the middle equivalence is given by transposition of the first two factors, and 
$\ev_{\cC}$ is as in Lemma~\ref{lemma-smooth-proper-dualizable} (which is well-defined since $\cC$ is proper over $S$). 
\end{remark} 

The claim of Lemma~\ref{lemma-smooth-proper-dualizable}\eqref{C-times-C} admits the following amplification. 

\begin{lemma}
\label{lemma-kernels}
Let $\cC$ and $\cD$ be $S$-linear categories, with $\cC$ smooth and proper over $S$. 
Then the functor $\FM_{\cC \shortrightarrow \cD}$ is an equivalence. 
Explicitly, the inverse equivalence sends $\Phi \in \Fun_S(\cC, \cD)$ to the image $K_{\Phi}$ of $\cO_S$ under the composition  
\begin{equation*}
\Dperf(S) \xrightarrow{\, \coev_{\cC} \,} 
\cC \otimes_{\Dperf(S)} \cC^{\vee} \xrightarrow{\, \Phi \otimes \id_{\cC^{\vee}} \, } 
\cD \otimes_{\Dperf(S)} \cC^{\vee} \simeq
\cC^{\vee} \otimes_{\Dperf(S)} \cD. 
\end{equation*} 
\end{lemma} 

\begin{proof}
This is a formal consequence of the dualizability of $\cC$, cf. \cite[Lemma 4.4]{NCHPD}. 
\end{proof} 

\begin{lemma}
\label{lemma-yoneda-kernels} 
Let $\cC$ be a smooth proper $S$-linear category, 
and let $K \in  \cC^{\vee} \otimes_{\Dperf(S)} \cC$ with associated Fourier--Mukai functor $\Phi_K \in \Fun_S(\cC,\cC)$. 
Then for $E, F \in \cC$ there is a functorial equivalence 
\begin{equation*}
\cHom_S(E \boxtimes F, K) \simeq \cHom_S(F, \Phi_K(E)) . 
\end{equation*} 
\end{lemma}

\begin{proof} 
Under the equivalence $\FM_{\cC} \colon \cC^{\op} \otimes_{\Dperf(S)} \cC \xrightarrow{\sim} \Fun_S(\cC,\cC)$ of Lemma~\ref{lemma-smooth-proper-dualizable}\eqref{C-times-C}, the kernel $E \boxtimes F$ corresponds to $\cHom_S(E,-) \otimes F$; thus, we find 
\begin{equation*}
\cHom_S(E \boxtimes F, K) \simeq \cHom_S(\cHom_S(E,-) \otimes F, \Phi_K) 
\end{equation*} 
On the other hand, a version of Yoneda's lemma says 
\begin{equation*}
\cHom_S(\cHom_S(E,-) \otimes F, \Phi_K) \simeq 
\cHom_S(F, \Phi_K(E)). 
\end{equation*} 
More precisely, this corresponds to the usual form of  Yoneda's lemma under the equivalence $\Fun_S(\cC,\cC) \simeq \Fun_S(\cC \otimes_{\Dperf(S)} \cC^{\vee}, \Dperf(S))$ deduced from the dualizability of $\cC$, under which 
$\Phi_K$ corresponds to the functor induced by 
\begin{equation*}
\cC \times \cC^{\op} \to \Dperf(S), \quad (C,D) \mapsto \cHom_S(D,\Phi_K(C)), 
\end{equation*}
and $\cHom_S(E,-) \otimes F$ corresponds to the representable functor 
\begin{equation*} 
\cHom_S( E\boxtimes F, -) \colon \cC \otimes_{\Dperf(S)} \cC^{\op} \to \Dperf(S). \qedhere 
\end{equation*} 
\end{proof} 

\subsection{Serre functors} 

\begin{definition}
Let $\cC$ be a proper $S$-linear category. 
A \emph{relative Serre functor} for $\cC$ over $S$ is an $S$-linear autoequivalence 
$\rS_{\cC/S} \colon \cC \to \cC$ such that there are equivalences 
\begin{equation}
\label{equation-relative-serre}
\cHom_S(E, F)^{\vee} \simeq \cHom_{S}(F, \rS_{\cC/S}(E))
\end{equation} 
functorial in $E, F \in \cC$, where $(-)^{\vee} = \cHom_S(-, \cO_S)$ is the derived dual on $\Dperf(S)$. 
\end{definition} 

The source of the terminology is the following example. 

\begin{example}
\label{example-serre-geometric}
Let $f \colon X \to S$ be a proper Gorenstein morphism of noetherian schemes. 
of relative dimension $n$. 
Then $- \otimes \omega_f[n] \colon \Dperf(X) \to \Dperf(X)$ is a relative Serre functor, where $\omega_f$ is the dualizing sheaf (a line bundle by the Gorenstein assumption). 
Indeed, for $E \in \Dperf(X)$, Grothendieck duality gives  
\begin{equation*}
(f_*G)^{\vee} \simeq f_* \cHom_X(G, \omega_f[n]). 
\end{equation*}
Now taking $G = \cHom_X(E, F)$ for $E, F \in \Dperf(X)$ we obtain the required equivalence
\begin{equation*}
\cHom_S(E, F)^{\vee} \simeq \cHom_S(F, E \otimes \omega_f[n]).
\end{equation*} 
More generally, if $\alpha \in \Br(X)$, then 
$- \otimes \omega_f[n] \colon \Dperf(X, \alpha) \to \Dperf(X, \alpha)$ is a relative Serre functor; 
this follows from the untwisted case, because 
for $E, F \in \Dperf(X, \alpha)$ we have 
\begin{equation*}
\cHom_X(E, F) = E^{\vee} \otimes F \in \Dperf(X) . 
\end{equation*}
\end{example} 

When $\cC$ is smooth proper $S$-linear category, a Serre functor always exists and 
is given by a Fourier--Mukai kernel (in the sense of Definition~\ref{definition-kernel-functors}) described explicitly in terms of the duality data of~$\cC$. 
To formulate this, note that the category $\cC^{\vee} \otimes_{\Dperf(S)} \cC$ is smooth and proper over $S$ by Lemma~\ref{lemma-smooth-proper-dualizable} and Remark~\ref{remark-product-dualizable}, and thus by  \cite[Lemma 4.13]{NCHPD} the functor $\ev_{\cC} \colon \cC^{\vee} \otimes_{\Dperf(S)} \cC \to \Dperf(S)$ admits an $S$-linear right adjoint $\ev_{\cC}^!$. 

\begin{lemma}
\label{lemma-S-FM-kernel}
Let $\cC$ be a smooth proper $S$-linear category. 
Then a relative Serre functor $\rS_{\cC/S}$ over $S$ exists, and 
its Fourier--Mukai kernel is given by $\ev_{\cC}^!(\cO_S) \in  \cC^{\vee} \otimes_{\Dperf(S)} \cC$, i.e. 
\begin{equation*}
\rS_{\cC/S} \simeq \Phi_{\ev_{\cC}^!(\cO_S)}. 
\end{equation*} 
\end{lemma} 

\begin{proof}
Let $E, F \in \cC$. 
Since $\ev_{\cC}(E \boxtimes F) = \cHom_S(E,F)$, we find 
\begin{equation*}
\cHom_S(E,F)^{\vee} \simeq \cHom_S(E \boxtimes F, \ev^!_{\cC}(\cO_S)) .  
\end{equation*}
On the other hand, by Lemma~\ref{lemma-yoneda-kernels} 
we have 
\begin{equation*}
\cHom_S(E \boxtimes F, \ev^!_{\cC}(\cO_S)) \simeq \cHom_S(F, \Phi_{\ev^!_{\cC}(\cO_S)}(E)) . \qedhere 
\end{equation*} 
\end{proof} 

\begin{lemma}
\label{lemma-bc-Serre}
Let $\cC$ be a smooth proper $S$-linear category. 
Let $T \to S$ be a morphism of perfect derived algebraic stacks.
Then the functor $(\rS_{\cC/S})_T \colon \cC_T \to \cC_T$ obtained by base change from the relative Serre functor for $\cC$ over $S$ 
is a relative Serre functor for $\cC_T$ over $T$. 
\end{lemma} 

\begin{proof}
Note that $\cC_T$ is smooth and proper over $S$ by Lemma~\ref{lemma-smooth-proper-bc}. 
By Lemma~\ref{lemma-S-FM-kernel}, the relative Serre functor $\rS_{\cC_T/T}$ is given by the kernel 
\begin{equation*}
\ev_{\cC_T}^!(\cO_T) \in \cC_T^{\vee} \otimes_{\Dperf(T)} \cC_T. 
\end{equation*} 
The duality data for $\cC_T$ (realizing it as a dualizable object of $\Cat_T$) are obtained by base change from those of $\cC$. 
In particular, there is a $T$-linear equivalence
\begin{equation}
\label{bc-C-C}
 \left(\cC^{\vee} \otimes_{\Dperf(S)} \cC \right)_T \simeq \cC_T^{\vee} \otimes_{\Dperf(T)} \cC_T 
\end{equation} 
under which $(\ev_{\cC})_T$ corresponds to $\ev_{\cC_T}$. 
The right adjoint of $(\ev_{\cC})_T$ is obtained by base change to $T$ from the right adjoint $\ev_{\cC}^!$ \cite[Lemma 2.12]{NCHPD}, 
and thus $(\ev_{\cC})_T^!(\cO_T) = (\ev_{\cC}^!(\cO_S))_T$ corresponds to $\ev_{\cC_T}^!(\cO_T)$ under the equivalence~\eqref{bc-C-C}. 
Here, $(\ev_{\cC}^!(\cO_S))_T$ denotes the image of $\ev_{\cC}^!(\cO_S)$ under 
$\cC^{\vee} \otimes_{\Dperf(S)} \cC \to \left(\cC^{\vee} \otimes_{\Dperf(S)} \cC \right)_T$. 
Unwinding the definitions we find that in general, for any kernel $K \in \cC^{\vee} \otimes_{\Dperf(S)} \cC$, its image $K_T \in \left(\cC^{\vee} \otimes_{\Dperf(S)} \cC \right)_T$ corresponds under the equivalence~\eqref{bc-C-C} to the kernel for the base changed functor $(\Phi_{K})_T \colon \cC_T \to \cC_T$. 
Since by Lemma~\ref{lemma-smooth-proper-bc} we have $\Phi_{\ev_{\cC}^!(\cO_S)} \simeq \rS_{\cC/S}$, 
we conclude $\rS_{\cC_T/T} \simeq (\rS_{\cC/S})_T$.  
\end{proof} 

\subsection{Presentable linear categories} 
\label{section-presentable-lin-cat}
It is sometimes\footnote{In this paper, we will only need such categories for our discussion of Hochschild homology in \S\ref{section-hochschild-homology}.} useful to consider a ``large'' version of linear categories, which corresponds to working over $\Dqc(S)$ instead of $\Dperf(S)$. 
Namely, $\Dqc(S)$ is a commutative algebra object in the category $\PrCat_{\mathrm{st}}$ of presentable stable $\infty$-categories, with morphisms concontinuous functors (i.e.\ those preserving small colimits). A \emph{presentable $S$-linear category} is a $\Dqc(S)$-module object of $\PrCat_{\mathrm{st}}$, and the collection of all such categories forms an $\infty$-category $\PrCat_S = \Mod_{\Dqc(S)}(\PrCat_{\mathrm{st}})$. 

The formalism of presentable $S$-linear categories is parallel to that of $S$-linear categories described in \S\ref{section-small-linear-cat}, which for clarity we sometimes refer to as ``small $S$-linear categories''. 
The category $\PrCat_S$ admits a symmetric monoidal structure 
with unit $\Dqc(S)$ and tensor product denoted 
\begin{equation*}
\cC \otimes_{\Dqc(S)} \cD . 
\end{equation*} 
A morphism $\cC \to \cD$ in $\PrCat_S$ is called a \emph{presentable $S$-linear functor}; the collection of all such forms a presentable $S$-linear category $\Fun_S(\cC, \cD)$, which is the internal mapping object in $\PrCat_S$. 
If $\cC$ is a presentable $S$-linear category and $E, F \in \cC$, then there is a mapping object $\cHom_S(E, F) \in \Dqc(S)$ characterized by equivalences as in \eqref{equation-cHom} where we allow $G \in \Dqc(S)$. 

For $\cC \in \Cat_S$ there is a category $\Ind(\cC) \in \PrCat_S$, called its \emph{ind-completion}, which is roughly obtained by freely adjoining all filtered colimits to $\cC$. This gives a symmetric monoidal functor 
\begin{equation*}
\Ind \colon \Cat_S \to \PrCat_S. 
\end{equation*} 
In fact, $\Ind$ factors through an equivalence onto the category $\PrCat_S^{\omega}$ of compactly generated presentable $S$-linear categories, with morphisms the cocontinuous $S$-linear functors which preserve compact objects; the inverse equivalence $(-)^c \colon \PrCat_S^{\omega} \to \Cat_S$ is given by passage to compact objects. 

\begin{remark}
In the above terms, conditions \eqref{das-2} and \eqref{das-3} in Definition~\ref{definition-perfect-stack} of a perfect derived algebraic stack $X$ are together equivalent to the condition that the canonical morphism 
\begin{equation*}
\Ind(\Dperf(X)) \to \Dqc(X) 
\end{equation*} 
is an equivalence. 
\end{remark} 

There is also notion of a presentable $S$-linear semiorthogonal decompositions, which is compatible with $S$-linear decompositions of small $S$-linear categories via ind-completion \cite[Lemma 3.12]{NCHPD}. 

\begin{lemma}[{\cite[Theorem 5.4]{bergh-BS}}]
\label{lemma-mun-gerbe-Dqc}
Let $X$ be a scheme, let $\alpha \in \Br(X)[n]$ with $n$ invertible on $X$, and let $\pi \colon \cX \to X$ a $\bmu_n$-gerbe of class~$\alpha$. 
Then there is a presentable $X$-linear orthogonal decomposition 
\begin{equation*}
\Dqc(\cX) = \llangle \Dqc^k(\cX) \rrangle_{k \in \bZ/n}. 
\end{equation*} 
\end{lemma}

\begin{lemma}
\label{lemma-compact-generation-twisted-D}
Let $X$ be a quasi-compact scheme with affine diagonal. 
Then there is an equivalence
\begin{equation*}
\Ind(\Dperf(X, \alpha)) \simeq \Dqc(X, \alpha). 
\end{equation*} 
\end{lemma}

\begin{proof}
By Example~\ref{example-gerbes-perfect} the stack $\cX$ is perfect. 
Thus, taking the ind-completion of the decomposition of Lemma~\ref{lemma-mun-gerbe} gives a presentable $X$-linear orthogonal decomposition 
\begin{equation*}
\Dqc(\cX) = \llangle \Ind(\Dperf^k(\cX)) \rrangle_{k \in \bZ/n}. 
\end{equation*} 
Since the fully faithful functor $\Ind(\Dperf^k(\cX)) \to \Dqc(\cX)$ has image contained in $\Dqc^k(\cX)$, it therefore must factor via an equivalence 
\begin{equation*}
\Ind(\Dperf^k(\cX)) \simeq \Dqc^k(\cX). 
\end{equation*} 
For $k = 1$ we obtain the claim of the lemma. 
\end{proof}

In the case of small $S$-linear categories, dualizability is a strong finiteness condition, equivalent to being smooth and proper. 
In the presentable case, however, dualizability is automatic for most categories of interest. 

\begin{lemma}[{\cite[Lemma 4.3]{NCHPD}}]
\label{lemma-cg-dualizable}
Let $\cC$ be a compactly generated presentable $S$-linear category. 
Then $\cC$ is dualizable as an object of $\PrCat_S$, with dual 
\begin{equation*}
\cC^{\vee} = \Ind((\cC^c)^{\op}). 
\end{equation*} 
Moreover, there is an equivalence
\begin{equation*}
\FM_{\cC} \colon \cC^{\vee} \otimes_{\Dqc(S)} \cC \xrightarrow{\sim} \Fun_S(\cC, \cC)
\end{equation*} 
under which $\coev_{\cC}$ corresponds to the object $\id_{\cC} \in \Fun_S(\cC, \cC)$, and 
$\ev_{\cC}$ is induced by the functor $\cHom_S(-,-) \colon (\cC^c)^{\op} \times \cC^c \to \Dqc(S)$
\end{lemma} 

By Lemma~\ref{lemma-cg-dualizable}, $\Dqc(X)$ is dualizable for any perfect derived algebraic stack $X$ over $S$; 
as in Example~\ref{example-duality-data-scheme}, we can explicitly describe the duality data in this geometric case. 

\begin{example}
\label{example-duality-data-perfectstack}
Let $\cC = \Dqc(X)$ for $f \colon X \to S$ a morphism of perfect derived algebraic stacks. 
By Example~\ref{example-geometric-lincat} combined with duality $(-)^{\vee} \colon \Dperf(X)^{\op} \simeq \Dperf(X)$, we have an equivalence 
\begin{equation*}
\Dperf(X)^{\op} \otimes_{\Dperf(S)} \Dperf(X) \simeq \Dperf( X \times_S X). 
\end{equation*}  
Passing to ind-completion gives an equivalence 
\begin{equation*}
\cC^{\vee} \otimes_{\Dqc(S)} \cC \simeq \Dqc(X \times_S X) ,
\end{equation*} 
under which 
$\coev_{\cC}$ and $\ev_{\cC}$ are identified with the functors 
\begin{equation*}
\Delta_* \circ f^* \colon \Dqc(S) \to \Dqc( X \times_S X) \quad \text{and} \quad 
f_* \circ \Delta^* \colon \Dqc( X \times_S X) \to \Dqc(S) 
\end{equation*} 
where $\Delta \colon X \to X \times_S X$ is the diagonal 
(cf. \cite[Corollary 4.8]{bzfn}). 
\end{example} 

We also have a twisted variant of Example~\ref{example-duality-data-perfectstack}. 

\begin{example}
\label{lemma-duality-twisted-scheme}
Let $\cC = \Dqc(X, \alpha)$  
where $f \colon X \to S$ is a flat quasi-compact morphism from a quasi-compact scheme with affine diagonal to a perfect scheme 
and $\alpha \in \Br(X)[n]$ with $n$ invertible on~$X$. 
By Lemma~\ref{lemma-tensor-twisted-schemes} (which applies since $X$ is perfect by Example~\ref{example-gerbes-perfect}) combined with duality $(-)^{\vee} \colon \Dperf(X,\alpha)^{\op} \simeq \Dperf(X,-\alpha)$, we have an equivalence 
\begin{equation*}
\Dperf(X, \alpha)^{\op} \otimes_{\Dperf(S)} \Dperf(X,\alpha) \simeq \Dperf( X \times_S X, \pr_2^*(\alpha) - \pr_1^*(\alpha) ). 
\end{equation*}  
Passing to ind-completion and using Lemma~\ref{lemma-compact-generation-twisted-D} (which applies since by our assumptions $X \times_S X$ is quasi-compact with affine diagonal) gives an equivalence 
\begin{equation*}
\cC^{\vee} \otimes_{\Dqc(S)} \cC \simeq \Dqc(X \times_{S} X, \pr_2^*(\alpha) - \pr_1^*(\alpha)),  
\end{equation*} 

We now describe $\coev_{\cC}$ and $\ev_{\cC}$. Note that since $\Delta^*(\pr_2^*(\alpha) - \pr_1^*(\alpha)) = 0$, there are by Remark~\ref{remark:pullback_killing_brauer_class} many choices of pushforward and pullback functors
\begin{align*}
\Dqc(X) \to \Dqc( X \times_S X, \pr_2^*(\alpha) - \pr_1^*(\alpha)) , \\  
\Dqc( X \times_S X, \pr_2^*(\alpha) - \pr_1^*(\alpha)) \to \Dqc(X) ,
\end{align*} 
which could assume the roles played by $\Delta_*$ and $\Delta^*$ in Example~\ref{example-duality-data-perfectstack}. 

In fact, in our particular situation, there is a natural choice. 
Let $\cX$ be a $\bmu_n$-gerbe of Brauer class $\alpha$. There is a \emph{dual gerbe} $\cX^\vee$ of Brauer class $- \alpha$. 
The dual gerbe may be described as the relative moduli stack of $1$-twisted line bundles $L$ for the morphism $\cX \to X$, equipped with trivializations $L^{\otimes n} \simeq \cO$.
The product $\cX^\vee \times_X \cX$ carries a canonical $(1,1)$-twisted line bundle $\cP$ (the universal bundle). Writing
\[
    \pi:\cX^\vee \times_X \cX \to X, \quad \iota\colon \cX^\vee \times_X \cX \to \cX^\vee \times_S \cX
\]
we obtain functors
\begin{align*}
    \Delta_*^{\alpha} &\colon \Dqc(X) \to \Dqc^{(1,1)}(\cX^\vee \times_S \cX), \quad E \mapsto \iota_*(\cP \otimes \pi^*E) \\
    \Delta^*_{\alpha} &\colon \Dqc^{(1,1)}(\cX^\vee \times_S \cX) \to \Dqc(X), \quad F \mapsto \pi_*\left(\iota^*(F) \otimes \cP^\vee\right),
\end{align*}
Note that $\Dqc^{(1,1)}(\cX^\vee \times_S \cX)$ is the twisted derived category $\Dqc(X \times_S X, \pr_2^* \alpha - \pr_1^* \alpha)$.

With this notation, $\coev_{\cC}$ and $\ev_{\cC}$ are identified 
with the functors 
\begin{align*}
\Delta^{\alpha}_* \circ f^* & \colon \Dqc(S) \to \Dqc( X \times_S X, \pr_2^*(\alpha) - \pr_1^*(\alpha)) , \\  
f_* \circ \Delta_\alpha^* & \colon \Dqc( X \times_S X, \pr_2^*(\alpha) - \pr_1^*(\alpha)) \to \Dqc(S) .
\end{align*}  
\end{example} 

As in the case of small linear categories, when $\cC$ is a dualizable presentable $S$-linear category, we can describe functors out of $\cC$ in terms of kernels. 
\begin{lemma}
\label{lemma-kernel-presentable}
Let $\cC, \cD$ be presentable $S$-linear categories, with $\cC$ dualizable. 
Then there is an equivalence
\begin{equation*}
\FM_{\cC \to \cD} \colon \cC^{\vee} \otimes_{\Dqc(S)} \cD \xrightarrow{\, \sim \,} \Fun_S(\cC, \cD)
\end{equation*} 
which sends $K \in  \cC^{\vee} \otimes_{\Dqc(S)}\cD$ to the composition 
\begin{equation*}
\Phi_K \colon \cC  
\xrightarrow{ \,  \id_{\cC} \otimes K \,} \cC \otimes_{\Dqc(S)} \cC^{\vee} \otimes_{\Dqc(S)} \cD \simeq 
\cC^{\vee} \otimes_{\Dqc(S)} \cC \otimes_{\Dqc(S)} \cD 
\xrightarrow{\, \ev_{\cC} \otimes \id_{\cD} \,} \cD,
\end{equation*} 
and whose inverse equivalence sends $\Phi \in \Fun_S(\cC, \cD)$ to the image $K_{\Phi}$ of $\cO_S$ under the composition  
\begin{equation*}
\Dqc(S) \xrightarrow{\, \coev_{\cC} \,} 
\cC \otimes_{\Dqc(S)} \cC^{\vee} \xrightarrow{\, \Phi \otimes \id_{\cC^{\vee}} \, } 
\cD \otimes_{\Dqc(S)} \cC^{\vee} \simeq
\cC^{\vee} \otimes_{\Dqc(S)} \cD. 
\end{equation*} 
\end{lemma}

Finally, we note a generalization of Lemma~\ref{lemma-yoneda-kernels}, which holds by the same argument, where the $S$-linear category is not required to be smooth and proper and the Fourier--Mukai kernel is only required to exist after ind-completion. 

\begin{lemma}
\label{lemma-presentable-yoneda-kernels} 
Let $\cC$ be small $S$-linear category, 
and let $K \in  \Ind(\cC)^{\vee} \otimes_{\Dqc(S)} \Ind(\cC)$ with associated Fourier--Mukai functor $\Phi_K \in \Fun_S(\Ind(\cC), \Ind(\cC))$. 
Then for $E, F \in \cC$ there is a functorial equivalence 
\begin{equation*}
\cHom_S(E \boxtimes F, K) \simeq \cHom_S(F, \Phi_K(E)) . 
\end{equation*} 
\end{lemma}


\section{(Co)homological invariants}

In this section, we study some (co)homological invariants of linear categories, namely Hochschild (co)homology and topological $\rK$-theory. 

Fix a perfect derived algebraic stack $S$, which will serve as the base space for our discussion. 
In our discussion of Hodge theory in \S\ref{section-Hodge-theory}, we will specialize to the case where $S$ is a complex variety. 

\subsection{Hochschild homology} 
\label{section-hochschild-homology} 
Given a symmetric monoidal $\infty$-category $(\ccA, \otimes, \one)$ and a dualizable (in the sense of Definition~\ref{definition-dualizable}) object $A \in \ccA$, the 
\emph{trace} of an endomorphism $F \colon A \to A$ is the map $\Tr(F) \in \Map_{\ccA}(\one, \one)$ given by the composition 
\begin{equation*}
\one \xrightarrow{\, \coev_A \,} A \otimes A^{\svee} \xrightarrow{\, F \otimes \id_{A^{\svee}} \,} A \otimes A^{\svee} \simeq A^{\svee} \otimes A \xrightarrow{\, \ev_{A} \,} \one. 
\end{equation*} 
When $\ccA = \PrCat_S$ then $\one = \Dqc(S)$ and the functor $\Tr(F)$ is determined by its value on $\cO_S$. 

\begin{definition}
If $\cC$ is a dualizable presentable $S$-linear category and $\Phi \in \Fun_S(\cC,\cC)$ is an endomorphism, 
the \emph{Hochschild homology of $\cC$ over $S$ with coefficients in $\Phi$} is the complex 
\begin{equation*}
\cHH_*(\cC/S, \Phi) = \Tr(\Phi)(\cO_S) \in \Dqc(S). 
\end{equation*} 
The \emph{Hochschild homology of $\cC$ over $S$} is the complex 
\begin{equation*}
\cHH_*(\cC/S) = \cHH_*(\cC/S, \id_{\cC}) \in \Dqc(S). 
\end{equation*} 
If instead $\cC$ is a small $S$-linear category and $\Phi \in \Fun_S(\cC,\cC)$ is an endomorphism, 
then by Lemma~\ref{lemma-cg-dualizable} the category $\Ind(\cC) \in \PrCat_S$ is dualizable, so we may define  
\begin{align*}
\cHH_*(\cC/S, \Phi) & = \cHH_*(\Ind(\cC)/S, \Ind(\Phi)) \\ 
\cHH_*(\cC/S) & = \cHH_*(\cC/S, \id_{\cC}). 
\end{align*} 
In either case, we use the notation  
\begin{align*}
\cHH_i(\cC/S,\Phi) & = \cH^{-i}(\cHH_*(\cC/S, \Phi)), \\  
\HH_*(\cC/S, \Phi) & = \rR\Gamma(\cHH_*(\cC/S, \Phi)),  \\ 
\HH_i(\cC/S, \Phi) & = \rH^{-i}(\HH_*(\cC/S,\Phi)), 
\end{align*} 
for the degree $-i$ cohomology sheaf, the derived global sections, and the degree $-i$ cohomology of $\cHH_*(\cC/S, \Phi)$, and when $\Phi = \id_{\cC}$ we omit it from all of these notations. 
\end{definition} 

\begin{remark}
Strictly speaking, for the applications in this paper the case of trivial coefficients $\Phi = \id_{\cC}$ suffices, but our results on Chern characters below hold for arbitrary coefficients and may be useful elsewhere in that generality. 
\end{remark} 

\begin{remark}
\label{remark-smooth-proper-HH} 
If $\cC$ is a smooth proper $S$-linear category, then by Lemma~\ref{lemma-smooth-proper-dualizable} 
it is a dualizable object of $\Cat_S$, so for any $\Phi \in \Fun_S(\cC,\cC)$ we may form the trace 
$\Tr(\Phi) \colon \Dperf(S) \to \Dperf(S)$. 
There is a canonical equivalence $\Ind(\Tr(\Phi)) \simeq \Tr(\Ind(\Phi)) \colon \Dqc(S) \to \Dqc(S)$ 
\cite[Remark~3.2]{IHC-CY2}; in particular $\cHH_*(\cC/S, \Phi) \simeq \Tr(\Phi)(\cO_S) \in \Dperf(S)$. 
\end{remark} 

\begin{remark}
\label{remark-HH-kernels}
Let $\cC$ be either a dualizable presentable $S$-linear category or a smooth proper small $S$-linear category. 
Let $\Phi \in \Fun_S(\cC,\cC)$ and let $K_{\Phi}$ be the corresponding kernel, given by Lemma~\ref{lemma-kernel-presentable} or Lemma~\ref{lemma-kernels}. 
From the description of $K_{\Phi}$ in those lemmas, it follows that 
\begin{equation*}
\cHH_*(\cC/S, \Phi) \simeq \ev_{\cC}(K_{\Phi}). 
\end{equation*} 
\end{remark} 

The following result summarizes some important properties of Hochschild homology. 

\begin{theorem}
\label{theorem-HH}
Let $\cC$ be an $S$-linear category. 
\begin{enumerate}
\item \label{theorem-mukai-pairing}
If $\cC$ is a smooth and proper over $S$, then $\cHH_*(\cC/S) \in \Dperf(S)$ 
and there is a canonical nondegenerate pairing $\cHH_*(\cC/S) \otimes \cHH_*(\cC/S) \to \cO_S$, called the \emph{Mukai pairing}, 
which thus induces an equivalence $\cHH_*(\cC/S) \simeq \cHH_*(\cC/S)^{\vee}$.  
\item \label{theorem-smooth-proper-HH}
If $\cC$ is smooth and proper over $S$ and $S$ is a $\bQ$-scheme, 
then $\cHH_i(\cC/S)$ is a finite locally free sheaf for any $i \in \bZ$. 
\item \label{base-change-HH} 
For any morphism of perfect derived algebraic stacks $g \colon T \to S$, there is a canonical equivalence 
\begin{equation*}
g^*\cHH_*(\cC/S) \simeq \cHH_*(\cC_T/T), 
\end{equation*} 
and if $\cC$ is smooth and proper over $S$ and $S$ is a $\bQ$-scheme, then for any $i \in \bZ$ there is a canonical isomorphism 
\begin{equation*} 
g^* \cHH_i(\cC/S) \simeq \cHH_i(\cC_T/T). 
\end{equation*} 
\end{enumerate} 
\end{theorem} 

\begin{proof}
See \cite[Lemma 3.6, Lemma 3.4, and Theorem 3.5]{IHC-CY2}. 
The main content is \eqref{theorem-smooth-proper-HH}, which is a 
consequence of the  
the degeneration of the noncommutative Hodge-to-de Rham spectral sequence 
\cite{kaledin1, kaledin2, akhil-degeneration}. 
\end{proof} 

For further background on Hochschild homology, see \cite[\S3]{IHC-CY2} and the references therein.\footnote{As a warning, in \cite{IHC-CY2} the complex $\cHH_*(\cC/S, \Phi)$ is denoted instead by $\HH_*(\cC/S, \Phi)$, which we have reserved for its global sections.}
Below we detail some additional properties for which we do not know an adequate reference. 
The first concerns the invariance of the Hochschild homology of a scheme under twisting by a Brauer class. 
In fact, such a result was proved for an arbitrary additive invariant by Tabuada and Van den Bergh \cite{HH-TVdB}, but their argument does not apply to Hochschild cohomology, which we will also need. 
We supply an alternative argument, due to Vadim Vologodsky, which gives a more precise statement and also applies to cohomology (see Lemma~\ref{lemma-hochschild-cohomology-twisted-variety}).  

\begin{lemma}[Vologodsky]
\label{lemma-hochschild-homology-twisted-variety} 
Let $f \colon X \to S$ be a flat morphism from a noetherian scheme with affine diagonal to a perfect noetherian scheme.
Let $\alpha \in \Br(X)[n]$, where $n$ is invertible on $S$.
Then there are equivalences 
\begin{equation}
\label{HH-twist}
\cHH_*(\Dperf(X, \alpha)/S) \simeq f_* \Delta_{\alpha}^* \Delta^{\alpha}_* \cO_X \simeq f_* \Delta^*\Delta_*\cO_X \simeq \cHH_*(\Dperf(X)/S) ,
\end{equation} 
where $\Delta \colon X \to X \times_S X$ is the diagonal and $\Delta_*^{\alpha}$, $\Delta^*_{\alpha}$ are the functors described in Example~\ref{lemma-duality-twisted-scheme}. In fact, there is an equivalence 
\begin{equation}
\label{HH-twist-pushpull}
        \Delta^*_{\alpha} \Delta_*^{\alpha} \cO_X \simeq \Delta^* \Delta_* \cO_X .
\end{equation}
\end{lemma} 

\begin{proof} 
The outer two equivalences in~\eqref{HH-twist} follow from Lemma~\ref{lemma-compact-generation-twisted-D} and Examples~\ref{example-duality-data-perfectstack} and~\ref{lemma-duality-twisted-scheme}, so it suffices to prove the identity~\eqref{HH-twist-pushpull}. 
The idea of the proof is to work over the formal completion of the diagonal. 

We adopt the notation of Example~\ref{lemma-duality-twisted-scheme}, so that $\cX$ is a $\bmu_n$-gerbe of class $\alpha$, $\cX^\vee$ is the dual gerbe, and $\cP$ is the universal bundle on $\cX^\vee \times_X \cX$. There is a diagram
\[
    \begin{tikzcd}
        \cX^\vee \times_X \cX \ar[dr, "\iota"'] \ar[r, "\hat\iota"] & \left(\cX^\vee \times_S \cX\right)^\wedge \ar[d, "g"] \\ & \cX^\vee \times_S \cX ,
    \end{tikzcd}
\]
where $(\cX^\vee \times_s \cX)^\wedge$ is the formal completion of $\cX^\vee \times_S \cX$ along the closed substack $\cX^\vee \times_X \cX$. The identity \eqref{HH-twist-pushpull} can be checked after pullback to the gerbe $\cX^\vee \times_X \cX$, where it may be written as
\begin{equation}
\label{equation:pulled_back_hh_identity}
\cP^\vee \otimes \iota^* \iota_*(\cP) \simeq \iota^* \iota_*(\cO_{\cX^\vee \times_X \cX}).
\end{equation}
The natural completion maps
\begin{align*}
    \cP^\vee \otimes \iota^* \iota_*(\cP) &\to \cP^\vee \otimes \hat{\iota}^* \hat{\iota}_*(\cP) , \\
    \iota^* \iota_*(\cO_{\cX^\vee \times_X \cX}) &\to \hat{\iota}^* \hat{\iota}_*(\cO_{\cX^\vee \times_X \cX}) 
\end{align*}
are equivalences, because of the observation that
\[
    g^* g_* \hat{\iota}_* (M) \simeq \hat{\iota}_* M,
 \] 
for any coherent sheaf $M$ on $\cX^\vee \times_X \cX$. \'Etale-locally on $\cX^\vee \times_S \cX$, the observation boils down to the fact that if $A$ is a noetherian ring and $M$ is a finitely generated $A/I$-module for an ideal $I$, then the natural map 
\[
    M \to M \otimes_{A} \widehat{A}
\]
is an isomorphism, where $\widehat{A}$ is the $I$-adic completion of $A$.

From the previous paragraph, it remains to show the completed version of \eqref{equation:pulled_back_hh_identity}:
\begin{equation}
\label{equation:completed_hh_identity}
    \cP^\vee \otimes \hat{\iota}^* \hat{\iota}_*(\cP) \simeq \hat{\iota}^* \hat{\iota}_*(\cO_{\cX^\vee \times_X \cX}).
\end{equation}
Suppose that there exists a line bundle $\widehat \cP$ on the completion $(\cX^\vee \times_S \cX)^\wedge$ which pulls back to $\cP$. (The existence of $\widehat{\cP}$ is shown in the next paragraph.) Then \eqref{equation:completed_hh_identity} follows:
\begin{align*}
    \cP^\vee \otimes \hat{\iota}^* \hat{\iota}_*(\cP) &\simeq \cP^\vee \otimes \hat{\iota}^* \hat{\iota}_*\hat{\iota}^*(\widehat{\cP}) \\
        &\simeq \cP^\vee \otimes \hat{\iota}^*(\hat{\iota}_* (\cO_{\cX^\vee \times_X \cX}) \otimes \widehat{\cP}) \\
        &\simeq \cP^\vee \otimes \hat{\iota}^*(\hat{\iota}_* (\cO_{\cX^\vee \times_X \cX})) \otimes \cP \\
        &\simeq \hat{\iota}^*(\hat{\iota}_* (\cO_{\cX^\vee \times_X \cX})).
\end{align*}
In the second line, we have used the projection formula for a ringed topos \cite[\href{https://stacks.math.columbia.edu/tag/0943}{Section 0943}]{stacks-project}. This proves the identity \eqref{equation:completed_hh_identity}, assuming the existence of $\widehat{\cP}$.

To conclude the proof, we show that there exists a line bundle $\widehat{\cP}$ on $(\cX^\vee \times_S \cX)^\wedge$ which pulls back to $\cP$ along $\hat{\iota}$. The $(\bmu_n \times \bmu_n)$-gerbe $\cX^\vee \times_S \cX$ over $X \times_S X$ may be ``multiplied'' to a $\bmu_n$-gerbe $\cG$ on $X \times_S X$ of class $\pr_2^*[\cX] - \pr_1^*[\cX]$. Let $\widehat{\cG}$ be the formal completion of $\cG$ along the preimage of the diagonal, so that $\widehat{\cG}$ is a $\bmu_n$-gerbe over $(X \times_S X)^\wedge$. In fact, $\widehat{\cG}$ is a trivial $\bmu_n$-gerbe. Indeed, on one hand the natural pullback map
\begin{equation*}
    \widehat{\Delta}^* \colon \rH^2_{\et}((X \times_S X)^\wedge, \bmu_n) \to \rH_{\et}^2(X, \bmu_n)
\end{equation*}
is an isomorphism, because of the invariance of the \'etale site under nilpotent thickening. On the other hand, $[\widehat{\cG}]$ lies in the kernel.

By the previous paragraph, there exists a $1$-twisted line bundle $L$ on $\widehat{\cG}$. (This is Lemma~\ref{lemma-essentially-trivial} for formal schemes; the proof is identical.) Let $\widehat{\cP}'$ be the pullback of $L$ to $(\cX^\vee \times_S \cX)^\wedge$. Then $\widehat{\cP}'$ is a $(1,1)$-twisted line bundle, and its restriction to $\cX^\vee \times_X \cX$ differs from $\cP$ by an element $L_0$ of $\Pic(X)$. Then
\[
    \widehat{\cP} = \widehat{\cP}' \otimes \pr_1^*(L_0^{-1})
\]
is the desired line bundle. 
\end{proof}

\begin{remark}
    If $n$ is not invertible on $S$, then the conclusion of Lemma~\ref{lemma-hochschild-homology-twisted-variety} does not hold in general \cite[Theorem~1.3]{HKR-charp}.
\end{remark}

\begin{remark}
    If $X$ is a Deligne--Mumford stack then the conclusion of Lemma~\ref{lemma-hochschild-homology-twisted-variety} often fails. For example, if $X = \rB G$ over $\bC$ for $G = \bZ/2 \times \bZ/2$, then $\Br(X) = \bZ/2$, and the nonzero class $\alpha$ comes from the $\bmu_2$-gerbe 
    \[
        \rB D_8 \to \rB G,
    \]
    where $D_8$ is the dihedral group of order $8$. The twisted derived category $\Dperf(X, \alpha)$ is identified with the subcategory $\langle V \rangle$ of $\Dperf(\rB D_8)$ generated by the irreducible $2$-dimensional representation of $D_8$. In particular, $\HH_0(\Dperf(X, \alpha)) = \bC$, whereas $\HH_0(\Dperf(X)) = \bC^4$.
\end{remark}

\begin{lemma}
\label{lemma-HH-Serre}
Let $\cC$ be a smooth proper $S$-linear category, 
let $\Phi \in \Fun_S(\cC,\cC)$, and let $\rS_{\cC/S}$ be the relative Serre functor (which exists by Lemma~\ref{lemma-S-FM-kernel}). 
Then there is an equivalence 
\begin{equation*} 
\cHH_*(\cC/S, \Phi)^{\vee} \simeq \cHom_S(\Phi, \rS_{\cC/S}) 
\end{equation*} 
where the right side is the mapping object computed in $\Fun_S(\cC,\cC)$. 
\end{lemma} 

\begin{proof}
By Remark~\ref{remark-HH-kernels} we find 
\begin{equation*}
\cHH_*(\cC/S, \Phi)^{\vee}  
 \simeq \cHom_S(\ev_{\cC}(K_\Phi), \cO_S) 
 \simeq \cHom_S(K_{\Phi}, \ev_{\cC}^!(\cO_S)). 
\end{equation*} 
By Lemma~\ref{lemma-smooth-proper-dualizable}\eqref{C-times-C} there is an equivalence $\FM_{\cC} \colon \cC^{\vee} \otimes_{\Dperf(S)} \cC \xrightarrow{\sim} \Fun_S(\cC, \cC)$, under which $\FM_{\cC}(K_{\Phi}) \simeq \Phi$ by the definition of $K_{\Phi}$ and $\FM_{\cC}(\ev_{\cC}^!(\cO_S)) \simeq \rS_{\cC/S}$ by Lemma~\ref{lemma-S-FM-kernel}, so the claim follows. 
\end{proof} 

Hochschild homology is a recipient for a generalized Chern character, defined out of the complex $\cHom_S(E, \Phi(E))$ for any object $E$ of an $S$-linear category. 
A construction can be found for instance in \cite[\S3.2]{IHC-CY2}, but we give a direct description in terms of Fourier--Mukai kernels below. 
 
\begin{construction}[Chern characters]
Let $\cC$ be an $S$-linear category, let $\Phi \in \Fun_S(\cC, \cC)$, and let $E \in \cC$. 
Let $K_{\Ind(\Phi)} \in \Ind(\cC)^{\vee} \otimes_{\Dqc(S)} \Ind(\cC)$ be the 
Fourier--Mukai kernel for the functor $\Ind(\Phi) \in \Fun_S(\Ind(\cC), \Ind(\cC))$, which exists by Lemma~\ref{lemma-cg-dualizable}. 
By Lemma~\ref{lemma-presentable-yoneda-kernels} we have an equivalence 
\begin{equation*}
\cHom_S(E \boxtimes \Phi(E), K_{\Ind(\Phi)}) \simeq \cHom_S(\Phi(E), \Phi(E)) , 
\end{equation*} 
and thus a canonical morphism 
\begin{equation}
\label{ch-kernel} 
E \boxtimes \Phi(E) \to K_{\Ind(\Phi)}
\end{equation} 
corresponding to $\id_{\Phi(E)}$. 
Note that $\ev_{\cC}(E \boxtimes \Phi(E)) \simeq \cHom_S(E, \Phi(E))$, while by Remark~\ref{remark-HH-kernels} we have $\ev_{\cC}(K_{\Ind(\Phi)}) \simeq \cHH_*(\cC/S, \Phi)$. 
The \emph{Chern character of $E$ with coefficients in $\Phi$} is the morphism 
\begin{equation*}
\ch_{E,\Phi} \colon \cHom_S(E, \Phi(E)) \to \cHH_*(\cC/S, \Phi) 
\end{equation*} 
obtained by applying $\ev_{\cC}$ to~\eqref{ch-kernel}. 
We write $\ch_{E} = \ch_{E, \id_{\cC}}$ for the case when $\Phi = \id_{\cC}$. 
\end{construction}

\begin{remark}
When $\cC$ is smooth and proper over $S$, 
then we may avoid the use of ind-completion in the definition of the Chern character, using the 
kernel $K_{\Phi} \in \cC^{\vee} \otimes_{\Dperf(S)} \cC$ for $\Phi$ given by Lemma~\ref{lemma-smooth-proper-dualizable}\eqref{C-times-C} in place of $K_{\Ind(\Phi)}$, and Lemma~\ref{lemma-yoneda-kernels} in place of Lemma~\ref{lemma-presentable-yoneda-kernels}. 
In this situation, the Chern character also admits the following simple description in terms of Serre duality. 
\end{remark} 

\begin{lemma}
\label{lemma-ch-dual} 
Let $\cC$ be a smooth proper $S$-linear category, 
let $\Phi \in \Fun_S(\cC, \cC)$, and let $E \in \cC$. 
Then there is a commutative diagram 
\begin{equation*}
    \begin{tikzcd}
    \cHH_*(\cC/S, \Phi)^{\vee} \ar[r, "\ch_{E, \Phi}^{\vee}"] \ar[d, "\sim" {anchor=south, rotate=90}] & \cHom_{S}(E, \Phi(E))^{\vee}  \ar[d, "\sim" {anchor=north, rotate=90}] \\ 
    \cHom_S(\Phi, \rS_{\cC/S}) \ar[r, "\eta_{E}"] & \cHom_S(\Phi(E), \rS_{\cC/S}(E)) 
    \end{tikzcd}
\end{equation*} 
where $\eta_E$ is the morphism given by evaluating a natural transformation at $E$, 
the left vertical arrow is the equivalence of Lemma~\ref{lemma-HH-Serre}, 
and the right vertical arrow is Serre duality. 
\end{lemma} 

\begin{proof}
Let $K_{\Phi} \in \cC^{\vee} \otimes_{\Dperf(S)} \cC$ be the 
Fourier--Mukai kernel for $\Phi$. 
We consider the diagram 
\begin{equation*}
\begin{tikzcd} 
 \cHH_*(\cC/S, \Phi)^{\vee} \ar[r, "\ch_{E, \Phi}^{\vee}"] \ar[d, "\sim" {anchor=south, rotate=90}] & \cHom_{S}(E, \Phi(E))^{\vee}   \ar[d, "\sim" {anchor=north, rotate=90}] \\ 
\cHom_S(\ev_{\cC}(K_{\Phi}), \cO_S) \ar[r] \ar[d, "\sim" {anchor=south, rotate=90}] 
& \cHom_S(\ev_{\cC}(E \boxtimes \Phi(E)), \cO_S) \ar[d, "\sim" {anchor=north, rotate=90}] \\ 
\cHom_S(K_{\Phi}, \ev^!_{\cC}(\cO_S)) \ar[r] \ar[d, "\sim" {anchor=south, rotate=90}] & 
 \cHom_S(E \boxtimes \Phi(E), \ev_{\cC}^!(\cO_S)) \ar[d, "\sim" {anchor=north, rotate=90}] \\ 
 \cHom_S(\Phi, \rS_{\cC/S}) \ar[r, "\eta_E"] & \cHom_S(\Phi(E), \rS_{\cC/S}(E)). 
 \end{tikzcd}
\end{equation*} 
where: 
\begin{itemize}
\item The top square is given by dualizing the definition of $\ch_{E,\Phi}$. 
\item The middle square is given by adjunction. 
\item The bottom left vertical arrow arises from the equivalence $\cC^{\vee} \otimes_{\Dperf(S)} \cC \xrightarrow{\sim} \Fun_S(\cC, \cC)$ of Lemma~\ref{lemma-smooth-proper-dualizable}\eqref{C-times-C}, 
which sends $K_{\Phi}$ to $\Phi$ by construction and $\ev^!_{\cC}(\cO_S)$ to $\rS_{\cC/S}$ by Lemma~\ref{lemma-S-FM-kernel}. 
\item The bottom right vertical arrow is given by Lemma~\ref{lemma-yoneda-kernels}, using 
again that $\ev^!_{\cC}(\cO_S)$ is the kernel for $\rS_{\cC/S}$. 
\end{itemize} 
The top square and middle square commute by definition. 
Unwinding the equivalences invoked in the construction of the diagram, 
it is straightforward to check that the bottom square commutes, 
the composition of the left vertical arrows is the equivalence of Lemma~\ref{lemma-HH-Serre}, 
and the composition of the right vertical arrows is Serre duality. 
The outer square thus gives the desired diagram. 
\end{proof} 

\subsection{Hochschild cohomology} 
\label{section-Hochschild-cohomology}

\begin{definition}
Let $\cC$ be a small or presentable $S$-linear category. 
The \emph{Hochschild cohomology of $\cC$ over $S$} is the complex 
\begin{equation*}
\cHH^*(\cC/S) = \cHom_S(\id_{\cC}, \id_{\cC}) \in \Dqc(S) . 
\end{equation*} 
We use the notation 
\begin{align*}
\cHH^i(\cC/S) & = \cH^{i}(\cHH^*(\cC/S)), \\  
\HH^*(\cC/S) & = \rR\Gamma(\cHH^*(\cC/S)),  \\ 
\HH^i(\cC/S) & = \rH^{i}(\HH^*(\cC/S)), 
\end{align*} 
for the degree $i$ cohomology sheaf, the derived global sections, and the 
degree $i$ cohomology of $\cHH^*(\cC/S)$. 
\end{definition} 

\begin{remark}
\label{remark-hochschild-cohomology-ind}
If $\cC$ is a small $S$-linear category, then $\cHH^*(\cC/S) \simeq \cHH^*(\Ind(\cC)/S)$ \cite[Remark 4.2]{IHC-CY2}. 
\end{remark}

\begin{remark} 
There is also a version of Hochschild cohomology with coefficients in an endomorphism $\Phi \in \Fun_S(\cC,\cC)$ \cite[Definition 4.1]{IHC-CY2}, but we do not discuss it since we do not need it for our applications and we have no new results to add in this context.
\end{remark} 

Below we recall some basic results about Hochschild cohomology, and compute it for twisted derived categories. 
For further background, see \cite[\S4]{IHC-CY2} or \cite[\S2.2]{Fano3fold} and the references therein.\footnote{The same notation warning about \cite{IHC-CY2} as for Hochschild homology applies: there the complex $\cHH^*(\cC/S)$ is denoted instead by $\HH^*(\cC/S)$.}

\begin{lemma}[{\cite[Remark 2.5 and Lemma 2.6]{Fano3fold}}]
\label{lemma-HH-co-bc}
Let $\cC$ be an $S$-linear category. 
\begin{enumerate}
\item If $\cC$ is smooth and proper over $S$, then $\cHH^*(\cC/S) \in \Dperf(S)$. 
\item For any morphism of perfect derived algebraic stacks $g \colon T \to S$, there is a canonical equivalence 
\begin{equation*}
g^* \cHH^*(\cC/S) \simeq \cHH^*(\cC_T/T). 
\end{equation*} 
\end{enumerate} 
\end{lemma} 

Vologodsky's argument from Lemma~\ref{lemma-hochschild-homology-twisted-variety} also applies to the cohomological setting: 

\begin{lemma}[Vologodsky]
\label{lemma-hochschild-cohomology-twisted-variety} 
Let $f \colon X \to S$ be a flat morphism from a noetherian scheme with affine diagonal to a perfect noetherian scheme, and let $\alpha \in \Br(X)[n]$ with $n$ invertible on~$X$. 
Then there are equivalences 
\begin{equation*}
\cHH^*(\Dperf(X, \alpha)/S) \simeq 
 \cHom_S(\Delta_*^\alpha \cO_X, \Delta_*^{\alpha} \cO_X)
 \simeq 
 \cHom_S(\Delta_*\cO_X, \Delta_*\cO_X) 
 \simeq 
\cHH^*(\Dperf(X)/S)  
\end{equation*} 
where $\Delta \colon X \to X \times_S X$ is the diagonal and $\Delta_*^{\alpha}$, $\Delta^*_{\alpha}$ are the functors described in Example~\ref{lemma-duality-twisted-scheme}. 
\end{lemma} 

\begin{proof}
By Remark~\ref{remark-hochschild-cohomology-ind} and Lemma~\ref{lemma-compact-generation-twisted-D}, we may pass to the unbounded quasi-coherent derived category for the computation. 
Then the outer two equivalences follow from 
Examples~\ref{example-duality-data-perfectstack} and~\ref{lemma-duality-twisted-scheme}. 
Adopting the notation of 
the proof of Lemma~\ref{lemma-hochschild-homology-twisted-variety}, we have 
\begin{align*}
    \cHom_S(\Delta_*^\alpha \cO_X, \Delta_*^{\alpha} \cO_X) & \simeq \cHom_S(\iota_*(\cP), \iota_*(\cP)) \\
    &\simeq \cHom_S(\iota^*\iota_*(\cP), \cP) \\
    &\simeq \cHom_S(\cP^\vee \otimes \iota^*\iota_*(\cP), \cO_{\cX^\vee \times_X \cX}) \\
    &\simeq \cHom_S(\Delta^*_{\alpha} \Delta_*^{\alpha} \cO_X, \cO_X), 
\end{align*}
and similarly 
\begin{equation*}
 \cHom_S(\Delta_*\cO_X, \Delta_*\cO_X) \simeq \cHom_S(\Delta^*\Delta_*\cO_X, \cO_X). 
\end{equation*}
Since $\Delta^*_{\alpha} \Delta_*^{\alpha} \cO_X \simeq \Delta^* \Delta_* \cO_X$ by Lemma~\ref{lemma-hochschild-homology-twisted-variety}, 
this finishes the proof. 
\end{proof}

Using Hochschild cohomology, Anel and To\"{e}n \cite{anel-toen} introduced an analog for categories of the geometric notion of connectedness. 
To formulate it, note that if $\cC$ is an $S$-linear category, then there is a canonical morphism $\cO_S \to \cHH^0(\cC/S)$ (corresponding to the identity in $\cHom_S(\id_{\cC}, \id_{\cC})$). 

\begin{definition}
\label{definition-connected}
Let $\cC$ be an $S$-linear category.  
Then $\cC$ is \emph{connected over $S$} if for every morphism $T \to S$ from a perfect scheme $T$, $\cHH^i(\cC_T/T) = 0$ for $i < 0$ and the morphism $\cO_T \to \cHH^0(\cC_T/T)$ is an isomorphism. 
\end{definition}

\begin{lemma}
\label{lemma-connected-Dperf}
Let $f \colon X \to S$ be a flat proper surjective morphism of noetherian schemes 
with geometrically reduced and connected fibers. 
Let $\alpha \in \Br(X)$ be a Brauer class. 
Then $\Dperf(X, \alpha)$ is a connected $S$-linear category. 
\end{lemma} 

\begin{proof}
This follows from Lemma~\ref{lemma-hochschild-cohomology-twisted-variety}.  
\end{proof} 

Many interesting semiorthogonal components of derived categories of varieties can be shown to be connected by combining Lemma~\ref{lemma-connected-Dperf} with Kuznetsov's results on heights of exceptional collections \cite{kuznetsov-heights}; 
see for instance \cite[\S2.3]{Fano3fold} and Example~\ref{example-cubic-sevenfold} below. 

\subsection{Hodge theory} 
\label{section-Hodge-theory} 

In what follows, all local systems are taken with respect to the analytic topology, although following our conventions (\S \ref{conventions}) we suppress analytification from the notation

\begin{theorem}
\label{theorem-Ktop}
Let $\cC$ be a smooth proper $S$-linear category of geometric origin, where $S$ is a complex variety. 
Then there is functorially associated local system of finitely generated abelian groups 
$\Ktop[0](\cC/S)$ on $S$ (for the analytic topology), 
which underlies a weight $0$ variation of Hodge structures and satisfies the following: 
\begin{enumerate}
\item If $\cC = \langle \cC_1, \dots, \cC_m \rangle$ is an $S$-linear semiorthogonal decomposition, 
then there is an isomorphism 
\begin{equation*}
\Ktop[0](\cC/S) \simeq \Ktop[0](\cC_1/S) \oplus \cdots \oplus \Ktop[0](\cC_m/S) 
\end{equation*} 
induced by the projection functors onto the components $\cC_i$. 
\item When $f \colon X \to S$ is a smooth proper morphism of complex varieties, $\Ktop[0](\Dperf(X)/S)$ is the local system $U \mapsto \Ktop[0](X_U)$ of topological $\rK$-theory groups of the family, and there is an isomorphism 
\begin{equation*}
\Ktop[0](\Dperf(X)/S) \otimes \bQ \simeq \bigoplus_{k \in \bZ} \rR^{2k} f_* \bQ(k)
\end{equation*} 
of variations of rational Hodge structures on $S$ (where $(k)$ denotes the Tate twist). 
\item If $S = \Spec(\bC)$, in which case we write  $\Ktop[0](\cC)$ for $\Ktop[0](\cC/S)$, then 
the $p$-th graded piece of the Hodge filtration on $\Ktop[0](\cC) \otimes \bC$ is isomorphic to $\HH_{2p}(\cC) \coloneqq \rH^{-2p}(\cHH_*(\cC))$. 
For general $S$, the fiber of $\Ktop[0](\cC/S)$ over $s \in S(\bC)$ is $\Ktop[0](\cC_s)$. 
\end{enumerate} 
\end{theorem} 

\begin{proof}
The construction of $\Ktop[0](\cC/S)$ is due to 
Blanc in the absolute case \cite{blanc} and to Moulinos \cite{moulinos} in general. 
The essential ingredient in constructing the Hodge structure is 
the degeneration of the noncommutative Hodge-to-de Rham spectral sequence \cite{kaledin1, kaledin2, akhil-degeneration}. 
For details, see \cite[\S5.1]{IHC-CY2}. 
\end{proof} 

\begin{remark}
    When $\cC = \Dperf(X, \alpha)$ for a Brauer class $\alpha \in \BrAz(X)$, and $X \to S$ is a smooth proper morphism of complex varieties, then Theorem~\ref{theorem-Ktop} implies that there is a variation of Hodge structure on the local system $\Ktop[0](\Dperf(X, \alpha)/S)$, which we shall denote simply by $\Ktop[0]((X,\alpha)/S)$. Indeed, $\Dperf(X, \alpha)$ is of geometric origin by Lemma~\ref{lemma-D-SB}. 

    When the Brauer class is topologically trivial (Definition~\ref{def:topologically_trivial}), the results of \cite{hotchkiss-pi} calculate the associated variation of Hodge structure in terms of \emph{twisted Mukai structures}. This is explained in the special case when $X \to S$ is a family of abelian varieties in \S\ref{sec:hodge_classes_on_twisted_abelian_varieties}.
 \end{remark}


\section{Calabi--Yau categories} 
\label{section-CY-cats}

In this short section we introduce (families of) Calabi--Yau categories and specialize some of our earlier 
results to this setting. 
The main results of this paper only require $3$-dimensional Calabi--Yau categories, 
but as it is no harder we consider arbitrary dimensions. 

Fix a perfect derived algebraic stack $S$, which will serve as the base space for our discussion.

\subsection{Definition and examples}

\begin{definition}
\label{definition-CYn}
Let $n \in \bZ$. 
An \emph{$n$-dimensional Calabi--Yau (CY$n$) category over $S$} is an 
$S$-linear category $\cC$ such that: 
\begin{enumerate}
\item \label{CYn-connected-smooth-proper}
$\cC$ is smooth, proper, and connected over $S$. 
\item \label{CYn-Serre} 
$\rS_{\cC/S} = (- \otimes L)[n]$ for a line bundle $L$ on $S$. 
\end{enumerate} 
\end{definition} 

This property is preserved under base change; in particular, the fibers $\cC_s$ of $\cC$ are CY$n$ over $\kappa(s)$ for all $s \in S$. 

\begin{lemma}
Let $\cC$ be a CY$n$ category over $S$. 
Let $T \to S$ be a morphism from a perfect derived algebraic stack. 
Then $\cC_T$ is a CY$n$ category over $T$. 
\end{lemma}

\begin{proof} 
Connectedness is preserved under base change by Definition~\ref{definition-connected}, 
smoothness and properness are preserved under base change by Lemma~\ref{lemma-smooth-proper-bc}, and condition~\eqref{CYn-Serre} in Definition~\ref{definition-CYn} is preserved under base change by Lemma~\ref{lemma-bc-Serre}. 
\end{proof} 

\begin{remark} 
A variant of Definition~\ref{definition-CYn} in case $n=2$ appears in 
\cite[\S6.1]{IHC-CY2}. The definition there is stronger in that $\cC$ is required to be smooth and proper of geometric origin, but weaker in that condition~\eqref{CYn-Serre} is replaced with its fibral version and connectedness is replaced by a slightly weaker fibral condition (only requiring that $\cHH^0$ is the scalars). 
Whenever we use Hodge theory for categories in the sense of \S\ref{section-Hodge-theory}, we will restrict to CY$n$ categories of geometric origin, but we have not included this condition in our definition as it is not needed for other results. 
\end{remark} 

The following is the most important example of a CY category for this paper. 
\begin{lemma}
\label{lemma-twisted-CYn}
Let $f \colon X \to S$ be a smooth proper morphism of noetherian schemes of relative dimension $n$, with geometrically connected fibers and 
$\omega_f = f^*L$ for a line bundle $L$ on $S$. 
Let $\alpha \in \BrAz(X)$. 
Then $\Dperf(X, \alpha)$ is a CY$n$ category of geometric origin over $S$. 
\end{lemma} 

\begin{proof}
By Example~\ref{example-twisted-smooth-proper}, the category $\Dperf(X, \alpha)$ is smooth and proper of geometric origin over $S$. 
By Lemma~\ref{lemma-connected-Dperf}, the category $\Dperf(X, \alpha)$ is connected over $S$. 
Finally, by Example~\ref{example-serre-geometric} the relative Serre functor of $\Dperf(X, \alpha)$ is $(- \otimes L)[n]$. 
\end{proof} 

Many other examples of CY categories can be constructed as semiorthogonal components of Fano varieties, 
using the results of~\cite{kuznetsov-CY}. 
For a survey of the known CY2 examples, see \cite[\S6.2]{IHC-CY2}. 
Below we spell out an interesting CY3 example.  

\begin{example}
\label{example-cubic-sevenfold} 
Let $f \colon X \to S$ be a smooth family of cubic sevenfolds, defined by a section of 
$L \boxtimes \cO_{\bP^8}(3)$ on $S \times \bP^8$ for some line bundle $L$ on $S$. 
Then $\cO_X, \cO_X(1), \dots, \cO_X(5)$ is a relative exceptional collection over $S$ (in the sense of \cite[\S3.3]{stability-families}). 
The \emph{Kuznetsov component} $\Ku(X) \subset \Dperf(X)$ is defined by the $S$-linear semiorthogonal decomposition 
\begin{equation*}
\Dperf(X) = \llangle 
\Ku(X), f^*\Dperf(S), f^*\Dperf(S) \otimes \cO_X(1), \dots, 
f^*\Dperf(S) \otimes \cO_X(5) \rrangle .
\end{equation*} 
The category $\Ku(X)$ is CY3 of geometric origin over $S$. 
Indeed, $\Ku(X)$ is smooth and proper of geometric origin over $S$ by definition, 
connected over $S$ by \cite[Corollary 2.11]{Fano3fold}, 
and has relative Serre functor $\rS_{\Ku(X)/S} \simeq (- \otimes L^{\otimes 2})[3]$ by the results of \cite{kuznetsov-CY}.\footnote{The results of \cite{kuznetsov-CY} are stated for smooth varieties over a field, but they extend directly to the relative setting.} 
\end{example} 

As mentioned in \S\ref{section-IHC-CY3}, the integral Hodge conjecture holds for cubic sevenfolds and their Kuznetsov components:  

\begin{lemma}
\label{lemma-cubic-sevenfold}
Let $X \subset \bP^8$ be a complex smooth cubic sevenfold. 
Then the integral Hodge conjecture holds for $X$ (in all degrees), as well as for $\Ku(X)$. 
\end{lemma}

\begin{proof}
By \cite[Proposition 5.16]{IHC-CY2}, it suffices to prove the integral Hodge conjecture for $X$ in all degrees. 
By the Lefschetz hyperplane theorem, it is enough to show that $X$ contains a line, a $2$-plane, and a $3$-plane. 
By a dimension count, the only nontrivial case is a $3$-plane. The expected dimension of the space of $3$-planes is $0$, and 
a standard Chern class computation shows that the virtual number of such planes is nonzero (equal to $321489$), so $3$-planes must exist on $X$. 
\end{proof}

\subsection{Hochschild (co)homology and Chern characters} 
The Calabi--Yau condition implies a close relation between 
Hochschild homology and cohomology. 
Note that for any object $E \in \cC$ of an $S$-linear category, 
there is a canonical morphism 
\begin{equation}
\label{equation-action-morphism}
a_E \colon \cHH^*(\cC/S) \to \cHom_S(E,E)
\end{equation} 
given by evaluation at $E$. 

\begin{corollary}
\label{CY-HH} 
Let $\cC$ be a CY$n$ category over $S$ with $\rS_{\cC/S} = (-\otimes L)[n]$ for a line 
bundle $L$ on $S$. 
\begin{enumerate}
\item \label{CY-HH-co-ho}
There is an equivalence $\cHH_*(\cC/S)^{\vee} \simeq \cHH^*(\cC/S) \otimes L [n]$. 
\item \label{CY-ch-dual}
For any $E \in \cC$, there is a commutative diagram 
\begin{equation*}
    \begin{tikzcd}
    \cHH_*(\cC/S)^{\vee} \ar[rr, "\ch_{E}^{\vee}"] \ar[d, "\sim" {anchor=south, rotate=90}] && \cHom_{S}(E, E)^{\vee}  \ar[d, "\sim" {anchor=north, rotate=90}] \\ 
    \cHH^*(\cC/S) \otimes L [n] \ar[rr, "a_{E} \otimes \id"] && \cHom_S(E,E) \otimes L [n] 
    \end{tikzcd}
\end{equation*} 
where the left vertical arrow is the equivalence of~\eqref{CY-HH-co-ho} 
and the right vertical arrow is Serre duality. 

\item \label{CY-HH-degrees}
If $S$ is a $\bQ$-scheme, then we have $\cHH_*(\cC/S) \in \Dqc^{[-n,n]}(S)$. 
\end{enumerate} 
\end{corollary}

\begin{proof}
\eqref{CY-HH-co-ho} and \eqref{CY-ch-dual} follow from Lemmas~\ref{lemma-HH-Serre} and~\ref{lemma-ch-dual} with $\Phi = \id_{\cC}$. 
For part \eqref{CY-HH-degrees}, note that $\cHH^*(\cC/S) \in \Dqc^{\geq 0}(S)$ by the 
connectedness of $\cC$ over $S$, and therefore $\cHH_*(\cC/S)^{\vee} \in \Dqc^{\geq -n}(S)$ by~\eqref{CY-HH-co-ho}. 
But $\cHH_*(\cC/S)$ is self-dual by Theorem~\ref{theorem-HH}\eqref{theorem-mukai-pairing} and has locally free cohomology sheaves if $S$ is a $\bQ$-scheme by Theorem~\ref{theorem-HH}\eqref{theorem-smooth-proper-HH}, so~\eqref{CY-HH-degrees} follows. 
\end{proof}  


\newpage 
\part{Moduli} 
\label{part-moduli} 


\section{Moduli of objects in a category} 
\label{section-moduli} 

Fix a perfect scheme $S$ and a smooth proper $S$-linear category $\cC$. 
In this section, we discuss the moduli space of objects in $\cC$, following work of Lieblich~\cite{lieblich-moduli} and Toen--Vaqui\'{e} \cite{toen-moduli}. 
We will consider three incarnations of this space: as a derived stack, as a stack, and as an algebraic space. 
For applications later in the paper, we will primarily be interested in the latter two more classical incarnations, but the enhancement as a derived stack will be useful for proving results, as its deformation theory is more natural.  

\subsection{Classical moduli spaces} 
\label{section-classical-moduli} 
We consider the moduli functor 
\begin{alignat}{3}
\nonumber \cM_{\gl}(\cC/S) \colon  &  (\Sch/S)^{\op} ~~ &  \to & ~~ \Grpd & \\ 
\label{MCS} & \qquad T & \mapsto 
& \set{E \in \cC_T ~ | ~  \Ext^{<0}_{\kappa(t)}(E_t, E_t) = 0 \text{ for all } t \in T} & ,  
\end{alignat} 
where the right-hand side is considered as a groupoid (by truncating the $\infty$-groupoid spanned by the displayed objects in $\cC_T$). 
Following Lieblich, we call $\cM_{\gl}(\cC/S)$ the \emph{moduli stack of gluable objects} in $\cC$. 
We also consider the \emph{moduli stack of simple gluable objects} 
\begin{equation*} 
s\cM_{\gl}(\cC/S) \subset \cM_{\gl}(\cC/S)
\end{equation*} 
defined by the extra condition that $E \in \cC_T$ is \emph{simple}, i.e. 
$\Hom_{\kappa(t)}(E_t,E_t) = \kappa(t)$ for all $t \in T$. 

\begin{theorem}
\label{theorem-MC}
The functor $\cM_{\gl}(\cC/S)$ is an algebraic stack which is locally of finite presentation over $S$. 
Moreover, $s\cM_{\gl}(\cC/S) \subset \cM_{\gl}(\cC/S)$ is an open substack (hence also algebraic and locally of finite presentation over $S$), and is a $\bG_m$-gerbe over an algebraic space $s\rM_{\gl}(\cC/S)$ locally of finite presentation over $S$. 
\end{theorem} 

\begin{proof}
In the case where $\cC = \Dperf(X)$ for $X \to S$ a smooth proper morphism, this was first proved in \cite[Theorem 4.2.1 and Corollary 4.3.3]{lieblich-moduli}. 
As explained in \cite[Proposition 9.2 and Lemma 9.8]{stability-families}, this can be used to deduce the result in case 
$\cC \subset \Dperf(X)$ is an $S$-linear semiorthogonal component (which is the only case needed for the applications in this paper). 
The general case is proved in \cite{toen-moduli}. 
\end{proof} 

\begin{remark}
As shown in~\cite{lieblich-moduli, toen-moduli}, a moduli stack of objects in $\cC$ can be defined under weaker hypotheses than $\cC$ being smooth and proper. 
However, the definition of the moduli stack then becomes more complicated. 
As we will not need this generality for our main results, we stick to the smooth and proper case. 
\end{remark} 

\subsection{Derived moduli spaces} 
\label{section-derived-moduli} 
We use the language of derived algebraic geometry, 
our conventions on which are recalled in \S\ref{section-DAG}. 
The \emph{derived moduli stack of objects} in $\cC$ is the functor 
\begin{equation*}
\fM(\cC/S) \colon (\dAff/S)^{\op} \to \Grpd_{\infty}
\end{equation*}
whose value on $T$ is the $\infty$-groupoid $(\cC_T)^{\simeq}$ obtained from $\cC_T$ by discarding non-invertible $1$-morphisms. 

\begin{remark}
The $T$-points of $\fM(\cC/S)$ for a derived algebraic stack $T$ can be described by the same formula as in the case when $T$ is a derived affine scheme. 
Indeed, we may write $T = \colim_{U \in \dAff/T} U$ as the colimit of the derived affine schemes over $T$. Then 
\begin{align*}
\Hom( T, \fM(\cC/S) ) & \simeq \lim_{U \in \dAff/T} \Hom(U, \fM(\cC/S)) \\ 
& \simeq \lim_{U \in \dAff/T} (\cC_U)^{\simeq} \\ 
& \simeq \left( \lim_{U \in \dAff/T} \cC_U \right)^{\simeq} , 
\end{align*} 
where the last line holds since the operation $(-)^{\simeq}$ of passing to the maximal sub-$\infty$-groupoid is right adjoint to the inclusion $\Grpd_{\infty} \to \Cat_{\infty}$ of $\infty$-groupids into $\infty$-categories. On the other hand, we have 
\begin{align*}
\lim_{U \in \dAff/T} \cC_U & = 
\lim_{U \in \dAff/T} \cC \otimes_{\Dperf(S)} \Dperf(U)  \\ 
& \simeq \cC \otimes_{\Dperf(S)} \left( \lim_{U \in \dAff/T} \Dperf(U) \right) \\ 
& \simeq \cC \otimes_{\Dperf(S)} \Dperf(T) \\ 
& = \cC_T,  
\end{align*} 
where the second line holds since by dualizability of $\cC \in \Cat_S$ the operation $\cC \otimes_{\Dperf(S)} -$ admits a left adjoint given by $\cC^{\vee} \otimes_{\Dperf(S)} -$, 
and the third and fourth lines hold by definition. 
\end{remark} 

Now we can state the main foundational result about the derived moduli stack of objects. 

\begin{theorem}[{\cite[Theorem 3.6]{toen-moduli}}] 
\label{theorem-derived-M-algebraic}
The functor $\fM(\cC/S)$ is a derived locally algebraic stack which is locally of finite presentation over $S$. 
\end{theorem} 

One of the key advantages of the derived moduli stack $\fM(\cC/S)$ is that its (co)tangent complex admits a simple description. 

\begin{theorem}[{\cite[Corollary 3.17]{toen-moduli}}] 
\label{theorem-LMder}
Let $g \colon T \to \fM(\cC/S)$ be a morphism from a derived algebraic stack $T$, 
corresponding to an object $E \in \cC_T$. 
Then the pullback of the relative cotangent complex $L_{\fM(\cC/S)/S}$ of $\fM(\cC/S) \to S$ to $T$ is given by 
\begin{equation*}
g^* L_{\fM(\cC/S)/S} \simeq \left( \cHom_T(E,E)[1] \right)^{\vee} . 
\end{equation*} 
\end{theorem}

\subsection{Derived enhancements} 
\label{subsection-derived-enhancement} 

We define the \emph{higher moduli stack of object} in $\cC$ to be the classical truncation 
\begin{equation*}
\cM(\cC/S) \coloneqq \fM(\cC/S)_{\cl} \colon (\Aff/S)^{\op} \to \Grpd_{\infty}, 
\end{equation*}
which by Theorem~\ref{theorem-derived-M-algebraic} is a higher locally algebraic stack locally of finite presentation over $S$. 

\begin{example}
\label{example-LM}
Let $\cM \to \cM(\cC/S)$ be a Zariski open immersion. 
Then by Remark~\ref{remark-derived-enhancements} there is a corresponding 
Zariski open immersion of derived locally algebraic stacks 
\begin{equation*}
\fM \to \fM(\cC/S)
\end{equation*} 
which recovers $\cM \to \cM(\cC/S)$ upon taking classical truncations. 
As observed in \cite[Corollary 3.21 and 3.22]{toen-moduli}, 
the natural morphisms $s\cM_{\gl}(\cC/S) \to \cM_{\gl}(\cC/S) \to \cM(\cC/S)$ are Zariski open immersions; thus for example we may take $\cM = s\cM_{\gl}(\cC/S)$ or $\cM = \cM_{\gl}(\cC/S)$, giving Zariski open immersions of derived enhancements 
\begin{equation*}
s\fM_{\gl}(\cC/S) \to \fM_{\gl}(\cC/S) \to \fM(\cC/S) . 
\end{equation*} 
\end{example} 

As illustrated by Theorem~\ref{theorem-LMder}, derived moduli problems often have 
well-behaved cotangent complexes. 
This leads to control on the cotangent complex of the underived moduli problem, via the 
following observation. 

\begin{lemma}[{\cite[Proposition 1.2]{toen-obstruction}}]
\label{lemma-enhancement-obs}
Let $\fX$ be a derived algebraic stack over $S$, let $\cX$ be its truncation, 
and let $i \colon \cX \to \fX$ be the canonical map. 
Then the cone $L_{\cX/\fX}$ of the morphism 
\begin{equation*}
i^*L_{\fX/S} \to L_{\cX/S}
\end{equation*} 
satisfies $L_{\cX/\fX} \in \Dqc^{\leq -2}(\cX)$. 
\end{lemma}


\section{Groups of autoequivalences} 
\label{section-autoequivalences}  
Fix a perfect scheme $S$ and a smooth proper $S$-linear category $\cC$. 
In this section, we discuss the geometric structure of the group of autoequivalences of $\cC$, which, like the moduli space of objects in $\cC$, comes in three guises. 
In particular, we describe its cotangent complex and study properties of its identity component.

\subsection{Groups of autoequivalences} 
We denote by $\Aut_S(\cC) \in \Grpd_{\infty}$ the space of automorphisms of $\cC$ considered as an object of the $\infty$-category $\Cat_S$ of $S$-linear categories. 
The \emph{stack of autoequivalences} of $\cC$ is the functor 
\begin{alignat}{3}
\nonumber \cAut(\cC/S) \colon  &  (\Sch/S)^{\op} ~~ &  \to & ~~ \Grpd_{\infty} & \\ 
\label{AutC}  & \qquad T & \mapsto 
& \Aut_T(\cC_T) & . 
\end{alignat}  

In Proposition~\ref{proposition-cAut} below we explain the basic properties of $\cAut(\cC/S)$.  
Recall $\Fun_S(\cC,\cC)$ is the $S$-linear category of $S$-linear endofunctors of $\cC$.  
\begin{lemma}
\label{lemma-Fun}
The $S$-linear category $\Fun_S(\cC,\cC)$ is smooth and proper over $S$. 
Further, if $T \to S$ is a morphism from a locally algebraic derived stack $T$, then there is an equivalence 
\begin{equation*}
(\Fun_S(\cC,\cC))_T \simeq \Fun_T(\cC_T, \cC_T). 
\end{equation*} 
of $T$-linear categories. 
\end{lemma} 

\begin{proof}
As $\cC$ is assumed smooth and proper over $S$, 
the first claim holds by the (proof of) \cite[Lemma 2.6]{Fano3fold}, 
while the second holds by the (proof of) \cite[Lemma 4.3]{IHC-CY2}. 
\end{proof} 

\begin{proposition}
\label{proposition-cAut} 
The functor $\cAut(\cC/S)$ is a higher locally algebraic stack which is locally of finite presentation over $S$ and admits a  
Zariski open immersion 
\begin{equation}
\label{equation-immersion-Aut-Fun}
\cAut(\cC/S) \to \cM(\Fun_S(\cC,\cC)/S). 
\end{equation} 
Moreover, if $\cC$ is a connected $S$-linear category (in the sense of Definition~\ref{definition-connected}), then $\cAut(\cC/S)$ is $1$-truncated (i.e. takes values in $\Grpd$), is an algebraic stack locally of finite presentation over $S$, and is a $\bG_m$-gerbe over a group algebraic space $\Aut(\cC/S)$ locally of finite presentation over $S$. 
\end{proposition} 

\begin{proof} 
This essentially follows from the arguments in the proof of \cite[Corollary 3.24]{toen-moduli}, but for convenience we sketch the proof. 
The morphism $\cAut(\cC/S) \to \cM(\Fun_S(\cC,\cC)/S)$ is given on $T \in \Sch/S$ by the natural morphism $\Aut_T(\cC_T) \to \Fun_T(\cC_T, \cC_T) \simeq (\Fun_S(\cC, \cC))_T$, where the equivalence is given by Lemma~\ref{lemma-Fun}. 
That this morphism is a Zariski open immersion is shown just after Lemma 3.25 in \cite{toen-moduli}. 
It follows that $\cAut(\cC/S)$ is a locally algebraic stack locally of finite presentation over $S$, as $\cM(\Fun_S(\cC,\cC)/S)$ is one. 

If $\cC$ is connected, then \cite[Lemma 3.3]{Fano3fold} implies that for any point $\Phi \in \Aut_T(\cC_T)$ we have 
\begin{equation*}
\pi_i(\Aut_T(\cC_T), \Phi) \cong 
\begin{cases}
\cO_T(T)^{\times} & \text{if } i = 1, \\
0 & \text{if } i \geq 2. 
\end{cases} 
\end{equation*} 
Therefore, $\cAut(\cC/S)$ is $1$-truncated. 
It follows that $\cAut(\cC/S)$ is an algebraic stack by \cite[Lemma 2.19]{toen-moduli}. 
Finally, the algebraic space $\Aut(\cC/S)$ can be obtained as the $\bG_m$-rigidification of $\cAut(\cC/S)$.  
\end{proof} 

By our discussion of derived enhancements in \S\ref{subsection-derived-enhancement}, it follows that the Zariski open immersion \eqref{equation-immersion-Aut-Fun} corresponds to a Zariski open immersion of derived locally algebraic stacks 
\begin{equation*}
\fAut(\cC/S) \to \fM(\Fun_S(\cC,\cC)/S)
\end{equation*} 
which recovers \eqref{equation-immersion-Aut-Fun} upon taking classical truncations.
Explicitly, 
\begin{equation*}
\fAut(\cC/S) \colon (\dAff/S)^{\op} \to \Grpd_{\infty}
\end{equation*} 
is defined by the same formula~\eqref{AutC} as $\cAut(\cC/S)$ but on derived schemes $T$. 
We call $\fAut(\cC/S)$ the \emph{derived stack of autoequivalences} of $\cC$. 

\begin{remark}
By construction, $\cAut(\cC/S)$ is a group stack over $S$ (i.e. a group object in the category $\St/S$) which acts on the stacks 
\begin{equation*}
s\cM_{\gl}(\cC/S) \subset \cM_{\gl}(\cC/S) \subset \cM(\cC/S)
\end{equation*} 
while $\fAut(\cC/S)$ is a group derived stack over $S$ (i.e. a group object in the category $\dSt/S$) which acts on the derived enhancements  
\begin{equation*}
s\fM_{\gl}(\cC/S) \subset \fM_{\gl}(\cC/S) \subset \fM(\cC/S). 
\end{equation*}  
When $\cC$ is a connected $S$-linear category, then the action of $\cAut(\cC/S)$ on $s\cM_{\gl}(\cC/S)$ induces an action of the group algebraic space $\Aut(\cC/S)$ on $s\rM_{\gl}(\cC/S)$. 
\end{remark} 

\subsection{The cotangent complex} 

\begin{lemma}
\label{lemma-LG}
Let $\fG$ be a group derived algebraic stack over $S$, with
identity section $e \colon S \to \fG$. 
Then there is an equivalence 
\begin{equation*}
L_{\fG/S} \simeq e^*L_{\fG/S} \otimes \cO_{\fG}. 
\end{equation*} 
\end{lemma} 

\begin{proof}
Consider the cartesian diagram 
\begin{equation*} 
    \begin{tikzcd}
\fG \times_S \fG \ar[r, "m"] \ar[d, "\pr_{1}"'] & \fG \ar[d, "\pi"] \\ 
\fG \ar[r, "\pi"] & S
    \end{tikzcd}
\end{equation*} 
where $m$ is the multiplication map, $\pr_1$ is the first projection, and $\pi$ is the structure morphism. 
If $b = (\id_{\fG}, e \circ \pi) \colon \fG \to \fG \times_S \fG$, then $m \circ b = \id_{\fG}$ and hence pulling back we find an equivalence 
\begin{equation*}
L_{\fG/S} \simeq b^*m^*L_{\fG/S}. 
\end{equation*}  
On the other hand, base change for the cotangent complex gives 
\begin{equation*}
m^*L_{\fG/S} \simeq L_{\fG \times_S \fG/\fG} \simeq \pr_2^* L_{\fG/S}. 
\end{equation*} 
Since $\pr_2 \circ b = e \circ \pi$, combining the above two observations gives 
\begin{equation*}
L_{\fG/S} \simeq b^*\pr_2^*L_{\fG/S} \simeq \pi^* e^* L_{\fG/S}. \qedhere
\end{equation*}
\end{proof}

\begin{corollary}
\label{corollary-LfAut}
Let $\fG$ be a group derived algebraic stack over $S$ which is an open subgroup of $\fAut(\cC/S)$. 
Then there is an equivalence 
\begin{equation*}
L_{\fG/S} \simeq (\cHH^*(\cC/S)[1])^{\vee} \otimes \cO_{\fG}. 
\end{equation*} 
\end{corollary} 

\begin{proof}
By Lemma~\ref{lemma-LG}, it suffices to show that 
$e^*L_{\fG/S} \simeq (\cHH^*(\cC/S)[1])^{\vee}$ where $e \colon S \to \fG$ is the identity section. 
By assumption, $\fG$ is a Zariski open inside the derived locally algebraic stack 
$\fM(\Fun_S(\cC,\cC)/S)$, and the composition $S \xrightarrow{ e } \fG \to \fM(\Fun_S(\cC,\cC)/S)$ is classified by the object $\id_{\cC} \in \Fun_S(\cC,\cC)$. 
Therefore, by Theorem~\ref{theorem-LMder} we have 
\begin{equation*}
e^*L_{\fG/S} \simeq (\cHom_{S}(\id_{\cC}, \id_{\cC})[1])^{\vee} = (\cHH^*(\cC/S)[1])^{\vee}, 
\end{equation*} 
where the final equality holds by the definition of Hochschild cohomology. 
\end{proof} 

\begin{lemma}
\label{lemma-L-cAut}
Let $\cG$ be a group higher algebraic stack over $S$ which is an open subgroup of 
$\cAut(\cC/S)$. 
Assume that $\cG \to S$ is smooth. 
Then there is an equivalence 
\begin{equation*}
L_{\cG/S} \simeq \tau^{\geq 0}((\cHH^*(\cC/S)[1])^{\vee} \otimes \cO_{\cG})
\end{equation*} 
\end{lemma} 

\begin{proof}
Let $\fG$ 
be the open subgroup of $\fAut(\cC/S)$ corresponding to $\cG \subset \cAut(\cC/S)$, 
so that $i \colon \cG \to \fG$ is a derived enhancement. 
As $\cG \to S$ is smooth we have $L_{\cG/S} \in \Dqc^{[0,1]}(\cG)$, 
and therefore the morphism $i^*L_{\fG/S} \to L_{\cG/S}$ factors via 
a morphism $\alpha \colon \tau^{\geq 0}i^*L_{\fG/S} \to L_{\cG/S}$. 
By Lemma~\ref{lemma-enhancement-obs} the morphism 
$\alpha$ induces an isomorphism on cohomology sheaves in degrees $\geq 0$, 
while by Corollary~\ref{corollary-LfAut} the source of $\alpha$ is 
precisely $\tau^{\geq 0}((\cHH^*(\cC/S)[1])^{\vee} \otimes \cO_{\cG})$. 
\end{proof} 

As we explain in Lemma~\ref{lemma-Aut0} below, in nice situations there is a canonical choice of $\cG$ satisfying the hypotheses of 
Lemma~\ref{lemma-L-cAut}.

\subsection{The identity component} 
\label{section-identity-component}
\begin{definition}[Identity component]
\label{definition-identity-component}
If $G \to \Spec(k)$ is a group algebraic space locally of finite type over a field, we denote by $G^{0} \subset G$ the \emph{identity component} of $G$, i.e. the open subgroup given by the connected component of the identity. 
If $G \to S$ is a group algebraic space locally of finite type over a scheme $S$, we consider the union of the the identity components $G_s^0 \subset G_s$, $s \in S$, of all fibers; when this is an open subset $G^0 \subset G$, then it is a subgroup called the \emph{identity component} of $G \to S$, and we say ``the identity component exists'' to signify this situation. 
\end{definition} 

\begin{remark}
When it exists, the formation of the identity component commutes with base change \cite[Lemma 5.1]{kleiman-picard}. 
\end{remark} 

\begin{definition}[Identity component of the stack of autoequivalences]
\label{definition-identity-component-C}
If $\cC$ is a connected $S$-linear category, then Proposition~\ref{proposition-cAut} shows there is a $\bG_m$-gerbe $\cAut(\cC/S) \to \Aut(\cC/S)$ of the stack of autoequivalences over the space of autoequivalences. 
If the identity component of $\Aut(\cC/S) \to S$ exists, then we denote it by $\Aut^0(\cC/S)$, and define the 
\emph{identity component} $\cAut^0(\cC/S)$ of $\cAut(\cC/S)$ as the subgroup given by the fiber product diagram 
\begin{equation*}
\xymatrix{
\cAut^0(\cC/S) \ar[r] \ar[d] & \cAut(\cC/S) \ar[d] \\ 
\Aut^0(\cC/S) \ar[r] & \Aut(\cC/S), 
}
\end{equation*} 
where the horizontal arrows are open immersions and the vertical arrows are $\bG_m$-gerbes. 
\end{definition} 

\begin{lemma}
\label{lemma-Aut0}
Let $\cC$ be a connected, smooth, and proper $S$-linear category over a locally noetherian $\bQ$-scheme $S$. 
Assume that the function $s \mapsto \dim_{\kappa(s)} \HH^1(\cC_s/\kappa(s))$ is constant on $S$, say equal to $d \in \bZ$. 
\begin{enumerate}
\item \label{lemma-Aut0-1}
The identity component $\Aut^0(\cC/S)$ of $\Aut(\cC/S) \to S$ exists, and is of finite type and relative dimension $d$ over $S$. 
\item \label{lemma-Aut0-2} 
If $S$ is reduced, then $\Aut^0(\cC/S)$ (and hence also $\cAut^0(\cC/S)$) is smooth over $S$.
\end{enumerate} 
\end{lemma} 

\begin{proof}
\eqref{lemma-Aut0-1} 
By construction, the fiber of $\Aut(\cC/S)$ over $s \in S$ is $\Aut(\cC_s/\kappa(s))$. 
This is a group algebraic space locally of finite presentation over the field $\kappa(s)$, 
which has characteristic $0$ by assumption. 
Therefore $\Aut(\cC_s/\kappa(s))$ is smooth over $\kappa(s)$ 
(see \cite[\href{https://stacks.math.columbia.edu/tag/047N}{Tag 047N}]{stacks-project}). 
Combining Corollary~\ref{corollary-LfAut}, Lemma~\ref{lemma-enhancement-obs}, and the fact that $\cAut(\cC_s/\kappa(s)) \to \Aut(\cC_s/\kappa(s))$ is a $\bG_m$-gerbe, we find that the tangent space to $\Aut(\cC_s/\kappa(s))$ at the identity is 
isomorphic to $\HH^1(\cC_s/\kappa(s))$, whose dimension is independent of $s \in S$. 
In particular, we have shown that the $\Aut^0(\cC_s/\kappa(s))$ are all smooth of the same dimension $d$. 
Now applying \cite[15.6.3 and 15.6.4]{EGAIV3}\footnote{Technically, here and below the results invoked from \cite{EGAIV3} are stated for schemes, but they are also valid for algebraic spaces, cf. \cite{romagny}.}, we find that the union $\Aut^0(\cC/S)$ of the $\Aut^0(\cC_s/\kappa(s))$ is open 
in $\Aut(\cC/S)$, and the structure morphism $\Aut^0(\cC/S) \to S$ is universally open. 
That $\Aut^0(\cC/S) \to S$ is of finite type follows by the same argument as in \cite[Proposition 5.20]{kleiman-picard} for Picard schemes. 
Finally, if $S$ is reduced, then $\Aut^0(\cC/S) \to S$ is flat by \cite[15.6.7]{EGAIV3}, and hence 
smooth by \cite[17.5.1]{EGAIV4}. 
\end{proof} 

The hypotheses of Lemma~\ref{lemma-Aut0} are automatic in the CY setting: 
 
\begin{lemma}
\label{lemma-HH1-CY}
Let $\cC$ be a CY$n$ category over a locally noetherian $\bQ$-scheme $S$. 
Then the function $s \mapsto \dim_{\kappa(s)} \HH^1(\cC_s/\kappa(s))$ is locally 
constant on $S$. 
In particular, the identity component $\Aut^0(\cC/S)$ exists and is of finite type over $S$, and if $S$ is reduced it (as well as $\cAut^0(\cC/S)$) is smooth over $S$. 
\end{lemma} 

\begin{proof}
By Theorem~\ref{theorem-HH}\eqref{theorem-mukai-pairing} and Corollary~\ref{CY-HH}\eqref{CY-HH-co-ho}, we have $\HH^1(\cC_s/\kappa(s)) \simeq \HH_{n-1}(\cC_s/\kappa(s))$. 
But by Theorem~\ref{theorem-HH} the vector spaces $\HH_{n-1}(\cC_s/\kappa(s))$ are the fibers of the finite locally free 
sheaf $\cHH_{n-1}(\cC/S)$, and hence their dimensions are locally constant. 
\end{proof} 

When it exists, the identity component of a linear category must preserve all semiorthogonal decompositions: 

\begin{lemma}
\label{lemma-autoequivalence-sod}
Let $\cC$ be a connected, smooth, and proper $S$-linear category over a scheme $S$ such that $\Aut^0(\cC/S) \to S$ exists. 
Let $\cC = \llangle \cC_1, \dots, \cC_n \rrangle$ be an $S$-linear semiorthogonal decomposition. 
Then restriction to any subcategory $\cC_i \subset \cC$ determines a homomorphism 
\begin{equation*}
\Aut^0(\cC/S) \to \Aut^0(\cC_i/S). 
\end{equation*} 
\end{lemma} 

\begin{proof}
For a scheme $T$ over $S$, a $T$-point of $\Aut^0(\cC/S)$ corresponds to a $T$-linear autoequivalence $\Phi$ of $\cC_T$ such that for every $t \in T$ the autoequivalence $\Phi_t$ of the fiber $\cC_t$ is contained in $\Aut^0(\cC_t)$. 
We must show that for any $i$, $\Phi$ restricts to a $T$-linear autoequivalence of $(\cC_i)_T$ 
(as it then follows that the restriction is indeed a $T$-point of $\Aut^0(\cC_i)$). 
By renaming, we may assume $T = S$. 
The claim is equivalent to the vanishing of the composite functor $\pr_j \circ \Phi \circ \pr_i$ for all $j \neq i$, where $\pr_k \colon \cC \to \cC_k$ is the projection functor of the semiorthogonal decomposition. 
In general, the vanishing of an $S$-linear functor $\Psi \colon \cC \to \cC$ may be checked on geometric fibers. 
Indeed, for any object $C \in \cC$ and geometric point $s$ of $S$, by base change we have 
\begin{equation*}
\cHom_S(\Psi(C), \Psi(C))_{s} \simeq \cHom_{\kappa(s)}(\Psi_s(C_s), \Psi_s(C_s)); 
\end{equation*}  
if $\Psi_s$ vanishes for every $s$, then so must the complex $\cHom_S(\Psi(C), \Psi(C)) \in \Dperf(S)$ as its fibers do, 
which implies $\Psi$ vanishes. 
Therefore, the claim that $\Phi$ restricts to an $S$-linear autoequivalence of $\cC_i$ reduces to the case where $S = \Spec(k)$ is the spectrum of an algebraically closed field, 
in which case it follows from Lemma~\ref{lemma-connected-family-aut-sod} below. 
\end{proof} 

\begin{lemma}
\label{lemma-connected-family-aut-sod} 
Let $\cC$ be a connected, smooth, and proper $k$-linear category over an algebraically closed field $k$.  
Let $\cD \subset \cC$ be  $k$-linear semiorthogonal component. 
Let $T$ be a connected scheme of finite type over $k$. 
Let $\Phi \in \Aut_T(\cC_T)$ be a $T$-linear autoequivalence. 
Assume there exists a point $0 \in T(k)$ such that $\Phi_0(\cD) = \cD$. 
Then $\Phi_t(\cD) = \cD$ for all $t \in T(k)$. 
\end{lemma} 

\begin{proof}
The proof is essentially the same as that of \cite[Theorem 3.9]{okawa-nonexistence}, but we include the argument for convenience.

First we claim that for any $k$-linear semiorthogonal component $\cD' \subset \cC$, the set 
\begin{equation*}
U(\cD') = \set{ t \in T(k) \st \Phi_t(\cD) = \cD' } 
\end{equation*} 
is open in $T(k)$. 
It suffices to show that the locus of $t \in T(k)$ with $\Phi_t(\cD) \subset \cD'$ is open, 
since then by replacing $\cD$ with $\cD'$ and $\Phi$ with $\Phi^{-1}$ we also find that the locus of $t \in T(k)$ with $\Phi_t^{-1}(\cD') \subset \cD$ is open. 
Let $\ccE' = {^\perp}\cD'$ be the left orthogonal of $\cD'$, so that there is a semiorthogonal decomposition $\cC = \llangle \cD' , \ccE' \rrangle$. 
Then we must show that the locus of $t \in T(k)$ such that $\cHom_k(E', \Phi_t(D)) = 0$ for all $E' \in \ccE', D \in \cD$ is open. 
It follows from \cite[Lemma~3.9]{antieau-gepner} (see also \cite[Lemma 2.6]{toen-generator}) that $\cC$ admits a generator, and hence (by projection) 
so does any semiorthogonal component. 
The vanishing $\cHom_k(E', \Phi_t(D)) = 0$ for all $E' \in \ccE', D \in \cD$ is equivalent to the vanishing when $E' \in \ccE'$ and $D \in \cD$ are chosen generators. 
Let $E'_T \in \cC_T$ and $D_T \in \cC_T$ be the pullbacks of $E'$ and $D$ along $T \to \Spec(k)$. 
Then by base change, the fiber of the complex
\begin{equation*}
\cHom_T(E'_T, \Phi(D_T)) \in \Dperf(T) 
\end{equation*} 
over $t \in T(k)$ is $\cHom_k(E', \Phi_t(D))$, and hence its vanishing is an open condition on $t$. 
This completes the proof that $U(\cD') \subset T(k)$ is open. 

Now let $\Lambda$ denote the set of semiorthogonal components of $\cC$ obtained as images of $\Phi_t$ for $t \in T(k)$, i.e. 
$\Lambda = \set{ \Phi_t(\cD) \st  t \in T(k) }$. 
Then $T(k) = \coprod_{\cD' \in \Lambda} U(\cD')$ is a disjoint union of open sets. 
By connectedness of $T$ we deduce that $T(k) = U(\cD)$. 
\end{proof}

\begin{lemma}
\label{lemma-Aut0-CYn-proper} 
Let $f \colon X \to S$ be a smooth proper morphism of constant relative dimension $n$ with geometrically connected fibers and 
$\omega_f = f^*L$ for a line bundle $L$ on $S$, where $S$ is a reduced locally noetherian $\bQ$-scheme. 
Let $\alpha \in \BrAz(X)$.  
Then the identity component $\Aut^0(\Dperf(X, \alpha)/S)$ exists and is smooth and proper over $S$. 
\end{lemma} 

\begin{proof} 
The category $\Dperf(X, \alpha)$ is CY$n$ over $S$ by Lemma~\ref{lemma-twisted-CYn}. 
Hence $\Aut^0(\Dperf(X, \alpha)/S)$ exists and is smooth and of finite type over $S$ by Lemma~\ref{lemma-HH1-CY}. 
To show properness, we first prove two related group spaces are proper.

Choose a $\bmu_n$-gerbe $\pi \colon \cX \to X$ of class $\alpha$. 
Consider the automorphism and Picard stacks of $\cX$ over $S$, denoted 
$\cAut(\cX/S)$ and $\cPic(\cX/S)$, which are gerbes over algebraic spaces $\Aut(\cX/S)$ and $\Pic(\cX/S)$. 
Arguing as in Lemma~\ref{lemma-HH1-CY} shows that the identity components $\Aut^0(\cX/S)$ and $\Pic^0(\cX/S)$ exist and are of finite type, smooth, and of constant relative dimension over $S$. 
We claim that both are also proper over~$S$. 
For $\Pic^0(\cX/S)$, this holds for instance because the pullback map $\Pic^0(X/S) \to \Pic^0(\cX/S)$ is an isomorphism. 
For $\Aut^0(\cX/S)$, we instead have a natural morphism $\Aut^0(\cX/S) \to \Aut^0(X/S)$. 
In order to show that $\Aut^0(\cX/S) \to S$ is proper, by \cite[15.7.11]{EGAIV3} we may pass to geometric fibers and assume that $S = \Spec(k)$ for an algebraically closed field $k$. 
On Lie algebras, the homomorphism $\Aut^0(\cX/k) \to \Aut^0(X/k)$ induces 
the map $\rH^0(\rT_{\cX/k}) \to \rH^0(\rT_{X/k})$, which is an isomorphism 
since $\rT_{\cX/k} \to \pi^* \rT_{X/k}$ is. 
It follows that $\Aut^0(\cX/k) \to \Aut^0(X/k)$ is \'{e}tale with kernel a finite group scheme $H$ over $k$; in fact, since $\Aut^0(X/k)$ is connected, the homomorphism $\Aut^0(\cX/k) \to \Aut^0(X/k)$ must be surjective, and is identified with the quotient by $H$. 
In particular, the morphism $\Aut^0(\cX/k) \to \Aut^0(X/k)$ is proper, so we reduce to showing that $\Aut^0(X/k)$ is proper over $k$. 
For this, it suffices to show that that $\Aut^0(X/k)$ does not contain a copy of $G = \bG_a$ or $G = \bG_m$. 
If so, then $X$ would be birational to a product $G \times Y$ by \cite[Theorem 10]{algebraic-groups}, but $X$ cannot be ruled since $\omega_X \cong \cO_X$. 

The actions of $\Aut^0(\cX/S)$ and $\Pic^0(\cX/S)$ on $\Dperf(\cX)$ give rise to a homomorphism 
\begin{equation*}
\Aut^0(\cX/S) \times_S \Pic^0(\cX/S) \to \Aut^0(\Dperf(\cX)/S). 
\end{equation*} 
By Lemma~\ref{lemma-mun-gerbe} the category $\Dperf(X, \alpha)$ is an $S$-linear semiorthogonal component of $\Dperf(\cX)$, 
so by Lemma~\ref{lemma-autoequivalence-sod} there is an induced homomorphism 
\begin{equation*}
\Aut^0(\Dperf(\cX)/S) \to \Aut^0(\Dperf(X, \alpha)/S). 
\end{equation*} 
Consider the composition 
\begin{equation*}
\rho \colon \Aut^0(\cX/S) \times_S \Pic^0(\cX/S)  \to \Aut^0(\Dperf(X, \alpha)/S)
\end{equation*} 
of the above homomorphisms. 
By the previous paragraph the source of $\rho$ is proper over $S$, 
so by \cite[\href{https://stacks.math.columbia.edu/tag/08AJ}{Tag 08AJ}]{stacks-project} to 
show that $\Aut^0(\Dperf(X, \alpha)/S)$ is proper over $S$ it suffices to show $\rho$ is surjective, 
for which we may reduce to the case where $S = \Spec(k)$ for an algebraically closed field $k$. 
On Lie algebras, this induces a map 
\begin{equation*} 
\rH^0(\rT_{\cX/k}) \oplus \rH^1(\cO_{\cX}) \to \HH^1(\Dperf(X, \alpha)/k). 
\end{equation*} 
Under the isomorphisms $\rH^0(\rT_{\cX/k})\cong \rH^0(\rT_{X/k})$ and $\rH^1(\cO_{\cX}) \cong \rH^1(\cO_X)$ and the HKR isomorphism $\HH^1(\Dperf(X)/k) \cong \rH^0(\rT_{X/k}) \oplus \rH^1(\cO_X)$, the map on Lie algebras is identified with the isomorphism $\HH^1(\Dperf(X)/k) \cong \HH^1(\Dperf(X, \alpha)/k)$ from Lemma~\ref{lemma-hochschild-cohomology-twisted-variety}.  
It follows that $\rho$ is \'{e}tale, and hence surjective as the target is a connected group. 
\end{proof} 

\begin{remark}
\label{remark-Aut0-proper}
It seems plausible that in the setting of Lemma~\ref{lemma-HH1-CY}, $\Aut^0(\cC/S)$ is necessarily proper over $S$. 
Lemma~\ref{lemma-Aut0-CYn-proper} gives the (Brauer twisted) geometric case of this statement. 
\end{remark}


\section{Moduli of objects modulo autoequivalences} 
\label{section-moduli-objects-modulo-autoequivalence}

Fix a perfect scheme $S$ and a smooth proper $S$-linear category $\cC$. 
So far in \S\ref{section-moduli}-\S\ref{section-autoequivalences} we have constructed various incarnations of the moduli space of objects in $\cC$, 
as well as corresponding group spaces that act on them. 
Our interest now is in their quotient spaces. 
In particular, in Proposition~\ref{proposition-quotient-enhancement} we describe under certain hypotheses  
the cotangent complex of the derived version of the quotient; 
this forms a key ingredient in our theory of reduced DT invariants developed in Part~\ref{part-DT-theory}. 

\subsection{Generalities on quotients} 
Let $\fX \in \dSt/S$ be a derived stack over $S$ and $\fG \in \dSt/S$ a group derived stack over $S$ that acts on~$\fX$. 
The action defines a simplicial diagram  
\begin{equation*}
\xymatrix{
 \cdots 
 \ar@<1.5ex>[r] \ar@<.5ex>[r] \ar@<-.5ex>[r]  \ar@<-1.5ex>[r] & 
 \fX \times_S \fG \times_S \fG 
\ar@<1ex>[r] \ar@<0ex>[r] \ar@<-1ex>[r] & 
\fX \times_S \fG 
\ar@<.5ex>[r]   \ar@<-.5ex>[r]  &  \fX
 }
\end{equation*} 
in $\dSt/S$, whose colimit $\fX/\fG$ is by definition the \emph{quotient derived stack} of $\fX$ by $\fG$. 
Note that $\dSt/S$ admits all colimits (and limits), so this definition makes sense. 
Similarly, for $\cX \in \St/S$ a higher stack acted on by a group higher stack $\cG \in \St/S$, 
we may form the \emph{quotient higher stack} $\cX/\cG$.  

\begin{remark}
\label{remark-quotient-spaces}
Suppose as above that $\fX$ is a derived stack acted on by a group derived stack~$\fG$. 
Let $\cX = \fX_{\cl}$ and $\cG = \fG_{\cl}$ be the classical truncations. 
Then $\cG$ is a group higher stack and we have $(\fX/\fG)_{\cl} \simeq \cX/\cG$. 
Indeed, this follows from the commutation of classical truncation 
with colimits and limits (Remark~\ref{remark-derived-colimits}). 

Conversely, let $\cX$ be a higher stack acted on by a group higher stack~$\cG$, 
and let \mbox{$\fX = \iota(\cX)$} and $\fG = \iota(\cG)$ be their derived extensions. 
Then $\fG$ does \emph{not} automatically inherit the structure of a group derived stack 
which acts on $\fX$, because in general derived extension does not preserve fiber products. 
However, if $\cG \to S$ is flat, then 
$\fG$ is naturally a group higher stack which acts on $\fX$ and $\iota(\cX/\cG) \simeq \fX/\fG$; 
indeed, this follows from the fact that derived extension preserves pullbacks along flat morphisms of higher algebraic stacks and commutes with colimits (Remark~\ref{remark-derived-colimits}). 
Thus, in this situation there is no harm in our usual abuse of notation by which we omit 
$\iota$ when thinking of a higher stack as a derived stack. 
\end{remark} 

\begin{lemma}
\label{lemma-quotient-algebraic}
Let $\fG$ be a group derived algebraic stack over $S$ whose structure morphism 
$\fG \to S$ is flat and locally of finite presentation. 
Let $\fX$ be a derived algebraic stack on which $\fG$ acts. 
Then the quotient $\fX/\fG$ is a derived algebraic stack over $S$. 
Similarly, if $\cG$ is a group higher algebraic stack over $S$, flat and locally of finite presentation 
over $S$, which acts on a higher algebraic stack $\cX$, then the quotient $\cX/\cG$ is a higher algebraic stack over $S$. 
\end{lemma} 

\begin{proof}
Consider the cartesian diagram 
\begin{equation} 
\vcenter{
\label{action-diagram}
\xymatrix{
\fX \times_S \fG \ar[r]^{a} \ar[d]_{\pr_{\fX}} & \fX \ar[d]^{q} \\ 
\fX \ar[r]^{q} & \fX/\fG
}}
\end{equation} 
where $q \colon \fX \to \fX/\fG$ is the quotient morphism, 
$a \colon \fX \times_S \fG \to \fX$ is the action morphism, and 
$\pr_{\fX} \colon \fX \times_S \fG \to \fX$ is the projection. 
The morphism $\pr \colon \fX \times_S \fG \to \fX$ is flat and locally of finite presentation, 
being a base change of $\fG \to S$, 
and thus $q \colon \fX \to \fX/\fG$ is a flat and locally of finite presentation surjection. 
It follows by a result of To\"{e}n \cite{toen-descent} that $\fX/\fG$ is a derived algebraic stack. 
(The use of \cite{toen-descent} can be avoided at the expense of assuming instead that $\fG \to S$ is smooth.) This proves the first claim of the lemma, and the second follows similarly. 
\end{proof} 

\begin{lemma}
\label{lemma-L-quotient} 
Let $\fG$ be a group derived algebraic stack over $S$ whose structure morphism 
$\fG \to S$ is flat and locally of finite presentation. 
Let $\fX$ be a derived algebraic stack on which $\fG$ acts. 
Then the relative cotangent complex of $\fX \to \fX/\fG$ is given by 
\begin{equation*}
L_{\fX/(\fX/\fG)} \simeq e^*L_{\fG/S} \otimes \cO_{\fX} 
\end{equation*} 
where $e \colon S \to \fG$ is the identity section and the right-hand side denotes 
the pullback of $e^*L_{\fG/S}$ along the structure morphism $\fX \to S$. 
\end{lemma} 

\begin{proof}
First note that by Lemma~\ref{lemma-quotient-algebraic}, the quotient $\fX/\fG$ is indeed a 
derived algebraic stack. 
Considering the diagram~\eqref{action-diagram}, base change for the cotangent 
complex gives 
\begin{equation*} 
a^*L_{\fX/(\fX/\fG)} \simeq L_{\fX \times_S \fG/\fX} \simeq \pr_{\fG}^*L_{\fG/S}
\end{equation*} 
where $\pr_{\fG} \colon \fX \times_S \fG \to \fG$ is the projection. 
If $b = (\id_{\fX}, e \circ \pi) \colon \fX \to \fX \times_S \fG$ where $\pi \colon \fG \to S$ is the structure morphism, then $a \circ b = \id_{\fX}$ and 
hence pulling back the above equivalence along $b$ gives 
\begin{equation*}
L_{\fX/(\fX/\fG)} \simeq b^*\pr_{\fG}^*L_{\fG/S}. 
\end{equation*} 
Now the result follows from Lemma~\ref{lemma-LG}.
\end{proof} 

\subsection{Quotient of the moduli space of objects} 
Now we can describe the cotangent complex of a quotient of a moduli stack of objects in $\cC$. 

\begin{proposition}
\label{proposition-quotient-enhancement}
Let $\cG$ be a group higher algebraic stack over $S$ which is an open 
subgroup of $\cAut(\cC/S)$, and assume that $\cG \to S$ is smooth. 
Let $\cM \to \cM(\cC/S)$ be a Zariski open which is preserved by the action of $\cG$. 
Let $\fM \to \fM(\cC/S)$ be the corresponding Zariski open. 
\begin{enumerate}
\item 
\label{quotient-enhancement} 
The derived algebraic stack $\fM$ inherits a $\cG$-action such that the canonical derived enhancement morphism 
$i \colon \cM \to \fM$ is $\cG$-equivariant. 
The quotient $\fM/\cG$ is a derived enhancement of $\cM/\cG$,  
with the derived enhancement morphism $\cM/\cG \to \fM/\cG$ induced by $i$ 
upon passing to quotients. 

\item 
\label{quotient-L}
Let $\cE \in \cC_{\fM}$ be the universal object. Consider the morphism 
\begin{equation}
\label{equation-beta}
\beta \colon (\cHom_{\fM}(\cE, \cE)[1])^{\vee} \xrightarrow{\, (a_{\cE}[1])^{\vee} \,} (\cHH^*(\cC_{\fM}/\fM)[1])^{\vee} \xrightarrow{\, \,}  
\tau^{\geq 0}((\cHH^*(\cC/S)[1])^{\vee}) \otimes \cO_{\fM} 
\end{equation} 
in $\Dqc(\fM)$, where $(a_{\cE}[1])^{\vee}$ is the dual of the shift of the canonical morphism~\eqref{equation-action-morphism} and the last morphism is the pullback to $\fM$ of the truncation map 
\begin{equation}
\label{HH-truncation} 
(\cHH^*(\cC/S)[1])^{\vee} \to \tau^{\geq 0}((\cHH^*(\cC/S)[1])^{\vee}) . 
\end{equation} 
Then if $q \colon \fM \to \fM/\cG$ is the quotient morphism, there is an exact triangle 
\begin{equation}
\label{fM-mod-G-cotangent}
q^*L_{(\fM/\cG)/S} \xrightarrow{\, \,} (\cHom_{\fM}(\cE, \cE)[1])^{\vee} \xrightarrow{\, \beta \,} 
\tau^{\geq 0}((\cHH^*(\cC/S)[1])^{\vee}) \otimes \cO_{\fM}  .  
\end{equation} 

\item 
\label{quotient-L-descend}
The exact triangle \eqref{fM-mod-G-cotangent} descends to an exact triangle 
\begin{equation}
\label{fM-mod-G-cotangent-descend}
L_{(\fM/\cG)/S} \xrightarrow{\, \,} (\cHom_{\fM}(\cE, \cE)[1])^{\vee} \xrightarrow{\, \overline{\beta} \,} 
\tau^{\geq 0}((\cHH^*(\cC/S)[1])^{\vee}) \otimes \cO_{\fM}  . 
\end{equation} 
in $\Dqc(\fM/\cG)$, where by abuse of notation we denote the second and third term by the same symbol as their pullbacks 
to $\fM$. 
\end{enumerate}
\end{proposition} 

\begin{proof}
\eqref{quotient-enhancement}
Because $\cG \to S$ is smooth, as noted in Remark~\ref{remark-quotient-spaces} its derived extension $\iota(\cG)$ is a group derived algebraic stack, and as usual we may safely conflate $\cG$ and $\iota(\cG)$. 
Note that $\cG$ acts on $\fM(\cC/S)$ via the homomorphism $\cG \to \cAut(\cC/S) \to \fAut(\cC/S)$. 
By the correspondence between Zariski opens in $\cM(\cC/S)$ and $\fM(\cC/S)$, it follows that 
$\cG$ acts on $\fM$ such that $\cM \to \fM$ is $\cG$-equivariant. That 
$\fM/\cG$ is a derived enhancement of $\cM/\cG$ follows from Remark~\ref{remark-quotient-spaces}. 

\eqref{quotient-L} Consider the exact triangle 
\begin{equation*}
q^*L_{(\fM/\cG)/S} \to L_{\fM/S} \to L_{\fM/(\fM/\cG)}
\end{equation*} 
associated to the morphism $q \colon \fM \to \fM/\cG$. 
By Theorem~\ref{theorem-LMder}, we have an identification 
\begin{equation*}
L_{\fM/S} \simeq (\cHom_{\fM}(\cE, \cE)[1])^{\vee}. 
\end{equation*} 
By Lemmas~\ref{lemma-L-quotient} and~\ref{lemma-L-cAut}, we have an identification  
\begin{equation*}
L_{\fM/(\fM/\cG)} \simeq \tau^{\geq 0}((\cHH^*(\cC/S)[1])^{\vee}) \otimes \cO_{\fM} . 
\end{equation*} 
Now we explain why, under these identifications, the map $\gamma \colon L_{\fM/S} \to L_{\fM/(\fM/\cG)}$ is given by $\beta$. 
Let $\fG$ be the open subgroup of $\fAut(\cC/S)$ corresponding to $\cG$, so that $\fG$ acts on $\fM$, 
and consider the 
commutative diagram 
\begin{equation} 
\vcenter{
\label{equation-M-action}
\xymatrix{
&& \fM \times_S \fG  \ar[dr]^{a'} & \\ 
\fM \ar[r]^-{b} & \fM \times_S \cG \ar[rr]^{a} \ar[d]_{\pr_{\fM}} \ar[ur]^{i} && \fM \ar[d]^{q} \\ 
& \fM \ar[rr]^{q} &&  \fM/\cG
}}
\end{equation} 
where $b = (\id_{\fM}, e)$ with $e \colon S \to \fG$ the identity section, 
$i$ is the closed embedding induced by $\cG \to \fG$, 
$a$ and $a'$ are the corresponding action morphisms, and the square is cartesian. 
Then $a \circ b = \id_{\fM}$, so $\gamma \colon L_{\fM/S} \to L_{\fM/(\fM/\cG)}$ is identified with its pullback along $a \circ b$. 
Let us consider its pullback along $a$. 
Under the base change isomorphism $a^*L_{\fM/(\fM/\cG)} \simeq L_{\fM \times_S \cG/\fM}$, $a^*\gamma$ is 
identified with the composition of the natural maps 
\begin{equation*}
a^*L_{\fM/S} \to L_{\fM \times_S \cG/S} \to L_{\fM \times_S \cG/\fM}. 
\end{equation*} 
In view of the upper triangle in~\eqref{equation-M-action}, 
this factors as 
\begin{equation*}
a^*L_{\fM/S} \to i^*L_{\fM \times_S \fG/S} \to i^*L_{\fM \times_S \fG/\fM} \to  L_{\fM \times_S \cG/\fM}.
\end{equation*} 
Using that $L_{\fM \times_S \fG/\fM} \simeq L_{\fG/S} \otimes \cO_{\fM \times_S \fG}$ and 
$L_{\fM \times_S \cG/\fM} \simeq L_{\cG/S} \otimes \cO_{\fM \times_S \fG}$, and pulling back along $b$, 
we find that $\gamma$ factors as 
\begin{equation*}
L_{\fM/S} \xrightarrow{\, \gamma_1 \,} L_{\fG/S} \otimes \cO_{\fM} \xrightarrow{\, \gamma_2 \,} L_{\cG/S} \otimes \cO_{\cM} , 
\end{equation*} 
where the $\gamma_2$ is the pullback to $\fM$ of the morphism $L_{\fG/S} \otimes \cO_{\cG} \to L_{\cG/S}$ 
induced by $\cG \to \fG$. 
By Corollary~\ref{corollary-LfAut} and (the proof of) Lemma~\ref{lemma-L-cAut}, $\gamma_2$ is 
identified with the pullback to $\fM$ of the truncation morphism~\eqref{HH-truncation}. 
Note that by base change for Hochschild cohomology (Lemma~\ref{lemma-HH-co-bc}), the pullback of 
$(\cHH^*(\cC/S)[1])^{\vee}$ to $\fM$ can be identified with $(\cHH^*(\cC_{\fM}/\fM)[1])^{\vee}$, the middle term in the composition~\eqref{equation-beta} defining $\beta$. 
Finally, tracing through the above description of $\gamma_1$ and using the identification 
$L_{\fM/S} \simeq (\cHom_{\fM}(\cE, \cE)[1])^{\vee}$ of Theorem~\ref{theorem-LMder} shows 
that $\gamma_1$ is dual to the natural morphism $\cHH^*(\cC_{\fM}/\fM)[1] \to \cHom_\fM(\cE, \cE)[1]$. 

\eqref{quotient-L-descend} 
As shown above, the triangle~\eqref{fM-mod-G-cotangent} is precisely the exact triangle 
\begin{equation*}
q^*L_{(\fM/\cG)/S} \to L_{\fM/S} \to L_{\fM/(\fM/\cG)}
\end{equation*} 
associated to the morphism $q \colon \fM \to \fM/\cG$. 
As the first morphism in this triangle naturally descends to $\fM/\cG$, so does the second. 
\end{proof} 

\begin{remark}
\label{remark-derived-obs-thy-rewritten}
Suppose that, in the situation of Proposition~\ref{proposition-quotient-enhancement}, 
the Hochschild cohomology $\cHH^*(\cC/S)$ has locally free cohomology sheaves; 
for instance, by Corollary~\ref{CY-HH} and Theorem~\ref{theorem-HH}\eqref{theorem-smooth-proper-HH}, this holds if $\cC$ is CY and $S$ is a $\bQ$-scheme. 
Then 
\begin{equation*}
\tau^{\geq 0}((\cHH^*(\cC/S)[1])^{\vee}) \otimes \cO_{\fM} \simeq 
\tau^{\geq 0}((\cHH^*(\cC_{\fM}/\fM)[1])^{\vee})). 
\end{equation*} 
Moreover, in this case all of the objects appearing in the triangles~\eqref{fM-mod-G-cotangent} 
and~\eqref{fM-mod-G-cotangent-descend} are in fact perfect, as 
the second two objects in each triangle are always so by our assumption that $\cC$ is smooth and proper over $S$. 
\end{remark} 

\subsection{Preservation by the identity component} 
The following observation is useful for constructing pairs $\cG$ and $\cM$ to which Proposition~\ref{proposition-quotient-enhancement} applies. 

\begin{lemma}
\label{lemma-aut0-preserve-open}
Let $\cC$ be a connected $S$-linear category such that 
$\Aut^0(\cC/S)$ exists (see Definition~\ref{definition-identity-component}). 
Let $\cM \subset \cM_{\gl}(\cC/S)$ be a Zariski open substack which is universally closed over $S$. 
Then the action of the identity component $\cAut^0(\cC/S)$ of the stack of autoequivalences on $\cM_{\gl}(\cC/S)$ preserves $\cM$. 
\end{lemma} 

\begin{remark} 
In particular, if $\cAut^0(\cC/S) \to S$ is smooth (a condition for which we gave simple criteria in Lemmas~\ref{lemma-Aut0} and~\ref{lemma-HH1-CY}), 
the hypotheses of Proposition~\ref{proposition-quotient-enhancement} are satisfied. 
\end{remark} 

\begin{proof} 
The question is topological so we may check it on fibers and reduce to the case where $S = \Spec(k)$. 
Let $\cG = \cAut^0(\cC/k)$ and $G = \Aut^0(\cC/k)$. 
Let $\cE \in \cC_{\cM}$ be the universal object, and let 
$\Phi \in \Aut_{\cG}(\cC_{\cG})$ be the universal autoequivalence. 
Let $\cE_{\cG \times_k \cM} \in \cC_{\cG \times_k \cM}$ be the pullback of $\cE$, let
$\Phi_{\cG \times_k \cM} \in \Aut_{\cG \times_k \cM}(\cC_{\cG \times_k \cM})$ be the base changed autoequivalence, and set 
\begin{equation*}
\cF \coloneqq \Phi_{\cG \times_k \cM}(\cE_{\cG \times_k \cM}) \in \cC_{\cG \times_k \cM}. 
\end{equation*} 
Then $\cF$ is a family of gluable objects in $\cC$ parameterized by $\cG \times_k \cM$, 
whose fiber over any point $(g,m) \in \cG \times_k \cM$ is 
\begin{equation*}
\cF_{(g,m)} = \Phi_{g}(\cE_m), 
\end{equation*} 
where $\Phi_g$ is the fiber of $\Phi$ over $g$ (i.e. the autoequivalence of $\cC$ corresponding to $g$). 

As $\cM \subset \cM_{\gl}(\cC/S)$ is an open substack, the locus of points $(g,m) \in \cG \times_k \cM$ such that $\cF_{(g,m)} \in \cM$ is open.  
As $\cM \to \Spec(k)$ is universally closed, the locus $\cU \subset \cG$ of points $g$ such that $\cF_{(g,m)} \in \cM$ for all $m \in \cM$ is therefore also open. 
If $1 \in \cG$ denotes the identity element, then 
$\cF_{(1,m)} = \cE_m$ for all $m \in \cM$, so $\cU$ is an open neighborhood of $1$. 
Let $U \subset G$ be the image of $\cU$ under the $\bG_m$-gerbe $\cG \to G$. 
Replacing $U$ by $U \cap U^{-1}$, it follows that for all $g \in U$ we have 
\begin{equation*}
\Phi_g(\cM) = \cM, 
\end{equation*} 
where here we regard $\Phi_g$ as an automorphism of $\cM_{\gl}(\cC/S)$. 
Therefore, the locus of $g \in G$ such that $\Phi_g(\cM) = \cM$ is an open subgroup of $G$, 
and hence equal to $G$, as $G$ is connected. 
\end{proof} 

Any moduli space $\cM$ of semistable objects of fixed class with respect to a reasonable notion of stability will satisfy the hypotheses of Lemma~\ref{lemma-aut0-preserve-open}. 
In \S\ref{section-moduli-twisted-sheaves} we will spell out the case of moduli of semistable twisted sheaves, and 
in \S\ref{section-moduli-ss} the case of moduli of semistable objects in a category. 


\newpage 
\part{Stability} 
\label{part-stability} 


\section{Stability of twisted sheaves} 
\label{section-stability-twisted-sheaves} 
In this section we set up our conventions on slope and Gieseker stability of twisted sheaves, mostly following \cite{lieblich-moduli-twisted}. 

\subsection{Chern characters}
\label{section-chern-character}
We will use Vistoli's theory of rational Chow groups for DM stacks over a field \cite{vistoli-chow},
which we denote by $\CH^*(-)_{\bQ}$. 
This theory is formally similar to that of Chow groups of varieties, and in particular comes equipped with a formalism of Chern classes for perfect complexes. 
As the notation suggests, when applied to a variety $\CH^*(X)_{\bQ}$ agrees with the rationalization of the usual Chow group. 

We apply this to a $\bmu_n$-gerbe $\pi \colon \cX \to X$ over a variety $X$ representing a class $\alpha \in \Br(X)$. 
In this case, the results of \cite{vistoli-chow} show that the pushforward map $\pi_* \colon \CH^*(\cX)_{\bQ} \to \CH^*(X)_{\bQ}$ is an isomorphism. 
As any $E \in \Dperf(X, \alpha)$ is in particular a perfect complex on $\cX$, we may consider its Chern character $\ch(E) \in \CH^*(\cX)_{\bQ}$; this construction gives a homomorphism 
\begin{equation*}
\ch \colon \rK_0(X, \alpha) \to \CH^*(\cX)_{\bQ}. 
\end{equation*} 
Of course, $\ch_0(E) = \rk(E)$ is simply the rank. 

\begin{lemma}
\label{lemma-ch-numerical}
If $X$ is a smooth proper variety, then the Chern character descends to a homomorphism 
\begin{equation*}
\ch \colon \rK_{\num}(X, \alpha) \to \CH_{\num}^*(\cX)_{\bQ},
\end{equation*} 
where $\CH^*_{\num}(\cX)_{\bQ}$ denotes the quotient of $\CH^*(\cX)_{\bQ}$ by numerical equivalence. 
\end{lemma} 

\begin{proof}
Given $E \in \Dperf(X, \alpha)$ be a numerically trivial object, 
we must show $\ch(E) \in \CH^*(\cX)_{\bQ}$ is numerically trivial. 
Let $f \colon Y \to X$ be a finite cover killing $\alpha$ as in Lemma~\ref{lemma-finite-cover-kill}, and consider the pullback diagram 
\begin{equation*}
\xymatrix{
\cY \ar[r]^{f'} \ar[d]_{\pi'} & \cX \ar[d]^{\pi} \\ 
Y \ar[r]^{f} & X
}
\end{equation*} 
Since $f'_* \circ f'^*$ is multiplication by $\deg(f)$ on $\CH^*(\cX)_{\bQ}$, 
it suffices to show that the class $f'^*\ch(E) = \ch(f'^*E) \in \CH^*(\cY)_{\bQ}$ 
is numerically trivial. 

Since $f^*(\alpha) = 0$, there exists an invertible twisted sheaf $L$ 
on $\cY$ (Lemma~\ref{lemma-essentially-trivial}). Then $f'^*E \otimes L^{\vee} \in \Dperf^0(\cY)$ is an untwisted complex on $\cY$, with $\ch(f'^*E \otimes L^{\vee}) = \ch(f'^*E)$. 
In this way, we reduce to proving the lemma when $\alpha = 0$. 
But in this case it follows from Hirzebruch--Riemann--Roch that $\ch \colon \rK_0(X) \to \CH^*(X)_{\bQ}$ 
induces a map $\rK_\num(X) \to \CH_{\num}^*(X)_{\bQ}$ (which is in fact an isomorphism after tensoring by $\bQ$). 
\end{proof} 

In our later discussion of stability conditions, it will be useful to consider Chern characters twisted by a real divisor.  
\begin{definition}
\label{definition-twisted-ch}
Let $X$ be a smooth proper variety equipped with a Brauer class \mbox{$\alpha \in \Br(X)$}, represented by a $\bmu_n$-gerbe $\pi \colon \cX \to X$, 
and a real divisor $D \in \Div(X)_{\bR}$. 
For a twisted perfect complex $E \in \Dperf(X, \alpha) = \Dperf^1(\cX)$, the \emph{$D$-twisted Chern character} is 
\begin{equation*}
\ch^D(E) \coloneqq e^{-\pi^*(D)}\ch(E) \in \CH^*(\cX)_{\bR}
\end{equation*} 
where $e^{-\pi^*(D)}$ is the formal exponential of $-\pi^*(D) \in \CH^1(\cX)_{\bR}$. 
\end{definition} 

By abuse of notation, we will often denote the pullback of a divisor along $\pi \colon \cX \to X$ by the same symbol. 
With this convention, in low degrees we have 
\begin{align*}
\ch_0^D & = \ch_0 = \rk , \\ 
\ch_1^D & = \ch_1 - D \ch_0 , \\ 
\ch_2^D & = \ch_2 - D \ch_1 + \frac{D^2}{2}  \ch_0 ,  \\ 
\ch_3^D & = \ch_3 - D  \ch_2 + \frac{D^2}{2}\ch_1 - \frac{D^3}{6} \ch_0.
\end{align*} 
Let us also note that by Lemma~\ref{lemma-ch-numerical}, the $D$-twisted Chern character defines a homomorphism 
\begin{equation*}
\ch^D \colon \rK_{\num}(X, \alpha) \to \CH_{\num}^*(\cX)_{\bR}. 
\end{equation*}

\begin{remark}
We have followed Lieblich \cite{lieblich-moduli-twisted} in our definition of the Chern character for twisted sheaves. There is another possible definition, going back to Yoshioka \cite{yoshioka-twisted-sheaves}, which can be made when $\alpha \in \Br_{\mathrm{Az}}(X)$ is represented by an Azumaya algebra, or equivalently when there exists a locally free twisted sheaf $V$ on a $\bmu_n$-gerbe $\pi \colon \cX \to X$ of class $\alpha$. 
Given such a $V$, for $E \in \Dperf(X, \alpha)$ we may define 
\begin{equation*}
\ch^V(E) \coloneqq \frac{\ch(\pi_*(E \otimes V^{\vee}))}{\sqrt{\ch(\pi_*(V \otimes V^{\vee}))}} \in \CH^*(X)_{\bQ}. 
\end{equation*} 
In general $\ch^V$ differs from the Chern character $\ch$ defined above. 
For instance, $\ch^V$ depends on the choice of $V$, whereas $\ch$ is canonically defined. 
However, for most purposes, including ours in this paper, one could use $\ch^V$ in place of $\ch$. 
For instance, both Chern characters lead to the same notion of slope (semi)stability for twisted sheaves, which satisfy a Bogomolov inequality, as studied in \S\ref{section-slope-stability-twisted-sheaves} below. 
\end{remark}

\subsection{Slope stability} 
\label{section-slope-stability-twisted-sheaves} 
Slope stability of twisted sheaves is defined analogously to the classical case. 
\begin{definition}
\label{definition-twisted-stability}
Let $X$ be a smooth projective variety equipped with a class \mbox{$\alpha \in \Br(X)$} represented by a $\bmu_n$-gerbe $\pi \colon \cX \to X$, 
a real ample divisor $\omega$, and a real divisor $D$. 
For a twisted coherent sheaf $E \in \Coh(X, \alpha) = \Coh^1(\cX)$, the $(\omega, D)$-\emph{slope} is  
\begin{equation*}
\mu_{\omega, D}(E) \coloneqq 
\begin{cases}
\frac{ \omega^{\dim X -1} \cdot \ch^D_1(E) }{ \omega^{\dim X} \cdot \ch^D_0(E) } & \text{if } \ch^D_0(E) \neq 0 , \\
+\infty & \text{if } \ch^D_0(E) = 0, 
\end{cases}
\end{equation*} 
We say $E$ is \emph{$\mu_{\omega, D}$-semistable} (or \emph{slope semistable} if $(\omega, D)$ are understood) if for all nontrivial $\alpha$-twisted subsheaves $F \subset E$, 
the inequality $\mu_{\omega,D}(F) \leq \mu_{\omega,D}(E)$ holds, 
and $E$ is \emph{$\mu_{\omega, D}$-stable} if for all $F$ the inequality is strict. 
When $D = 0$ we write $\mu_{\omega} = \mu_{\omega,0}$. 
\end{definition} 

\begin{remark}
The notion of $\mu_{\omega, D}$-(semi)stability is independent of $D$. 
Classically only the case $D = 0$ was considered, but the ability to vary $D$ is useful when studying stability conditions. 
\end{remark}

Many results for semistable sheaves have analogs in the twisted setting. For instance, the same arguments as in the untwisted case show the existence of a \emph{Harder--Narasimhan (HN) filtration} with respect to $\mu_{\omega, D}$-stability for any $E \in \Coh(X, \alpha)$. 
That is, there exists a filtration 
\begin{equation*}
0 = E_0 \subset E_1 \subset \cdots \subset E_m = E
\end{equation*} 
such that the factors $A_i = E_i/E_{i-1}$ are $\mu_{\omega,D}$-semistable and 
\begin{equation*}
\mu_{\omega,D}(A_1) > \mu_{\omega,D}(A_2) > \cdots > \mu_{\omega,D}(A_m). 
\end{equation*}  
We use the notation $\mu_{\omega,D}^+ \coloneqq \mu_{\omega,D}(A_1)$ and $\mu_{\omega,D}^{-}(E) = \mu_{\omega,D}(A_m)$. 

Combined with Lemma~\ref{lemma-finite-cover-kill}, the following lemma sometimes allows one to reduce statements to the untwisted setting, as we will illustrate in the proof of the twisted Bogomolov inequality below. 

\begin{lemma}
\label{lemma-pullback-ss}
Let $f \colon Y \to X$ be a finite surjective morphism of smooth projective varieties over a field of characteristic $0$. 
Let $\alpha \in \Br(X)$ and let $H$ be an ample divisor on $X$. 
Let $E \in \Coh(X, \alpha)$ be a torsion free twisted sheaf. 
Then $E$ is $\mu_H$-semistable if and only if the pullback $f^*E \in \Coh(Y, f^*\alpha)$ is $\mu_{f^*H}$-semistable. 
\end{lemma} 

\begin{proof}
This follows by the same argument as in \cite[Lemma 3.2.2]{huybrechts-lehn} which treats the untwisted version of the assertion. 
\end{proof}

\subsection{Bogomolov inequality} 
The Bogomolov inequality plays an important role in the construction of stability conditions. 
To state it, we need to introduce the discriminant. 

\begin{definition}
Let $X$ be a smooth projective variety with a class \mbox{$\alpha \in \Br(X)$} represented by a $\bmu_n$-gerbe $\pi \colon \cX \to X$. 
For $E \in \Dperf(E,\alpha) = \Dperf^1(\cX)$, the \emph{discriminant} is 
\begin{equation*}
\Delta(E) \coloneqq \ch_1(E)^2 - 2\ch_0(E) \ch_2(E) \in \CH^2(\cX)_{\bQ}. 
\end{equation*}  
\end{definition} 

\begin{remark}
If $D$ is a real divisor on $X$, then we have 
\begin{equation*}
\Delta(E) = \ch^D_1(E)^2 - 2\ch^D_0(E) \ch^D_2(E), 
\end{equation*} 
so allowing a twist by $D$ leads to the same notion of discriminant. 
\end{remark} 

\begin{theorem}[Twisted Bogomolov inequality]
\label{theorem-twisted-bogomolov}
Let $X$ be a smooth projective variety over a field of characteristic $0$, 
equipped with a Brauer class $\alpha \in \Br(X)$ and an ample divisor $H$.  
Then for any a torsion free $\mu_H$-semistable twisted sheaf $E \in \Coh(X, \alpha)$, we have 
\begin{equation*}
 H^{\dim X -2} \Delta(E)  \geq 0 
\end{equation*} 
\end{theorem} 

\begin{proof}
This is a special case of \cite[Proposition 3.2.3.13]{lieblich-moduli-twisted} where the base field is allowed to be of arbitrary characteristic; in that generality, the inequality is more complicated as it requires a correction term.  
In our case, the argument is simple. 
Indeed, using Lemmas~\ref{lemma-finite-cover-kill} and~\ref{lemma-pullback-ss}, 
by the argument from Lemma~\ref{lemma-ch-numerical} we may reduce to proving the claim in the untwisted case, which is nothing but the well-known Bogomolov inequality \cite[Theorem 3.4.1]{huybrechts-lehn}. 
\end{proof} 

\subsection{Moduli of semistable twisted sheaves} 
\label{section-moduli-twisted-sheaves}
There is also a version of Gieseker stability for twisted sheaves. 
If $\alpha \in \Br(X)$ is a class represented by a $\bmu_n$-gerbe $\pi \colon \cX \to X$ and $
H$ is an ample divisor on $X$, 
then for $E \in \Coh(X, \alpha) = \Coh^1(\cX)$ one defines Gieseker $H$-stability analogously to the untwisted case, using the \emph{geometric Hilbert polynomial} 
\begin{equation*}
P_H(E,t) \coloneqq n \int_{\cX} \ch(E \otimes \cO_{\cX}(\pi^*(tH))) \, \td_\cX
\end{equation*} 
in place of the Hilbert polynomial of a coherent sheaf. 
This reduces to usual Gieseker stability for coherent sheaves when $\alpha = 0$. 
We refer to \cite{lieblich-moduli-twisted} for the details of this theory. 
Similarly to slope stability, one can more generally replace $H$ by a real ample divisor $\omega$ and consider a twist by a real divisor $D$ in Gieseker stability, but we will not need this. 

There is a well-behaved moduli space of Gieseker semistable sheaves:

\begin{example}[Moduli of semistable twisted sheaves]
\label{example-moduli-sheaves} 
Let $X$ be a smooth projective variety over a field $k$, equipped with a Brauer class $\alpha \in \Br(X)$ and an ample divisor $H$. 
Let $\Knum(X)$ denote the numerical K-theory of $X$, i.e. the quotient of the 
Grothendieck group $\rK_0(X)$ of coherent sheaves by the kernel of the Euler pairing. 
For a class $v \in \Knum(X, \alpha)$, let $\cM_H(v)$ denote the moduli stack of Gieseker $H$-semistable $\alpha$-twisted coherent sheaves of class $v$. 
Then $\cM_H(v)$ is an open substack of the stack $\cM_{\gl}(\Dperf(X,\alpha)/k)$ of gluable objects in $\Dperf(X,\alpha)$, 
and $\cM_H(v) \to \Spec(k)$ is finite type and universally closed. 
This holds by the results of \cite{lieblich-moduli-twisted}, generalizing well-known results from the case when $\alpha = 0$ \cite{huybrechts-lehn}. 
\end{example} 

To formulate a relative version of moduli of semistable twisted sheaves, we need to fix a relative version of the numerical class $v$. 
There are various ways to do this. 
For instance, one may fix the geometric Hilbert polynomial. 
Another possibility is to consider classes in the relative numerical Grothendieck group, or a finite rank abelian group to which it maps, as in our discussion of stability conditions relative to a base in \S\ref{section-stability-families} below. 
For applications to the noncommutative variational Hodge conjecture or period-index problem, 
it will be useful to consider a topological variant where we fix a section of the relative topological K-theory. 

\begin{example}[Relative moduli of semistable twisted sheaves]
\label{example-relative-moduli-sheaves}
Let $f \colon X \to S$ be a smooth proper morphism of complex varieties equipped with a 
Brauer class $\alpha \in \Br(X)$ and a relatively ample divisor $H$ on $X$ over $S$.  
Let $v$ be a section of the local system $\Ktop[0](\Dperf(X, \alpha)/S))$. 
Note that any object $E \in \Dperf(X, \alpha)$ gives a section $v_{E}$ of $\Ktop[0](\Dperf(X, \alpha)/S))$ 
whose fiber over $s \in S(\bC)$ is the class of $E_s$ in $\Ktop[0](\Dperf(X_s, \alpha_s))$. 
We denote by $\cM_{H}(v)$ the moduli stack of Gieseker $H$-semistable $\alpha$-twisted coherent sheaves of class $v$. 
Since $v$ determines the geometric Hilbert polynomial, it follows again from the results of \cite{lieblich-moduli-twisted} that 
$\cM_{H}(v)$ is an open substack of $\cM_{\gl}(\Dperf(X, \alpha)/S)$ which is finite type and universally closed over $S$. 
\end{example} 

\section{Stability conditions}
\label{section-stability} 

In this section, we discuss stability conditions on categories and their associated moduli spaces of semistable objects, which provide the main examples to which we will apply the general theory from \S\ref{section-moduli-objects-modulo-autoequivalence}. 
We also give a new proof of a variant of a result of Polishchuk \cite{polishchuk-t-structure}, which says  
stability conditions are invariant under the identity component of the group of autoequivalences (Proposition~\ref{proposition-stability-invariant}). 

\subsection{Stability conditions in a family}
\label{section-stability-families} 
The notion of a stability condition on a triangulated category $\cC$, due to Bridgeland \cite{bridgeland-stability}, gives rise to a well-behaved theory of semistable objects in $\cC$. 
The definition depends on the choice of a homomorphism $\bv \colon \rK_0(\cC) \to \Lambda$, called a \emph{Mukai homomorphism}, from the Grothendieck group of $\cC$ to a finite rank free abelian group. 
A stability condition $\sigma$ on $\cC$ with respect to $\bv$ then consists of a pair $(Z, \cA)$ where $Z \colon \Lambda \to \bC$ is a group homomorphism (the \emph{central charge} of $\sigma$) and $\cA$ is the heart of a bounded t-structure, which satisfy certain compatibility conditions. 
When $\cC$ is (the homotopy category of) a smooth proper $k$-linear category for a field $k$, then it is often assumed that $\bv$ factors through the map $\rK_0(\cC) \to \Knum(\cC)$ to the numerical Grothendieck group, defined as the quotient by the kernel of the Euler pairing $\chi(E,F) = \sum_i (-1)^i \dim_k \Ext^i(E,F)$; in this case, $\sigma$ is called a \emph{numerical} stability condition. 
We say that $\sigma$ is a \emph{full} numerical stability condition when we can take $\Lambda = \Knum(\cC)$, i.e. 
the support property holds with respect to $\Knum(\cC)$. 

In \cite{stability-families} this notion was generalized to the relative setting, where instead of a triangulated category we consider an $S$-linear category $\cC$. 
For simplicity, we will consider the case where $X \to S$ is a smooth proper morphism of varieties and $\cC \subset \Dperf(X)$ is an $S$-linear semiorthogonal component. 
We consider the \emph{relative numerical Grothendieck group} $\Knum(\cC/S)$, which is roughly the group obtained from $\bigoplus_{s \in S} \Knum(\cC_s)$ by identifying elements in different summands that are the restriction of a global object $E \in \cC$; see \cite{stability-families} for the precise definition. 
Then we fix a Mukai homomorphism $\bv \colon \Knum(\cC/S) \to \Lambda$ to a finite rank free abelian group. 
Note that for any $s \in S$, we obtain a homomorphism $\bv_s \colon \Knum(\cC_s) \to \Lambda$ given as the composition of the canonical map $\Knum(\cC_s) \to \Knum(\cC/S)$ with $\bv$. 

A \emph{stability condition on $\cC$ over $S$} with respect to $\bv \colon \Knum(\cC/S) \to \Lambda$ 
consists of a group homomorphism $Z \colon \Lambda \to \bC$ and a collection ${\sigma} = (\sigma_s = (Z_s, \cA_s))_{s \in S}$ of numerical stability conditions on the fibers $\cC_s$ satisfying various compatibility and tameness conditions; 
in particular, we require that $Z_s = Z \circ \bv_s$. 
For the precise definition we refer to \cite[Definitions 20.5 and 20.15]{stability-families}. 

\begin{warning}
\label{warning-stability-k}
In the case where $S = \Spec(k)$ for a field $k$, we can consider a stability condition on $\cC$ in the sense introduced by Bridgeland, or a stability condition on $\cC$ \emph{over $k$} in the sense of \cite{stability-families}. These notions are not a priori equivalent. A stability condition on $\cC$ over $k$ is a stability condition in Bridgeland's sense satisfying some additional properties, most importantly the existence of proper moduli spaces of semistable objects. 
However, the results of \cite{stability-families} show that all known constructions of stability conditions in Bridgeland's sense actually give stability conditions over $k$. 
\end{warning} 

\subsection{Topological Mukai homomorphisms} 
Over the complex numbers, it will be useful to consider stability conditions over $S$ with respect to Mukai homomorphisms that are topological in the following sense. 
Recall from Theorem~\ref{theorem-Ktop} that we have a local system $\Ktop[0](\cC/S)$ on $S$ whose fibers are $\Ktop[0](\cC_s)$ for $s \in S(\bC)$. 

\begin{definition}
\label{definition-topological-v}
Let $\cC \subset \Dperf(X)$ be an $S$-linear semiorthogonal component, where 
$X \to S$ is a smooth proper morphism of complex varieties. 
A Mukai homomorphism $\bv \colon \Knum(\cC/S) \to \Lambda$ is called \emph{topological} if there is given an abelian group $\Lambda^{\rtop}$ with an inclusion $\Lambda \subset \Lambda^{\rtop}$ 
and a morphism $\bv^{\rtop} \colon \Ktop[0](\cC/S) \to \underline{\Lambda}^{\rtop}$ of local systems, where $\underline{\Lambda}^{\rtop}$ is the constant local system with value $\Lambda^{\rtop}$, 
such that for every $s \in S(\bC)$ the diagram
\begin{equation*}
\begin{tikzcd}
\rK_0(\cC_s) \arrow{r} \arrow{dr} & \Knum(\cC_s) \arrow{rr}{\bv_s} && \Lambda \arrow[hook]{d} \\ 
& \Ktop[0](\cC_s) \arrow{rr}{\bv_s^{\rtop}} && \Lambda^{\rtop}
\end{tikzcd}
\end{equation*} 
commutes, where the maps out of $\rK_0(\cC_s)$ are the canonical ones and $\bv_s^{\rtop}$ is the fiber of $\bv^{\rtop}$ over $s$. 
\end{definition}

\begin{example}
\label{example-topological-v}
	Let $f:X \to S$ be a smooth, projective morphism, where $S$ is a quasi-projective complex variety, and let $\pi:\cX \to X$ be a $\bmu_n$-gerbe whose restriction to the fibers of $f$ is essentially topologically trivial. Let $\omega$ and $D$ be real divisors on $X$, with $\omega$ relatively ample over $S$. 

	For each $s \in S$, there is a commutative diagram
	\[
		\begin{tikzcd}
			\rK_0(\Dperf(\cX_s)) \ar[r] \ar[dr] & \Knum(\Dperf(\cX_s)) \ar[r] & \CH^{*}(\cX_s)_{\bQ} \ar[d] \ar[r, "\sim", "\pi_*"'] & \CH^{*}(X_s)_{\bQ} \ar[d] \\
			& \Ktop[0](\Dperf(\cX_s)) \ar[r] & \rH^{\ev}(\cX_s, \bQ) \ar[r, "\sim", "\pi_*"'] & \rH^{\ev}(X_s, \bQ).
		\end{tikzcd}
	\]
Consider the two homomorphisms (given by the same formula)
\begin{align*}
 \ch^{\omega, D} \colon \Knum(\Dperf(\cX_s)) & \to \bR^{n+1} \\
 \ch^{\omega, D, \topo} \colon \Ktop[0](\Dperf(\cX_s)) & \to \bR^{n + 1} \\
 E & \mapsto ( \omega^n \ch^D_0(E), \dots, \omega^{n-i} \ch^D_i(E) , \dots, \ch^D_n(E)), 
\end{align*}
where we have used the Chern character for a Deligne--Mumford stack as in \S\ref{section-chern-character}. Both homomorphisms are induced (via the above diagram) by the map
\begin{equation} 
	\rH^{\ev}(X_s, \bQ) \to \bR^{n + 1}, \quad x \mapsto \prod_{k = 0}^{n} \int_X \exp(D) \cdot x \cdot \omega^k, 
\end{equation}
which (as $s$ varies) assembles to a morphism of local systems
\begin{equation} \label{eq:twisting_chern}
	\rR^{\ev} f_* \bQ \to \underline{\bR}^{n + 1}.
\end{equation}
We write $\Lambda_{\omega, D}^{\topo}$ for the image of $\Ktop[0](\cX_s)$ in $\bR^{n + 1}$ under $\ch^{\omega, D, \topo}$, which does not depend on the point $s$.

Ranging over $s$, the maps $\ch^{\omega, D}$ form a homomorphism $\Knum(\Dperf(\cX)/S) \to \bQ^{n+1}$, 
whose image we denote by $\Lambda_{\omega,D}$, so that we have a Mukai homomorphism
\begin{equation*}
\ch^{\omega,D}_{\cX/S} \colon \Knum(\Dperf(\cX)/S) \to \Lambda_{\omega,D}. 
\end{equation*} 
On the other hand, the results of \cite{hotchkiss-pi} show that there is a relative Chern character
\[
	\ch^{\topo}_{\cX/S}:\Ktop[0](\Dperf(\cX)/S) \to \rR^{\ev} f_* \bQ,
\]
defined as the relative $\theta^k$-twisted Chern character on each summand $\Dperf(X, \alpha^k)$ of $\Dperf(\cX)$.

The composition of $\ch^{\topo}_{\cX/S}$ with \eqref{eq:twisting_chern} gives a homomorphism $\ch_{\cX/S}^{\omega, D, \topo}$ landing in the constant local system $\Lambda_{\omega, D}^{\topo}$, which at each fiber $\cX_s$ fits into a commutative diagram
\[
	\begin{tikzcd}
		\rK_0(\Dperf(\cX_s)) \ar[dr] \ar[r] & \Knum(\Dperf(\cX_s)) \ar[r, "\ch^{\omega, D}_{\cX/S}"] & \Lambda_{\omega, D} \ar[d, hook] \\
		& \Ktop[0](\Dperf(\cX_s)) \ar[r, "\ch^{\omega, D, \topo}_{\cX/S}"] &  \Lambda_{\omega, D}^{\topo}.
	\end{tikzcd}
\]
In particular, the Mukai homomorphism $\ch^{\omega, D}_{\cX/S}$ is topological.

Finally, when $\omega = D = H$, to simplify notation we write $\ch^H$ for $\ch^{\omega, D}$. 
\end{example}

\subsection{Moduli of semistable objects} 
\label{section-moduli-ss}
For a stability condition over a field, we have the following analog of Example~\ref{example-moduli-sheaves}. 

\begin{example}[Moduli of semistable objects]
\label{example-moduli-objects} 
Let $\cC \subset \Dperf(X)$ be $k$-linear 
semiorthogonal component where $X$ is a smooth proper $k$-scheme. 
Let $\sigma$ be a numerical stability condition on $\cC$ over $k$ in the sense of \cite{stability-families}, with respect to a Mukai homomorphism $\mathbf{v} \colon \Knum(\cC) \to \Lambda$ to a finite rank abelian group. 
Fix $v \in \Lambda$ and a compatible phase $\phi \in \bR$, i.e. $Z(v) \in \bR_{> 0} e^{i\pi\phi}$ where $Z$ is the central charge of $\sigma$. 
Let $\cM_{\sigma}(v, \phi)$ denote the moduli stack of $\sigma$-semistable objects in $\cC$ of class $v$ and phase $\phi$. 
Then $\cM_{\sigma}(v, \phi) \subset \cM_{\gl}(\cC/k)$ is an open substack by \cite[Lemma 21.12]{stability-families}, 
and $\cM_{\sigma}(v, \phi)$ is universally closed over $k$ by (the proof of) \cite[Theorem 21.24]{stability-families}. 
\end{example}

Similar to Example~\ref{example-relative-moduli-sheaves}, we can also consider a relative variant of Example~\ref{example-moduli-objects}. 

\begin{example}[Relative moduli of semistable objects]
\label{example-relative-stability}
Let \mbox{$\cC \subset \Dperf(X)$} be an $S$-linear semiorthogonal component where  
$f \colon X \to S$ a smooth proper morphism of complex varieties. 
Let $v \in \Gamma(S, \Ktop[0](\cC/S))$ be a section. 
Suppose $\cC$ is equipped with a stability condition $\sigma$ over a connected base $S$ 
with respect to a topological Mukai homomorphism $\bv \colon \Knum(\cC/S) \to \Lambda$ in the sense of Definition~\ref{definition-topological-v}. 
Let $\lambda \in \Lambda^{\rtop}$ be the image of $v$ under the map on global sections $\Gamma(\bv^{\rtop}) \colon \Gamma(\Ktop[0](\cC/S)) \to \Lambda^{\rtop}$. 
 
Then, on the one hand, we can consider the moduli stack $\cM_{\sigma}(\lambda, \phi)$ of $\sigma$-semistable objects in $\cC$ of class $\lambda$ and phase $\phi$, where $\phi$ is a phase compatible with $\lambda$ (we will also say $\phi$ is compatible with $v$ in this case); 
this is an open substack of $\cM_{\gl}(\cC/S)$  \cite[Lemma 21.12]{stability-families} 
which is finite type and universally closed over $S$ \cite[Theorem 21.24]{stability-families}. 
Note that by definition $\cM_{\sigma}(\lambda, \phi)$ is empty if $\lambda \notin \Lambda$. 
On the other hand, we may consider a topological variant: 
the moduli stack $\cM_{\sigma}(v, \phi)$ of $\sigma$-semistable objects in $\cC$ of class $v$ and phase $\phi$. 
This is a connected component of $\cM_{\sigma}(\lambda, \phi)$, and hence also open in $\cM_{\gl}(\cC/S)$ and finite type and universally closed over $S$. 

Often one requires the fibers $v_s \in \Ktop[0](\cC_s)$ of $v$ to be Hodge classes for all~$s \in S(\bC)$, 
as otherwise the above moduli spaces are empty. 
\end{example}

\subsection{Action of the identity component on stability conditions} 
\label{section-identity-component-stab}
We digress to explain an interesting consequence of Lemma~\ref{lemma-aut0-preserve-open} for 
the action of the identity component of the group of autoequivalences on stability conditions. 

Let $\cC$ and $\sigma$ be as in Example~\ref{example-moduli-objects}. 
The group of pairs 
\begin{equation*}
\Aut_{\Lambda}(\cC/k)(k) \coloneqq \set{ (g, a) \in \Aut(\cC/k)(k) \times \Aut(\Lambda) ~ \st ~ \mathbf{v} \circ \Phi_g^{\Knum} = a \circ \mathbf{v}  } 
\end{equation*} 
acts on the space $\Stab_{\Lambda}(\cC/k)$ of all $\sigma$, where $\Phi_g^{\Knum}$ denotes the automorphism of $\Knum(\cC)$ induced by the autoequivalence $\Phi_g \in \Aut_k(\cC)$ corresponding to $g$. Namely, if $\sigma = (\cP, Z)$, i.e. $\cP$ is the slicing and $Z$ the central charge for $\sigma$, then 
\begin{equation*}
(g, a) \cdot \sigma \coloneqq (\Phi_g(\cP), Z \circ a^{-1})
\end{equation*} 
where $\Phi_g(\cP)$ is the slicing $\Phi_g(\cP(\phi))$, $\phi \in \bR$. 
Let $\Aut^0_{\Lambda}(\cC/k)(k) \subset \Aut_{\Lambda}(\cC/k)(k)$ be the subgroup where $g \in \Aut^0(\cC/k)(k)$. 

\begin{proposition}
\label{proposition-stability-invariant} 
Let $\cC \subset \Dperf(X)$ be a semiorthogonal component where $X$ is a smooth proper scheme over an algebraically closed field $k$, and assume that $\cC$ is a connected as a $k$-linear category. 
Let $\sigma$ be a numerical stability condition on $\cC$ over $k$
with respect to a surjective homomorphism $\mathbf{v} \colon \Knum(\cC) \to \Lambda$. 
Then $\Aut^0_{\Lambda}(\cC/k)(k)$ acts trivially on $\Stab_{\Lambda}(\cC/k)$, i.e. 
for every $(g, a)  \in \Aut^0_{\Lambda}(\cC/k)(k)$ and $\sigma \in \Stab_{\Lambda}(\cC/k)$ we have 
$(g,a) \cdot \sigma = \sigma$. 
\end{proposition} 

\begin{proof}
For a scheme $T$ over $k$, an autoequivalence $\Phi \in \Aut_T(\cC_T)$ over $T$, and objects 
\mbox{$E, F \in \cC$}, consider the function $t \mapsto \chi(F_{t}, \Phi_t(E_{t}))$ on $T$, where a subscript $t$ denotes base change to the residue field $\kappa(t)$. This function is locally constant, as it computes the rank of the object $\cHom_T(\Phi(E_T), F_T)$. 
It follows that any $g \in \Aut^0_{\Lambda}(\cC/k)(k)$ acts trivially on $\Knum(\cC)$. 
As $\mathbf{v} \colon \Knum(\cC) \to \Lambda$ is surjective, we then have 
$\Aut_{\Lambda}^0(\cC/k)(k) = \Aut^0(\cC/k)(k)$. 
Therefore, if $\sigma = (\cP, Z)$, we just need to show that $\Phi_g(\cP(\phi)) = \cP(\phi)$ for $g \in \Aut^0(\cC/k)(k)$. But the objects of $\cP(\phi)$ are precisely the $\sigma$-semistable objects of phase $\phi$, so we conclude by Lemma~\ref{lemma-aut0-preserve-open} and Example~\ref{example-moduli-objects}. 
\end{proof} 

\begin{remark}
Polishchuk proved a variant of Proposition~\ref{proposition-stability-invariant} in \cite[Theorem 3.5.1 and Corollary 3.5.2]{polishchuk-t-structure}, where $\cC = \Dperf(X)$ and the heart of the t-structure for $\sigma$ is assumed noetherian, but $\sigma$ is only required to be a stability condition in the sense of Bridgeland (as opposed to ``over $k$'' in the sense of \cite{stability-families}). 
\end{remark} 


\section{Construction of stability conditions} 
\label{section-construction-stability} 

In general, it is a difficult problem to construct a stability condition on a given category. 
For a smooth projective complex variety $X$, stability conditions conjecturally always exist on $\Dperf(X)$. 
This is currently only known in general when $\dim X \leq 2$ \cite{bridgeland-K3, arcara-bertram}. 
When $\dim X = 3$, a construction is known for special classes of varieties, like Fano threefolds \cite{Chunyi, Macri-et-al-Fano}, abelian threefolds \cite{Piyaratne-abelian1, Piyaratne-abelian2, BMS}, and quintic threefolds \cite{chunyi-quintic}, while in higher dimensions much less is known. 
The strategy for proving these results was laid out in \cite{BMT}, which via a tilting construction reduced the problem to establishing a generalized Bogomolov--Gieseker inequality for certain complexes. 
Our goal is to explain how these results extend to the twisted setting. 
In most cases, the twisted result follows from the same proof as in the untwisted case, or can be deduced from the untwisted case by a covering argument. 

\subsection{Central charge} 
Let $X$ be a smooth projective variety equipped 
with a Brauer class $\alpha \in \Br(X)$. 
Let $\omega, D$ be real divisors on $X$ with $\omega$ ample. 
We define a central charge $Z_{\omega, D} \colon \rK_0(X, \alpha) \to \bC$ by 
\begin{equation*}
Z_{\omega,D}(E) = -\int_X e^{-i \omega} \cdot \ch^D(E) . 
\end{equation*} 
Explicitly, in low dimensions we have: 
\begin{alignat*}{2}
Z_{\omega,D} & = -\ch_1^D + i \omega \ch^D_0  & & \text{if } \dim X = 1 , \\
Z_{\omega,D} & = \left( -\ch_2^D + \frac{\omega^2}{2} \ch_0^D\right) +  i \omega \ch_1^D \qquad   & &  \text{if } \dim X = 2, \\ 
Z_{\omega,D} & = \left(-\ch_3^D + \frac{\omega^2}{2} \ch_1^D \right) + i \left( \omega \ch_2^D - \frac{\omega^3}{6} \ch_0^D \right) \qquad &&  \text{if } \dim X =3. 
\end{alignat*} 
Note that $Z_{\omega,D}$ factors via the Mukai homomorphism 
\begin{equation*}
\ch^{\omega,D} \colon \rK_\num(X, \alpha) \to \Lambda_{\omega,D}
\end{equation*} 
from Example~\ref{example-topological-v} (with $S$ a point). 

The following is a precise statement of the conjectural existence of stability conditions, extending \cite[Conjecture 2.1.2]{BMT} to the twisted case. 

\begin{conjecture}
\label{conjecture-Z-heart-stab}
There exists a bounded t-structure on $\Dperf(X, \alpha)$ with heart $\cA_{\omega, D}$ such that 
the pair $(Z_{\omega,D}, \cA_{\omega, D})$ is a full numerical stability condition. 
\end{conjecture} 

When $\dim X = 1$ the standard heart $\cA_{\omega,D} = \Coh(X, \alpha)$ works, but in higher dimensions the proposal is to take an $(\dim X -1)$-fold tilt of the category $\Coh(X, \alpha)$. 

\subsection{The first tilt} 
We consider $\mu_{\omega,D}$-stability for objects in $\Coh(X, \alpha)$, as reviewed in \S\ref{section-slope-stability-twisted-sheaves}. 
There is a torsion pair $(\cT_{\omega,D}, \cF_{\omega,D})$ in $\Coh(X, \alpha)$ defined by 
\begin{align*}
\cT_{\omega,D}  & \coloneqq \set{E \in \Coh(X,\alpha) \st \mu_{\omega,D}^{-}(E) > 0} , \\ 
\cF_{\omega,D}  & \coloneqq \set{E \in \Coh(X,\alpha) \st \mu_{\omega,D}^{+}(E) \leq 0} . 
\end{align*} 
Since $\Coh(X, \alpha)$ is the heart of the standard bounded t-structure on $\Dperf(X,\alpha) = \Db(X, \alpha)$, tilting at the above torsion pair produces a new bounded t-structure on $\Dperf(X, \alpha)$, with heart the extension closure 
\begin{equation*}
\Coh^{\omega, D}(X,\alpha) \coloneqq \langle \cT_{\omega, D}, \cF_{\omega,D}[1] \rangle.  
\end{equation*}
This t-structure gives a solution to Conjecture~\ref{conjecture-Z-heart-stab} when $\dim X = 2$ and the base field has characteristic $0$: 

\begin{theorem}
Let $X$ be a smooth projective surface defined over a field of characteristic $0$, equipped with a Brauer class $\alpha \in \Br(X)$, 
a real ample divisor $\omega$, and a real divisor $D$. 
Then the pair $(Z_{\omega, D}, \Coh^{\omega, D}(X, \alpha))$ is a full numerical stability condition. 
\end{theorem} 

\begin{proof}
When $\alpha = 0$, this result goes back to \cite{bridgeland-K3} in the case of K3 surfaces and \cite{arcara-bertram} in general; see also \cite[\S6]{macri-schmidt} for an exposition. 
The key input in the proof is the Bogomolov inequality for semistable sheaves. 
When $\alpha \in \Br(X)$ is arbitrary, it is straightforward to check that the same proof works, using instead the twisted Bogomolov inequality of Theorem~\ref{theorem-twisted-bogomolov}. 
\end{proof}

\subsection{Tilt stability on threefolds} 
Now we focus on the case of threefolds. To construct a tilt of $\Coh^{\omega, D}(X, \alpha)$, we will need a notion of stability which plays the role of slope stability in the construction of the first tilt. 
\begin{definition}
Let $X$ be a smooth projective threefold equipped with a Brauer class $\alpha \in \Br(X)$, 
a real ample divisor $\omega$, and a real divisor $D$. 
For $E \in \Coh^{\omega, D}(X, \alpha)$, the \emph{tilt slope} is 
\begin{equation*}
\nu_{\omega, D}(E) \coloneqq 
\begin{cases}
\frac{ \omega \ch_2^D(E) - \frac{\omega^3}{6} \ch_0^D(E) }{  \omega^2\ch_1^D(E) } & \text{if }  \omega^2\ch_1^D(E) \neq 0 , \\
+\infty & \text{if } \omega^2\ch_1^D(E) = 0. 
\end{cases}
\end{equation*} 
We say $E$ is \emph{$\nu_{\omega,D}$-semistable} (or \emph{tilt semistable} if $(\omega, D)$ are understood) if for all nontrivial subobjects $F \hookrightarrow E$ in $\Coh^{\omega, D}(X, \alpha)$, the inequality 
$\nu_{\omega,D}(F) \leq \nu_{\omega,D}(E/F)$ holds, and $E$ is \emph{$\nu_{\omega, D}$-stable} if for all $F$ the inequality is strict. 
\end{definition} 

By the argument of \cite[Lemma 3.2.4]{BMT} (which treats the untwisted case), Harder--Narasimhan filtrations exist with respect to $\nu_{\omega,D}$-stability when $D$ is rational. That is, for any $E \in \Coh(X, \alpha)$ there exists a filtration 
\begin{equation*}
0 = E_0 \subset E_1 \subset \cdots \subset E_m = E
\end{equation*} 
such that the factors $A_i = E_i/E_{i-1}$ are $\nu_{\omega,D}$-semistable and 
\begin{equation*}
\nu_{\omega,D}(A_1) > \nu_{\omega,D}(A_2) > \cdots > \nu_{\omega,D}(A_m). 
\end{equation*}  
We use the notation $\nu_{\omega,D}^+ \coloneqq \nu_{\omega,D}(A_1)$ and $\nu_{\omega,D}^{-}(E) = \nu_{\omega,D}(A_m)$. 

Tilt stability behaves well under finite covers. 
This is analogous to Lemma~\ref{lemma-pullback-ss} for slope stability, except that we require the finite cover to be \'{e}tale for one of the implications. 

\begin{lemma}
\label{lemma-tilt-stability-cover}
Let $X$ be a smooth projective threefold equipped with a Brauer class $\alpha \in \Br(X)$, 
a real ample divisor $\omega$, and a real divisor $D$. 
Let $f \colon Y \to X$ be a finite surjective morphism from a smooth variety $Y$. 
Let $E \in \Dperf(X, \alpha)$. 
\begin{enumerate}
\item \label{CohomegaD-pullback}
$E \in \Coh^{\omega, D}(X, \alpha)$ if and only if $f^*E \in \Coh^{f^*\omega, f^*D}(Y, f^*\alpha)$. 
\item \label{tiltss-pullback}   
If $f^*E$ is $\nu_{f^*\omega, f^*D}$-semistable, then $E$ is $\nu_{\omega, D}$-semistable. 
\item \label{tiltss-etalepullback} 
If $D$ is rational, $f$ is \'{e}tale, and $E \in \Coh^{\omega, D}(X, \alpha)$ is $\nu_{\omega, D}$-semistable, then $f^*E$ is $\nu_{f^*\omega, f^*D}$-semistable. 
\end{enumerate}
\end{lemma}  

\begin{proof}
This is essentially a twisted version of \cite[Proposition 6.1]{BMS}; 
the proof is similar but we include it for convenience. 

By definition, for $E \in \Dperf(X, \alpha)$ we have $E \in \Coh^{\omega, D}(X, \alpha)$ if and only if $\cH^{0}(E) \in \cT_{\omega, D}$, $\cH^{-1}(E)  \in \cF_{\omega,D}$, and $\cH^{i}(E) = 0$ for $i \notin \set{0,1}$, and there is an analogous criterion for $f^*E \in \Coh^{f^*\omega, f^*D}(Y, f^*\alpha)$. 
Therefore, to prove~\eqref{CohomegaD-pullback} it suffices to show that for $E \in \Coh(X, \alpha)$, we have 
\begin{align*}
E \in \cT_{\omega,D} & \iff f^*E \in \cT_{f^*\omega, f^*D} \\ 
E \in \cF_{\omega, D} & \iff f^*E \in \cF_{f^*\omega, f^*D}. 
\end{align*} 
But $\mu_{f^*\omega, f^*D}(f^*E) = \mu_{\omega,D}(E)$ for any $E \in \Coh(X, \alpha)$ and 
Lemma~\ref{lemma-pullback-ss} implies that $f^*$ preserves HN filtrations of objects in $\Coh(X, \alpha)$, 
so both claims follow. 

Similar to the usual slope, for $E \in \Coh^{\omega, D}(X, \alpha)$ the tilted slope satisfies
\begin{equation}
\label{pullback-tilt-slope}
\nu_{f^*\omega, f^*D}(f^*E) = \nu_{\omega, D}(E). 
\end{equation} 
Since by~\eqref{CohomegaD-pullback} the pullback functor $f^* \colon \Dperf(X, \alpha) \to \Dperf(X, \alpha)$ is t-exact with respect to the tilt t-structures, any destabilizing subobject of $E$ must therefore destabilize $f^*E$; 
that is, $E$ is tilt semistable if $f^*E$ is tilt semistable. This proves~\eqref{tiltss-pullback}. 

Now assume $f$ is \'{e}tale. 
By taking a finite \'{e}tale Galois cover of $X$ which dominates $f \colon Y \to X$ and using~\eqref{tiltss-pullback}, we may reduce to proving \eqref{tiltss-etalepullback} when $f$ itself is Galois, say with Galois group $G$. 
In this case, if $E \in \Coh^{\omega, D}(X, \alpha)$, then let $F \hookrightarrow f^*E$ be the first step of the HN filtration with respect to $\nu_{f^*\omega, f^*D}$. 
As $f^*E$ is $G$-equivariant and HN filtrations are unique and functorial, the morphism $F \to f^*E$ is $G$-equivariant (for a natural $G$-equivariant structure on~$F$), and thus descends to a morphism $F' \to E$ in $\Dperf(X, \alpha)$. 
It follows from~\eqref{CohomegaD-pullback} that $F' \to E$ is an injection in $\Coh^{\omega, D}(X, \alpha)$, 
which must destabilize $E$ in view of~\eqref{pullback-tilt-slope}.  
\end{proof}

\subsection{The generalized Bogomolov--Gieseker inequality}
In the construction of stability conditions on threefolds, the role of the Bogomolov inequality for sheaves is replaced by the following 
inequality involving $\ch_3$.  
\begin{definition}
Let $X$ be a smooth projective threefold 
equipped with a Brauer class $\alpha \in \Br(X)$, 
a real ample divisor $\omega$, and a real divisor $D$. 
Let $E \in \Coh^{\omega, D}(X, \alpha)$ be a $\nu_{\omega,D}$-semistable object with $\nu_{\omega,D}(E) = 0$. 
We say $E$ satisfies the \emph{generalized Bogomolov--Gieseker inequality} if 
\begin{equation*}
\ch_3^D(E) \leq \frac{\omega^2}{18} \ch_1^D(E). 
\end{equation*} 
If this holds for all such $E$, we say $(X, \alpha)$ satisfies the generalized Bogomolov--Gieseker inequality with respect to $\omega$ and $D$.
\end{definition} 

For $\alpha = 0$, the above inequality first appeared in \cite[Conjecture 1.3.1]{BMT}, where it was conjectured to hold in general in characteristic $0$. 
There are a number of threefolds, like abelian threefolds \cite{BMS} or Fano threefolds of Picard number 1 \cite{Chunyi}, for which the inequality has been proved, but in \cite{schmidt-generalizedBG} an example was given where it fails. 
A more flexible modified conjecture, which is expected to always hold, is stated in \cite[Conjecture 4.7]{bayer-macri-ICM} and would have similar consequences for the existence of stability conditions. 
We focus on the original form of the generalized Bogomolov--Gieseker inequality as it suffices for our case of interest: 

\begin{theorem}
\label{theorem-BG-AV}
Let $X$ be an abelian threefold over a field of characteristic $0$ equipped with a Brauer class $\alpha \in \Br(X)$. 
Then for any real ample divisor $\omega$ and rational divisor $D$ on $X$, 
$(X,\alpha)$ satisfies the generalized Bogomolov--Gieseker inequality with respect to $\omega$ and $D$.
\end{theorem} 

\begin{proof}
The crucial result is \cite[Theorem 1.1]{BMS}, which handles the untwisted case, to which we will reduce. 
If $n = \per(\alpha)$ then the multiplication-by-$n$ map $[n] \colon A \to A$ is an \'{e}tale cover of degree $n^6$ such that $[n]^*\alpha = 0$. 
If $E \in \Coh^{\omega, D}(X, \alpha)$ is $\nu_{\omega,D}$-semistable, 
then by Lemma~\ref{lemma-tilt-stability-cover}\eqref{tiltss-etalepullback} the object $[n]^*E \in \Coh^{\omega, D}(X)$ is $\nu_{n^2\omega, n^2D}$-semistable. 
By \cite[Theorem 1.1]{BMS} we conclude that $[n]^*E$ satisfies the generalized Bogomolov--Gieseker inequality. 
But this is nothing but the generalized Bogomolov--Gieseker inequality for $E$ itself multiplied by $n^6 = \deg([n])$. 
\end{proof} 

\begin{remark}
In fact, \cite[Theorem 1.1]{BMS} shows that when $\alpha = 0$, Theorem~\ref{theorem-BG-AV} holds even when $D$ is a real divisor, not necessarily rational. 
Using deformation arguments as in \cite[\S7]{BMS}, it would be possible to extend Theorem~\ref{theorem-BG-AV} to the case of real divisors for arbitrary $\alpha$, but we will not need this. 
\end{remark} 

\subsection{The second tilt} 
Let $X$ be a smooth projective threefold equipped with a Brauer class $\alpha \in \Br(X)$, a real ample divisor $\omega$, and a rational divisor $D$. 
Then HN filtrations exist with respect to $\nu_{\omega,D}$-stability, so we obtain a torsion pair $(\cT^{'}_{\omega,D}, \cF'_{\omega,D})$ in $\Coh^{\omega,D}(X, \alpha)$ defined by 
\begin{align*}
\cT'_{\omega,D}  & \coloneqq \set{E \in \Coh^{\omega,D}(X,\alpha) \st \nu_{\omega,D}^{-}(E) > 0} , \\ 
\cF'_{\omega,D}  & \coloneqq \set{E \in \Coh^{\omega,D}(X,\alpha) \st \nu_{\omega,D}^{+}(E) \leq 0} . 
\end{align*} 
Tilting at this torsion pair produces a bounded t-structure on $\Dperf(X, \alpha)$ with heart the 
extension closure 
\begin{equation*}
\cA^{\omega, D}(X, \alpha) \coloneqq 
\langle \cT'_{\omega, D}, \cF'_{\omega, D}[1] \rangle. 
\end{equation*} 

In \cite[Corollary 5.2.4]{BMT} it is shown that if $(X, \alpha)$ satisfies the generalized Bogomolov--Gieseker inequality with respect to $\omega$ and $D$, then 
$(Z_{\omega, D}, \cA_{\omega,D})$ satisfies the requirements of being a stability condition, except possibly the support property. 
In \cite{BMS}, the support property is proved, at least when $\omega$ and $D$ are suitably proportional to an ample divisor. 

For ease of reference, we will follow the conventions of \cite{BMS}, where the stability condition is parameterized in a slightly different form. 
For an ample divisor $H$ on $X$ and rational numbers $a,b,c,d \in \bQ$ with $a > 0$, we write
\begin{equation*}
\cA^{a,b}(X, \alpha) \coloneqq \cA^{\sqrt{3}aH,bH}(X, \alpha), \quad \ch^{b} \coloneqq \ch^{bH}, 
\end{equation*} 
and consider the the central charge 
\begin{equation*}
Z_{a,b}^{c,d} \coloneqq \left( - \ch_3^{b} + d H \ch_2^{b} + c H^2 \ch_1^b \right) + i \left( H \ch_2^b - \frac{a^2}{2} H^3 \ch_0^b \right) 
\end{equation*} 
which factors through the Mukai homomorphism $\ch^H \colon \Knum(X, \alpha) \to \Lambda_H$ from Example~\ref{example-topological-v} (with $S$ a point). 

\begin{theorem}
\label{theorem-stability-from-BG}
Let $X$ be a smooth projective threefold with a Brauer class $\alpha \in \Br(X)$ and an ample divisor $H$. 
Let $a,b,c,d \in \bQ$ such that 
\begin{equation}
\label{abcd}
a > 0 \quad \text{and} \quad c > \frac{a^2}{6} + \frac{|d|a}{2}. 
\end{equation} 
Assume that for any real ample divisor $\omega$ and rational divisor $D$ on $X$ which are both proportional to $H$, 
the generalized Bogomolov--Gieseker inequality holds for $(X, \alpha)$ with respect to $\omega$ and $D$. 
Then $(\cA^{a,b}(X, \alpha), Z_{a,b}^{c,d})$ is a stability condition on $\Dperf(X, \alpha)$ with respect to 
$\ch^H$. 
\end{theorem} 

\begin{proof}
In the untwisted case when $\alpha = 0$, this is \cite[Theorem 8.2]{BMS} combined with \cite[Theorem 4.2]{BMS}. 
The same arguments work when $\alpha$ is general. 
\end{proof} 

\begin{remark}
In Theorem~\ref{theorem-stability-from-BG}, the assumption that the generalized Bogomolov--Gieseker inequality holds 
for all real $\omega$ and rational $D$ could be replaced with an a priori stronger inequality for a particular choice of $\omega$ and $D$ --- see \cite[Conjecture 4.1 and Theorem 4.2]{BMS}. 
\end{remark} 

\begin{remark}
\label{remark-distinguished-Stab}
In view of Theorem~\ref{theorem-BG-AV}, we in particular find that stability conditions exist on twisted abelian threefolds. 
By \cite{OPT}, in the case of an untwisted abelian threefold $X$, this construction gives full numerical stability conditions. In the space of all such, there is a distinguished connected component $\Stab^{\dagger}(X) \subset \Stab(X)$ which contains the stability conditions arising from this construction and is preserved by autoequivalences of $X$.  
\end{remark}

\subsection{Relative case} 
In \cite[Part V]{stability-families}, the untwisted versions of the above results were upgraded to produce relative stability conditions in the sense of \S\ref{section-stability-families}. 
The same arguments go through directly in the twisted setting. 
We state the result for threefolds, but a similar statement also holds for surfaces. 

\begin{theorem}
Let $f \colon X \to S$ be a smooth projective morphism of varieties of relative dimension $3$. 
Let $\alpha \in \Br(X)$ be a Brauer class, and let $H$ be a relatively ample divisor on $X$ over $S$. 
Assume that for every $s \in S$, $(X_s, \alpha_s)$ satisfies the generalized Bogomolv--Gieseker inequality with 
respect to $\omega$ and $D$, whenever $\omega$ is a real ample divisor and $D$ is a rational divisor, both of which are proportional to $H_s$. 
Then for any $a,b,c,d \in \bQ$ satisfying \eqref{abcd}, 
there is a stability condition $\sigma$ on $\Dperf(X, \alpha)$ over $S$ with respect to the homomorphism 
$\ch^H \colon \Knum(\Dperf(X, \alpha)/S) \to \Lambda_H$ from Example~\ref{example-topological-v} whose 
fiber $\sigma_s$ over $s \in S$ is the stability condition on $\Dperf(X_s, \alpha_s)$ given by Theorem~\ref{theorem-stability-from-BG}. 
\end{theorem} 

Together with Theorem~\ref{theorem-BG-AV}, this has the following important consequence. 
\begin{corollary}
\label{cor:existence_of_relative_stability_ab_threefolds}
Let $f \colon X \to S$ be a smooth projective family of complex abelian threefolds. 
Let $\alpha \in \Br(X)$ be a Brauer class. 
Then there exists a stability condition $\sigma$ on $\Dperf(X, \alpha)$ over $S$ with respect to the 
topological Mukai homomorphism $\ch^H \colon \Knum(\Dperf(X, \alpha)/S) \to \Lambda_H$. 
\end{corollary} 


\newpage 
\part{Donaldson--Thomas theory} 
\label{part-DT-theory} 

\section{Obstruction theories and virtual fundamental classes} 
\label{section-obstruction-theories-and-vfc}
In this section we cover some preliminaries on obstruction theories and their associated virtual fundamental classes. 

\subsection{Obstruction theories}
\label{section-obs-thy}

\begin{definition}
\label{definition-obstruction-theory}
Let $\cX$ be an algebraic stack over $S$. 
An \emph{obstruction theory} for $\cX$ over $S$ is a morphism 
$\phi \colon \cF \to \tau^{\geq -1} L_{\cX/S}$ in $\Dqc(\cX)$ such that 
$\cofib(\phi) \in \Dqc^{\leq -2}(\cX)$. 
We say $\phi$ is \emph{perfect} if 
$\cF$ is a perfect complex of Tor-amplitude $[-1, \infty]$. 
\end{definition} 

\begin{remark} 
\label{remark-tor-amp}
As $\cX$ is an algebraic stack, we have $L_{\cX/S} \in \Dqc^{\leq 1}(\cX)$. 
If $\phi \colon \cF \to \tau^{\geq -1} L_{\cX/S}$ is a relative obstruction theory, then it follows that 
$\cF$ has Tor-amplitude in $[-\infty, 1]$; 
in particular, if $\phi$ is perfect, then $\cF$ has Tor-amplitude in $[-1,1]$. 
If $\cX$ is Deligne--Mumford, then we have the stronger connectivity $L_{\cX/S} \in \Dqc^{\leq 0}(\cX)$, so if $\phi$ is an obstruction theory then $\cF$ has Tor-amplitude in $[-\infty, 0]$, and in $[-1,0]$ if $\phi$ is perfect.  
\end{remark} 

\begin{remark}
There are some slight variants of Definition~\ref{definition-obstruction-theory} in the literature. 
Let us compare with the original definition of Behrend and Fantechi \cite{BF-normal-cone}. 
There the target of $\phi$ is taken to be the full cotangent complex $L_{\cX/S}$ instead of its truncation $\tau^{\geq -1}L_{\cX/S}$. 
However, given such a $\cF \to L_{\cX/S}$, the composition 
$\cF \to L_{\cX/S} \to \tau^{\geq -1}L_{\cX/S}$ gives an obstruction theory in the sense of Definition~\ref{definition-obstruction-theory}, and this is sufficient for all of the constructions 
in \cite{BF-normal-cone}. 
In \cite{BF-normal-cone} Behrend and Fantechi also only consider obstruction theories for Deligne--Mumford stacks; in this case our definition recovers theirs, modulo the discrepancy about cotangent complexes just explained. 
More recently, others have studied obstruction theories in the setting of (higher) algebraic stacks \cite{normal-cone-artin-stacks, poma-vclass}. 
For our main applications, we will in fact only need the Deligne--Mumford case, but some intermediate results are more naturally formulated in the context of general algebraic stacks. 
\end{remark} 

\begin{remark}
\label{remark-enhancement-obs}
Obstruction theories can be thought of as shadows of derived enhancements: by Lemma~\ref{lemma-enhancement-obs}, if $\fX$ is a derived enhancement of an algebraic stack $\cX$ over $S$ and $i \colon \cX \to \fX$ is the canonical map, then 
$i^* L_{\fX/S} \to L_{\cX/S} \to \tau^{\geq -1}L_{\cX/S}$ is an obstruction theory. 
\end{remark} 

\begin{remark} 
\label{remark-bc-obs-thy}
Let $\phi \colon \cF \to \tau^{\geq -1} L_{\cX/S}$ be an obstruction theory for an algebraic stack $\cX$ over $S$. 
Let $u \colon S' \to S$ be a morphism, and consider the \emph{underived} base change diagram 
\begin{equation} 
\label{bc-obs-thy} 
\vcenter{
\xymatrix{
\cX' \ar[r]^{u'} \ar[d] & \cX \ar[d] \\ 
S' \ar[r]^{u} & S .
}}
\end{equation} 
We claim that there is a natural way to base change $\phi$ to obtain an obstruction theory for $\cX'$ over $S'$. 
To explain this, it is convenient to consider the derived fiber product 
\begin{equation*} 
\vcenter{
\xymatrix{
\fX' \ar[r]^{\mu'} \ar[d] & \cX \ar[d] \\ 
S' \ar[r]^{u} & S .
}}
\end{equation*} 
which is a derived enhancement of $\cX'$, i.e. $t_0(\fX') = \cX'$. 
By base change for the cotangent complex we have 
$(\mu')^*L_{\cX/S} \simeq L_{\fX'/S'}$; note that for this to hold, it is important that we use 
the derived base change $\fX'$ instead of $\cX'$. 
It follows that 
\begin{equation*}
\tau^{\geq -1}L_{\fX'/S'} \simeq \tau^{\geq -1}(\mu')^*\tau^{\geq -1}L_{\cX/S} .
\end{equation*} 
Therefore, pullback of $\phi$ along $\mu'$ followed by truncation in degrees $\geq -1$ gives a morphism 
\begin{equation}
\label{phi-fX}
(\mu')^*\cF \to  (\mu')^*\tau^{\geq -1}L_{\cX/S} \to \tau^{\geq -1}L_{\fX'/S'}
\end{equation}
whose cone is easily seen to lie in $\Dqc^{\leq -2}(\fX')$ using $\cofib(\phi) \in \Dqc^{\leq -2}(\cX)$. 
This can be considered as an obstruction theory for $\fX'$ over $S'$. 

To induce an obstruction theory for $\cX'$, let $i \colon \cX' \to \fX'$ be the canonical closed immersion and consider the composition $i^*L_{\fX'/S'} \to L_{\cX'/S'} \to \tau^{\geq -1} L_{\cX'/S'}$. 
For degree reasons, this factors via a morphism $i^* \tau^{\geq -1} L_{\fX'/S'} \to \tau^{\geq -1} L_{\cX'/S'}$, which by Lemma~\ref{lemma-enhancement-obs} has cone lying in $\Dqc^{\leq -2}(\cX')$. 
Composing this morphism with the pullback of~\eqref{phi-fX} along $i$ we obtain a morphism 
\begin{equation*}
\phi_{S'} \colon (u')^*\cF \to \tau^{\geq -1} L_{\cX'/S'}. 
\end{equation*} 
Using the connectivity estimates above, we see $\cofib(\phi_{S'}) \in \Dqc^{\leq -2}(\cX')$, i.e. 
$\phi_{S'}$ is an obstruction theory for $\cX'$ over $S'$ which we call the base change of $\phi$. 
Note that $\phi_S$ is perfect if $\phi$ is. 
For any point $s \in S$, we denote by $\phi_s$ the base change of $\phi$ along $\Spec(\kappa(s)) \to S$. 
\end{remark}

\subsection{Virtual fundamental classes} 
The main upshot of an obstruction theory is that it gives rise 
to a virtual fundamental class. 

\begin{theorem}
\label{theorem-vclass}
Let $\cX$ be a Deligne--Mumford stack over a pure-dimensional scheme $S$, 
with both $\cX$ and $S$ of finite type over a field. 
Let $\phi \colon \cF \to \tau^{\geq -1} L_{\cX/S}$ be a perfect obstruction theory for $\cX$ over $S$. 
Assume that the rank of $\cF$ (a priori a locally constant function on $\cX$) is constant, denoted $\rk \cF$. 
Then there is a canonically associated \emph{virtual fundamental class}  
\begin{equation*}
[\cX]_{\phi}^{\vir} \in \CH_{\dim S+\rk \cF}(\cX)
\end{equation*} 
in the Chow group of $\cX$, which is compatible with base change in the following sense: if $u \colon S' \to S$ is a regular immersion or flat morphism from a pure-dimensional scheme of finite type over the base field, and 
$u' \colon \cX' \to \cX$ is the (underived) base change as in~\eqref{bc-obs-thy}, then 
\begin{equation*}
(u')^*[\cX]_{\phi}^{\vir} = [\cX']_{\phi_{S'}}^{\vir}. 
\end{equation*} 
\end{theorem} 

\begin{proof}
The construction of the virtual fundamental class is the main result of \cite{BF-normal-cone}, 
with the caveat that there $\cF$ is assumed to admit a global resolution, 
but this assumption was removed in \cite{kresch-chow}. 
The base change property is \cite[Proposition 7.2]{BF-normal-cone}. 
\end{proof} 

\begin{remark}
A virtual fundamental class can also be defined for algebraic stacks that are not necessarily Deligne--Mumford \cite{normal-cone-artin-stacks, poma-vclass}, but 
we will not need this. 
\end{remark} 

In the situation of Theorem~\ref{theorem-vclass}, 
the number $\rk \cF$ is called the \emph{virtual dimension} of $\cX$ over $S$ with respect to $\phi$. 
Note that the virtual dimension is preserved by base change on $S$ (see Remark~\ref{remark-bc-obs-thy}), 
and hence determined on any fiber of $\cX \to S$.  
In the case where $S$ is a point, we see that the virtual fundamental class is a cycle of (homological) degree given by the virtual dimension. 
A particularly nice situation is when the virtual dimension is $0$. 
In this case, if $S$ is a point and $\cX \to S$ is proper, the \emph{virtual count} of $\cX$ with 
respect to $\phi$ is the number 
\begin{equation*}
\#^{\vir}_{\phi}(\cX) = \int_{[\cX]_{\phi}^{\vir}} 1 ,  
\end{equation*} 
i.e. the degree of the virtual fundamental class. 
A fundamental property is that virtual counts of fibers are deformation invariant when 
$\cX$ is proper over $S$: 

\begin{corollary}
\label{corollary-vnumber-constant}
Let $\cX$ and $\phi$ be as in Theorem~\ref{theorem-vclass}. 
Assume that $\cX \to S$ is proper and the virtual dimension of $\cX$ over $S$ with respect to $\phi$ is $0$. 
Then for $s \in S$ a regular closed point, the virtual count 
$\#^{\vir}_{\phi_s}(\cX_s)$ of the (underived) fiber $\cX_s$ is independent of $s$. 
\end{corollary} 

\begin{proof}
This follows from Theorem~\ref{theorem-vclass}, as the virtual fundamental classes $[\cX_s]^{\vir}_{\phi_s}$ are the fibers of the family of $0$-cycles $[\cX]_{\phi}^{\vir}$ on $\cX$ over $S$, and hence have the same degree. 
\end{proof} 

\subsection{Symmetric obstruction theories} 
The vanishing of the virtual dimension can be guaranteed by imposing a 
symmetry on the obstruction theory.  

\begin{definition}
Let $\cX$ be an algebraic stack over $S$. 
A \emph{symmetric obstruction theory} for $\cX$ over $S$ is a 
triple $(\phi, L, \theta)$ where $\phi \colon \cF \to \tau^{\geq -1} L_{\cX/S}$ is a 
relative obstruction theory, $\cF$ is a perfect complex, $L$ is a line bundle on $S$, 
and $\theta \colon \cF \to \cF^{\vee} \otimes L[1]$ is an isomorphism 
satisfying $\theta^{\vee} \otimes L[1] = \theta$. 
\end{definition} 

\begin{remark}
The data of $\theta$ is equivalent to that of a degree $1$ nondegenerate symmetric bilinear form $\beta \colon \cF \otimes \cF \to \cO_{\cX} \otimes L[1]$; see \cite[Remark 1.2]{BF-symmetric}. 
\end{remark} 

\begin{remark} 
The virtual dimension of a Deligne--Mumford stack $\cX$ over $S$ with respect to a symmetric perfect obstruction theory is indeed $0$, because then 
$\rk \cF = \rk(\cF^{\vee} \otimes L[1]) = - \rk \cF$. 
The symmetry of an obstruction theory also has other important consequences by Behrend's work \cite{behrend}, 
cf. Remark~\ref{remark-behrend-function} below. 
\end{remark} 

\begin{remark}
Parallel to Remark~\ref{remark-bc-obs-thy}, 
if $(\phi, L, \theta)$ is a symmetric obstruction theory for an algebraic stack $\cX$ over $S$, 
then for any morphism $u \colon S' \to S$ 
we can form in an evident way the base change $(\phi_{S'}, L_{S'}, \theta_{S'})$ 
which is a symmetric obstruction theory for $\cX'$ over $S'$. 
\end{remark} 

In the Deligne--Mumford case, any symmetric obstruction theory is automatically perfect: 

\begin{lemma}
\label{lemma-DM-spos}
Let $(\phi \colon \cF \to \tau^{\geq -1} L_{\cX/S}, L, \theta)$ be a symmetric obstruction theory 
for an algebraic stack $\cX$ over $S$. 
Then $\cF$ has Tor-amplitude $[-2,1]$. 
If $\cX$ is Deligne--Mumford, then $\cF$ has Tor-amplitude $[-1,0]$, and hence 
$\phi$ is a perfect obstruction theory. 
\end{lemma} 

\begin{proof}
By Remark~\ref{remark-tor-amp}, $\cF$ has Tor-amplitude in $[-\infty, 1]$. 
Dualizing and using symmetry, we find that $\cF \cong \cF^{\vee} \otimes L[1]$ also has Tor-amplitude $[-2, \infty]$, proving the first claim. 
If $\cX$ is Deligne--Mumford, then by Remark~\ref{remark-tor-amp} again $\cF$ has Tor-amplitude in $[-\infty, 0]$, so 
the same symmetry argument shows that $\cF$ has Tor-amplitude $[-1,0]$.  
\end{proof} 


\section{Obstruction theory for moduli of objects in a CY3 category} 
\label{section-obstruction-theory-CY3} 

In this section, we construct a reduced symmetric perfect obstruction theory on a suitable quotient of the moduli of objects in a CY3 category. 
This is the main technical ingredient in our definition of reduced DT invariants in \S\ref{section-DT-CY3}.

\subsection{Setup}
\label{section-setup}  
Now we describe the setup for the rest of this section. 
Let $\cC$ be a CY3 category of geometric origin over a smooth complex variety $S$. 
By Theorem~\ref{theorem-Ktop}, we have a local system $\Ktop[0](\cC/S)$ on $S$ whose fibers are $\Ktop[0](\cC_s)$ for $s \in S(\bC)$. 
Note that any object $E \in \cC$ defines a section \mbox{$v_{E} \in \Gamma(S, \Ktop[0](\cC/S))$} whose fiber over $s \in S(\bC)$ is the class of $E_s$ in $\Ktop[0](\cC_s)$, which is a Hodge class for the Hodge structure described in Theorem~\ref{theorem-Ktop} (see \cite[Lemma 5.10]{IHC-CY2}). 

Fix a section $v \in \Gamma(S, \Ktop[0](\cC/S))$ whose fibers $v_s \in \Ktop[0](\cC_s)$ are Hodge classes for all $s \in S(\bC)$. 
We define $s\cM_{\gl}(\cC/S, v) \subset s\cM_{\gl}(\cC/S)$ to be the open substack parameterizing objects with class $v$, 
where $\cM_{\gl}(\cC/S)$ is the moduli stack of simple gluable objects introduced in \S\ref{section-classical-moduli}. 
Note that as $\cC$ is connected over $S$, by Proposition~\ref{proposition-cAut} the stack of autoequivalences 
$\cAut(\cC/S)$ is an algebraic stack. 
Finally, 
let $\cM$ be an open substack of $s\cM_{\gl}(\cC/S, v)$, 
and let $\cG$ be an open subgroup stack of $\cAut(\cC/S)$ which is smooth over $S$ and whose action on $s\cM_{\gl}(\cC/S, v)$ preserves $\cM$. 

\begin{remark}
As explained in \S\ref{section-moduli-twisted-sheaves} and \S\ref{section-moduli-ss}, many such examples can be obtained by taking 
$\cM$ to be a moduli space of semistable objects and $\cG = \cAut^0(\cC/S)$ the identity component of $\cAut(\cC/S)$. 
We will spell this out more explicitly in Examples~\ref{example-moduli-sheaves-proper} and~\ref{example-moduli-objects-proper} below. 
\end{remark} 

\begin{remark}
\label{remark-quotient-algebraic}
Our goal is to construct a symmetric obstruction theory for the quotient stack $\cM/\cG$, which is perfect when this quotient is Deligne--Mumford. 
A priori the quotient $\cM/\cG$ is a higher algebraic stack, but as we explain now our assumptions 
actually imply it is $1$-truncated, i.e. an ordinary algebraic stack. 
Indeed, by Proposition~\ref{proposition-cAut} the stack $\cG$ is naturally a $\bG_m$-gerbe $\cG \to G$ over a group algebraic space $G$, and 
by Theorem~\ref{theorem-MC} the stack $\cM$ is naturally a $\bG_m$-gerbe 
$\cM \to M$ over an algebraic space $M$. 
Said differently, the group stack $B \bG_m$ is a subgroup of $\cG$ with 
quotient $G$, and acts on $\cM$ with quotient $M$. 
In particular, we find that $G$ acts on $M$, and there is an isomorphism 
\begin{equation*}
\cM/\cG \cong M/G. 
\end{equation*} 
As $M/G$ is an algebraic stack, this proves the above claim. 
\end{remark} 

\begin{remark}
Let $\cE \in \cC_{\cM}$ be the universal object. 
Before passing to the quotient, the stack $\cM$ has a natural obstruction theory over $S$, 
\begin{equation}
\label{phiM}
\phi_{\cM} \colon (\cHom_{\cM}(\cE, \cE)[1])^{\vee} \to \tau^{\geq -1} L_{\cM/S} , 
\end{equation} 
which arises via Remark~\ref{remark-enhancement-obs} from the derived enhancement 
$\fM \subset \fM(\cC/S)$ (see Remark~\ref{example-LM}) and the description of the cotangent 
complex of $\fM(\cC/S)$ (Theorem~\ref{theorem-LMder}). 
This can be thought of as the starting point for the obstruction theory we will construct for $\cM/\cG$ below. 
\end{remark} 

\subsection{Obstruction theory for $\cM/\cG$} 
\label{section-obs-thy-MG}
Now we can formulate the main result of this section. 
Let $\cE \in \cC_{\cM}$ be the universal object and 
let $L$ be the line bundle on $S$ such that \mbox{$\rS_{\cC/S} = (- \otimes L)[n]$}. 
Consider the morphism 
\begin{equation}
\label{alpha}
\alpha \colon \tau^{\leq 1} \cHH^*(\cC_{\cM}/\cM) \otimes L[2] \to \cHom_\cM(\cE,\cE) \otimes L[2]
\end{equation} 
induced by the canonical morphism $a_{\cE} \colon \cHH^*(\cC_{\cM}/\cM) \to \cHom_{\cM}(\cE,\cE)$, and the morphism 
\begin{equation}
\label{beta}
\beta \colon \cHom_\cM(\cE,\cE) \otimes L[2] \to \tau^{\geq 2}(\cHH_*(\cC_{\cM}/\cM)) \otimes L[2] 
\end{equation} 
induced by the Chern character $\ch_{\cE} \colon \cHom_\cM(\cE,\cE) \to \cHH_*(\cC_{\cM}/\cM)$. For degree reasons, the composition 
\begin{equation}
\label{complex-obs-theory}
\tau^{\leq 1} \cHH^*(\cC_{\cM}/\cM) \otimes L[2] \xrightarrow{\, \alpha \,} \cHom_\cM(\cE,\cE) \otimes L[2] \xrightarrow{\, \beta \,} \tau^{\geq 2}(\cHH_*(\cC_{\cM}/\cM)) \otimes L[2] 
\end{equation} 
vanishes, so we can consider the ``cohomology'' of this sequence defined by 
\begin{equation*}
\cF \coloneqq \cofib(  \tau^{\leq 1} \cHH^*(\cC_{\cM}/\cM) \otimes L[2] \xrightarrow{\, \alpha' \,} \fib( \beta)) 
\end{equation*} 
where $\alpha' \colon \tau^{\leq 1} \cHH^*(\cC_{\cM}/\cM)  \to \fib(\beta)$ is the morphism through which $\alpha$ factors. 

\begin{theorem}
\label{theorem-obs-theory}
In the setup of \S\ref{section-setup}, there is a canonical symmetric obstruction theory $(\phi \colon \overline{\cF} \to \tau^{\geq -1} L_{(\cM/\cG)/S}, L, \theta)$ for $\cM/\cG$ over $S$, such that if 
$q \colon \cM \to \cM/\cG$ is the quotient morphism then $q^* \overline{\cF} \simeq \cF$. 
Moreover, if $\cM/\cG$ is Deligne--Mumford, then this obstruction theory is perfect. 
\end{theorem}  

We will prove the theorem in several steps. 

\begin{step}{1}
\label{step-1} 
Description of an obstruction theory coming from the derived enhancement of $\cM/\cG$.  
\end{step} 

By Proposition~\ref{proposition-quotient-enhancement}, Remark~\ref{remark-derived-obs-thy-rewritten},  
and Remark~\ref{remark-enhancement-obs}, we obtain an obstruction theory 
\begin{equation*}
\phi' \colon \overline{\cF}' \to \tau^{\geq -1} L_{(\cM/\cG)/S} 
\end{equation*}  
where 
\begin{equation*}
\overline{\cF}' = \fib(  (\cHom_{\cM}(\cE, \cE)[1])^{\vee} \xrightarrow{\, \overline{\beta} \,} 
\tau^{\geq 0}((\cHH^*(\cC_{\cM}/\cM)[1])^{\vee})  )  
\end{equation*} 
and all terms are understood as objects in $\Dperf(\cM/\cG)$ (with their natural equivariant structure). 
Using Corollary~\ref{CY-HH}\eqref{CY-ch-dual}, we see that the pullback to $\cM$ of 
$\overline{\cF}'$ is given by  
\begin{equation*}
q^* \overline{\cF}' \simeq  
\fib(\cHom_\cM(\cE,\cE) \otimes L[2] \xrightarrow{\, \beta \,} \tau^{\geq 2}(\cHH_*(\cC_{\cM}/\cM)) \otimes L[2]) 
\end{equation*} 
where $\beta$ is the morphism~\eqref{beta} described above. 
The rest of the proof consists of ``making $\phi'$ symmetric''. 

\begin{step}{2}
\label{step-2}
Construction of $\overline{\cF} \in \Dperf(\cM/\cG)$. 
\end{step} 

By construction, taking the dual of the morphism $\overline{\beta}$ from Step~\ref{step-1} and tensoring by $L[1]$ gives a morphism 
\begin{equation*} 
\overline{\alpha} \colon \tau^{\leq 1} \cHH^*(\cC_{\cM}/\cM) \otimes L [2] \to \cHom_\cM(\cE,\cE) \otimes L[2]
\end{equation*} 
in $\Dperf(\cM/\cG)$ that descends the morphism $\alpha$ from~\eqref{alpha} above.  
Note that by Serre duality, we have 
\begin{equation*}
\cHom_{\cM}(\cE, \cE) \otimes L[2] \simeq (\cHom_{\cM}(\cE, \cE)[1])^{\vee}. 
\end{equation*} 
Using this, we can define $\overline{\cF}$ as the ``cohomology'' of the sequence 
formed by $\overline{\alpha}$ and $\overline{\beta}$ (whose composition vanishes for degree reasons), i.e. 
\begin{equation*}
\overline{\cF} \coloneqq \cofib( \tau^{\leq 1} \cHH^*(\cC_{\cM}/\cM) \otimes L [2] \xrightarrow{\, \overline{\alpha}' \,} \overline{\cF}') 
\end{equation*}  
where $\overline{\alpha}'$ is the morphism through which $\overline{\alpha}$ factors. 
Then $\overline{\cF} \in \Dperf(\cM/\cG)$ because all of the terms in the sequence are perfect, and $q^* \overline{\cF} \simeq \cF$ by construction. 

\begin{step}{3}
\label{step-construction-phi}
Construction of the morphism $\phi \colon \overline{\cF} \to \tau^{\geq -1} L_{(\cM/\cG)/S}$. 
\end{step} 

We claim that the composition 
\begin{equation}
\label{composition-vanish}
\tau^{\leq 1} \cHH^*(\cC_{\cM}/\cM) \otimes L [2] \xrightarrow{\, \overline{\alpha}' \,} \overline{\cF}' \xrightarrow{\, \phi' \,} \tau^{\geq -1} L_{(\cM/\cG)/S} 
\end{equation} 
vanishes. 
Since for degree reasons we also have 
\begin{equation*} 
\Hom(\tau^{\leq 1} \cHH^*(\cC_{\cM}/\cM) \otimes L [3], \tau^{\geq -1} L_{(\cM/\cG)/S})) = 0, 
\end{equation*} 
this implies that there is a unique morphism $\phi \colon \overline{\cF} \to \tau^{\geq -1} L_{(\cM/\cG)/S}$ such that $\phi' = \phi \circ \gamma$ where $\gamma \colon \overline{\cF}' \to \overline{\cF}$ is the canonical map. 

It remains to prove the claim. 
It suffices to prove that the pullback of the map~\eqref{composition-vanish} along the faithfully flat cover $q \colon \cM \to \cM/\cG$ vanishes. 
Consider the exact triangle 
\begin{equation*}
q^*L_{(\cM/\cG)/S} \to L_{\cM/S} \to L_{\cM/(\cM/\cG)}. 
\end{equation*} 
Combining Lemma~\ref{lemma-L-quotient} with Lemma~\ref{lemma-L-cAut}, 
we see that $L_{\cM/(\cM/\cG)} \in \Dqc^{[0,1]}(\cM)$. 
In particular, the map $\cH^{-1}(q^*L_{(\cM/\cG)/S}) \to \cH^{-1}(L_{\cM/S})$ on cohomology sheaves in 
degree $-1$ is an isomorphism. On the other hand, for degree reasons the map~\eqref{composition-vanish} factors via a map 
\begin{equation*}
\cHH^{1}(\cC_{\cM}/\cM) \otimes L[1] \to \cH^{-1}(L_{(\cM/\cG)/S}). 
\end{equation*} 
Altogether, these observations imply that~\eqref{composition-vanish} vanishes if and only if the composition 
\begin{equation}
\label{composition-vanish-2} 
\cHH^{1}(\cC_{\cM}/\cM) \otimes L[1] \to \tau^{\geq -1} q^*L_{(\cM/\cG)/S} \to \tau^{\geq -1} L_{\cM/S} 
\end{equation}  
vanishes. 
As the source and target have bounded coherent cohomology, 
by Nakayama's lemma it suffices to check that~\eqref{composition-vanish-2} vanishes after pullback along any 
locally of finite type $\bC$-point $m \colon \Spec(\bC) \to \cM$ 
(such a point is sometimes called a ``finite type point'' of $\cM$, see \cite[\href{https://stacks.math.columbia.edu/tag/06FW}{Tag 06FW}]{stacks-project}).
By Corollary~\ref{CY-HH}\eqref{CY-HH-co-ho} and Theorem~\ref{theorem-HH}\eqref{base-change-HH} 
we have the base change formula $m^*(\cHH^{1}(\cC_{\cM}/\cM) \otimes L[1]) = \cHH^1(\cC_{m}/\bC)[1]$. 
Therefore, it suffices to check the map 
\begin{equation}
\label{composition-vanish-3} 
\HH^1(\cC_m/\bC) \to \rH^{-1}(m^*\tau^{\geq -1} L_{\cM/S} ) = \rH^{-1}(m^* L_{\cM/S})
\end{equation} 
vanishes. 
By construction this map can also be described as $\rH^{-1}$ of the pullback along $m$ of 
the composition
\begin{equation*}
\cHH^*(\cC_{\cM}/\cM) \otimes L[2] \xrightarrow{\, a_{\cE} \otimes L[2] \,} 
\cHom_{\cM}(\cE, \cE) \otimes L[2] \simeq 
(\cHom_{\cM}(\cE, \cE)[1])^{\vee} \xrightarrow{\, \phi_{\cM} \,} \tau^{\geq -1} L_{\cM/S} ,
\end{equation*} 
where the equivalence is Serre duality and $\phi_{\cM}$ is the obstruction theory~\eqref{phiM} for $\cM$. 
Therefore, using base change,~\eqref{composition-vanish-3} identifies with 
\begin{equation*}
\HH^1(\cC_m/\bC) \xrightarrow{\, a_{\cE_{m}} \,} \Ext^1(\cE_m, \cE_m) 
\cong \Ext^2(\cE_m, \cE_m)^{\vee} 
\to \rH^{-1}(m^*L_{\cM/S}). 
\end{equation*} 
Dualizing and using Corollary~\ref{CY-HH}\eqref{CY-ch-dual}, this identifies with  
\begin{equation}
\label{composition-vanish-4}
\Ext^1(m^*L_{\cM/S}, \bC) \to \Ext^2(\cE_m, \cE_m) \xrightarrow{\, \ch_{\cE_m} \,} \HH_{-2}(\cC_m/\bC). 
\end{equation} 

To show that~\eqref{composition-vanish-4} vanishes, we will use a general result which allows us to interpret the 
source in terms of obstruction theory. 
Recall that for an algebraic stack $\cX$ over $S$, if $j \colon T \to T'$ is a square-zero thickening 
of schemes defined by an ideal $I$, then given any solid commutative diagram 
\begin{equation}
\label{obstruction-diagram}
\vcenter{
\xymatrix{
T \ar[d]_-{j} \ar[r]^-{g} & \cX \ar[d] \\
T' \ar[r]_{u} \ar@{-->}[ur]^{g'} & S 
}}
\end{equation} 
there is an obstruction $\omega(g,j,u) \in \Ext^1(g^*L_{\cX/S}, I)$ which vanishes if and only if a dotted arrow $g'$ making the diagram 
commute exists, in which case the set of such lifts $g'$ forms a torsor under $\Ext^0(g^*L_{\cX/S}, I)$. 
We say that a homomorphism $A' \to A$ of artinian local $\bC$-algebras with residue field $\bC$ is a 
\emph{small extension} if its kernel $I$ is annihilated by $\fm_{A'}$ and isomorphic to $\bC$. 
Note that if $T = \Spec(A)$ and $T' = \Spec(A')$ above and $x \colon \Spec(\bC) \to \cX$ is the restriction of $g \colon T \to \cX$ to the closed point, then $\Ext^1(g^*L_{\cX/S}, \bC) = \Ext^1(x^*L_{\cX/S}, \bC)$, so the obstruction space only depends on the $\bC$-point $x$. 

\begin{lemma}
\label{lemma-obstruction}
Let $\cX$ be an algebraic stack locally of finite type over $S$. 
Let $x \colon \Spec(\bC) \to \cX$ be a locally of finite type $\bC$-point. 
Then every element of $\Ext^1(x^*L_{\cX/S}, \bC)$ arises as $\omega(g,j,u)$ for a diagram~\eqref{obstruction-diagram} with $T \to T'$ the map induced by a small extension of artinian local $\bC$-algebras with residue field $\bC$ such that $x \colon \Spec(\bC) \to \cX$ is restriction of $g$ to the closed point. 
\end{lemma} 

\begin{remark}
We work with $\bC$ because that is the setup of this section, but the result holds more generally for any algebraically closed base field. 
\end{remark} 

\begin{proof}
Choose $p \colon U \to \cX$ a smooth surjection from a scheme $U$. 
As $x$ is a locally of finite type $\bC$-point, we can lift it along $p$ to a $\bC$-point $u \colon \Spec(\bC) \to U$ 
\cite[\href{https://stacks.math.columbia.edu/tag/06FX}{Tag 06FX}]{stacks-project}. 
The morphism $p^*L_{\cX/S} \to L_{U/S}$ induces a map 
\begin{equation}
\label{obstruction-U-X}
\Ext^1(u^*L_{U/S}, \bC) \to \Ext^1(x^*L_{\cX/S}, \bC)
\end{equation}
which is easily seen to be an isomorphism. 
Given a diagram
\begin{equation*} 
\label{obstruction-diagram-U}
\vcenter{
\xymatrix{
T \ar[d]_-{j} \ar[r]^-{\tilde{g}} & U \ar[d] \\
T' \ar[r]_{u}  & S 
}}
\end{equation*} 
with $T \to T'$ a square-zero thickening induced by a small extension, we obtain a diagram as in~\eqref{obstruction-diagram} by taking $g = p \circ \tilde{g}$, and by functoriality the map~\eqref{obstruction-U-X} takes the obstruction $\omega(\tilde{g},j,u)$ to $\omega(g,j,u)$. Therefore, to prove the result we may assume that $\cX$ is a scheme. 

Let $\pi \colon \cX \to S$ be the structure morphism, and let $s = \pi \circ x \colon \Spec(\bC) \to S$. 
The exact triangle $\pi^*L_{S/\bC} \to L_{\cX/\bC} \to L_{\cX/S}$ induces an exact sequence 
\begin{equation}
\label{obstruction-es-SX}
\Ext^0(s^*L_{S/\bC}, \bC) \to \Ext^1(x^*L_{\cX/S}, \bC) \to \Ext^1(x^*L_{\cX/\bC}, \bC) \to 0
\end{equation} 
where the final term is $0$ due to the smoothness of $S$ over $\bC$. 
Given a solid commutative diagram 
\begin{equation*}
\vcenter{
\xymatrix{
T \ar[dd]_-{j} \ar[r]^-{g} & \cX \ar[d]^{\pi} \\
 & S  \ar[d] \\  
T' \ar[r]_-{v} \ar@{-->}[ur]^{u} & \Spec(\bC)  
}}
\end{equation*} 
with $T \to T'$ a square-zero thickening induced by a small extension, 
by the smoothness of $S$ over $\bC$ a dotted arrow $u$ making the diagram commute exists and the set of such $u$ form a $\Ext^0(s^*L_{S/\bC}, \bC)$-torsor. 
For elements arising as obstructions from such a diagram, the sequence~\eqref{obstruction-es-SX} may be interpreted as follows: 
for varying $u$ the elements $\omega(g,j,u)$ form the fiber of 
the map $\Ext^1(x^*L_{\cX/S}, \bC) \to \Ext^1(x^*L_{\cX/\bC}, \bC)$ over $\omega(g,j,v)$. 
Therefore, by the sequence~\eqref{obstruction-es-SX}, to prove the result it suffices to 
treat the case $S = \Spec(\bC)$, which is handled in \cite[Proposition 4.7]{BF-normal-cone}. 
\end{proof} 

Finally, we can finish the proof that~\eqref{composition-vanish-4} vanishes. 
By Lemma~\ref{composition-vanish-4}, it suffices to prove the map vanishes 
when evaluated on $\omega(g,j,u) \in \Ext^1(m^*L_{\cM/S}, \bC)$ for a diagram 
\begin{equation*} 
\vcenter{
\xymatrix{
\Spec(A) \ar[d]_-{j} \ar[r]^-{g} & \cM \ar[d] \\
\Spec(A') \ar[r]_{u}  & S 
}}
\end{equation*} 
where $A' \to A$ is a small extension of artinian local $\bC$-algebras with residue field $\bC$ such that $m \colon \Spec(\bC) \to \cM$ is restriction of $g$ to the closed point. 
Note that the morphism $g$ is classified by the pullback $\cE_A \in \cC_{A}$ of the universal object $\cE \in \cC_{\cM}$.  
By \cite[\S2.5]{pridham}
(combined with the description from Theorem~\ref{theorem-Ktop} of the variation of Hodge structures $\Ktop[0](\cC/S)$ as a summand of that of $X \to S$), 
we find that the image of $\omega(g,f,u)$ under~\eqref{composition-vanish-4} vanishes due to our assumptions that $\cM \subset s\cM_{\gl}(\cC/S,v)$ and $v \in \Gamma(S, \Ktop[0](\cC/S))$ is of Hodge type along~$S$. 
This completes Step~\ref{step-construction-phi} of the proof.  

\begin{step}{4}
$\phi \colon \overline{\cF} \to \tau^{\geq -1} L_{(\cM/\cG)/S}$ is an obstruction theory. 
\end{step}

By construction, we have $\phi' = \phi \circ \gamma$ where $\phi'$ is the obstruction theory constructed in Step~\ref{step-1} and $\gamma \colon \overline{\cF}' \to \overline{\cF}$ is the canonical map. Therefore, we obtain an exact triangle 
\begin{equation*}
\cofib(\gamma) \to \cofib(\phi') \to \cofib(\phi) .
\end{equation*} 
As $\phi'$ is an obstruction theory we have $\cofib(\phi') \in \Dqc^{\leq -2}(\cX/\cG)$, 
and by construction we have $\cofib(\gamma) =  \tau^{\leq 1} \cHH^*(\cC_{\cM}/\cM) \otimes L [3]$, so also $\cofib(\gamma) \in \Dqc^{\leq -2}(\cX/\cG)$. 
We conclude that $\cofib(\phi) \in \Dqc^{\leq -2}(\cX/\cG)$, i.e. $\phi$ is an obstruction theory. 

\begin{step}{5}
\label{step-5} 
Construction of $\theta$ such that $(\phi, L, \theta)$ is a symmetric obstruction theory. 
\end{step} 

By the construction of Step~\ref{step-2}, $\overline{\cF}$ is the perfect complex given by the ``cohomology'' of the sequence 
\begin{equation*}
 \tau^{\leq 1} \cHH^*(\cC_{\cM}/\cM) \otimes L [2] \xrightarrow{\, \overline{\alpha} \,} \cHom_\cM(\cE,\cE) \otimes L[2] \simeq (\cHom_{\cM}(\cE, \cE)[1])^{\vee} \xrightarrow{\, \overline{\beta} \,} 
\tau^{\geq 0}((\cHH^*(\cC_{\cM}/\cM)[1])^{\vee}),  
\end{equation*} 
which is self-dual up to tensoring with $L[1]$. 
It follows formally from this that there is an induced isomorphism 
$\theta \colon \overline{\cF} \to \overline{\cF}^{\vee} \otimes L[1]$ 
satisfying $\theta^{\vee} \otimes L[1] = \theta$, i.e. 
$(\phi, L, \theta)$ is a symmetric obstruction theory. 

\begin{step}{6}
$\phi$ is perfect if $\cM/\cG$ is Deligne--Mumford. 
\end{step}

In view of Step~\ref{step-5}, the claim follows from Lemma~\ref{lemma-DM-spos}.
This completes the proof of Theorem~\ref{theorem-obs-theory}. \qed 

\begin{remark}
\label{remark-derived-enhancement-MG}
A natural question is whether the obstruction theory of Theorem~\ref{theorem-obs-theory} arises from a derived enhancement of $\cM/\cG$. 
Our proof goes halfway toward answering this question: 
in Step~\ref{step-1} we obtain an obstruction theory using the derived enhancement of $\cM/\cG$ from Proposition~\ref{proposition-quotient-enhancement}, and in Steps~\ref{step-2}-\ref{step-5} we ``symmetrize'' it by hand to obtain Theorem~\ref{theorem-obs-theory}. 
The question is whether the derived enhancement from Proposition~\ref{proposition-quotient-enhancement} can itself be modified to realize the symmetrized obstruction theory. 
We will address this in future work, where we will provide a positive answer using ideas from \cite{pridham} (already invoked in Step~\ref{step-construction-phi}) and derived symplectic geometry. 
\end{remark}

\section{Reduced DT invariants of CY3 categories } 
\label{section-DT-CY3} 

In this section, we define reduced DT invariants of CY3 categories that are preserved under deformations, using the obstruction theory constructed in~\S\ref{section-obstruction-theory-CY3}. 
In particular, 
we prove Theorems~\ref{theorem-intro-DT-definition} and \ref{theorem-intro-DT-defo-invariant} from the introduction. 

\subsection{DT invariant of $\cM/\cG$} 
\begin{definition}
\label{definition-DT}
In the setup of \S\ref{section-setup}, assume that $S = \Spec(\bC)$ is a point and $\cM/\cG$ is Deligne--Mumford and proper over $\bC$. 
Then the \emph{Donaldson-Thomas (DT) invariant} of $\cM/\cG$ is the number 
\begin{equation*}
\DT(\cM/\cG) \coloneqq \#_{\phi}^{\vir}(\cM/\cG)
\end{equation*} 
where $\phi$ is the obstruction theory constructed in Theorem~\ref{theorem-obs-theory}. 
\end{definition} 

\begin{corollary}
\label{corollary-DT-constant}
In the setup of \S\ref{section-setup}, assume that $\cM/\cG$ is Deligne--Mumford and proper over $S$. 
Then the DT invariants $\DT(\cM_s/\cG_s)$ of the fibers are independent of $s \in S(\bC)$. 
\end{corollary} 

\begin{proof}
Combine Theorem~\ref{theorem-obs-theory} and Corollary~\ref{corollary-vnumber-constant}. 
\end{proof} 

\begin{remark}
\label{remark-behrend-function}
In the situation of Definition~\ref{definition-DT}, if $\cM/\cG$ is also quasi-projective in the sense of \cite[Definition 0.1]{behrend}, 
then by \cite[Proposition 4.16 and Theorem 4.18]{behrend} the DT invariant $\DT(\cM/\cG)$ can be computed in terms of the Behrend function, and in particular is an intrinsic invariant of $\cM/\cG$. 
\end{remark} 

The following observation is useful for checking the properness hypothesis of Corollary~\ref{corollary-DT-constant} in examples. 
We recall that $\cM \to M$ is a $\bG_m$-gerbe over an algebraic space $M$ over $S$ 
and $\cG \to G$ is a $\bG_m$-gerbe over a group algebraic space $G$ over $S$. 

\begin{lemma}
\label{lemma-MG-proper} 
In the setup of \S\ref{section-setup}, assume that $M \to S$ and $G \to S$ are proper. 
Then $\cM/\cG \to S$ is proper. 
\end{lemma} 

\begin{proof}
By Remark~\ref{remark-quotient-algebraic} we have $\cM/\cG \cong M/G$, so it suffices to prove properness of the latter. 
As $M \to M/G$ is a smooth cover and $M \to S$ is proper, if follows that $M/G \to S$ is finite type (being locally finite type as $M \to S$ is and quasi-compact by \cite[\href{https://stacks.math.columbia.edu/tag/050X}{Tag 050X}]{stacks-project}) and universally closed by \cite[\href{https://stacks.math.columbia.edu/tag/0CQK}{Tag 0CQK}]{stacks-project}. 
For separatedness of $M/G \to S$, we must show that the diagonal is proper; equivalently, if $T$ is an $S$-scheme and $\rho_i \colon P_i \to M_T$, $i=1,2$, are two $T$-points of $M/G$ ($G_T$-torsors equipped with an equivariant map to $M_T$), we must show the Isom-space $I \coloneqq \mathrm{Isom}(P_1, P_2) \to T$ is proper. 
Working \'{e}tale locally on $T$, we may assume the $P_i$ are trivial, and then it is easy to see that $I$ is given by the (underived) fiber product diagram 
\begin{equation*}
\xymatrix{
I \ar[r] \ar[d] & G_T \ar[d]^{(\rho_1(1), \rho_2)} \\ 
M_T \ar[r]^-{\Delta} & M_T \times_T M_T
}
\end{equation*} 
where $\Delta$ is the diagonal and $1 \in G_T(T)$ is the identity section. 
The bottom arrow is a closed immersion as $M \to S$ is separated, 
so $I \to G_T$ is a closed immersion. 
As $G \to S$ is proper, we conclude that the composition $I \to G_T \to T$ is proper. 
\end{proof} 

\subsection{Reduced DT invariants for moduli of semistable objects}
In the rest of this section, we discuss the DT invariants defined above in the two main cases of interest: moduli of Giesker semistable twisted sheaves and Bridgeland semistable objects. 

\begin{example}[Moduli of semistable twisted sheaves]
\label{example-moduli-sheaves-proper}
Let $f \colon X \to S$ be a smooth proper morphism of relative dimension $3$ with geometrically connected fibers and 
$\omega_f = f^*L$ for a line bundle $L$ on $S$, where $S$ is a smooth complex variety.  
Let $\alpha \in \BrAz(X)$ be a Brauer class. 
Then $\cC = \Dperf(X, \alpha)$ is a CY3 category of geometric origin over $S$ by Lemma~\ref{lemma-twisted-CYn}. 

Let $v \in \Gamma(S, \Ktop[0](\cC/S))$ be a section whose fibers $v_s \in \Ktop[0](\cC_s)$ are Hodge classes for all $s \in S(\bC)$, and let $H$ be a relatively ample divisor on $X \to S$.
Let $\cM_H(v) \to S$ denote the moduli stack of Gieseker $H$-semistable $\alpha$-twisted sheaves of class $v$. 
Then, as noted in Example~\ref{example-relative-stability}, $\cM_H(v) \to S$ is an open substack of $\cM_{\gl}(\cC/S, v)$, universally closed over $S$.

By Lemma~\ref{lemma-Aut0-CYn-proper}, the identity component $\Aut^0(\cC/S)$ of $\Aut(\cC/S)$ exists and is smooth and proper over $S$. 
The identity component $\cAut^0(\cC/S)$ of the stack of autoequivalences (Definition~\ref{definition-identity-component}) is a $\bG_m$-gerbe over $\Aut^0(\cC/S)$, and in particular smooth over $S$. Moreover, it acts on the stack $\cM_H(v)$ by 
Lemma~\ref{lemma-aut0-preserve-open}. 

If for all $s \in S(\bC)$ there do not exist strictly $H_s$-semistable $\alpha_s$-twisted sheaves on $X_s$ of class $v_s$,  
then $\cM_H(v)$ parameterizes simple objects and hence is an open substack of $s\cM_{\gl}(\cC/S, v)$ 
with the structure of a $\bG_m$-gerbe 
$\cM_H(v) \to M_H(v)$ over an algebraic space; moreover,  
by standard results on (twisted) sheaves, $M_H(v)$ is proper over $S$ 
\cite{huybrechts-lehn, lieblich-moduli-twisted}. 
Altogether, when this holds we have shown that $\cM = \cM_H(v)$ and $\cG = \cAut^0(\cC/S)$ satisfy the hypotheses of the setup in \S\ref{section-setup}, as well as Lemma~\ref{lemma-MG-proper}. 
Therefore, we may make the following definition. 
\end{example} 

\begin{definition}
\label{definition-DT-twisted}
In the setup of Example~\ref{example-moduli-sheaves-proper}, assume that $S = \Spec(\bC)$ is a point, 
there do not exist strictly $H$-semistable $\alpha$-twisted sheaves of class $v$, 
and the quotient stack 
$\cM_H(v)/\cAut^0(\Dperf(X, \alpha)/\bC)$ is Deligne--Mumford. 
Then we define the \emph{reduced DT invariant} for $H$-semistable sheaves of class $v$ by 
\begin{equation*}
\DT_{H}(v) \coloneqq \DT(\cM_H(v)/\cAut^0(\Dperf(X, \alpha)/\bC)). 
\end{equation*} 
\end{definition} 

For simplicity, we have chosen the notation $\DT_{H}(v)$ without any adornment to indicate it is the reduced invariant. 
This should not cause any confusion as we will only consider reduced DT invariants. 

\begin{corollary}
\label{corollary-DT-twisted-constant} 
In the setup of Example~\ref{example-moduli-sheaves-proper}, assume that 
for all $s \in S(\bC)$ there do not exist strictly $H_s$-semistable $\alpha_s$-twisted sheaves of class $v_s$, 
and $\cM_H(v)/\cAut^0(\Dperf(X, \alpha)/S)$ is Deligne--Mumford. 
Then the reduced DT invariants $\DT_{H_s}(v_s)$ of the fibers are independent of $s \in S(\bC)$. 
\end{corollary} 

\begin{remark}
In the absence of a Brauer class (i.e. when $\alpha = 0$), reduced DT invariants have been defined using 
virtual fundamental classes in special cases by various authors. 
Definition~\ref{definition-DT-twisted} recovers these previous ad hoc definitions in a uniform way. 
Let us spell out the comparison in two important cases: 
\begin{enumerate}
\item Suppose $X$ is strict Calabi--Yau threefold, i.e. a smooth projective complex variety $X$ with $\omega_X \cong \cO_X$ and $h^1(\cO_X) = 0$. 
Then by the HKR isomorphism $\HH^1(\Dperf(X)/\bC)$ vanishes, so by Lemma~\ref{lemma-Aut0} the group $\Aut^0(\Dperf(X)/\bC)$ is trivial. 
Thus, in this case Definition~\ref{definition-DT-twisted} recovers the classical DT invariant defined by Thomas \cite{thomas-DT, BF-symmetric}. 
\item Suppose $X$ is an abelian threefold. 
In this case, Gulbrandsen \cite{gulbrandsen} defined reduced DT invariants under certain assumptions which imply that the conditions in our Definition~\ref{definition-DT-twisted} are satisfied. 
His construction goes through a slice for the $\Aut^0(\Dperf(X)/\bC)$-action on $M_H(v)$. 
More precisely, he constructs a space $M'$ equipped with a symmetric perfect obstruction theory $\psi$ and an action by a finite group $G'$ such that 
\begin{equation*}
M'/G' \cong M_H(v)/\Aut^0(\Dperf(X)/\bC), 
\end{equation*} 
and then defines $\DT_H(v)$ as $\#^{\vir}_{\psi}(M')/|G'|$ \cite[Definition 3.2]{gulbrandsen}. 
It is easy to see that symmetric perfect obstruction theories pull back to symmetric perfect obstruction theories along \'{e}tale morphisms. 
Let $\phi'$ be the pullback of the obstruction theory $\phi$ from Theorem~\ref{theorem-obs-theory} along the \'{e}tale morphism 
\begin{equation*}
M' \to M'/G' \cong M_H(v)/\Aut^0(\Dperf(X)/\bC). 
\end{equation*}  
By Remark~\ref{remark-behrend-function} we have $\#^{\vir}_{\phi'}(M') = \#^{\vir}_{\psi}(M')$, 
so by functoriality of virtual fundamental classes \cite[Proposition 5.10]{BF-normal-cone} we conclude 
\begin{equation*}
\#^\vir_{\phi}(M_H(v)/\Aut^0(\Dperf(X)/\bC)) = \#^{\vir}_{\phi'}(M')/|G'| = \#^{\vir}_{\psi}(M')/|G'|. 
\end{equation*} 
This shows the agreement of Gulbrandsen's definition with ours. 
\end{enumerate}
\end{remark} 

\begin{remark}
Using Hall algebra techniques, in \cite{OPT} reduced generalized DT invariants are defined for Calabi--Yau threefolds. In the presence of strictly semistable objects, it is not clear that their invariants are deformation-invariant. However, when there are no strictly semistable objects, the invariants defined here coincide with theirs, and Corollary~\ref{corollary-DT-twisted-constant} implies deformation-invariance. 
\end{remark} 

Reduced DT invariants for moduli of semistable objects can be defined parallel to the case of sheaves: 

\begin{example}[Moduli of semistable objects] 
\label{example-moduli-objects-proper}
Let $\cC$ be a CY3 category of geometric origin over a smooth complex variety $S$. 
Let $v \in \Gamma(S, \Ktop[0](\cC/S))$ be a section whose fibers $v_s \in \Ktop[0](\cC_s)$ are Hodge classes for all $s \in S(\bC)$. 
Let $\sigma$ be a stability condition on $\cC$ over $S$ with respect to a topological Mukai homomorphism. 
Let $\phi \in \bR$ be a phase compatible with $v$. 
Then, as noted in Example~\ref{example-relative-stability}, 
the moduli space $\cM_{\sigma}(v, \phi)$ of $\sigma$-semistable objects in $\cC$ of class $v$ and phase $\phi$ is an open substack $\cM_{\gl}(\cC/S, v)$, universally closed over $S$. 

By Lemma~\ref{lemma-HH1-CY}, the identity component $\Aut^0(\cC/S)$ of $\Aut(\cC/S)$ exists, and both it and $\cAut^0(\cC/S)$ are smooth over $S$. 
Further, $\cAut^0(\cC/S)$ acts on the stack $\cM_{\sigma}(v, \phi)$ by 
Lemma~\ref{lemma-aut0-preserve-open}. 

If for all $s \in S(\bC)$ there do not exist strictly $\sigma_s$-stable objects of class $v_s$, 
then $\cM_{\sigma}(v, \phi)$ is an open substack of $s\cM_{\gl}(\cC/S, v)$ with the structure of a 
$\bG_m$-gerbe $\cM_{\sigma}(v, \phi) \to M_{\sigma}(v, \phi)$ over an algebraic space, 
which is proper over $S$ by \cite[Theorem 21.24]{stability-families}. 
Altogether, 
when this holds 
we have shown that $\cM = \cM_\sigma(v,\phi)$ and $\cG = \cAut^0(\cC/S)$ satisfy the hypotheses of the setup in \S\ref{section-setup}, and  $\cM$ satisfies the hypothesis of Lemma~\ref{lemma-MG-proper}. 
If $\cC = \Dperf(X, \alpha)$ is as in Example~\ref{example-moduli-sheaves-proper}, then as we discussed there $\Aut^0(\cC/S)$ also satisfies the hypothesis of Lemma~\ref{lemma-MG-proper}, i.e. is proper over $S$, by Lemma~\ref{lemma-Aut0-CYn-proper}; 
in general, we do not know whether this holds (cf. Remark~\ref{remark-Aut0-proper}), so we take it as an assumption in the definition and corollary below. 
\end{example}

The following definition and corollary complete the proofs of Theorems~\ref{theorem-intro-DT-definition} and \ref{theorem-intro-DT-defo-invariant} promised in the introduction, 
which we state here in terms of moduli stacks and automorphism stacks instead of their corresponding algebraic spaces, cf. Remark~\ref{remark-quotient-algebraic}. 

\begin{definition}
\label{definition-DT-CY3}
In the setup of Example~\ref{example-moduli-objects-proper}, assume that $S = \Spec(\bC)$ is a point, 
$\Aut^0(\cC/\bC)$ is proper, there do not exist strictly $\sigma$-semistable objects of class $v$, and 
the quotient stack $\cM_\sigma(v,\phi)/\cAut^0(\cC/\bC)$ is Deligne--Mumford. 
Then we define the \emph{reduced DT invariant} for $\sigma$-semistable objects of class $v$ by 
\begin{equation*}
\DT_{\sigma}(v,\phi) \coloneqq \DT(\cM_\sigma(v,\phi)/\cAut^0(\cC/\bC)). 
\end{equation*} 
\end{definition} 

\begin{corollary}
\label{corollary-DT-CY3-constant} 
In the setup of Example~\ref{example-moduli-objects-proper}, assume that 
$\Aut^0(\cC/S) \to S$ is proper, for all $s \in S(\bC)$ there do not exist strictly $\sigma_s$-semistable objects of class $v_s$, and 
the quotient stack $\cM_\sigma(v,\phi)/\cAut^0(\cC/S)$ is Deligne--Mumford. 
Then the reduced DT invariants $\DT_{\sigma_s}(v_s, \phi)$ of the fibers are independent of $s \in S(\bC)$. 
\end{corollary} 

\begin{remark}
It is easy to see that $\DT_{\sigma}(v, \phi)$ is independent of the phase $\phi \in \bR$ compatible with $v$, so it is natural to define 
\begin{equation*}
\DT_{\sigma}(v) \coloneqq \DT_{\sigma}(v, \phi)
\end{equation*}
for any such $\phi$. 
\end{remark}


\section{The variational integral Hodge conjecture for CY3 categories} 
\label{section-VIHC-CY3}

In this section we prove the following criterion for the validity of the noncommutative variational integral Hodge conjecture, and deduce Theorem~\ref{theorem-intro-VIHC-sigma} from the introduction. 

\begin{theorem}
\label{theorem-VIHC-general}
Let $\cC$ be a CY3 category of geometric origin over a smooth complex variety $S$. 
Let $v$ be a section of $\Ktop[0](\cC/S)$ whose fibers $v_s \in \Ktop[0](\cC_s)$ are Hodge classes for all $s \in S(\bC)$. 
Let $\cM \subset \cM_{\gl}(\cC/S, v)$ be an open substack which is universally closed over $S$. 
Let $U \subset S$ be a nonempty open subset such that $\cM_U$ is contained in the substack $s\cM_{\gl}(\cC/S, v)$ of simple objects. 
Let $\cG_U$ be an open subgroup stack of $\cAut(\cC_U/U)$ which is smooth over $U$ and whose action on $s\cM_{\gl}(\cC/S, v)$ preserves $\cM_U$, such that $\cM_U/\cG_U \to U$ is proper. 
Assume there exists a point $0 \in S(\bC)$ such that: 
\begin{enumerate}
\item $\cM_0/\cG_0$ is Deligne--Mumford. 
\item $\DT(\cM_0/\cG_0) \neq 0$. 
\end{enumerate}
Then for every $s \in S(\bC)$ the space $\cM_s$ is nonempty and, in particular, the class $v_s \in \Ktop[0](\cC_s)$ is algebraic. 
\end{theorem} 

\begin{proof}
It suffices to show that the morphism $\cM \to S$ is surjective. 
As it is universally closed by assumption, it suffices to show that it is dominant. 
By Lemma~\ref{lemma-fiber-DM-open} below, up to shrinking $U$ we may assume that 
$\cM_U/\cG_U$ is Deligne--Mumford. 
Then by Corollary~\ref{corollary-DT-constant} the numbers $\DT(\cM_s/\cG_s)$ are constant for $s \in U(\bC)$, 
and hence nonzero as $\DT(\cM_0/\cG_0) \neq 0$. 
In particular, the space $\cM_s/\cG_s$, and hence also $\cM_s$, is nonempty for $s \in U(\bC)$. 
\end{proof} 

The next lemma, invoked above, says that 
for an algebraic stack universally closed over a base, the Deligne--Mumford locus is open on the base. 
\begin{lemma}
\label{lemma-fiber-DM-open}
Let $\pi \colon \cX \to S$ be a universally closed morphism from an algebraic stack $\cX$ to a scheme $S$. 
Let $0 \in S$ be a point such that the fiber $\cX_0$ is Deligne--Mumford. 
Then there exists an open neighborhood $0 \in U \subset S$ such that 
$\cX_U$ is Deligne--Mumford. 
\end{lemma} 

\begin{proof}
Consider the diagonal morphism of $\pi$, which fits into a diagram 
\begin{equation*}
\xymatrix{
\cX \ar[rr]^{\Delta_{\pi}} \ar[dr]_{\pi} && \cX \times_S \cX \ar[dl] \\ 
&  S & 
}
\end{equation*} 
The fiber over $0$ is $\Delta_{\pi_0} \colon \cX_0 \to \cX_0 \times_{\Spec(\kappa(0))} \cX_0$, 
the diagonal of $\cX_0$, which by assumption is unramified. 
As $\Delta_{\pi}$ is locally of finite type  \cite[\href{https://stacks.math.columbia.edu/tag/04XS}{Tag 04XS}]{stacks-project}, 
there is an open subset $V \subset \cX$ where $\Delta_{\pi}$ is unramified, which by the previous observation 
contains $\cX_0$. 
Consider the closed complement $Z = \cX \setminus V$ and set $U = S \setminus \pi(Z)$, which is open as $\pi$ is universally closed. 
By construction, $0 \in U$ and $\Delta_{\pi_U} \colon \cX_U \to \cX_U \times_U \cX_U$ is unramified, 
i.e. $\cX_U$ is Deligne--Mumford. 
\end{proof} 

Now we specialize Theorem~\ref{theorem-VIHC-general} to the case of moduli spaces of sheaves. 

\begin{corollary}
\label{corollary-VIHC-Gieseker}
Let $f \colon X \to S$ be a smooth proper morphism to a smooth variety $S$, 
such that $f$ has relative dimension $3$ with geometrically connected fibers 
and  $\omega_f = f^*L$ is the pullback of a line bundle on $S$.
Let $\alpha \in \BrAz(X)$ be a Brauer class. 
Let $v$ be a section of $\Ktop[0]((X, \alpha)/S)$ whose fibers $v_s \in \Ktop[0](X_s, \alpha_s)$ are Hodge classes for all $s \in S(\bC)$. 
Let $H$ be a relatively ample divisor on $X \to S$. 
Assume there exists a point $0 \in S(\bC)$ such that: 
\begin{enumerate}
\item There do not exist strictly $H_0$-semistable $\alpha_0$-twisted sheaves of class $v_0$. 
\item $\cM_{H_0}(v_0)/\cAut^0(\Dperf(X_0, \alpha_0)/\bC)$ is Deligne--Mumford. 
\item $\DT_{H_0}(v_0) \neq 0$. 
\end{enumerate} 
Then for every $s \in S(\bC)$ the moduli space $\cM_{H_s}(v_s)$ is nonempty, and in particular, the class 
$v_s \in \Ktop[0](X_s, \alpha_s)$ is algebraic. 
\end{corollary} 

\begin{proof}
We claim there exists a neighborhood $0 \in U \subset S$ such that 
for every $u \in U(\bC)$ there do not exist strictly $H_u$-semistable $\alpha_u$-twisted sheaves of class $v_u$. 
Indeed, let $\cM_H^{\mathrm{st}}(v) \subset \cM_{H}(v)$ denote the open subspace parameterizing $H$-stable sheaves, let $Z \subset \cM_H(v)$ be its complement, and 
let $Y \subset S$ be the image of $Z$, which is closed as $\cM_H(v) \to S$ is universally closed; 
then $U = S \setminus Y$ is the desired neighborhood. 

Set $\cM \coloneqq \cM_{H}(v)$ and $\cG \coloneqq \cAut^0(\Dperf(X,\alpha)/S)$.  
By the discussion in Example~\ref{example-moduli-sheaves-proper} and Lemma~\ref{lemma-MG-proper}, it follows that $\cM_U = \cM_{H_U}(v_U)$ and $\cG_U$ satisfy the hypotheses of Theorem~\ref{theorem-VIHC-general}. 
Thus Theorem~\ref{theorem-VIHC-general} applies and gives the result. 
\end{proof} 

By the same argument, we also obtain the following analogous result for moduli of stable objects in a category, 
which gives the promised Theorem~\ref{theorem-intro-VIHC-sigma} from the introduction; 
for consistency with the above, we restate the result here in terms of moduli stacks and automorphism stacks, cf. Remark~\ref{remark-quotient-algebraic}. 

\begin{corollary}
\label{corollary-VIHC-sigma}
Let $\cC$ be a CY3 category of geometric origin over a smooth complex variety $S$. 
Let $v$ be a section of $\Ktop[0](\cC/S)$ whose fibers $v_s \in \Ktop[0](\cC_s)$ are Hodge classes for all $s \in S(\bC)$. 
Let $\sigma$ be a stability condition on $\cC$ over $S$ with respect to a topological Mukai homomorphism. 
Let $\phi \in \bR$ be a phase compatible with $v$. 
Assume there exists a point $0 \in S(\bC)$ such that: 
\begin{enumerate}
\item \label{sigma0-v0-generic}
There do not exist strictly $\sigma_0$-semistable objects of class $v_0$. 
\item $\cM_{\sigma_0}(v_0, \phi)/\cAut^0(\cC_0/\bC)$ is Deligne--Mumford. 
\item $\DT_{\sigma_0}(v_0) \neq 0$. 
\end{enumerate} 
Then for every $s \in S(\bC)$ the moduli space $\cM_{\sigma_s}(v_s, \phi)$ is nonempty, and in particular, the class 
$v_s \in \Ktop[0](\cC_s)$ is algebraic. 
\end{corollary} 

Assumption~\eqref{sigma0-v0-generic} 
in Corollary~\ref{corollary-VIHC-sigma} is mild. 
Indeed, if the image of $v_0$ in the lattice for the stability condition is primitive, then the assumption can be achieved up to deforming $\sigma$, due to the following general observation. 

\begin{lemma}
\label{lem:v_generic_relative_stability}
Let $\cC \subset \Dperf(X)$ be an $S$-linear semiorthogonal component, where $X \to S$ is a smooth proper morphism of complex varieties. 
Let $\sigma$ be a stability condition on $\cC$ over $S$ with respect to a Mukai homomorphism 
$\bv \colon \Knum(\cC/S) \to \Lambda$. 
Let $\lambda \in \Lambda$ be a primitive class. 
Then for any $0 \in S(\bC)$, there exists a stability condition $\sigma'$ on $\cC$ over $S$ with respect to $\bv$ 
such that there do not exist strictly $\sigma'_0$-semistable objects of class $\lambda$. 
\end{lemma} 

\begin{proof}
Let $\Stab_{\Lambda}(\cC/S)$ denote the space of stability conditions over $S$ with respect to $\bv$. 
Let $\bv_0 \colon \Knum(\cC_0) \to \Lambda$ be the Mukai homomorphism obtained by composing the canonical map $\Knum(\cC_0) \to \Knum(\cC/S)$ with $\bv$, and 
let $\Stab_{\Lambda}(\cC_0/\bC)$ denote the space of stability conditions on $\cC_0$ over $\bC$ with respect to $\bv_0$. 
We have maps 
\begin{equation*}
\Stab_{\Lambda}(\cC/S) \to \Stab_{\Lambda}(\cC_0/\bC) \to \Hom(\Lambda, \bC), 
\end{equation*} 
where the first map sends a stability condition over $S$ to its fiber over $0$, 
and the second map sends a stability condition to its central charge. 
By the deformation theorem for relative stability conditions \cite[Theorem 1.2]{stability-families}, 
$\Stab_{\Lambda}(\cC/S)$ and $\Stab_{\Lambda}(\cC_0/\bC)$ are complex manifolds such that the above two maps are local isomorphisms. 
Therefore, 
it suffices to show there exists a small deformation $\sigma'_0$ of $\sigma_0$ 
such that there do not exist strictly $\sigma'_0$-semistable objects of class $\lambda$. 
But this follows from the wall and chamber structure for $\Stab_{\Lambda}(\cC_0/\bC)$ \cite[Lemma 12.13]{stability-families}: it suffices to take $\sigma'_0$ to be a nearby stability condition that does not lie on a wall for $\lambda$. 
\end{proof}

\begin{remark}
Even if $\lambda \in \Lambda$ is not primitive, one can 
often construct a stability condition over $S$ for which there do not exist strictly semistable objects of class $\lambda$, after possibly enlarging the lattice $\Lambda$, cf. \cite[Example 2.17 and Remark 30.7]{stability-families}. 
However, for our applications, we will not need such a result.
\end{remark}



\newpage 
\part{The period-index conjecture for abelian threefolds} 
\label{part-PI-abelian-3folds}

With the exception of Corollary~\ref{cor:over_other_fields}, throughout Part~\ref{part-PI-abelian-3folds} we work over the complex numbers. 


\section{The main result} 
\label{section-main-result-abelian-3folds}

We now state our main result on the period-index conjecture. 

\begin{theorem}
\label{thm:period_index_conjecture}
    Let $X$ be a complex abelian threefold and let $\alpha \in \Br(X)$ be a Brauer class. Then
    \[
      \ind(\alpha) \mid \per(\alpha)^2.
    \]
\end{theorem}

Since the index of a Brauer class may only drop under specialization, 
we obtain an analogous result 
in positive characteristic by lifting, which was stated as Theorem~\ref{theorem-main} in the introduction.

\begin{corollary}  
\label{cor:over_other_fields}
    Let $X$ be an abelian threefold over an algebraically closed field $k$. 
    For any $\alpha \in \Br(X)$ such that $\per(\alpha)$ is prime to the characteristic of $k$, we have  
    \[
      \ind(\alpha) \mid \per(\alpha)^2.
    \]
\end{corollary}

\begin{proof}
The case of an algebraically closed field $k$ of characteristic $0$ follows from Theorem~\ref{thm:period_index_conjecture} and the Lefschetz principle. When $k$ may have positive characteristic, $X$ lifts to an abelian scheme $X_R$ over an integral local domain $R$ of characteristic $0$ with residue field $k$ \cite{oort_norman}. 

The class $\alpha$ lies in the image of a class $\theta \in \rH^2(X, \bmu_n)$, where $n$ is the period of $\alpha$. Since $n$ is invertible on $R$, smooth and proper base change implies that $\theta$ lifts to a class $\theta_R \in \rH^2(X_R, \bmu_n)$. Writing $\eta \in \Spec R$ for the generic point, the restriction of $\theta_R$ to the geometric generic fiber $X_{\bar \eta}$ gives a Brauer class $\alpha_{\bar \eta} \in \Br(X_{\bar \eta})[n]$. Since $\kappa(\bar{\eta})$ is algebraically closed of characteristic $0$, by the previous paragraph we have 
\[
        \ind(\alpha_{\bar \eta}) \mid \per(\alpha_{\bar \eta})^2,
\]
so $\ind(\alpha_{\bar \eta})$ divides $n^2$. Then observe that $\ind(\alpha) \mid \ind(\alpha_{\bar \eta})$, since the index of a Brauer class may only decrease under specialization \cite[Lemma 6.2]{dJ-period-index}.
\end{proof}

\begin{remark}
\label{remark-lifting-AV}
As the proof shows, when $\per(\alpha)$ is divisible by the characteristic of $k$, the result still holds as long as $\alpha \in \Br(X)$ admits a lift along with $X$ to characteristic $0$. It is plausible that such a lift always exists. 
\end{remark}

We make a few comments on Theorem~\ref{thm:period_index_conjecture}. 
First, for an abelian threefold  $X$ and a Brauer class $\alpha \in \Br(X)$, it is not difficult to see that
\[
        \ind(\alpha) \mid \per(\alpha)^3.
\]
This follows from symbol length considerations, and the analogous result holds more generally for an abelian variety of arbitrary dimension. Symbol length bounds are discussed in \S \ref{sec:symbol_length}, where we show that they do not suffice to prove Theorem~\ref{thm:period_index_conjecture}.

A different approach to proving Theorem~\ref{thm:period_index_conjecture} might be to show that for any twisted abelian threefold $(X, \alpha)$, the twisted derived category $\Dperf(X, \alpha)$ is equivalent to $\Dperf(X')$, for $X'$ an abelian threefold. Since the integral Hodge conjecture holds for $X'$ \cite{voisin-IHC, grabowski-IHC, totaro} and the cohomology of $X'$ is torsion free, the integral Hodge conjecture holds for $\Dperf(X')$ \cite[Corollary 5.18]{IHC-CY2} and thus also for $\Dperf(X, \alpha)$. From there, Theorem~\ref{thm:period_index_conjecture} would follow from the observation that there are Hodge classes in $\Ktop[0](X, \alpha)$ of rank $\per(\alpha)^2$. Unfortunately, we show in \S \ref{sec:fourier_mukai_partners} that $\Dperf(X, \alpha)$ is  not generally equivalent to $\Dperf(Y)$ for any smooth  proper variety $Y$.

\subsection{Outline of the proof} 
\label{outline-of-PI-proof}
The basic premise of the argument is the following. For each abelian threefold $X_\init$ and Brauer class $\alpha_\init \in \Br(X_\init)$, the goal is to construct a family of twisted abelian threefolds $(X, \alpha)$ over a base $S$, a section $v \in \Gamma(\Ktop[0]((X, \alpha) / S))$ whose fibers are Hodge classes, 
a stability condition $\sigma$ on $(X, \alpha)$ over $S$ (with respect to a topological Mukai homomorphism), and points $\fin,\init \in S(\bC)$ such that the following conditions hold:
\begin{enumerate} 
         \item \label{step:construct_hodge} The fiber over $\init \in S(\bC)$ is the original twisted abelian threefold $(X_\init, \alpha_\init)$, and the rank of $v_\init$ is equal to $\per(\alpha_\init)^2$. 
         \item \label{step:brauer_class_vanish} One has $\alpha_\fin = 0 \in \Br(X_\fin)$, where $X_\fin$ is the fiber over $\fin \in S(\bC)$.
         \item \label{step:v1_generic} At $\fin \in S(\bC)$, there do not exist strictly $\sigma_\fin$-semistable objects of class $v_\fin$. 
         \item \label{step:dm} The quotient stack $\cM_{\sigma_\fin}(v_\fin)/\cAut^0(\Dperf(X_\fin)/\bC)$ is Deligne--Mumford. 
         \item \label{step:compute_dt} $\DT_{\sigma_\fin}(v_\fin) \neq 0$.
 \end{enumerate} 
If all of these conditions are satisfied, then Corollary~\ref{corollary-VIHC-sigma} shows that $v_\init$ is algebraic. In particular, there exists an object in $\Dperf(X_\init, \alpha_\init)$ of rank $\per(\alpha_\init)^2$, so $\ind(\alpha_\init) \mid \per(\alpha_\init)^2$. The condition that $\alpha_\fin = 0$ is not necessary in order to apply Corollary~\ref{corollary-VIHC-sigma}; rather, we have included it here because it plays an important role in making the computation \eqref{step:compute_dt} possible.

Let us now address how to arrange for conditions \eqref{step:construct_hodge}--\eqref{step:compute_dt}. Condition~\eqref{step:construct_hodge} is handled through the theory of twisted Mukai structures, which computes explicitly the variation of Hodge structure $\Ktop[0]((X, \alpha)/S)$ and allows one to construct global sections of Hodge type. This is explained in \S\ref{sec:hodge_classes_on_twisted_abelian_varieties} below. The precise choice of $v$ is rather intricate, as will be explained when we address condition~\eqref{step:compute_dt}.

Condition~\eqref{step:brauer_class_vanish} is arranged as follows. The Brauer class $\alpha_\init$ on $X_\init$ lies in the image of a class
\[
        \theta_\init \in \rH^2(X_\init, \bmu_n),
\]
where $n = \per(\alpha)$. If $X_\fin$ is a deformation of $X_\init$, then $\alpha_\init$ deforms to the trivial Brauer class on $X_\fin$ if and only if a parallel transport $\theta_\fin$ of $\theta_\init$ lies in the image of $\Pic(X_\fin)/n$. Therefore, in order to find such a deformation of $X_\init$, we first lift $\theta_\init$ to a class $\tilde \theta_\init \in \rH^2(X_\init, \bZ)$, and then choose $X_\fin$ to lie in the Hodge locus of $\tilde \theta_\init$ in a suitable moduli space of polarized abelian varieties. 

Condition~\eqref{step:v1_generic} is responsible for the use of the theory of stability conditions over a base in the argument, as opposed to the more familiar stability of (twisted) sheaves with respect to a polarization. This is for a purely practical reason, which we now explain. Suppose we wished to run the argument using stability of sheaves with respect to a relative polarization $H$ on $X$ over $S$. Generically, $X_\init$ has Picard rank $1$, so we have no choice when it comes to the relative polarization itself. On the other hand, $X_{\fin}$ has Picard rank at least $2$ from the previous step, which means we have little control over whether strictly $H_\fin$-semistable twisted sheaves of class $v_\fin$ exist. 
Instead, we use Lemma~\ref{lem:v_generic_relative_stability} to construct a relative stability condition $\sigma$ such that strictly $\sigma_\fin$-semistable objects of class $v_\fin$ do not exist.

Condition~\eqref{step:dm} is the first of three appearances in the proof of a remarkable feature of the cohomology of abelian threefolds, which also plays an important role in \cite{OPT}. There is a $\bQ$-valued quartic form $\Delta$, called \emph{Igusa's discriminant}, on the rational even cohomology of an abelian threefold. It turns out that if $X_\fin$ is simple, then for $v_\fin \in \Ktop(X_\fin)$, the condition
\[
        \Delta(\ch(v_\fin)) \neq 0
\]
is enough to guarantee that the quotient stack in~\eqref{step:dm} is Deligne--Mumford (Lemma~\ref{lem:discr_no_stabilizer}). One can arrange that $X_\fin$ is simple from the beginning of the proof; this amounts to choosing $\tilde \theta_1$ in~\eqref{step:brauer_class_vanish} so that the general member of its Hodge locus has multiplication by a totally real cubic field.

Condition~\eqref{step:compute_dt} is the most intricate step, and impacts a number of choices made at the beginning of the proof. Very little is directly known about the DT invariants of higher rank classes on abelian threefolds, so the strategy is to transform the problem of computing $\DT_{\sigma_\fin}(v_\fin)$ into the problem of counting certain curves on $X_\fin$. 

To explain how this works, we first recall a result from \cite{OPT}: If $\Delta(v_\fin) \geq 0$, where $\Delta$ is Igusa's discriminant, then $\DT_{\sigma_\fin}(v_\fin)$ is invariant under wall-crossing of stability conditions and the action of autoequivalences of $\Dperf(X_\fin)$ on $v_\fin$. With this in mind, we arrange (through the choice of $v$ at the beginning of the proof) that $v_\fin$ can be transformed by autoequivalences into a class of Chern character $(1, 0, -\beta, -n)$, which is the shape of the Chern character of an ideal sheaf of a curve. By invariance under wall-crossing, it is enough to calculate $\DT_{H_\fin}(1,0, - \beta, -n)$, where $H_\fin$ is any polarization on $X_\fin$.

We have therefore reduced to a curve count. If $\beta \in \rH_2(X_\fin, \bZ)$ is a curve class and $H_0$ is any polarization, then a formula for the generating series
\begin{equation}
\label{eq:summary_generating_series}
        \sum_{n} \DT_{H_0}(1, 0, - \beta, -n) q^n
\end{equation}
is known by work of Oberdieck--Shen \cite{obshen} under the assumption that $\beta$ is of type $(1, 1, d)$ for some $d > 0$. This condition means that $\beta$ is of the form $e_1 \wedge e_2 + e_3 \wedge e_4 + d\cdot e_5 \wedge e_6$, for $\{e_i\}$ a basis of $\rH_1(X_\fin, \bZ)$. An analysis of the series \eqref{eq:summary_generating_series} leads to the observation that one has
\[
        \DT_{H_0}(1, 0, -\beta, -n) \neq 0
\]
whenever Igusa's discriminant is nonnegative: $\Delta(1, 0, -\beta, -n) \geq 0$. (Though we have no need of it, the converse is nearly true.)  
This curve counting argument is carried out in \S\ref{section-nonvanishing-curve-invariants}. 

Arranging that $\Delta(\ch(v_\fin)) \geq 0$ is not difficult, and since Igusa's discriminant is invariant under the action of autoequivalences, one gets that $\Delta(1, 0, - \beta, -n) \geq 0$. Unfortunately, choosing the initial class $v$ which guarantees that $\beta$ is ultimately of type $(1, 1, d)$ is subtle. At first glance, the only way to check that $\beta$ is of type $(1, 1, d)$ seems to be to verify that it has at least $4$ after reduction modulo every prime number $p$. Fortunately, it is possible to arrange the situation so that there is a \emph{finite} set of primes $\{\ell_1, \dots, \ell_n\}$, depending rather loosely on $v$, such that if $\beta$ has rank at least $4$ modulo each $\ell_i$, then $\beta$ is of type $(1, 1, d)$. This, along with some flexibility afforded by the choice of $v$, allows us to arrange for $\beta$ to have type $(1, 1, d)$. 
We have collected in Appendix~\ref{appendix-abelian-3folds} a number of auxiliary results on abelian varieties which are used in this argument and in arranging condition~\eqref{step:brauer_class_vanish}.


\section{Hodge classes on twisted abelian varieties}
\label{sec:hodge_classes_on_twisted_abelian_varieties}

\begin{definition}
\label{def:twisted_mukai_structure}
    Let $X$ be an abelian variety. Given a class $B \in \rH^2(X, \bQ(1))$, the \emph{twisted Mukai structure} $\muk(X, B; \bZ)$ is the integral Hodge structure of weight $0$ on the even cohomology $\rH^{\ev}(X, \bZ)$ such that the homomorphism
    \[
            \muk(X, B; \bZ) \otimes \bQ \to \bigoplus \rH^{2i}(X, \bQ)(i), \quad v \mapsto \exp(B) \cdot v.
    \]
    is an isomorphism of $\bQ$-Hodge structures. 
    We also set $\muk(X, B; \bQ) = \muk(X, B; \bZ) \otimes \bQ$. 
\end{definition}

\begin{remark}
        In the K3 surfaces literature (e.g. \cite{Huy_stell}), it is standard to consider twisted Mukai structures as having weight $2$, which differs from our convention.
\end{remark}

\begin{remark}
        The case of abelian varieties is rather special, since the topological $K$-theory $\Ktop[0](X)$ coincides with the integral cohomology $\rH^{\ev}(X, \bZ)$ under the Chern character. For a general variety with a topologically trivial Brauer class, the two differ, and one should use the lattice of topological $K$-theory to define twisted Mukai structures \cite{hotchkiss-pi}.
\end{remark}

\begin{definition}
\label{def:theta_field}
    Let $X$ be an abelian variety. Given a class $\theta \in \rH^2(X, \bmu_n)$, a \emph{$\theta$-field} is a class $B \in \rH^2(X, \bQ(1))$ such that the $n$th multiple is integral, i.e., $n B \in \rH^2(X, \bZ(1))$, and the image of $nB$ under the map 
    \[
            \exp(-/n):\rH^2(X, \bZ(1)) \to \rH^2(X, \bmu_n)
    \]
    is $\theta$.
\end{definition}

\begin{lemma}
\label{lem:twisted_mukai_abelian-families}
    Let $f:X \to S$ be a flat family of abelian varieties, and let $\cX \to X$ be a $\bmu_n$-gerbe of class $\theta \in \rH^2(X, \bmu_n)$, with Brauer class $\alpha$. 
    \begin{enumerate} 
            \item There is an isomorphism of $\bQ$-variations of Hodge structure
            \[
                \ch:\Ktop[0]((X, \alpha)/S) \otimes \bQ \to \rR^{\ev} f_* \bQ \coloneqq \bigoplus_{k \geq 0} \rR^{2k}f_* \bQ(k). 
             \] 
             The formation of $\ch$ is compatible with base change $S' \to S$. Here, $\Dperf(X, \alpha)$ implicitly refers to the category of $1$-twisted sheaves on a $\bmu_n$-gerbe of class $\alpha$.
        \item \label{item:identify_with_tw_muk} For $s \in S(\bC)$ and any $\theta_s$-field $B_s$, there is an isomorphism of integral Hodge structures
        \[
                \varphi_s:\Ktop[0](X_s, \alpha_s) \to  \muk(X_s, B_s; \bZ) 
        \]
        such that the diagram
        \begin{equation} \label{eq:compatibility_in_family_tw_muk}
             \begin{tikzcd}
                     \Ktop[0](X_s, \alpha_s) \ar{d}[swap] {\varphi_s} \ar[r, "\ch_s"] & \rH^{\ev}(X_s, \bQ) \\
                     \rH^\ev(X_s, B_s; \bZ) \ar[ur, "\exp(B_s) \cdot"']
             \end{tikzcd}
        \end{equation}
        commutes.
        \item (Parallel transport) Given $s, s' \in S(\bC)$ and a continuous path $\gamma$ from $s$ to $s'$, the diagram \eqref{eq:compatibility_in_family_tw_muk} is compatible with parallel transport. More precisely, if $\Phi_{\gamma}$ is the parallel transport isomorphism from $\rH^{\ev}(X_s, \bQ)$ to $\rH^{\ev}(X_{s'}, \bQ)$, then  
        \[
                \Phi_{\gamma}\left(\exp(B_s) \cdot \varphi_s \right) = \exp(\Phi_{\gamma}(B_s)) \cdot \varphi_{s'}
        \]
        \item In the situation of \eqref{item:identify_with_tw_muk}, a class $v \in \Ktop[0](X_s, \alpha)$ lifts to a global section of the variation $\Ktop[0]((X, \alpha)/S)$ if and only if $\exp(B_s)\cdot \varphi_s(v)$ lifts to a global section of $\rR^{\ev} f_* \bQ$.
    \end{enumerate}
\end{lemma}

\begin{proof}
    This is a reformulation of the results of \cite{hotchkiss-pi}.
\end{proof}

\begin{remark}[Fibers with trivial Brauer class]
\label{remar:fibers_with_trivial_brauer}
    In the situation of Lemma~\ref{lem:twisted_mukai_abelian-families}, suppose that there is a point $0 \in S(\bC)$ and a $\theta_0$-field $B_0$ which is algebraic, i.e., lies in $\NS(X_0)$. Then we may choose an equivalence
    \begin{equation}
    \label{eq:twist_down_by_twisted_lb}
        \Dperf(X_0, \alpha_0) \to \Dperf(X_0)
    \end{equation}
    such that the induced isomorphism $\muk(X_0, B_0; \bZ) \simeq \muk(X_0, \bZ)$ is the identity on the abelian group $\rH^{\ev}(X_0, \bZ)$ which underlies both Hodge structures.

    The point is that if $\cX \to X$ is a $\bmu_n$-gerbe of class $\theta$, then $\cX_0 = \cX \times_X X_0$ is isomorphic to the gerbe of $n$th roots of a line bundle $L$ on $X_0$ with first Chern class $n \cdot B_0$ \cite[2.3.4.2]{lieblich-moduli-twisted}. Then $\cX_0$ carries an $n$-twisted line bundle $L^{1/n}$ such that $(L^{1/n})^n$ descends to $L$. The equivalence \eqref{eq:twist_down_by_twisted_lb} is given by twisting down by $L^{1/n}$, and the diagram
    \begin{equation}
        \begin{tikzcd}
            \Ktop[0](X_0, \alpha_0) \ar[d, "\cdot {(L^{1/n})^\vee}"] \ar[r, "\varphi_0"] & \muk(X_0, B_0; \bZ) = \rH^{\ev}(X_0, \bZ) \ar[d, "\id"] \\
            \Ktop[0](X_0) \ar[r, "\ch"] & \muk(X_0, \bZ) = \rH^\ev(X_0, \bZ).
        \end{tikzcd}
    \end{equation}
    commutes: $\varphi_0$ is given by twisting down by any topological $1$-twisted line bundle with twisted first Chern class $B_0$, and $L^{1/n}$ is an example of such \cite{hotchkiss-pi}.
    
 \end{remark}

We now specialize to the case of threefolds. 

\begin{lemma}
\label{lem:constructing_rational_hodge_class}
    Let $f:(X,H) \to S$ be a flat, polarized family of abelian threefolds over a connected base, with a class $\theta \in \rH^2(X, \bmu_n)$ of Brauer class $\alpha \in \Br(X)$. Fix a basepoint $\init \in S(\bC)$, and a $\theta_{\init}$-field $B_{\init}$. 
    \begin{enumerate} 
            \item \label{z-hodge} For any $x, y \in \bQ$, the class
            \[
                    \left( n^{2}, -n^2 B_{\init}, \frac{n^2}{2} B_\init^2 + x \cdot \frac{H^2}{2}, y \cdot [\pt] \right) \in \muk(X_\init, B_{\init}; \bQ)
            \]
            is equal to $\varphi_\init(v_{\init})$ for a global section $v$ of $\Ktop[0]((X, \alpha)/S) \otimes \bQ$ which is everywhere Hodge.
            \item Given an arbitrary point $s \in S(\bC)$ and the choice of a parallel transport $B_s$ of $B_\init$, one has
            \[
                    \varphi_s(v_s) = \left( n^{2}, -n^2 B_{s}, \frac{n^2}{2} B_s^2 + x \cdot \frac{H^2}{2}, y \cdot [\pt] \right) \in \muk(X_s, B_{s}; \bQ).
            \]
    \end{enumerate}
\end{lemma}

\begin{proof}
        For~\eqref{z-hodge}, let $z$ be the displayed class. Note that
                \[ 
                        z = n^2 \exp(-B_1) + x' H^2 + y'[\pt] \in \rH^{\ev}(X, \bQ)
                \]
        for appropriate constants $x', y' \in \bQ$, so
        \[
                \exp(B_1)\cdot z = n^2 + x' H^2 +  y''[\pt]
        \]
        for $y'' = y' + x' B_1 H^2$. Then it is clear that $\exp(B_1) \cdot z$ is Hodge in each degree and extends to a global section of $\rR^{\ev} f_* \bQ$; from the Theorem of the Fixed Part \cite[Th\'eor\`eme 4.1.1]{Deligne1971ThorieDH}, the latter must be everywhere Hodge, since it is Hodge at $\init \in S(\bC)$. Then everything follows from Lemma~\ref{lem:twisted_mukai_abelian-families}.
\end{proof}


\section{The discriminant}
\label{section-discriminant} 

\subsection{Spin representations}

We begin by recalling some basic facts about spin representations. We refer the reader to \cite[\S 20]{fulton_harris}. 

Let $\rH$ be a $\bQ$-vector space of dimension $n$. The vector space $V = \rH \oplus \rH^\vee$ is equipped with a natural quadratic form $q$ from the pairing between $\rH$ and $\rH^\vee$. The even part of the Clifford algebra $\Cliff^+(V, q)$ admits a decomposition of algebras
\begin{equation} \label{eq:cliff_split}
    \Cliff^+(V, q) = \End \left(\textstyle{\bigwedge}^{\mathrm{ev}} \rH\right) \oplus \End \left(\textstyle{\bigwedge}^{\mathrm{odd}} \rH\right).
\end{equation}
There is a natural action of the spin group $\Spin(V, q)$ on $\textstyle{\bigwedge}^{\ev} H$ and $\textstyle{\bigwedge}^{\mathrm{odd}} H$; these are the half-spin representations of $\Spin(V, q)$.

\subsection{Igusa's formula}

In this section, we specialize to the case when $\dim \rH = 6$.

\begin{theorem}[Igusa]
\label{thm:invariant_polynomial}
    Let $\rH$ be a $\bQ$-vector space of dimension $6$, and let $\omega \in \bigwedge^6 \rH$ be a nonzero volume element. There is a unique quartic form
    \[
        \Delta: \medwedge^{\ev} \rH \to \bQ,
    \]
    satisfying the following properties:
    \begin{enumerate} 
        \item $\Delta$ is invariant under the action of $\Spin(V, q)$.
        \item $\Delta(1 + \omega) = - \frac{1}{4}$.
    \end{enumerate}
\end{theorem}

\begin{proof}
    See \cite[Appendix A]{OPT}. 
\end{proof}

Igusa \cite[\S 3, Proposition 3]{igusa} gives an explicit formula for $\Delta$, as follows. Let $x_1, \dots, x_6$ be a basis for $\rH$ such that 
\[
    \omega = x_1 \wedge x_2 \wedge \cdots \wedge x_6.
\]
Any element $v$ of $\bigwedge^{\mathrm{ev}} \rH$ may be written 
\[
    v = a + \sum_{i < j} b_{ij} x_{ij} + \sum_{i < j} c_{ij} x_{ij}^* + d \cdot \omega, \quad a, b_{ij}, c_{ij}, d \in \bQ
\]
where $x_{ij} = x_i \wedge x_j$, and $x^*_{ij}$ satisfies $x_{ij} \wedge x^*_{ij} = \omega$. We use the following notation:
\begin{itemize}
    \item $B$ is the $6 \times 6$ alternating matrix with $B_{ij} = b_{ij}$ for $i < j$.
    \item $C$ is the $6 \times 6$ alternating matrix with $C_{ij} = c_{ij}$ for $i < j$.
    \item For any pair of integers $1 \leq i, j \leq 6$, $B_{(i, j)}$ is the $4 \times 4$ alternating matrix obtained from $B$ by simultaneously deleting the $i$th row and column and the $j$th row and column from $B$.
    \item The $4 \times 4$ alternating matrix $C_{(i, j)}$ is defined analogously.
\end{itemize}

\begin{theorem}[Igusa]
\label{thm:igusa_explicit_formula}
    With notation as above, 
    \begin{equation}
    \begin{split}
        \Delta(v) = - a &\Pf(C) - d \Pf(B) \\
            &+ \sum_{i < j} \Pf(B_{(i, j)})\Pf(C_{(i,j)}) -  \frac{1}{4}\left( ad - \sum_{i < j} b_{ij}c_{ij} \right)^2, 
    \end{split}
    \end{equation}
    where $\Pf$ denotes the Pfaffian. 
\end{theorem}

\begin{proof}
    See \cite[\S 3, Proposition 3]{igusa}. We note, however, that the formula there appears to contain an inaccuracy in the sign of the first two terms (as one may see from the fact that it does not vanish for $v = \exp(x_{12} + x_{34} + x_{56})$, as it ought).
\end{proof}

\subsection{Discriminants for abelian threefolds}

\begin{definition}
\label{def:abelian_threefold}
Let $X$ be an abelian threefold. 
The \emph{discriminant} on $\Ktop[0](X)$ is the quartic form
\[
    \Delta:\Ktop[0](X) \otimes \bQ \simeq \medwedge^{\ev} \rH^1(X, \bQ)  \longrightarrow \bQ
\] 
associated to $[\pt] \in \rH^6(X, \bQ)$ as in Theorem~\ref{thm:invariant_polynomial}. 
\end{definition}

The group of autoequivalences $\Aut(\Dperf(X)/\bC)$ acts on $\Ktop[0](X)$ via Fourier--Mukai transforms. On the other hand, $\Spin(V, q)$ acts on $\Ktop[0](X) \otimes \bQ$ through the half-spin representation, where 
\[
    V = \rH^1(X, \bQ) \oplus \rH_1(X, \bQ).
\]

\begin{theorem}[Mukai]
\label{thm:mukai_thm}
    The group $\Aut(\Dperf(X)/\bC)$ acts on $\Ktop[0](X)$ via half-spin transformations. 
\end{theorem}

\begin{proof}
    See \cite{mukai_spin}.
\end{proof}

\begin{corollary}    
\label{cor:auteq_preserves_discr}
    The discriminant $\Delta$ is invariant under the action of $\Aut(\Dperf(X)/\bC)$.
\end{corollary}

\begin{lemma}
\label{lem:discr_no_stabilizer}
    Let $X$ be a simple abelian threefold over $\bC$, and let $v \in \Ktop[0](X)$ be a Hodge class with $\Delta(v) \neq 0$ and $\rk(v) \neq 0$. 
    Let $\sigma$ be a stability condition on $\Dperf(X)$ over $\bC$ such that there do not exist strictly $\sigma$-semistable objects of class $v$. 
    Let  $\phi \in \bR$ be a phase compatible with $v$. 
    Then the  stack
    \begin{equation*} 
        \cM_{\sigma}(v, \phi)/\cAut^0(\Dperf(X)/\bC)
    \end{equation*}
    is proper and Deligne--Mumford. 
\end{lemma}

\begin{remark}
    If $H$ is a polarization on $X$ such that there do not exist strictly $H$-semistable sheaves of class $v$, then the same result holds with $\cM_H(v)$ in place of $\cM_{\sigma}(v, \phi)$. 
\end{remark}

\begin{proof}
Note that $\cM_{\sigma}(v, \phi)$ consists of stable (in particular simple) objects, and its good moduli space $M_{\sigma}(v,\phi)$ is proper \cite[Theorem 21.24]{stability-families}. Moreover, by Lemma~\ref{lemma-Aut0-CYn-proper} the identity component of the space of autoequivalences 
    $\Aut^0(\Dperf(X)/\bC)$ is proper; in fact, it is isomorphic to $X \times X^{\vee}$.  
By Remark~\ref{remark-quotient-algebraic}, we have an isomorphism 
\begin{equation*}
        \cM_{\sigma}(v, \phi)/\cAut^0(\Dperf(X)/\bC) \cong 
        M_{\sigma}(v,\phi)/\Aut^0(\Dperf(X)/\bC). 
\end{equation*}
Therefore the properness of the quotient stack follows from Lemma~\ref{lemma-MG-proper}. 

    To prove the quotient is also Deligne--Mumford, 
    we show that any stable object $E$ with nonzero rank and positive-dimensional stabilizer in $\Aut^0(\Dperf(X)/\bC)$ has $\Delta(\ch(E)) = 0$. This is essentially \cite[Lemma 4.7]{OPT}, which uses Fourier--Mukai transforms to reduce to the classification of semihomogeneous vector bundles due to Mukai \cite{mukai_semihom}, but we give a direct proof using the classification of semihomogeneous complexes from \cite{dejong2022point}.

    Let $G_E \subset X \times X^\vee = \Aut^0(\Dperf(X)/\bC)$ be the stabilizer of $E$. Concretely, $G_E$ is the subgroup of pairs $(x, L)$ such that there exists an isomorphism $t_x^* E \otimes L \cong E$. 
    We claim that the projection $G_E \to X$ is surjective. The fiber over $0 \in X$ consists of line bundles $L$ such that $E \otimes L \cong E$. Observe, however, that
    \begin{equation*}
        \det(E \otimes L) = \det(E) \otimes L^{\rk E},
    \end{equation*}
    so $L$ is contained in the finite subgroup $X^{\vee}[\rk E]$. It follows that the projection $G_E \to X$ is finite over its image; since $G_E$ is positive-dimensional and $X$ is simple, the image is $X$.
    
    It follows that $E$ is a \emph{semihomogeneous complex} on $X$ (cf. \cite{dejong2022point}). From the classification of semihomogeneous complexes (cf. \cite[Proposition 4.7]{dejong2022point}), there is an isogeny $\pi:X' \to X$ and a line bundle $L$ on $X'$ such that $E \simeq \pi_* L'[s]$, for some shift $s \in \bZ$. Writing $c_1(L') = \pi^* w$ for some rational class $w \in \rH^2(X, \bQ)$, one sees that $\ch(E) = \pm (\deg \pi)\exp(w)$. Then 
    \begin{equation*}
        \Delta(\ch(E)) = (\deg \pi)^{4}\Delta(1,0,0,0) = 0
    \end{equation*}
    because multiplication by $\exp(w)$ is a spin transformation (cf. \cite[Appendix A]{obPT}). 
\end{proof}


\section{Nonvanishing of curve invariants}
\label{section-nonvanishing-curve-invariants}

Given an abelian threefold $X$, the type of a class in $\rH^4(X, \bC)$ is an associated tuple of integers $(d_1, d_2, d_3) \in \bZ_{\geq 0}^3$, whose definition is recalled in \S\ref{section-rank-type}. In what follows, by \emph{$\bQ$-effective curve class} $\beta \in \rH^4(X, \bZ)$, we mean that $\beta$ is a positive $\bQ$-linear combination of cycle classes $[C_i], C_i \subset X$ (cf. Lemma~\ref{lemma:poincare_dual_of_effective}).\footnote{We mention in passing that if $\beta \in \rH^4(X, \bZ)$ is $\bQ$-effective, it is not a priori clear that $\beta$ is $\bZ$-effective in the evident sense. However, when $\beta$ has type $(1,1,d$), $\bZ$-effectivity follows a posteriori from Proposition~\ref{prop:nonvanishing_theorem}.}

\begin{proposition}
\label{prop:nonvanishing_theorem}
    Let $X$ be an abelian threefold. Let $\beta \in \rH^4(X, \bZ)$ be a $\bQ$-effective curve class
    of type $(1, 1, d)$ for some $d > 0$, and let $n$ be an integer. Consider the class
    \[
        v = (1, 0, - \beta, -n),
    \]
    and assume that $\Delta(v) \geq 0$. Then, for any polarization $H$ on $X$, $\DT_H(v) > 0$.
\end{proposition}

In the above statement, we slightly abuse notation by identifying $(1,0,-\beta, -n) \in \rH^{\ev}(X, \bZ)$ with a class $v \in \Ktop[0](X)$ via the Chern character isomorphism $\Ktop[0](X) \cong \rH^{\ev}(X, \bZ)$. 
Moreover, implicit in Proposition~\ref{prop:nonvanishing_theorem} is the claim that the assumptions of Definition~\ref{definition-DT-twisted} are satisfied, so that $\DT_H(v)$ is well-defined:

\begin{remark}[Well-definedness]
\label{remark:well-defined}
    We explain why the invariant $\DT_H(v)$ is well-defined, which amounts to showing that the quotient stack 
    \begin{equation*}
        M_H(v)/X \times X^\vee
    \end{equation*}
    is Deligne--Mumford, so the assumptions of Definition~\ref{definition-DT-twisted} are met. In fact, the action of $X^\vee$ on $M_H$ is free, and the embedding $\mathrm{Hilb}_X(v) \to M_H(v)$ (sending $Z \subset X$ to its ideal sheaf) induces an isomorphism of quotient stacks
    \begin{equation}
        \mathrm{Hilb}_X(v)/X \to M_H(v)/X \times X^\vee.
    \end{equation}
    The left-hand side is Deligne--Mumford: this is due to Gulbrandsen \cite{gulbrandsen} when $n \neq 0$, while for $n = 0$, a case which is unimportant for our application, it follows from a variant of Gulbrandsen's argument due to Oberdieck (cf. \cite[\S 2.1]{obPT} for the analog for stable pairs, which adapts straightforwardly to Hilbert schemes).
\end{remark}

\begin{remark}[Deformation type of curve classes]
\label{remark:def_type_of_curve_class}
    Any two pairs $(X_1, \beta_1)$, $(X_2, \beta_2)$, where $X_1$, $X_2$ are abelian $g$-folds and $\beta_1$, $\beta_2$ are effective curve classes of rank $2g$ and the same type $(d_1, \dots, d_g)$, can be connected by a deformation which keeps $\beta_1, \beta_2$ of Hodge type and transports $\beta_1$ to $\beta_2$.

    Indeed, passing to dual abelian varieties and applying Lemma~\ref{lemma:poincare_dual_of_effective}, this is equivalent to the claim that the polarized abelian $g$-folds $(X_1^\vee, L_1)$ and $(X_2^\vee, L_2)$ (where $L_i$ is the dual to $\beta_i$) can be connected by a polarized deformation. Since $L_1$ and $L_2$ have the same type, the claim follows from the fact that the moduli space of abelian $g$-folds with fixed polarization type is connected (e.g., the classical analytic construction presents the moduli space as a quotient of Siegel's upper half-space \cite[\S 8]{birk_lange}).
\end{remark}

From Remark~\ref{remark:def_type_of_curve_class}, $\DT_H(v)$ depends only on the type of the effective class $\beta$ and the integer $n$.
For simplicity, in what follows we write
\[
    \DT_{d, n} = \DT_H(1, 0, - \beta, -n),
\]
where $\beta$ is an effective class of type $(1, 1, d)$ on an abelian threefold $X$. 

We prove Proposition~\ref{prop:nonvanishing_theorem} by analyzing the generating function for DT invariants of curve classes of type $(1, 1, d)$ on abelian threefolds. The generating function has recently been determined by Oberdieck--Shen \cite{obshen}, building on earlier work of Bryan--Oberdieck--Pandharipande--Yin \cite{BOPR} on the generating function for topological Euler characteristics:

\begin{theorem}[Oberdieck--Shen]
\label{thm:curve_count_on_abelian_threefold}
The generating series for the invariants $\DT_{d, n}$ is given as follows:
    \begin{equation}
    \label{eq:generating_series_dt}
        \sum_{d, n} \DT_{d, n}  q^n t^d = (q + 2 + q^{-1}) \prod_{m \geq 1} \frac{(1 + qt^m)^2(1 + q^{-1} t^m)^2}{(1 - t^m)^4}
    \end{equation}
\end{theorem}

\begin{proof}
    The version of Theorem~\ref{thm:curve_count_on_abelian_threefold} for PT invariants is described in \cite[\S 5.7]{obshen}. In what follows, we merely compare their notation and setup with ours. 
    
    They consider a series
    \[
            \mathrm{PT}_2^{A \times E}(q, t) = \sum_{d = 0}^{\infty} \sum_{n \in \bZ} \mathrm{P}_{n, H + dF}^{\mathrm{red}}q^n t^d,
    \]
    where:
    \begin{itemize}
        \item $A$ is a simple, principally polarized abelian surface, and $E$ is an elliptic curve;
        \item $F$ is the cohomology class of any $\{\pt\} \times E$ embedded in $A \times E$;
        \item $H$ is the pushforward of the curve of type $(1,1)$ (also called $H$) on $A$ embedded via any $A \times \{\pt\}$ into $A \times E$, so that $H + dF$ is a curve class of type $(1,1,d)$ on $A \times E$;
        \item the subscript ``2'' refers to the genus of a curve $C$ in $A$ of class $H$.
    \end{itemize}
    The definition of $\rP_{n, H + dF}^{\mathrm{red}}$ is the following:
    Fix a curve $C$ in $A$ of class $(1,1)$; such a curve turns out to be smooth of genus $2$, uniquely determined up to translation, with Jacobian $A$. 
    Inside the moduli space of stable pairs $\rP_n(A \times E, H + dF)$, 
    there is a locus $\rP_n(A \times E, (C, d))$ consisting of pairs $(F, s)$ such that the set-theoretic support of $F$ pushes forward to $C$ along the projection $A \times E \to A$. (In \cite{obshen}, $C$ is written simply as $H$.) 
    Then one defines
    \begin{equation}
    \label{eq:pt_invariant_def}
       \rP_{n, H + dF}^{\mathrm{red}} = \int_{\rP_n(A \times E, (C, d))/E} \nu,
    \end{equation}
    where $\rP_n(A \times E, (C, d))/E$ is the quotient by the translation action of $\{0\} \times E$, and $\nu$ is the Behrend function. 

    On the other hand, the inclusion of $\rP_n(A \times E, (C, d))$ into $\rP_n(A \times E, H + dF)$ induces an isomorphism of Deligne--Mumford stacks
    \begin{equation*}
        \rP_n(A \times E, (C, d))/E \simeq \rP_n(A \times E, H + dF)/A \times E,
    \end{equation*}
    so \eqref{eq:pt_invariant_def} may be written
    \begin{equation*}
        \rP_{n, H + dF}^{\mathrm{red}} = \int_{\rP_n(A \times E, H + dF)/A \times E} \nu,
    \end{equation*}
    where $\nu$ is now the Behrend function on $\rP_n(A \times E, H + dF)/A \times E$.

    Next, we explain why $\rP_{n, H + dF}^{\mathrm{red}} = \DT_{d, n}$. 
    From Remark~\ref{remark:def_type_of_curve_class}, the class $\beta$ which appears in the definition of $\DT_{d, n}$ is deformation equivalent to $H + dF$, since both are effective and of type $(1,1,d)$.
    In other words, we may take $X = A \times E$ and $\beta = H + dF$ in the definition of $\DT_{d, n}$.
    We claim that
    \begin{equation*}
        \rP_{n, H + dF}^{\mathrm{red}} = \int_{\mathrm{Hilb}_{A \times E}(H + dF, n)/A \times E} \nu, 
    \end{equation*}
    where $\mathrm{Hilb}_{A \times E}(H + dF, n)$ is the Hilbert scheme of $1$-dimensional subschemes $Z \subset A \times E$ with $[Z] = H + dF$ and $\chi(\cO_Z) = n$. 
    This is essentially \cite[Theorem 3(i)]{obPT}, but the result is stated there for the product of a K3 surface with an elliptic curve $E$. Nonetheless, the proof goes through for $A \times E$ after replacing the action of $E$ with the action of $A \times E$ everywhere in sight. 
    As in Remark~\ref{remark:well-defined}, there is an isomorphism of Deligne--Mumford stacks 
    \begin{equation*}
        \mathrm{Hilb}_{A \times E}(H + dF, n)/A \times E \cong 
        M_H(1,0,-H-dF, -n)]/A \times E \times (A \times E)^{\vee}
    \end{equation*}
    induced by the embedding $\mathrm{Hilb}_{A \times E}(H+dF, n) \to M_H(1,0,-H,-dF,-n)$ sending $Z \subset X$ to its ideal sheaf. 
    Now a result of Behrend~\cite{behrend} asserts that 
    \begin{equation*}
        \DT_{d, n} = \int_{M_H(1,0,-H-dF, -n)]/A \times E \times (A \times E)^{\vee}} \nu,
    \end{equation*}
    where $\nu$ is the Behrend function on $M_H(1,0,-H-dF, -n)]/A \times E \times (A \times E)^{\vee}$. 
    Altogether, this proves the claimed equality $\rP_{n, H + dF}^{\mathrm{red}} = \DT_{d, n}$. 
     
    Finally, in \cite[\S 5.7, Proposition 5]{obshen}, it is shown that 
    \[
           \mathrm{PT}_2^{A \times E}(q, t) = \phi_{-2, 1}(q, t),
    \]
    where the right-hand side (as defined in \cite[\S 0.5]{obshen}) is precisely the right-hand side of \eqref{eq:generating_series_dt}.
\end{proof}

\begin{remark}
\label{rem:triangular_number}
    The coefficient of $q^a t^b$ in the expansion of the product
    \begin{equation} \label{eq:distinct_parts}
        \prod_{m \geq 1} (1 + qt^m)
    \end{equation}
    is the number of partitions of $b$ into $a$ distinct parts. For example, if $T_k$ is the $k$th triangular number, then the coefficient of $q^{n} t^{T_k}$ in \eqref{eq:distinct_parts} is nonzero for all $1 \leq n \leq k$.
\end{remark}

\begin{lemma}
\label{lem:cases_lemma}
    Let $d, n > 0$ be integers such that $n^2 \leq 4 d$. In the expansion of the product
    \begin{equation} \label{eq:squared_in_lemma}
        \prod_{m \geq 1} (1 + q t^m)^2,
    \end{equation}
    the coefficient of $q^{n - 1}t^d$ is positive.
\end{lemma}

\begin{proof}
    Let $k = \lfloor \sqrt{d} \rfloor$, and write $d = k^2 + \epsilon$. Since $d < (k + 1)^2$, we have $\epsilon < 2k + 1$, which implies $n \leq 2k + 1$. 

    Consider the case $n \leq 2k$. Then $d = T_k + T_{k - 1} + \epsilon$, with notation as in Remark~\ref{rem:triangular_number}. Since $n - 1 \leq 2k - 1$, we may write $n - 1 = n_1 + n_2$, where $n_1 \leq k$ and $n_2 \leq k - 1$. In particular, the coefficients of $q^{n_1} t^{T_k}$ and $q^{n_2} t^{T_{k - 1} + \epsilon}$ in \eqref{eq:distinct_parts} are positive by Remark~\ref{rem:triangular_number}, so their multiple $q^{n - 1} t^d$ has a nonzero coefficient in \eqref{eq:squared_in_lemma}.

    Finally, consider the case $n = 2k + 1$. From the inequality
    \[
        4k^2 + 4k + 1 \leq 4k^2 + 4 \epsilon,
    \]
    we see that $k < \epsilon$. It follows that we may write $d = k^2 + k + \epsilon'$ for some $\epsilon' > 0$. Equivalently, $d = T_k + T_k + \epsilon'$. A similar argument as in the previous paragraph shows that the coefficient of $q^{n - 1} t^d$ in \eqref{eq:squared_in_lemma} is positive. 
\end{proof}

\begin{proof}[Proof of Proposition~\ref{prop:nonvanishing_theorem}]
    Let $d > 0$, $n$ be integers. We observe that $\Delta(v) = d - n^2/4$, so our assumption is that $4d - n^2 \geq 0$. Since the coefficients of \eqref{eq:generating_series_dt} are symmetric under $q \mapsto q^{-1}$, we may suppose that $n \geq 0$. The case $n = 0$ is straightforward to see, so we take $n \geq 1$. In that case, from Lemma~\ref{lem:cases_lemma}, the product \eqref{eq:squared_in_lemma} contributes a nonzero $q^{n - 1} t^d$ term, which we multiply with $(q + 1 + q^{-1})$ to obtain the desired $q^n t^d$ term. 
\end{proof}


\section{Proof of the main result}
\label{section-proof-main-result-abelian-3folds} 

\subsection{Preliminary reduction}
Given an abelian threefold $X$, the type of a class in $\rH^2(X, \bC)$ is an associated tuple of integers $(d_1, d_2, d_3) \in \bZ_{\geq 0}^3$, whose definition is recalled in \S\ref{section-rank-type}. 

\begin{definition}
\label{def:type_mod_n}
    A class $\theta \in \rH^2(X, \bmu_n)$ has \emph{type} $(d_1, d_2, d_3)$ if it admits a lift to an integral class $\class \in \rH^2(X, \bZ(1))$ of type $(d_1, d_2, d_3)$. In fact, the reduction map
    \begin{equation*}
        \SL_6(\bZ) \to \SL_6(\bZ/n)
    \end{equation*}
    is surjective (transvections, which generate the right-hand side, may be lifted \cite[XIII, \S 9]{Lang}), so $\theta$ has type $(d_1, d_2, d_3)$ if and only if there is an integral basis $x_1, \dots, x_6$ of $\rH^1(X, \bmu_n)$ such that 
    \begin{equation*}
        \theta = d_1 x_1 \wedge x_4 + d_2 x_2 \wedge x_5 + d_3 x_3 \wedge x_6.
    \end{equation*} 
    The integers $d_1, d_2, d_3$ are only well-defined up to multiplication by units in $\bZ/n$, and if $n = p^e$ is a prime power, then any class has type $(p^a, p^b, p^c)$ for some $a,b,c > 0$.
    
    If we may choose a lift $\class$ of type $(d_1, d_2, d_3)$ with $\class^3 > 0$, then we say that $\theta$ is of type $(d_1, d_2, d_3)$ and \emph{admits a positive lift}.
\end{definition}

Given an abelian threefold $X$ and a class $\theta \in \rH^2(X, \bmu_n)$, we say that the period-index conjecture holds for the pair $(X, \theta)$ if $\ind(\alpha) \mid n^2$, where $\alpha$ is the image of $\theta$ in $\Br(X)$.

\begin{remark}[Isogeny trick]
\label{remark:isogeny_trick}
    In the proof of Lemma~\ref{lem:period_index_reduction} below, we repeatedly use the following: for an integral class $u$ of type $(d_1, d_2, d_3)$ on an abelian threefold $X$ and nonzero integers $e_1, e_2, e_3$ with $e_1 \mid e_2 \mid e_3$, there is an isogeny $\pi:X' \to X$ of degree $e_1e_2e_3$ such that $\pi^* u$ is of type $(e_1d_2, e_2d_2, e_3d_3)$. The same holds for classes $\theta \in \rH^2(X, \bmu_n)$.

    For the proof, writing
    \begin{equation*}
        u = d_1 x_1 \wedge x_4 + d_2 x_2 \wedge x_5 + d_3 x_3 \wedge x_6,
    \end{equation*}
    we define the sublattice
    \begin{equation*}
        \Lambda' = \left\langle \frac{x_1}{e_1}, \frac{x_2}{e_2}, \frac{x_3}{e_3}, x_4, x_5, x_6 \right\rangle \subset \rH^1(X, \bQ).
    \end{equation*}
    Then $\pi:X' \to X$ is dual to the isogeny $\rH^{0,1}(X)/\Lambda \to \rH^{0,1}(X)/\Lambda'$, where $\Lambda = \rH^1(X, \bZ)$.
\end{remark}

\begin{lemma}
\label{lem:period_index_reduction}
    Suppose that the period-index conjecture holds for all pairs $(X, \theta)$ where $X$ is an abelian threefold and $\theta \in \rH^2(X, \bmu_n)$ such that: 
    \begin{enumerate} 
        \item \label{X-polarization} $X$ admits a polarization $H$ of type $(d_1, d_2, d_3)$ for integers $d_i$ dividing a power of $n$.
        \item \label{theta-reduction} The class $\theta \in \rH^2(X, \bmu_n)$ is of type $(1, 1, 1)$, and admits a positive lift.
    \end{enumerate}
    Then for an arbitrary abelian variety $X$ with a class $\theta \in \rH^2(X, \bmu_n)$, the period-index conjecture holds for $(X, \theta)$.
\end{lemma}

\begin{proof}
    Let $(X, \theta)$ be arbitrary. The class $\theta$ admits a prime decomposition
    \[
        \theta = \theta_1 + \theta_2 + \cdots + \theta_k
    \]
    into $p_i$-power torsion classes $\theta_i$, for a set of primes $p_i$ dividing $n$. The period-index conjecture for all $(X,{\theta}_i)$ is equivalent to the period-index conjecture for $(X, \theta)$, so we may suppose that $n = p^e$ is a prime power \cite[Proposition 4.5.16]{gillet_sza}. 

    Let $H$ be a polarization on $X$ of type $(d_1, d_2, d_3)$. There is an isogeny $X' \to X$ of degree prime to $p$ such that $X'$ admits a polarization $H'$ of type $(d_1', d_2', d_3')$, where each $d'_i$ divides a power of $n$. (Use the isogeny trick from Remark~\ref{remark:isogeny_trick} to make the pullback of $H$ a multiple of a class $H'$ of the desired form.) The period-index conjecture for $(X', \theta')$, where $\theta'$ is the pullback of $\theta$, implies the period-index conjecture for $(X, \theta)$. Replace $(X, H, \theta)$ with $(X', H', \theta')$.

Now, $\theta$ is of type $(p^a, p^b, p^c)$ for some $a,b,c > 0$, as explained in Definition~\ref{def:type_mod_n}. 
Then there is an isogeny $X' \to X$ of degree $p^{2c - a - b}$ so that the pullback $\theta'$ is of type $(p^c, p^c, p^c)$. If $\alpha'$ is the image of $\theta'$ in $\Br(X')$, then $\alpha'$ is the image of a class $\theta'_0 \in \rH^2(X', \bmu_{p^{e - c}})$ of type $(1, 1, 1)$. After replacing $\theta$ with $\theta^{-1}$ (which does not change the index), we may suppose that $\theta'_0$ is of type $(1,1,1)$ and admits a positive lift. Since the pullback $H'$ of $H$ remains of the required type, the triple $(X', H', \theta'_0)$ satisfies hypotheses \eqref{X-polarization} and \eqref{theta-reduction}, so we may assume that the period-index conjecture holds for $\theta_0'$. In particular, the index of $\alpha'$ is at most $p^{2e - 2c}$. 

    If $\alpha$ is the image of $\theta$ in $\Br(X)$, then
    \[
        \ind(\alpha) \mid  p^{2c - a - b} \cdot p^{2e - 2c} \mid p^{2e - a - b},
    \]
    which implies that the period-index conjecture holds for $(X, \theta)$.
\end{proof}

\subsection{Choosing initial data}

\begin{situation}
\label{sit:initial_data}
    Let $(X_\init, H_\init, \theta_\init)$ be a triple, where:
    \begin{enumerate}
        \item $X_\init$ is an abelian threefold.
        \item For an integer $n > 1$, $\theta_\init \in \rH^2(X_\init, \bmu_n)$ is of type $(1, 1, 1)$ and admits a positive lift.
        \item The polarization $H_\init$ is of type $(d_1, d_2, d_3)$, where each $d_i$ divides a power of $n$.
    \end{enumerate}
\end{situation}

\begin{lemma}
\label{lem:choosing_b_field}
    In Situation~\ref{sit:initial_data}, there exists a lift $
    \class_\init \in \rH^2(X_\init, \bZ(1))$ of $\theta_\init$ which satisfies the following properties:
    \begin{enumerate} 
        \item \label{lem:choosing_b_field-1} The class $\class_\init$ is of type $(1, 1, 1)$ with $\class_\init^3 > 0$.
        \item \label{lem:choosing_b_field-2} There exists a polarized deformation from $(X_\init, H_\init)$ to $(X_\fin, H_\fin)$ such that $X_\fin$ is a simple abelian threefold and a parallel transport $\class_\fin$ of $\class_\init$ lies in $\NS(X_0)$; moreover, $u_0$ is a principal polarization.
        \item \label{lem:choosing_b_field-3}
        There exists an integer $A > 0$, prime to $n$, and an integer $k > 0$, which is arbitrarily large with respect to $A$, such that the class
        \begin{equation}
        \label{eq:curve_class_choosing}
            \frac{1}{2} \left( A^2 \cdot \frac{1}{2}(\class_\init^2)^* - n^k \cdot \frac{1}{2} (H_\init^2)^* \right)^2  + n \cdot A \class^*_\init \in \rH^4((X_\init^{\vee}), \bZ)
        \end{equation}
    is of type $(1,1,d)$, for some integer $d > 0$, where $(-)^*$ denotes the duality isomorphism $\rH^{i}(X_\init, \bZ) \simeq \rH^{6 - i}((X_\init^{\vee}), \bZ)$.
    \end{enumerate}
\end{lemma}

\begin{proof}
    From Proposition~\ref{prop:level_subgroup}, we see that the set of lifts $\class_\init$ satisfying~\eqref{lem:choosing_b_field-1} and~\eqref{lem:choosing_b_field-2} is nonempty, and its image in $\bP = \bP(\rH^2(X_\init, \bC))$ is Zariski dense. 
    In particular, if $C \subset \bP$ is the join of $\Gr(2, \rH^1(X, \bC))$ with the point $[H_\init]$, then we may choose $\class_\init$ to lie in the complement of $C$. For such a choice of $\class_\init$, the vector space $\langle \class_\init, H_\init \rangle_{\bC}$ does not contain any rank $2$ elements.  

    We show that $\class_\init$ satisfies~\eqref{lem:choosing_b_field-3}. First, the statement of~\eqref{lem:choosing_b_field-3} is deformation invariant, so by~\eqref{lem:choosing_b_field-2} we may assume that $\class_\init, H_\init \in \NS(X_\init)$, where $X_\init$ is simple. Let $\class = (\class_\init^2)^*/2$ and $H = (H_\init^2)^*/2$. We must show that Proposition~\ref{prop:prime_avoidance} applies to \eqref{eq:curve_class_choosing}, which amounts to checking that the following hold: 
    \begin{enumerate}
        \item $\class$ is of type $(1,1,1)$: This one computes from the fact that $\class_{\init}$ is of type $(1,1,1)$.
        \item $H$ is of type $(d_1', d_2', d_3')$ where each $d_i'$ divides a power of $n$: This one computes from the fact that $H$ is of type $(d_1, d_2, d_3)$, where each $d_i$ divides a power of $n$.
        \item $\langle \class^2, H^2 \rangle_{\bC}$ does not contain rank $2$ elements: This is because $\class^2$ and $H^2$ are identified with $\class_{\init}$ and $H_{\init}$ under the given duality isomorphism, up to scaling. We chose $\class_{\init}$ above so that $\langle \class_{\init}, H_{\init} \rangle_{\bC}$ does not contain rank $2$ elements.
        \item $\class^2/2 = \class_\fin^*$: As above, one chooses a basis and computes, noting that $\class^3 > 0$. \qedhere
    \end{enumerate}
\end{proof}

\subsection{The deformation step}

Let $(X_\init, H_\init, {\theta}_\init)$ be as in Situation~\ref{sit:initial_data}. Choose a class $\class_\init$ satisfying the conclusions of Lemma~\ref{lem:choosing_b_field}, and integers $A, k$ as in Lemma~\ref{lem:choosing_b_field}\eqref{lem:choosing_b_field-3}. By Igusa's formula (Theorem~\ref{thm:igusa_explicit_formula}), we may take $k$ to be sufficiently large so that the class
\begin{equation} \label{eq:hodge_class_at_zero}
    v_\init = \left(n^2, -n \cdot A \class_\init, \frac{1}{2}(A^2 \class_\init^2 - n^k H_\init^2), 1\right) \in \rH^{\ev}(X_\init, \bZ)
\end{equation}
has positive discriminant. In fact, we ultimately take $k$ to be sufficiently large so that, in the distant future, Remark~\ref{remark:effectivity_of_the_curve} below applies.

From Lemma~\ref{lem:choosing_b_field}\eqref{lem:choosing_b_field-2}, we may choose a flat, polarized family $f \colon (X, H) \to S$ of abelian threefolds over a connected affine curve $S$, with points $\init, \fin \in S(\bC)$, satisfying the following properties:
\begin{itemize} 
    \item The fiber over $\init \in S(\bC)$ is $(X_\init, H_\init)$.
    \item The fiber over $\fin \in S(\bC)$ is $(X_\fin, H_\fin)$, where $X_\fin$ is a simple abelian threefold.
    \item A parallel transport $\class_\fin \in \rH^2(X_\fin, \bZ(1))$ of $\class_\init$ is algebraic. 
\end{itemize}
After finite base change, we may suppose that $\theta_\init$ lifts to a global section of $\rR^2 f_* \bmu_n$. From the Leray spectral sequence, it lifts to a class $\theta \in \rH^2(X, \bmu_n)$.

\begin{remark}
\label{rem:making deformation}
    Let $B_\init = A \class_\init / n$. Since $A$ is prime to $n$, we may regard $B_\init$ as a $\theta$-field for a power $\theta^t$ of $\theta$, where $t$ is prime to $n$. Observe that $v_\init$ is of the form described in  Lemma~\ref{lem:constructing_rational_hodge_class}. After applying Lemma~\ref{lem:constructing_rational_hodge_class}, we conclude that there is a global section $v$ of $\Ktop[0](\Dperf(X, \alpha^t)/S)$ (here $\alpha \in \Br(X)$ is the image of $\theta$), such that the following hold:
    \begin{enumerate} 
        \item The section $v$ is everywhere of Hodge type. 
        \item The restriction of $v$ to the stalk at $\init$ corresponds (under the isomorphism from Lemma~\ref{lem:twisted_mukai_abelian-families}) to the class $v_1 \in \muk(X_\init, B_\init; \bZ)$ described above.
        \item The restriction of $v$ to the stalk at $\fin$ corresponds to the class $v_\fin \in \muk(X_\fin, B_\fin; \bZ)$, where $B_\fin = A \class_\fin / n$ and 
        \[
            v_\fin = \left(n^2, -n \cdot A \class_\fin, \frac{1}{2}(A^2 \class_\fin^2 - n^k H_\fin^2), 1\right).
        \]
    \end{enumerate}
\end{remark}

\begin{remark}
    In the situation above, the class $B_0 = Au_0/n$ is algebraic, so the restriction $\alpha_0^t \in \Br(X_0)$ is trivial, where $\alpha^t$ is the image of $\theta^t$ in $\Br(X)$. 
    The upshot is that the twisted derived category $\Dperf(X_0, \alpha^t_0)$ is non-uniquely equivalent to $\Dperf(X_0)$. 
    From Remark~\ref{remar:fibers_with_trivial_brauer}, we may choose an equivalence between categories such that the induced isomorphism between Mukai structures $\muk(X_0, B_0; \bZ)$ and $\muk(X_0, \bZ)$ is the identity on the abelian group $\rH^{\ev}(X, \bZ)$ which underlies both Hodge structures. The chosen equivalence is used implicitly throughout the remainder of the proof, e.g., to identify Bridgeland-stable objects of class $v \in \muk(X_0, B_0; \bZ)$ in $\Dperf(X_0, \alpha_0^t)$ with Bridgeland-stable objects of class $v \in \muk(X_0, \bZ)$ in $\Dperf(X_0)$.
\end{remark}

We have finally arrived at a situation where we may verify the hypotheses of Corollary~\ref{corollary-VIHC-sigma}. 

\begin{proposition}
\label{prop:tying_it_together}
    In the situation of Remark~\ref{rem:making deformation}, there exists a stability condition $\sigma$ on $\Dperf(X, \alpha)$ over $S$ with respect to a topological Mukai homomorphism, such that the following conditions hold:
    \begin{enumerate} 
            \item There do not exist strictly $\sigma_\fin$-semistable objects of class $v_\fin$. 
            \item $\cM_{\sigma_\fin}(v_\fin, \phi) / \cAut(\Dperf(X_\fin, \alpha_\fin^t)/\bC)$ is Deligne--Mumford, where $\phi$ is a phase compatible with $v_\fin$.
            \item $\DT_{\sigma_\fin}(v_\fin) \neq 0$.
    \end{enumerate}
\end{proposition}

The proof of Proposition~\ref{prop:tying_it_together} will be given in the next section. Using Corollary~\ref{corollary-VIHC-sigma}, we are now able to prove Theorem~\ref{thm:period_index_conjecture}:

\begin{proof}[Proof of Theorem~\ref{thm:period_index_conjecture}]
For the convenience of the reader, we trace the thread of the argument. By Lemma~\ref{lem:period_index_reduction}, we reduce to proving the desired period-index bound for $(X_\init, \theta_\init)$ as in Situation~\ref{sit:initial_data}. In fact, it is enough to prove the desired period-index bound for $\theta_\init^t$, since $t$ is prime to $n$. (The period-index problem for $\alpha$ is equivalent to the period-index problem for $\alpha^t$, whenever $t$ is prime to the period.)

We have constructed a Hodge class $v_\init$ of rank $n^2$ in the twisted Mukai structure for $(X_\init, \theta_\init^t)$, and it suffices to show that $v_\init$ is algebraic. Proposition~\ref{prop:tying_it_together} shows that the hypotheses of Corollary~\ref{corollary-VIHC-sigma} are satisfied, so $v_\init$ is algebraic.
\end{proof}

\subsection{Proof of Proposition~\ref{prop:tying_it_together}}
\label{ssec:proof_of_nonemptiness}

First, the existence of a relative stability condition is granted by Corollary~\ref{cor:existence_of_relative_stability_ab_threefolds}. Once the existence of a single stability condition for $\Dperf(X, \alpha)$ over $S$ is known, we can find by Lemma~\ref{lem:v_generic_relative_stability} a (potentially different) stability condition $\sigma$ over $S$ such that strictly $\sigma_\fin$-semistable objects of class $v_\fin$ do not exist. 

Recall from Lemma~\ref{lem:choosing_b_field} that the class $u_0 \in \rH^2(X_{\fin}, \bZ)$ is a principal polarization on $X_\fin$. Let $\Phi:\Dperf(X_\fin) \to \Dperf(X_\fin^{\vee})$ be the standard Fourier--Mukai transform, and let $\phi:X \to X^{\vee}$ be the isomorphism provided by the principal polarization. We write $\Psi = \phi^* \circ \Phi$, which is an autoequivalence of $\Dperf(X_\fin)$.

By our assumptions on $A$ and $k$, the class $v_\fin$ has positive discriminant, so the moduli space $\cM_{\sigma_\fin}(v_\fin, \phi)/\cAut^0(\Dperf(X_\fin)/\bC)$ is Deligne--Mumford by Lemma~\ref{lem:discr_no_stabilizer}, and we may define the DT invariant $\DT_{\sigma_\fin}(v_\fin)$. In fact, by \cite[Theorem 4.9 and Lemma 4.12]{OPT}, $\DT_{\sigma'}(v_\fin)$ does not depend on the stability condition $\sigma' \in \Stab^\dagger(X_\fin)$, where $\Stab^\dagger(X_\fin)$ is the distinguished component of the stability manifold (see Remark~\ref{remark-distinguished-Stab}), so we may simply write $\DT(v_\fin)$. 

According to \cite[Theorem 1.1 and Proposition 1.2]{OPT}, we have the identity
\[
    \DT(v_\fin) =  \DT(\Psi_* v_\fin).
\]
In addition, by Corollary~\ref{cor:auteq_preserves_discr}, $\Delta(\Psi_* v_\fin) = \Delta(v_\fin) > 0$. On the other hand, since $\phi^*$ comes from an isomorphism of the underlying varieties $X$ and $X^{\vee}$, it preserves DT invariants and the discriminant $\Delta$, so one has
\[
        \DT(v_\fin) = \DT(w_\fin), \quad \Delta(v_\fin) = \Delta(w_\fin),
\]
where $w_\fin = \Phi_* v_\fin$. Therefore, it suffices to show that $\DT(w_\fin) \neq 0$. (We work with $w_\fin$ over $\Psi_* v$ simply to avoid additional notational clutter.)

From \cite[Lemma 9.23]{huybrechts_fm}, one computes that
\begin{align*}
    w_\fin &= \left(1, - \frac{1}{2}(A^2 (\class_\fin^2)^* - n^k (H_\fin^2)^*), -n \cdot A \class_\fin^*, -n^2  \right),
\end{align*}
where $(-)^*$ denotes the duality isomorphism 
$\rH^{i}(X_\init, \bZ) \simeq \rH^{6 - i}((X_\init^{\vee}), \bZ)$. 
Again by \cite[Theorem 1.1 and Proposition 1.2]{OPT}, we have $\DT(w_0) = \DT(\Theta_*w_0)$ for any autoequivalence $\Theta$ of $\Dperf(X_0^{\vee})$; taking $\Theta = - \otimes L$, where $L$ is a line bundle on $X_0^{\vee}$ of class $c_1(L) \in \rH^2(X_0^{\vee}, \bZ)$ equal to  the degree $2$ term of $w_0$, we see that 
we may replace $w_\fin$ with $\exp(-c_1(L)) \cdot w_\fin$. Computing, we find 
\begin{equation*}
    \exp(-c_1(L)) \cdot w_\fin 
        = \left( 1, 0,  -\frac{1}{2} \left( A^2 \cdot \frac{1}{2}(\class_\fin^2)^* - n^k \cdot \frac{1}{2} (H_\fin^2)^* \right)^2  - n \cdot A \class^*_\fin, \dots \right) . 
\end{equation*}
We write it simply as $(1, 0, -\beta, -n)$. 

In  Lemma~\ref{lem:choosing_b_field}, we arranged it so that $\beta$ is of type $(1, 1, d)$ for some $d > 1$, and as above,
\[
    \Delta(1, 0, - \beta, -n) = \Delta(v_\fin) > 0.
\]
From Remark~\ref{remark:effectivity_of_the_curve} below, the class $\beta$ is $\bQ$-effective for $k$ sufficiently large, which we may arrange from the beginning.
From Proposition~\ref{prop:nonvanishing_theorem} and the existence of a Gieseker chamber, \cite[Proposition 3.26]{OPT} (see Remark~\ref{rem:opt} below), we have
\[
    \DT(v_\fin) = \DT(1, 0, - \beta, -n) > 0,
\]
as needed. \qed

\begin{remark}[Effectivity of the curve class]
\label{remark:effectivity_of_the_curve}
    For $k \gg 0$, the curve class
    \begin{equation*}
        \beta = \frac{1}{2}\left( A^2 \cdot \frac{1}{2}(\class_\fin^2)^* - n^k \cdot \frac{1}{2} (H_\fin^2)^* \right)^2  + n \cdot A \class^*_\fin
    \end{equation*}
    is $\bQ$-effective in the sense of Lemma~\ref{lemma:poincare_dual_of_effective}. 
    
    Indeed, since $u_0$ and $H_0$ are ample, Lemma~\ref{lemma:poincare_dual_of_effective} shows that $u_0^*$ is $\bQ$-effective and the classes $x := (u_0^2)^*/2$ and $y:= (H_0^2)^*/2$ are ample. Expanding the first term, we need to show that 
    \begin{equation*}
        n^{2k} y^2 - 2A^2 n^k xy + A^4 x^2 = n^k(n^k y^2 - 2A^2xy) + A^4x^2
    \end{equation*}
    is effective for $k \gg 0$. This reduces to showing that $n^ky^2 - 2A^2xy$ is effective for $k \gg 0$, which follows from a claim: $y^2$ lies in the interior of the cone of $\bQ$-effective curve classes. 
    
    To prove the claim, observe that since $y$ is ample, 
    multiplication by $y$ sends the ample cone in $\NS(X)_{\bQ}$ isomorphically onto an open cone in $\Hdg^4(X, \bQ)$ contained in the cone of $\bQ$-effective curve classes. Since $y$ lies in the interior of the ample cone, $y^2$ lies in the interior of the cone of $\bQ$-effective curve classes.
\end{remark}

\begin{remark}
\label{rem:opt}
        In \cite{OPT}, there is a sign convention in play (cf. \cite[Remark 2.16]{OPT}), which we do not adopt here. As a result, one cannot na\"ively invoke \cite[Proposition 4.2]{OPT} to ensure the existence of a Gieseker chamber. Instead, \cite[Proposition 3.26]{OPT} ensures its existence when the leading term of the Chern character is an effective algebraic class. 
\end{remark}


\begin{appendix}
    
\section{Auxiliary results on abelian threefolds}
\label{appendix-abelian-3folds}

\subsection{Rank and type}
\label{section-rank-type}
Let $X$ be an abelian variety of dimension $g$, and let $u \in \rH^2(X, \bC)$. The \emph{rank} of $u$ is the rank of the corresponding map
\begin{equation}
\label{eq:alternating_hom}
    \varphi_u:\rH_1(X, \bC) \to \rH^1(X, \bC)
\end{equation}
given by the skew-symmetric form associated to $u \in \bigwedge^2 \rH_1(X, \bC)^\vee$. In particular, the rank belongs to $\{0, 2, 4, \dots, 2g\}$. An element $u$ has rank $\leq 2r$ if and only if $u^{r + 1} = 0$.

If $u$ lies in $\rH^2(X, \bZ)$, then there is a unique sequence of nonnegative integers $d_1, \dots, d_g$ such that $d_1 \mid \cdots \mid d_g$ and 
\[
    u = d_1 e_1 \wedge e_{g + 1} + \cdots + d_g e_g \wedge e_{2g}
\]
for some integral basis $e_1, \dots, e_{2g}$ of $\rH^1(X, \bZ)$. The sequence $(d_1, \dots, d_g)$ is the \emph{type} of $u$. 

The rank of an element of $\rH_2(X, \bC) \cong \bigwedge^2 \rH_1(X, \bC)$ and type of an element of $\rH_2(X, \bZ)$ are defined similarly.  

A class $\beta \in \rH_2(X, \bQ)$ is \emph{$\bQ$-effective} if it is a positive $\bQ$-linear combination of cycle classes $[C_i]$, for curves $C_i \subset X$. 

\begin{lemma}
\label{lemma:poincare_dual_of_effective}
    Let $\beta \in \rH_2(X, \bQ)$ be a class of rank $2g$ on an abelian $g$-fold $X$. 
    The duality $\rH_2(X, \bQ) \simeq \rH^2(X^\vee, \bQ)$ induced by the natural isomorphism $\rH_1(X, \bQ) \simeq \rH^1(X^\vee, \bQ)$ exchanges the following two cones:
    \begin{enumerate}
        \item The cone of $\bQ$-effective curve classes of rank $2g$ in $\rH_2(X, \bQ)$.
        \item The ample cone in $\NS(X^\vee)_{\bQ}$.
    \end{enumerate}
\end{lemma}

\begin{proof}
    Before embarking on the proof, we note that if $\pi:X \to Y$ is an isogeny between abelian varieties, and $\beta \in \rH_2(X, \bQ)$ is a class, then 
    \begin{equation}
    \label{eq:pd_equation}
        (\pi_* \beta)^* = (\pi^{\vee})^* (\beta^*),
    \end{equation}
    where $(-)^{*}$ denotes the duality isomorphism and $\pi^\vee:Y^\vee \to X^\vee$ is the dual isogeny. 

    Let $\beta$ be a $\bQ$-effective curve class. After scaling, we may assume that $\beta$ is a positive $\bZ$-linear combination $\sum n_i [C_i]$, where $C_i \subset X$ are integral curves. The fact that $\beta$ has rank $2g$ implies that $\sum C_i$ is nondegenerate, i.e., is not contained in any abelian subvariety of $X$.
    From \eqref{eq:pd_equation}, the dual of $\beta$ is the pullback of $\sum n_i [\Theta_i]$ along the induced map 
    \begin{equation}
    \label{eq:map_to_prod_of_jac}
        \pi:X^\vee \to \prod \mathrm{Jac}(C_i^{\nu}).
    \end{equation}
    where $C_i^{\nu} \to C_i$ is the normalization, and $\Theta_i$ is the theta divisor on $\mathrm{Jac}(C_i^{\nu})$, which is dual to $[C_i^{\nu}] \in \rH_2(\mathrm{Alb}(C_i^\nu), \bZ)$. Since $\sum C_i$ is nondegenerate, $\pi^\vee$ is surjective, so $\pi$ has a finite kernel. In particular, the pullback of a polarization along $\pi$ remains a polarization.

    Let $L \in \NS(X)_{\bQ}$ be an ample class, and write $\beta = L^*$. From \eqref{lemma:poincare_dual_of_effective}, we are allowed to work up to finite isogeny, so we may assume that $L$ is a multiple of a principal polarization $P$. Then the dual of $P$ (hence $L$) is effective, since it is identified with $P^{g - 1}/(g - 1)!$ under the isomorphism $\phi_P:X \to X^\vee$. 
\end{proof}

\subsection{The characteristic Pfaffian}

    Let $(X, H)$ be a polarized abelian variety of dimension $g$. Throughout, we abuse notation and write $H$ for both the line bundle and its first Chern class. Given an element $u \in \rH^2(X, \bC)$, the \emph{characteristic Pfaffian} of $u$ is 
    \[
        p_u(t) = \Pf(tH - u) \in \bC[t].
    \]
    Here, the right-hand side is the Pfaffian of a complex-valued skew-symmetric matrix associated to $tH - u$; our sign convention in forming the Pfaffian is that $p_u(t)$ has positive leading coefficient. The roots of $p_u(t)$ correspond to values of $t$ for which $tH - u$ is a degenerate form. Although it is not reflected in the notation, $p_u(t)$ depends on the choice of polarization.

    \begin{lemma}
    \label{lem:alternative_formula_for_pfaffian}
        For $u \in \rH^2(X, \bC)$,
        \begin{equation}
        \label{eq:alternative_formula}
            p_u(t) = \sum_{i = 0}^g (-1)^{g - i} \frac{(H^i u^{g - i})}{(g - i)!i!} t^i.
        \end{equation}
    \end{lemma}

    \begin{proof}
        Let $q_u(t)$ be the right-hand side of \eqref{eq:alternative_formula}. From the binomial theorem,
        \[
            q_u(t) = \frac{1}{g!} \int_X (tH - u)^g.
        \]
        The roots of $q_u(t)$ are the values of $t$ where $(tH - u)^g = 0$, which occurs precisely when $\rk (tH - u) < 2g$. It follows that $p_u(t)$ and $q_u(t)$ have the same roots; since the leading coefficients match, the lemma follows. 
    \end{proof}

    The following is an analogue of the fact that a matrix with distinct eigenvalues which belong to the base field can be diagonalized:

    \begin{lemma}
    \label{lem:symplectic_orbit}
        Let $k$ be a field, and let $(V, H)$ be a symplectic vector space of dimension $2g$ over $k$. Let $u \in \bigwedge^2 V^\vee$ be an alternating form such that $p_u(t) = \Pf(tH - u)$ has distinct roots $a_1, \dots, a_g \in k$. 
        Then there is a symplectic basis $x_1, \dots, x_{2g}$ of $V$ such that
        \begin{align*}
            H &= x_1^\vee \wedge x^\vee_{g + 1} + \cdots + x^\vee_g \wedge x^\vee_{2g} \\
            u &= a_1 x^\vee_1 \wedge x^\vee_{g + 1} + \cdots + a_g x^\vee_g \wedge x^\vee_{2g},
        \end{align*}
        where $x_1^\vee, \dots, x_{2g}^\vee$ is the dual basis of $x_1, \dots, x_{2g}$. 
    \end{lemma}

    \begin{proof}
        For each $i$, let $V_i$ be the kernel of the alternating form $u - a_i H$. For each $x \in V_i$ and any $y \in V$,
        \[
            \langle x, y \rangle_{u} = a_i \langle x, y \rangle_{H}.
         \] 
        For $x \in V_i$ and $y \in V_j$,
        \[
            \langle x, y \rangle_{H} = a_i^{-1} \langle x, y \rangle_{u} = a_j^{-1} \langle x, y \rangle_{u}.
        \]
        If $i \neq j$ (so that $a_i \neq a_j$, by assumption), then $\langle x, y \rangle_{H} = \langle x, y \rangle_{u} = 0$. In particular, the subspaces $V_i$ and $V_j$ are $H$-orthogonal for $i \neq j$. On the other hand, since $H$ is nondegenerate, $H|_{V_i}$ is nondegenerate for each $i$. It follows that the map 
        \[
             \bigoplus_{i = 1}^g V_i \to V
         \] 
        is an isomorphism, since for each $i$, $ \bigoplus_{j \neq i} V_j$ is $H$-orthogonal to $V_i$.

        Putting everything together, there is an isomorphism of symplectic vector spaces
        \begin{equation}
            (V, H) = \bigoplus_{i = 1}^{g} (V_i, H|_{V_i}).
        \end{equation} 
        To conclude, one takes $x_i, x_{g + i} \in V_i$ to be any symplectic basis for $H|_{V_i} = a_i^{-1} \cdot u|_{V_i}$.
    \end{proof}

    \subsection{Endomorphisms}

    Let $(X, H)$ be a polarized abelian $g$-fold. 
    Consider the first homology $\rH_1 = \rH_1(X, \bZ)$, regarded as a Hodge structure of weight $1$, with a Hodge decomposition
    \[
        \rH_{1, \bC} = \rH_{1, 0} \oplus \rH_{0, 1}.
    \]
    Given an element $\class \in \rH^{1,1}(X)$, consider the homomorphism
    \[
        \varphi_\class:\rH_{1, \bC} \to \rH^1_{\bC} = \rH^1(X, \bC)
    \] 
    induced by the alternating form associated to $\class$. The composition
    \begin{equation} \label{eq:endom_of_complex_hs}
        \begin{tikzcd}
            \rho(\class):\rH_{1, \bC} \ar[r, "\varphi_\class"] &  \rH^1_{\bC} \ar[r, "\varphi_{H}^{-1}"] & \rH_{1, \bC},
        \end{tikzcd}
    \end{equation}
    preserves $\rH_{1, 0}$, and from the Hodge decomposition one sees that the induced map
    \begin{equation}
    \label{eq:map_to_end}
        \rho:\rH^{1, 1}(X) \to \End(\rH_{1, 0})
    \end{equation}
    is an isomorphism which sends $H$ to $\id$. Although it is not reflected in the notation, $\rho$ depends on the choice of polarization $H$.

    Let $\End^s(X)$ be the subring of $\End(X)$ fixed by the Rosati involution 
    \cite[\S 5.1]{birk_lange}. There is a commutative diagram
        \begin{equation} \label{eq:ns_big_diagram}
        \begin{tikzcd}[column sep=huge]
            \NS(X)_{\bQ} \ar[d] \ar[r, "\tilde \rho"', "\sim"] & \End^s(X)_{\bQ} \ar[d] \\
            \NS(X)_{\bC} \ar[r, "\sim"] \ar[d] & \End^s(X)_{\bC} \ar[d] \\
            \rH^{1, 1}(X) \ar[r, "\rho"', "\sim"] & \End(\rH_{1, 0}),
        \end{tikzcd}
    \end{equation}
    where $\tilde \rho$ is given by sending $\class$ to the rational endomorphism 
    of $X$ induced by the Hodge-structure endomorphism $\varphi_{\class} \circ \varphi_{H}^{-1}$ of $\rH_{1,{\bQ}}$ from \eqref{eq:endom_of_complex_hs}. Moreover, the right vertical arrows are ring homomorphisms.

    \begin{lemma}
    \label{lem:characteristic_polynomial}
        Let $u \in \rH^{1,1}(X)$. The characteristic polynomial of $\rho(u) \in \End(\rH_{1,0})$ is given by $p_u(t)$, up to a nonzero scalar. 
    \end{lemma}

    \begin{proof}
        Let $\chi_u(t) = \det(t \cdot \id - \rho(u))$ be the characteristic polynomial of $\rho(u)$. 
        It is enough to show that $\chi_u(t)$ and $p_u(t)$ have the same roots. 

        For any $v \in \rH^{1,1}(X)$, the compatibility of $v$ with the Hodge decomposition implies that $v$ is a degenerate form if and only if $\rho(v)$ has rank $< g$. Therefore, $\Pf(v)$ vanishes if and only if $\det(\rho(v)) = 0$. Taking $v = aH - u \in \rH^{1,1}(X)$ for each $a \in \bC$, we see that $\chi_u(t)$ and $p_u(t)$ have the same roots. 
    \end{proof}

    \begin{lemma}
\label{lem:char_poly_embeds_into_end}
    Let $X$ be a polarized abelian $g$-fold, and let $u \in \NS(X)_{\bQ}$. Then there is a ring homomorphism
    \[
        \bQ[t]/p_u(t) \to \End^s(X)_{\bQ}
    \]
    sending $t$ to $\tilde \rho(u)$. In particular, if $p_u(t)$ is irreducible over $\bQ$, then $\End^s(X)_{\bQ}$ contains the field $\bQ[t]/p_u(t)$.
\end{lemma}

\begin{proof}
    From Lemma~\ref{lem:characteristic_polynomial}, the map $\bQ[t] \to \End(\rH_{1,0})$ factors through $\bQ[t]/p_u(t)$. There is a commutative diagram of ring homomorphisms
    \[
        \begin{tikzcd}
            \bQ[t] \ar[d, "t \mapsto \tilde \rho(\class)"] \ar[r] & \bQ[t]/p_u(t) \ar[d, "t \mapsto \rho(\class)"] \ar[dl, dotted] \\
            \End^s(X)_{\bC} \ar[r, hook] & \End(\rH_{1, 0}).
        \end{tikzcd}
    \]
    The injectivity of the bottom arrow, and hence the existence of the dotted arrow, follows from chasing the diagram \eqref{eq:ns_big_diagram}.
\end{proof}

\begin{remark}
\label{rem:end_of_simple}
    If $X$ is a simple abelian threefold, then $\End(X)_{\bQ}/\bQ$ is either the trivial extension, an imaginary quadratic extension, a totally real cubic extension, or a CM extension of degree $6$. It follows that the rank of $\NS(X)_{\bQ} \simeq \End^s(X)_{\bQ}$ is either $1$ or $3$, and in the latter case, $\End^s(X)_{\bQ}$ is a totally real cubic field \cite[Proposition 5.5.7]{birk_lange}.
\end{remark}

    \begin{lemma}
\label{lem:idempotent_image}
    Let $X$ be a simple abelian $g$-fold such that $\End^s(X)_{\bQ}$ is a product of fields (e.g., a simple abelian threefold). If $u \in \NS(X)_{\bC}$ has rank $2$, then there exists a constant $c \in \bC$ so that $c \cdot \tilde \rho(\class)$ is idempotent.
\end{lemma}

\begin{proof}
    Since the rank of $\class$ is $2$, we have $\class^2 = 0$. Thus, it follows from Lemma~\ref{eq:alternative_formula} that 
    \begin{equation*}
        \bC[t]/p_u(t) \simeq \bC[t]/(t^{g-1}(t - \lambda))
    \end{equation*} 
    for some constant $\lambda$. Since $\End^s(X)_{\bC}$ is nilpotent-free and $\tilde \rho(\class)$ is nonzero, $\lambda$ is nonzero. Similarly, since $t(t - \lambda)$ is nilpotent in $R_\class$, $\tilde \rho(\class)^2 - \lambda \cdot \tilde \rho(\class) = 0$. It follows that $1/\lambda \cdot \tilde \rho(\class)$ is idempotent.
\end{proof}

    \begin{proposition}
    \label{prop:rank_two_points}
    Let $X$ be an abelian $g$-fold such that $\End^s(X)_{\bQ}$ is a product of fields (e.g., a simple abelian threefold). Then the intersection 
    \begin{equation} \label{eq:intersect_with_gr}
        \bP(\NS(X)_{\bC}) \cap \Gr(2, \rH^1_{\bC}) \subset \bP(\rH^2(X, \bC))
    \end{equation}
    is finite.
    \end{proposition}

    \begin{proof}
        The intersection \eqref{eq:intersect_with_gr} is the set of rank $2$ points in $\bP(\NS(X)_{\bC})$. We choose a polarization on $X$. From Lemma~\ref{lem:idempotent_image}, the image of \eqref{eq:intersect_with_gr} in $\bP(\End^s(X)_{\bC})$ under $\bP(\tilde \rho)$ is contained in the image in $\bP(\End^s(X)_{\bC})$ of the set of nonzero idempotents in $\End^s(X)_{\bC}$. The set of idempotents of $\End^s(X)_{\bC}$ is in bijection with the set of clopen subsets of $\Spec \End^s(X)_{\bC}$, hence is finite.
    \end{proof}

    \subsection{Hodge loci}

    Let $f\colon Y \to S$ be a smooth projective morphism of complex varieties, and let $u \in \rH^{2k}(Y_s, \bQ)$ for some $s \in S(\bC)$. The \emph{Hodge locus} $\Hdg(u) \subset S$ is the set of points $t$ such that there exists a path $\gamma$ from $s$ to $t$ which transports $u$ to a Hodge class in $\rH^{2k}(Y_t, \bQ)$. 

    Given a polarized abelian $g$-fold $(X, H)$ and a class $u \in \rH^2(X, \bQ)$, there is an equation for the Hodge locus of $u$ in the moduli space of polarized abelian $g$-folds, or rather its universal cover $\bH_g$. Recall that Siegel's upper half-space $\bH_g$ is the space of symmetric complex matrices $Z$ such that $\mathrm{Im} \ Z$ is positive-definite. 

    For simplicity, we treat only the case of principally polarized abelian varieties. On the other hand, it is convenient to allow real classes $u \in \rH^2(X, \bR)$. Each $Z \in \bH_g$ determines a principally polarized abelian variety $X$ is follows: The $g \times 2g$ matrix $(Z, 1_g)$ induces an $\bR$-linear map $\bR^{2g} \to \bC^g$. Then $X = \bC^g/\Lambda$, where $\Lambda$ is the image of $\bZ^{2g}$. 

    Let $M$ be the real skew-symmetric matrix corresponding to $u \in \rH^2(X, \bR)$ in an $H$-symplectic basis of $\bZ^{2g}$. We write $M$ as matrix of $g \times g$ blocks
    \[
        M = \begin{pmatrix}
            A & B \\
            -B^t & C
        \end{pmatrix},
    \]
    where $A$ and $C$ are skew-symmetric.

     \begin{lemma}
    \label{lem:birkenhake--lange}
        In the above situation, the class $u$ lies in $\rH^{1,1}(X)$ if and only if  
        \begin{equation}
        \label{eq:hodge-locus-equation}
            A - BZ + ZB^t + ZCZ = 0
        \end{equation}
    \end{lemma}

    \begin{proof}
        This is \cite[\S 1, Proposition 3.4]{bl_complex_tori}. Note, however, that our matrix $Z$ is symmetric, and also that the cited result is a criterion for an integral class $u \in \rH^2(X, \bZ)$ to belong to the N\'eron--Severi group  $\NS(X) = \rH^2(X, \bZ) \cap \rH^{1,1}(X)$. The argument applies to real classes $u \in \rH^2(X, \bR)$ without change. 
    \end{proof}

    \begin{remark}
    \label{remark:codimension_of_hodge_locus}
        Lemma~\ref{lem:birkenhake--lange} shows that the codimension of $\Hdg(u)$ inside of $\bH_g$ is at most $g(g - 1)/2$. Indeed, the expression \eqref{eq:hodge-locus-equation} is skew-symmetric in the entries of $Z$, so it imposes at most $g(g - 1)/2$ conditions on $\bH_g$. The same holds in the (coarse) moduli space $\cA_g$ of principally polarized abelian varieties: the Hodge locus $\Hdg(u) \subset \cA_g$, which is simply the image of $\Hdg(u) \subset \bH_g$, is a subvariety of codimension $\leq g(g - 1)/2$. By working up to polarized isogeny, one gets the same bound for Hodge loci in a moduli space of abelian varieties with any fixed polarization type.
    \end{remark}

    \begin{proposition}
    \label{prop:main_prop}
        Let $(X, H)$ be a polarized abelian variety, and let $u \in \rH^2(X, \bZ)$ be an element such that $p_u(t)$ is irreducible over $\bQ$ and has real roots. Then there is a polarized deformation from $(X, H)$ to $(X', H')$ such that the following hold:
        \begin{enumerate} 
            \item $X'$ is simple.
            \item A parallel transport $u'$ of $u$ from $X$ to $X'$ lies in $\NS(X')$. If, in addition, the roots of $p_u(t)$ are positive, then $u'$ is a polarization. 
         \end{enumerate} 
    \end{proposition}

   \begin{remark}
       \label{remark:polarizations}
       Let $(X,H)$ be a polarized abelian $g$-fold, and let $u \in \NS(X)$ be a class such that $p_u(t)$ has positive real roots. Then $u$ is a polarization. Indeed, if $p_u(t)$ has positive real roots, then the coefficients of $t^i$ in $p_u(t)$ are nonzero and alternate in sign. By Lemma~\ref{eq:alternative_formula}, 
       \begin{equation*}
           H^i u^{g - i} > 0,
       \end{equation*}
       for all $i$, so it follows from \cite[4.3.3]{birk_lange} that $u$ is a polarization.
    \end{remark}

     \begin{remark}[The real symplectic group]
    \label{rem:real_symplectic}
        Given a principally polarized abelian variety $(X, H)$ and a real symplectic transformation $g \in \Sp(\rH_1(X, \bR), H)$,  
        one may construct a new principally polarized abelian variety $(X^g, H)$, given as follows: 
        the weight-one Hodge structure corresponding to $X^g$ is determined by the real Hodge structure $\rH_1(X, \bR)$ with lattice $g(\rH_1(X, \bZ))$. 
        Since $g$ is symplectic, $X^g$ is principally polarized; moreover, 
        an element $u \in \rH^2(X, \bQ)$ lies in $\NS(X)_{\bQ}$ if and only if $g^*u$ lies in $\rH^{1,1}(X^g)$. 
        The association $(g, X) \mapsto X^g$ corresponds to the natural action of $\Sp_{2g}(\bR)$ on Siegel's upper half-space $\bH_g$. 
    \end{remark}

    \begin{proof}[Proof of Proposition~\ref{prop:main_prop}]
        Assume for the moment that we have found a parallel transport $u'$ of $u$ which is Hodge. Then we may take $X'$ to be simple. Indeed, $p_u(t) = p_{u'}(t)$, so $\End_{\bQ}(X')$ will contain the totally real field $F = \bQ[t]/p_u(t)$ by Lemma~\ref{lem:char_poly_embeds_into_end}. In the moduli space of abelian $g$-folds with fixed polarization type, the locus of abelian varieties with multiplication by $F$ is a union of codimension $g(g - 1)/2$ subvarieties whose generic elements are simple abelian varieties of Picard rank $g$ \cite{debarre_laszlo}. 
        On the other hand, the Hodge locus of $u$ is contained in this locus and has codimension at most $g(g - 1)/2$ (Remark~\ref{remark:codimension_of_hodge_locus}), so we may indeed take $X'$ to be simple.  
        If, in addition, the roots of $p_u(t)$ are positive, then $u'$ is a polarization by Remark~\ref{remark:polarizations}.

        Working up to polarized isogeny, we may assume that $H$ is a principal polarization. By Lemma~\ref{lem:symplectic_orbit} and the assumption that $p_u(t)$ has real roots (which are in fact distinct, since $p_u(t)$ is assumed to be irreducible over $\bQ$), there is a real symplectic transformation $g \in \Sp_{2g}(\bR)$ which diagonalizes $u$. Replacing $X$ by $X^g$ and $u$ by $g^* u$ as in Remark~\ref{rem:real_symplectic}, it is enough to show that if $u \in \rH^2(X, \bR)$ is a diagonal form, then then there is a polarized deformation $(X', H')$ such that $u$ is Hodge. Suppose that the skew-symmetric matrix corresponding to $u$ is of the form
        \begin{equation}
            M = \begin{pmatrix}
                0 & B \\
                -B & 0
            \end{pmatrix} \quad \text{where} \quad B = \begin{pmatrix}
                b_1 \\
                & b_2 \\
                & & b_3
            \end{pmatrix}.
        \end{equation}
        Let $X_Z$ be a principally polarized abelian variety associated to an element $Z \in \bH_g$. From Lemma~\ref{lem:birkenhake--lange},  a parallel transport of $u$ to $X_Z$ belongs to $\rH^{1,1}(X_Z)$ if and only if 
        \[
            -BZ + Z B = 0.
         \] 
        The relation is satisfied for any diagonal matrix $Z \in \bH_g$. 
    \end{proof}

    \subsection{Forms with positive discriminant}

    In this section, we specialize to the case of abelian threefolds. Roughly, the goal is to show that a class $\bar u \in \rH^2(X, \bZ/n)$ can be lifted to a large number of classes $u \in \rH^2(X, \bZ)$ such that the cubic polynomial $p_u(t)$ has positive real roots. 

    For an integer $N > 0$, we write $\Gamma_6(N)$ for the finite-index subgroup of $\SL_6(\bZ)$ consisting of matrices which reduce to the identity modulo $N$. For $u \in \rH^2(X, \bZ)$ and $g \in \SL_6(\bZ)$ (regarded as $\SL_6(\rH_1(X, \bZ))$), we write $g^* u$ for the action of $g$.

    \begin{lemma}
    \label{lem:orbit_trick}
        Let $(X, H)$ be a polarized abelian $3$-fold, $u \in \rH^2(X, \bZ)$ be a class with $u^3 > 0$, and $N > 0$ be an integer. 
        \begin{enumerate} 
            \item \label{item:density} Let $\rO(u)$ be the $\Gamma_6(N)$-orbit of $u$. Then the image of $\rO(u)$ in $\bP(\rH^2(X, \bC))$ is Zariski dense.
            \item \label{item:positivity_of_disc} Let $\rO^+(u) \subset \rO(u)$ be the subset of classes $v$ such that $p_v(t)$ has positive real roots 
            Then $\rO^+(u)$ is nonempty and has Zariski dense image in $\bP(\rH^2(X, \bC))$.
        \end{enumerate}
    \end{lemma}

    \begin{proof}
        For \eqref{item:density}, $\Gamma_6(N)$ is Zariski dense in $\SL_6(\bC)$ (since $\SL_6(\bZ)$ is), so it is enough to show that the $\SL_6(\bC)$-orbit of $u$ is Zariski dense in $\bP(\rH^2(X, \bC))$. Since $u$ has rank $6$, the orbit is the complement of the Pfaffian hypersurface $\Pf = 0$.

        For \eqref{item:positivity_of_disc}, we first let $V \subset \bP(\rH^2(X, \bC))$ be an arbitrary closed proper subset. To prove Zariski-density of $\rO^+(u)$, we will show below that there is an element of $\rO^+(u)$ whose image in $\bP(\rH^2(X, \bC))$ does not lie in $V$. 
        From \eqref{item:density}, we may choose $v \in \rO(u)$ whose image does not lie in $V$. 
        Let $A = (a_{ij})$ be the alternating matrix associated to $v$ in an integral $H$-symplectic basis. 
        Applying \eqref{item:density} once more, we may assume that 
        \begin{equation}
        \label{eq:genericity_on_v}
            a_{16} \neq 0, \quad a_{13}a_{26} - a_{16}a_{23} - a_{12}a_{36} \neq 0.
        \end{equation}

        In our chosen integral $H$-symplectic basis, the matrix for $H$ (also denoted by $H$) is 
        \begin{equation*}
            H = \begin{pmatrix}
                0 & D \\
                -D & 0
            \end{pmatrix} \quad \text{where} \quad D = \begin{pmatrix}
                d_1 \\
                & d_2 \\
                & & d_3
            \end{pmatrix},
        \end{equation*}
        and $(d_1, d_2, d_3)$ is the type of $H$. 
        To conclude the setup, consider the upper-triangular matrix
        \begin{equation}
            E = \begin{pmatrix}
            1 & x & xy & 0 & 0 & 0 \\
            & 1 & y & 0 & 0 & 0  \\
            & & 1 & z & 0 & 0  \\
            & & & 1 & 0 & 0 \\
            & & & & 1 & 0 \\
            & & & & & 1
        \end{pmatrix},
        \end{equation}
        where $x,y,z$ are variables. If we specialize $x,y,z$ to integer multiples of $N$, then $E$ specializes to an element of $\Gamma_6(N)$.

        Consider the characteristic Pfaffian
        \[
            \Pf(tH - E^t A E) = at^3 + bt^2 + ct + d.
        \]
        The coefficients are polynomials in $x,y,z$. 
        A computation shows that:
        \begin{align*}
            a &= d_1d_2d_3, \\
            b &= -d_1d_2a_{16}xy + \cdots, \\ 
            c &= d_2(a_{13}a_{26} - a_{16}a_{23} - a_{12}a_{36})yz + \cdots, \\
            d &= -\Pf(A) = -u^3/6,
        \end{align*}
        where the omitted terms are of lower degree in $x,y,z$.  
        Note that $b$ and $c$ are polynomials in $x,y,z$, whereas $a$ and $d$ are constants. 
        Writing $\Delta$ for the discriminant of the cubic polynomial $\Pf(tH - E^tAE)$, we get that 
        \[
            \Delta = b^2c^2 + \cdots,
        \]
        where the omitted terms have degree $< 8$ in $x,y,z$.

        By our assumption \eqref{eq:genericity_on_v} on $v$ above, the leading coefficients of $b$ and $c$ are nonzero. 
        From the formulas above, it is easy to see that we may specialize $(x,y,z)$ to $(x_0, y_0, z_0) \in N \cdot \bZ^3$ so that 
        \begin{equation}
        \label{eq:sign_conditions}
             b < 0, \quad c > 0, \quad  \Delta > 0.
        \end{equation}
        In fact, if $(x'_0, y'_0, z'_0) \in N \cdot \bZ^3$ is any element so that
        \[
            -d_1d_2a_{16}x'_0y'_0 < 0, \quad d_2(a_{13}a_{26} - a_{16}a_{23} - a_{12}a_{36})y'_0z'_0 > 0,
        \]
        then we can take $(x_0,y_0,z_0) = m \cdot (x'_0, y'_0, z'_0)$ for $m \gg 0$. We refer to a choice of element $(x_0, y_0, z_0) \in N \cdot \bZ^3$ such that \eqref{eq:sign_conditions} holds as a \emph{good choice}.

        A good choice $(x_0, y_0, z_0)$ determines an element $g \in \Gamma_6(N)$ by specializing the coefficients of $E$. The cubic $\Pf(tH - E^t AE)$ specializes to the characteristic Pfaffian $p_{g^* v}(t)$. From \eqref{eq:sign_conditions}, $p_{g^* v}(t)$ has positive discriminant $\Delta > 0$, hence real roots. Moreover, $p_{g^* v}(t)$ has alternating coefficients $a > 0, b < 0, c > 0, d < 0$, which implies that the roots of $p_{g^* v}(t)$ are positive. It follows that $g^* v$ lies in $\rO^+(u)$.

        We show that $\rO^+(u)$ is Zariski dense. The set of good choices $(x_0, y_0, z_0)$ is Zariski dense in $\bC^3$: as explained above, it contains every element of $\bZ^3$ with nonzero coordinates, modulo flipping signs in each coordinate and scaling by a sufficiently large positive integer. Given  $(x_0, y_0, z_0) \in \bC^3$ with associated $g \in \SL_6(\bC)$, it is a Zariski closed condition on the entries of $(x_0, y_0, z_0)$ for $[g^* v] \in \bP(\rH^2(X, \bC))$ to lie in the closed subset $V$ chosen above. By assumption, $[v] \notin V$, so the Zariski closed condition is nontrivial. By Zariski density of the set of good choices, there is a good choice $(x_0, y_0, z_0)$ such that $[g^* v]$ does not lie in $V$.
    \end{proof}

    \subsection{Generic lifting}

    \begin{proposition}[Generic lifting]
    \label{prop:level_subgroup}
        Let $(X, H)$ be a polarized abelian threefold. Fix a class $\bar \class \in \rH^2(X, \bZ/n)$ for an integer $n > 1$, and suppose that $\bar u$ admits a lift to a class $u_0 \in \rH^2(X, \bZ)$ of type $(1,1,1)$ with $u_0^3 > 0$. Let $\cL$ be the set of $\class \in \rH^2(X, \bZ)$ satisfying the following properties:
        \begin{enumerate} [label = (\alph*)]
            \item \label{item:oneoneone} $\class$ is of type $(1,1,1)$,
            \item \label{item:lift} $\class \equiv \bar \class \mod n$,
            \item \label{item:hodgelocus} There exists a polarized deformation from $(X, H)$ to $(X', H')$, where $X'$ is simple and some parallel transport $\class'$ of $\class$ to $X$ lies in $\NS(X')$ and is a polarization. 
        \end{enumerate}
        Then the image of $\cL$ in $\bP(\rH^2(X, \bC))$ is Zariski dense.
    \end{proposition}

    \begin{remark}
    \label{rem:inert}
    Let $(X, H)$ be a simple, polarized abelian threefold with multiplication by a cyclic cubic field $K/\bQ$. In particular, $K$ is totally real, and the construction of such $(X, H)$ can be found in \cite[Ch. 9]{birk_lange}. Let $v \in \NS(X)$ be a class which does not lie in $\langle H \rangle$. 
    Then $\bZ[t]/p_v(t)$ is an order in $K$, and there exists an infinite set $\cP$ of primes $\ell$ such that the reduction of $p_v(t)$ modulo $\ell$ is irreducible.
    Indeed, up to a finite set, $\cP$ coincides with the inert primes for the cyclic extension $K/\bQ$, and the latter is an infinite set of primes with density $2/3$. This is a consequence of Chebotarev's density theorem \cite[Ch. IV, Corollary 5.3]{janusz}.
    \end{remark}

    \begin{proof}
        We begin by choosing a prime $\ell$ which does not divide $n$ and a class $\bar v \in \rH^2(X, \bZ/\ell)$ such that the polynomial $p_{\bar v}(t) = \Pf(t\bar H - \bar v)$ is irreducible. The fact that $\bar v$ may be found for an infinite set of primes follows, for instance, from Remark~\ref{rem:inert}.

        Let $u \in \rH^2(X, \bZ)$ be a class of type $(1,1,1)$ with $u^3 > 0$ lifting both $\bar u$ and $\bar v$. One way to construct $u$ is through the observation that the reduction-mod-$\ell$ map
        \[
            \Gamma_6(n) \to \SL_6(\bZ/\ell)
        \]
        is surjective, since the right-hand side is generated by transvections, which can be lifted \cite[XIII, \S 9]{Lang}. To construct $u$, recall that we assumed the existence of a class $u_0$ of type $(1,1,1)$ with $u_0^3 > 0$ lifting $\bar u$. Since $\SL_6(\bZ/\ell)$ acts transitively on the set of rank $6$ classes in $\rH^2(X, \bZ/\ell)$, we can find an element $g \in \Gamma_6(n)$ such that
        \[
            g^* u_0 \equiv \bar v \mod \ell.
        \]
        Take $u = g^* u_0$. 

        We observe that $p_u(t)$ is irreducible, and the same is true for any $v$ in the $\Gamma_6(N)$-orbit $\rO(u)$ of $u$, where $N = n \cdot \ell$. Indeed, irreducibility may be checked after reduction mod $\ell$, and the reduction of $p_{v}(t)$ modulo $\ell$ is $p_{\bar v}(t)$. Moreover, for any $v \in \rO(u)$, we have $v^3 = u^3 = u_0^3 > 0$, since the Pfaffian $\Pf(u) = u^3/6$ is preserved under orientation-preserving change of basis.

        Finally, let $\rO^+(u)$ be the set considered in Lemma~\ref{lem:orbit_trick}. The image of $\rO^+(u)$ in $\bP(\rH^2(X, \bC))$ is Zariski dense by Lemma~\ref{lem:orbit_trick}, so if we can show that $\rO^+(u)$ is a subset of $\cL$, i.e., any element $v \in \rO^+(u)$ satisfies \ref{item:oneoneone}--\ref{item:hodgelocus}, then the image of $\cL$ in $\bP(\rH^2(X, \bC))$ is Zariski dense and we are done. 

        Item~\ref{item:oneoneone} holds because any element of the $\Gamma_6(N)$-orbit (or, indeed, the $\SL_6(\bZ)$-orbit) of $u$ is of type $(1,1,1)$. Item~\ref{item:lift} holds because the action of $\Gamma_6(N)$ on $\rH^2(X, \bZ)$ preserves the reduction modulo $N = n\cdot \ell$, hence also the reduction modulo $n$. Finally, for any $v \in \rO^+(u)$, $p_{v}(t)$ is irreducible as explained above, and has positive real roots by assumption. Then Proposition~\ref{prop:main_prop} gives \ref{item:hodgelocus}.
    \end{proof}

\subsection{Prime avoidance}

\begin{proposition}
\label{prop:prime_avoidance}
    Let $X$ be a simple abelian threefold, and let $n > 1$ be an integer. Let $\class, H \in \NS(X)$ be classes satisfying the following properties:
    \begin{enumerate}
        \item The class $\class \in \NS(X)$ is of type $(1,1,1)$.
        \item The class $H \in \NS(X)$ is of type $(d_1, d_2, d_3)$, where each $d_i$ is nonzero and divides a power of $n$.
        \item The $\bC$-linear span $\langle \class^2, H^2 \rangle_{\bC} \subset \rH^4(X, \bC)$ has dimension $2$, and does not contain any rank $2$ elements.
    \end{enumerate}
    There exists an integer $A > 0$, relatively prime to $n$, and an integer $k > 0$, which is arbitrarily large with respect to $A$, such that the class
    \begin{equation} \label{eq:beta_class_prop}
        \frac{1}{2}(A^2 \class - n^k H)^2 + \frac{1}{2}  A n \cdot \class^2
    \end{equation}
    is of type $(1,1,d)$, for some integer $d > 0$.
\end{proposition}

\begin{notation} With the assumptions of Proposition~\ref{prop:prime_avoidance}, we write $V \subset \rH^4(X, \bZ)$ for the sublattice of integral Hodge classes. We observe that $\rk_{\bZ}V = 3$, by Remark~\ref{rem:end_of_simple}.
\end{notation}

\begin{notation}
\begin{itemize} 
    \item We write $\cP_0$ for the set of primes which divide the order of the torsion part of $V/\langle \class^2/2, \class \cdot H, H^2/2 \rangle$.
    \item Let $Z \subset \bP(V)$ be a closed subscheme. We write $\cP(Z)$ for the (possibly infinite) set of primes $p \in \bZ$ such that the $\bZ$-scheme
    \begin{equation} \label{eq:intersect_with_gr_z}
        Z \cap \Gr(2, 6) \subset \bP(\rH^4(X, \bZ)).
    \end{equation}
    has a nonempty fiber over $\bF_p$, where $\Gr(2, 6) = \Gr(2, \rH_1(X, \bZ))$ is embedded via Poincar\'e duality.
\end{itemize}
\end{notation}

\begin{remark}
\label{rem:finiteness_of_bad_primes}
    Let $Z \subset \bP(V)$ be a closed subscheme such that the complex fiber $Z_{\bC} \cap \Gr(2, 6)$ of the intersection above is empty. Then $\cP(Z)$ is finite.
\end{remark}

\begin{lemma}
\label{lem:checking_lemma}
    Let $a, b$, and $c$ be integers, and consider the class $z = a \class^2/2 + b \class H + c H^2/2$. Assume the following:
    \begin{enumerate} 
        \item $\gcd(a, b, c) = 1$.
        \item For each $p \in \cP_0$, the reduction of $z$ modulo $p$ is nonzero.
        \item The point $[z] \in \bP(V)(\bZ)$ is contained in $Z(\bZ)$, for some closed subscheme $Z \subset \bP(V)$, and for each $p \in \cP(Z)$, the reduction of $z$ modulo $p$ has rank at least $4$. 
    \end{enumerate}
    Then $z$ is of type $(1, 1, d)$, for some $d > 0$.
\end{lemma}

\begin{proof}
    Conditions (1) and (2) imply that $z$ is primitive, i.e., that it is of type $(1, d, e)$ for some $d > 0$. If $p$ is a prime which divides $d$, then the reduction $[z]_p$ of $[z]$ modulo $p$ lies in the intersection $Z_p \cap \Gr(2, 6)$, so $p$ is contained in $\cP(Z)$. On the other hand, $[z]_p$ has rank at most $2$, contradicting (3).
\end{proof}

\begin{lemma}
\label{lem:rank_two_case}
    Proposition~\ref{prop:prime_avoidance} holds if $\rk_{\bZ}\langle \class^2/2, \class \cdot H, H^2 /2 \rangle = 2$.
\end{lemma}

\begin{proof}
    By assumption (3) of Proposition~\ref{prop:prime_avoidance}, $\bP(V)$ satisfies the conditions of Remark~\ref{rem:finiteness_of_bad_primes}, so $\cP(\bP(V))$ is finite.

    Let $A$ be the product of the primes in $\cP_0 \cup \cP(\bP(V))$ which do not divide $n$, and consider the class
    \begin{align*}
        z &= \frac{1}{2}(A^2  \class - n^k  H)^2 + \frac{1}{2}   n A \cdot \class^2 \\
            &= \frac{1}{2}(A^4 + n A) \class^2 - n^k A^2 \cdot \class H + \frac{1}{2} n^{2k} \cdot H^2
    \end{align*}
    for any $k > 0$. 

    Given a prime $p$ in $\cP_0 \cup \cP(\bP(V))$, $p$ either divides $n$ or $A$, and the reduction of $z$ modulo $p$ is nonzero by assumptions (1) or (2) of Proposition~\ref{prop:prime_avoidance}. Since $A$ is prime to $n$, $z$ satisfies the assumptions of Lemma~\ref{lem:checking_lemma}, and $k$ may be taken to be arbitrarily large. 
\end{proof}

The case when $\rk_{\bZ}\langle \class^2/2, \class \cdot H, H^2 /2 \rangle = 3$ is more complicated, because the complex fiber $\Gamma$ of its intersection with $\Gr(2, 6)$ is equal to $\bP(V)_{\bC} \cap \Gr(2, 6)$, which is not empty. However, by Proposition~\ref{prop:rank_two_points}, $\Gamma$ is supported on a finite set. Therefore, our strategy in the rank $3$ case is to produce subvarieties of $\bP(V)$ which avoid $\Gamma$ and contain the classes \eqref{eq:beta_class_prop} from the proposition.

\begin{situation}
Throughout the remainder of the section, we work in the context of Proposition~\ref{prop:prime_avoidance}. In particular, $n > 1, \class$, and $H$ are fixed. We write
\[
    \pi:\bP_{\bZ}^2 \to \bP(V)
\]
for the morphism over $\Spec(\bZ)$ induced by the map $\bZ^3 \to V$ sending the standard basis to $\class^2/2, \class \cdot H, H^2/2$. 
\end{situation}

\begin{example}[Conics]
For any integer $t$, consider the map
\begin{align*} 
    f_t: \bP^1_{\bZ} \to \bP^2_{\bZ}, \hspace{5mm} [x, y] \mapsto [(t^4 + nt)x^2, -t^2 xy, y^2].
\end{align*}
Let $C_t'$ be the scheme-theoretic image of $f_t$ in $\bP^2_{\bZ}$, and let $C_t \subset \bP(V)$ be the image of $C'_t$ under $\pi$.
\end{example}

\begin{example}[Quartics]
Let $t$ be an integer, and consider the map
\[
    g_t:\bP^1_{\bZ} \to \bP^2_{\bZ}, \hspace{5mm} [x, y] \mapsto [x^4 + n x y^3, -t \cdot x^2 y^2, t^2 \cdot y^4].
\]
Let $Q'_t$ be the scheme-theoretic image of $g$ in $\bP_{\bZ}^2$, and let $Q_t$ be the image of $Q'_t$ under $\pi$.
\end{example}

\begin{remark}
    In the context of Proposition~\ref{prop:prime_avoidance}, the point \eqref{eq:beta_class_prop} is contained in the intersection $C_A \cap Q_{n^k}(\bZ)$.
\end{remark}

\begin{lemma}
\label{lem:bpf_conics_quartics}
    Suppose that $\rk_{\bZ}\langle \class^2/2, \class \cdot H, H^2 /2 \rangle = 3$. For all but finitely many integers $t$, the intersections
    \[
        (C_t)_{\bC} \cap \Gr(2, 6), \hspace{5mm} (Q_t)_{\bC} \cap \Gr(2, 6)
    \]
    are empty.
\end{lemma}

\begin{proof}
    By Proposition~\ref{prop:rank_two_points}, the intersection $\Gamma = \Gr(2, 6) \cap \NS(X)_{\bC}$ is supported on a finite set of points, and it suffices to show that the basepoints of the families $\{C_t\}$ and $\{Q_t\}$ do not intersect $\Gamma$ over $\bC$. We claim that (in each case) the only basepoints are $[1, 0, 0]$ and $[0, 0, 1]$. The lemma follows, since both points correspond to forms of rank $6$. For the duration of the proof, we write $C_t$ and $Q_t$ to refer to the complex fibers $(C_t)_{\bC}$ and $(Q_t)_{\bC}$, respectively.

    We begin with the conic case. Suppose that $[a, b, c]$ is a point which is contained in $C_t$ for all $t \in \bZ$. Let $D \subset \bP^2_{\bC}$ be the standard conic, defined as the image of the map $\bP^1_{\bC} \to \bP^2_{\bC}$ given by 
    \[
        [x,y] \mapsto [x^2, xy, y^2] .
    \]
    For a general choice of $t \in \bC$, there is a linear change of coordinates $L_t:\bP^2_{\bC} \to \bP^2_{\bC}$ carrying $C_t$ to $D$, given by 
    \[
        [x, y, z] \mapsto \left[ (t^4 + nt)^{-1} \cdot x, -t^{-2} y, z \right].
    \]
    It follows that $L_t([a, b, c])$ lies on $D$ for generic $t$, so we obtain a family of points 
    \[
        p_t = [(t^4 + nt)^{-1}a, -t^{-2}b, c]
    \]
    contained in $D$, for almost all $t \in \bC$. 

    If both $a, b \neq 0$, then the Zariski closure of $\{p_t\}$ is positive-dimensional, contained in $D$, and is given by the set
    \[
        Z = \{[s^4ta, -(s^2t^3 + ns^5)b, t(t^4 + ns^3t)c]\}, \quad [s, t] \in \bP^1.
    \]
    Setting $t = 0$, we see that $Z$ contains the point $[0, nb, 0]$. But $[0,nb, 0]$ is not contained in $D$, contradicting our assumption that both $a, b \neq 0$. Therefore, either $a$ or $b$ is $0$; from the formula for the curves $C_t$ which contain $[a,b,c]$ by assumption, only $[1,0,0]$ and $[0,0,1]$ are possible. This concludes the conic case.

    We now turn to the quartic case. Arguing as above, let $[a, b, c]$ be a basepoint, and for generic $t$ let $L_t:\bP^2_{\bC} \to \bP^2_{\bC}$ be the linear change of coordinates 
    \[
        [x, y, z] \mapsto [x, -t^{-1} y, t^{-2} z],
    \]
    which carries $Q_t$ onto $Q_{-1} = \{[x^4 + nxy^3, x^2y^2, y^4]\}$. Clearing denominators, one sees that the Zariski closure of the points $\{q_t = L_t([a, b, c])\}$ is given by 
    \[
        W = \{[at^2, -bst, cs^2]\}, \quad [s, t] \in \bP^1.
    \]
    There are three possible cases:
    \begin{enumerate} 
        \item If $a, b, c \neq 0$, then $W$ is a conic which is not contained in $Q_{-1}$, so this case cannot occur.
        \item If exactly one of $a,b,c$ is $0$, then $W$ is a line. But we claim that $Q_{-1}$ does not contain a line, so this case cannot occur. Since $Q_{-1}$ is irreducible (as it is the image of a map from $\bP^1$), it is enough to show that $Q_{-1}$ is not supported on a line. But the intersection $Q_{-1} \cap \{y = z\}$ contains at least the points $[1 + n, 1, 1]$ and $[1 - n, 1, 1]$, which are distinct since we have assumed that $n > 1$. 
        \item The case $[a,b,c] = [0,1,0]$ cannot occur, since $[0,1,0]$ is not contained in $Q_t$ for any $t$.
    \end{enumerate}
    The only remaining possibilities are $[a,b,c] = [1,0,0]$ or $[0,0,1]$.
\end{proof}

We are now in a position to complete the proof of Proposition~\ref{prop:prime_avoidance}. Some care is needed in the proof to ensure that one may take $k$ to be arbitrarily large with respect to $A$.

\begin{proof}[Proof of Proposition~\ref{prop:prime_avoidance}]
    By Lemma~\ref{lem:rank_two_case}, we may suppose that $\rk_{\bZ}\langle \class^2/2, \class \cdot H, H^2 /2 \rangle = 3$. From Lemma~\ref{lem:bpf_conics_quartics}, we may choose an integer $k_0$ such the intersection
    \[
        (Q_{n^{k_0 }})_{\bC} \cap \Gr(2, 6)
    \]
    is empty. From Remark~\ref{rem:finiteness_of_bad_primes}, the set $\cP(Q_{n^{k_0}})$ is finite. 

    Let $A_0$ be the product of the primes in $\cP(Q_{n^{k_0}}) \cup \cP_0$ which do not divide $n$. By Lemma~\ref{lem:bpf_conics_quartics}, we may choose a power $A$ of $A_0$ such that the intersection
    \[
        (C_A)_{\bC} \cap \Gr(2, 6)
    \]
    is empty. Again from Remark~\ref{rem:finiteness_of_bad_primes}, $\cP(C_A)$ is finite. 

    The class
    \begin{align*}
        z &= \frac{1}{2}(A^2  \class - n^{k_0}  H)^2 + \frac{1}{2}   n A \cdot \class^2 \\
            &= (A^4 + n A) \cdot \frac{\class^2}{2} - n^{k_0} A^2 \cdot \class H +  n^{2k_0} \cdot \frac{H^2}{2}
    \end{align*}
    lies in the intersection of $C_A(\bZ)$ and $Q_{n^{k_0}}(\bZ)$. An identical calculation as in Lemma~\ref{lem:rank_two_case} shows that $z$ is of type $(1, 1, d)$ for some $d > 1$. However, it is not clear that we can take $k_0$ to be sufficiently large with respect to $A$.

    To remedy this, let $\cP'$ be the set of primes in $\cP_0 \cup \cP(C_A)$ which do not divide $n$, and let $\epsilon$ be any integer such that
    \[
        n^\epsilon \equiv 1 \mod p
    \]
    for any $p \in \cP'$. Then the class
    \[
        z' = \frac{1}{2}(A^4 + n A) \class^2 - n^{k_0 + \epsilon} A^2 \cdot \class H + \frac{1}{2} n^{2k_0 + 2 \epsilon} \cdot H^2
    \]
    lies in $C_A(\bZ)$, and
    \[
        z' \equiv z \mod p
    \]
    for any $p \in \cP'$. In particular, $z'$ satisfies the assumptions of Lemma~\ref{lem:checking_lemma}, so $z'$ is of type $(1, 1, d')$ for some $d' > 0$. Since $\epsilon$ may be arbitrarily large, we can arrange for $k = k_0 + \epsilon$ to be arbitrarily large with respect to $A$.
\end{proof}

\end{appendix}

\newpage

\setcounter{section}{21}

\part{Complements}
\label{part-complements}

Throughout Part~\ref{part-complements}, we work over the complex numbers. 

\section{Applications to the integral Hodge conjecture}
\label{section-applications-IHC}

In this section, we prove Theorem~\ref{theorem-voisin-group-twisted-CY3} from the introduction, concerning the Voisin group of a twisted Calabi--Yau threefold.  
Using Corollary~\ref{corollary-voisin-group-abelian3fold} from the introduction, we also construct an example of a Severi--Brauer variety $P$ for which the integral Hodge conjecture holds for $\Dperf(P)$ but fails for $P$. 

\subsection{The twisted Atiyah--Hirzebruch spectral sequence}
\label{section-twisted-AH}
To prove Theorem~\ref{theorem-voisin-group-twisted-CY3}, we need to compute $\Ktop[0](X, \alpha)$ when $\alpha$ is not necessarily topologically trivial. This is done using the twisted Atiyah--Hirzebruch spectral sequence \cite{atiyah_segal}. 

In the following lemma, we write
    \begin{equation}
    \label{eq:odd_even}
        \bZ(q/2) = \begin{cases}
            \bZ & q \equiv 0 \mod 2 \\
            0 & q \equiv 1 \mod 2.
        \end{cases}
    \end{equation}
Succinctly, $\bZ(q/2) = \Ktop[-q](\pt)$. 

\begin{lemma}[Twisted Atiyah--Hirzebruch]
\label{lem:ahss_analysis}
    Let $X$ be a smooth quasiprojective complex variety, with $\alpha \in \Br(X)$. There is a strongly convergent spectral sequence
    \begin{equation}
    \label{eq:ahss}
        \rE_2^{p, q} = \rH^p(X, \bZ(q/2)) \implies \Ktop[-p - q](X, \alpha),
    \end{equation}
    satisfying the following properties:
    \begin{enumerate}
        \item The differentials are torsion.
        \item If $X$ is projective, the resulting filtration of $\Ktop[-i](X, \alpha)$ is a filtration by pure Hodge substructures of weight $i$. 
    \end{enumerate}
\end{lemma}

\begin{proof}
    Since $X$ is quasi-projective, we may choose a Severi--Brauer variety $\pi:P \to X$ of class $\alpha$. 
    The spectral sequence \eqref{eq:ahss} is constructed in \cite{ant_will} as the descent spectral sequence for the sheaf of spectra $\Ktop((X, \alpha)/X)$, which is the relative topological $K$-theory discussed in \S \ref{section-Hodge-theory}.
    
    From Bernardara's decomposition (Lemma~\ref{lemma-D-SB}), there is a direct sum decomposition
    \[
        \Ktop(P/X) = \Ktop(X/X) \oplus \Ktop((X, \alpha)/X) \oplus \cdots
    \]  
    for the sheaf of spectra $\Ktop(P/X)$ \cite[Theorem 1.3]{moulinos}.
    Then \eqref{eq:ahss} is a summand of the descent spectral sequence
    \begin{equation}
    \label{eq:descent_ss_for_sb}
        \rE_2^{p, q} = \rH^p(X, \Ktop[-q](P/X)) \implies \Ktop[-p-q](P)
    \end{equation}
    for $\Ktop(P/X)$, so it suffices to show that \eqref{eq:descent_ss_for_sb} has torsion differentials and is a spectral sequence of Hodge structures. This is a general fact, as discussed in Remark~\ref{remark:ahss_hodge_structures} below.
\end{proof}

\begin{remark}
\label{remark:ahss_hodge_structures}
    Let $\pi:Y \to X$ be a smooth, projective morphism of complex varieties. The descent spectral sequence for $\Ktop(Y/X)$ takes the form  
    \begin{equation}
    \label{eq:descent_ss_for_morphism}
        \rE_2^{p, q} = \rH^p(X, \Ktop[-q](Y/X)) \implies \Ktop[-p-q](Y).
    \end{equation}
    Then \eqref{eq:descent_ss_for_morphism} has torsion differentials and is a spectral sequence of Hodge structures.

    Indeed, the Chern character induces an isomorphism of generalized cohomology theories 
    \begin{equation}
    \label{eq:chern_coh_theory}
        \ch: \Ktop[-q](-) \otimes \bQ \simeq \rH^{2* + q}(-, \bQ),
    \end{equation}
    as in \cite[\S 2.4]{atiyah-hirzebruch}.
    Then \eqref{eq:chern_coh_theory} induces an isomorphism between \eqref{eq:descent_ss_for_morphism} tensored with $\bQ$ and the descent spectral sequence for $\rH^{2* + q}(-, \bQ)$:
    \begin{equation}
    \label{eq:descent_ss_for_sum}
        \rE_2^{p, q} = \rH^p(X, \rR^{2* + q} f_* \bQ) \implies \rH^{2*+p+q}(Y, \bQ).
    \end{equation}
    In fact, \eqref{eq:descent_ss_for_sum} is a direct sum of vertical shifts of the usual Leray spectral sequence
    \begin{equation}
    \label{eq:std_leray}
        \rE_2^{p, q} = \rH^p(X, \rR^{q} f_* \bQ) \implies \rH^{p + q}(Y, \bQ),
    \end{equation}
    Then \eqref{eq:std_leray} degenerates by a classical result of Deligne \cite{degenerescence}, and the filtration from the Leray spectral sequence \eqref{eq:std_leray} is a filtration by mixed Hodge substructures \cite{leray_motivic}, which are pure if $X$ is proper. Unwinding the identifications, we get that the differentials of \eqref{eq:descent_ss_for_morphism} are torsion and the resulting filtration is by Hodge substructures (as the Hodge structures on $\Ktop[-q](Y)$ correspond to those on $\rH^{2* + q}(Y, \bQ)$ under the Chern character).
\end{remark}

\begin{lemma}
\label{lemma-leading-terms}
    Let $X$ be a smooth projective complex variety over $\bC$ with $\alpha \in \Br(X)$. Let $E$ be a torsion-free $\alpha|_Z$-twisted coherent sheaf on an integral closed subscheme $i:Z \to X$ of codimension $c$. Then the leading term of $[i_*E] \in \Ktop[0](X)$ is the image of $\rk(E)[Z]$ in $\rE^{2c, 0}_\infty$.
\end{lemma}

\begin{proof}
    The leading term of $i_* E$ lies in the kernel of the pullback map
    \begin{equation}
        \rH^{2c}(X, \bZ) \to \rH^{2c}(X - Z, \bZ),
    \end{equation}
    which is generated by the cycle class $[Z]$ \cite[\S 11.1.2]{voisin-hodge-theory-i}. Therefore, the leading term of $i_* E$ is of the form $m[Z]$ for some $m \in \bZ$.
    To check that $m = \rk(E)$, we reduce to the case $\alpha = 0$ by pulling back along a morphism $f:Y \to X$ from a smooth projective variety $Y$ such that $f^* \alpha = 0$, e.g., a Severi--Brauer variety of class $\alpha$. 
    
    When $\alpha = 0$, the desired claim is that the leading term of $[i_* E] \in \Ktop[0](X)$ is $\rk(E)[Z]$, for the filtration on $\Ktop[0](X)$ obtained from the Atiyah--Hirzebruch spectral sequence
    \begin{equation*}
        \rH^p(X, \bZ(q/2)) \implies \Ktop[-p-q](X).
    \end{equation*}
    After tensoring with $\bQ$, the Chern character induces an isomorphism between the spectral sequence for $\Ktop[*](-) \otimes \bQ$ and the spectral sequence for $\rH^*(-, \bQ)$ 
    (\cite[\S 2.4]{atiyah-hirzebruch}, or Remark~\ref{remark:ahss_hodge_structures} for the relative situation). 
    The isomorphism identifies the leading term of $[i_* E]$ with the leading term of $\ch(E)$.
\end{proof}

\begin{lemma}
    \label{lemma-describe-ahss}
    Let $X$ be a smooth projective complex threefold, and let $\rF^{\bullet}$ be the double-speed filtration on $\Ktop[0](X, \alpha)$ from the Atiyah--Hirzebruch spectral sequence~\eqref{eq:ahss}. Then
        \begin{align*}
                \gr_{\rF}^0 &= \ind(\bar \alpha) \cdot \bZ \subset \rH^0(X, \bZ) \\
                \gr_{\rF}^1 &= \ker(- \cup \bar \alpha: \rH^2(X, \bZ) \to \rH^5(X, \bZ))  \\
                \gr_{\rF}^2 &= \rH^4(X, \bZ)/\langle \rH^1(X, \bZ) \cup \bar{\alpha} \rangle \\
                \gr_{\rF}^3 &= \rH^6(X, \bZ).
        \end{align*}
        Here, $\bar \alpha \in \rH^3(X, \bZ)_{\tors}$ is the topological Brauer class associated to $\alpha$, and $\ind(\bar \alpha)$ is its (topological) index. 
\end{lemma}

\begin{proof}
    Up to periodicity (i.e., vertical shift), the only nonzero differentials which contribute to $\Ktop[0](X, \alpha)$ are:
    \begin{itemize}
        \item $d_3^{0,0}$ and $d_5^{0, 0}$, which determine $\gr^0_{\rF}$;
        \item $d_{3}^{2,0}$, which determines $\gr^1_{\rF}$;
        \item $d_3^{1, 0}$, which determines $\gr^2_{\rF}$.
    \end{itemize}
    The description of $\gr^0_{\rF}$ is \cite[Lemma 2.23]{ant_will}, which says that $\ind(\bar \alpha)$ is the generator of $\rE_{\infty}^{0,0} \subset \rH^0(X, \bZ)$, i.e., the product of the orders of the differentials originating at $\rE^{0,0}$.
    
    For $\gr^1_{\rF}$, we invoke a formula for $d_3$ due to Atiyah and Segal \cite[\S 4]{atiyah-segal-2}: 
        \begin{equation}
        \label{eq:atiyah_segal_formula}
            d_3(x) = (\beta \circ \mathrm{Sq}^2)(\bar x) - \bar{\alpha} \cup x,
        \end{equation}
        Here, $\bar x$ is reduction of $x$ mod $2$, $\mathrm{Sq}^2$ is the Steenrod square (whose definition and basic properties may be found in \cite[\S 4.L]{hatcher}, for instance), and $\beta:\rH^*(X,\bZ/2) \to \rH^{* + 1}(X, \bZ)$ is the mod $2$ Bockstein, i.e., the connecting homomorphism in cohomology coming from the sequence
        \begin{equation*}
            \begin{tikzcd}
                0 \ar[r] & \bZ \ar[r, "\cdot 2"] & \bZ \ar[r] & \bZ/2 \ar[r] & 0.
            \end{tikzcd}
        \end{equation*}
        
        In degree $2$ cohomology, $\mathrm{Sq}^2$ sends $\bar x \in \rH^2(X, \bZ/2)$ to $\bar x \cup \bar x \in \rH^4(X, \bZ/n)$. 
        If $\bar x$ is the reduction mod $2$ of some integral class $x \in \rH^2(X, \bZ)$, then $\bar x \cup \bar x$ is the reduction mod $2$ of $x \cup x \in \rH^4(X, \bZ)$, hence is killed by the Bockstein. 
        In other words, 
        \begin{equation*}
            d_3(x) = \beta(\bar x \cup \bar x) - \bar{\alpha} \cup x = 0 - \bar{\alpha} \cup x,
        \end{equation*}
        as needed. 

        We using the same description of $d_3$ to calculate the image of $d_3:\rH^1(X, \bZ) \to \rH^3(X, \bZ)$. It is a general fact that $\mathrm{Sq}^2$ vanishes on $\rH^{< 2}(X, \bZ/2)$, so $d_3^{0,1}$ is again given by cup-product with $\bar \alpha$.
\end{proof}

Given a nonzero class $v \in \Ktop[0](X, \alpha)$, the \emph{leading term} of $v$ is the image of $v$ in $\gr^{i}_{\rF}$, for $i$ the least integer such that $v$ does not lie in $\rF^{i + 1}$. If $v$ is nonzero, then the leading term of $v$ is nonzero. 
From Lemma~\ref{lem:ahss_analysis},  $\rF^{\bullet}$ is a filtration by Hodge substructures, so if $v$ is Hodge, then the leading term of $v$ is Hodge.

\subsection{Proof of Theorem~\ref{theorem-voisin-group-twisted-CY3}}
\label{section-proof-theorem-voisin-group}

    We freely use Lemma~\ref{lemma-describe-ahss} and Lemma~\ref{lemma-leading-terms} throughout.
    Theorem~\ref{theorem-voisin-group-twisted-CY3}  is equivalent to the statement that any Hodge class $v$ in $\rF^1\Ktop[0](X, \alpha)$, i.e. any class $v \in \Hdg(X, \alpha, \bZ)$ of rank $0$, is algebraic. Indeed, the rank map induces a surjection 
    \begin{equation*}
    \rV(X, \alpha) = \frac{\Hdg(X, \alpha, \bZ)}{\rK_0(X, \alpha)} 
    \to \frac{\ind_{\Hdg}(\alpha) \cdot \bZ}{\ind(\alpha) \cdot \bZ} 
    \end{equation*}
    whose kernel is precisely the subgroup of rank $0$ Hodge classes in $\Ktop[0](X,\alpha)$ modulo the algebraic classes of rank $0$. 

    We proceed by induction on the degree of the leading term of $v$. If the leading term lies in $\gr^3_{\rF}$, then $v$ is a multiple of the class of a (twisted) skyscraper sheaf, so $v$ is algebraic. 

    Suppose that the leading term, say $y$, lies in $\gr^2_{\rF} \simeq \rH^4(X, \bZ)/\langle \rH^1(X, \bZ) \cup \bar \alpha \rangle$. From the integral Hodge conjecture for $X$ \cite{totaro}, $y$ is the image of a $\bZ$-linear combination of cycle classes of curves in $\rH^4(X, \bZ)$. It suffices to show that for each integral curve $C \subset X$, there is an algebraic class in $\Ktop[0](X, \alpha)$ with leading term $[C]$. The pullback of the Brauer class $\alpha$ to the function field of $C$ is trivial. If $M$ is the coherent extension to $X$ of an $\alpha|_{C}$-twisted line bundle on an open subscheme of $C$, then the leading term of $[M]$ is $[C]$.

    The only remaining case is when the leading term, say $x$, lies in $\gr^1_{\rF} \subset \rH^2(X, \bZ)$. First, let $|H|$ be a very ample linear system on $X$, and consider the case when $x = n[H]$, where $n = \per(\alpha)$. If $S \in |H|$ is a smooth surface, then the restriction of $\alpha$ to $S$ has period at most $n$, so the index of $\alpha|_S$ is at most $n$ by de Jong's theorem \cite{dJ-period-index}. Therefore, there exists an $\alpha|_S$-twisted sheaf $E$ on $S$ of rank $n$, and the pushforward of $E$ to $X$ gives an algebraic class of leading term $n[H]$.

    Now we treat the general case when $x \in \gr^1_{\rF}$. From the previous paragraph, we may replace $x$ with $x + n [H']$, for $H'$ a very ample divisor. For instance, we may assume that $x = [H]$, for $H$ an ample divisor. Again from the previous paragraph, it suffices to show that there is an integer $m$, relatively prime to $n$, and an algebraic class of leading term $m[H]$. By Lemma~\ref{lem:totaro} below, we may choose $m$ so that the following holds: 

    \begin{lemma}[{\cite[Proposition 5.3]{totaro}}]
\label{lem:totaro}
    For $m \gg 0$ and $S_0 \in |mH|$ a very general smooth surface, there is a nonempty open cone $\rC \subset \rH^2(S_0, \bR)_{\mathrm{van}}$ such that the following holds: 
    every element of $\rH^2(S_0, \bZ)_{\mathrm{van}}$ whose image in $\rH^2(S_0, \bR)_{\mathrm{van}}$ lies in $\rC$
    becomes a Hodge class after parallel transport to a smooth surface $S_t \in |mH|$. 
    Here, $\rH^2(S_0, \bZ)_{\mathrm{van}}$ denotes the kernel of the Gysin map $\rH^2(S_0, \bZ) \to \rH^4(X, \bZ)$ (and similarly for $\bR$-coefficients).  
    
\end{lemma}

        We choose $m$ and $S_0$ such that $m$ is relatively prime to $n$ and Lemma~\ref{lem:totaro} applies. We may assume further that $S_0$ has minimal Picard rank among surfaces in $|mH|$. Let $\alpha_0$ be the restriction of $\alpha$ to $S_0$. 

        We claim that $\alpha_0$ is topologically trivial. To prove the claim, let $\bar \alpha \in \rH^3(X, \bZ)$ be the topological part of $\alpha$. Observe that $\bar \alpha \cup [S_0] = 0$, by our description of $\gr_{\rF}^1$ as the kernel of $- \cup \bar \alpha$ in $\rH^2(X, \bZ)$. On the other hand, the Gysin map
    \[
            \iota_*:\rH^3(S_0, \bZ) \to \rH^5(X, \bZ), \quad \iota:S_0 \to X
    \]
    is an isomorphism by Lefschetz's hyperplane theorem, and $\iota_* \iota^* \bar \alpha = \bar \alpha \cup [S_0]$.

    \begin{lemma}\label{lem:exact_sequence_ktop}
        There is a short exact sequence
        \begin{equation} \label{eq:ses_hodge}
            \begin{tikzcd}
                  0 \ar[r] & \rH^2(S_0, \bZ)_{\mathrm{van}}   \ar[r] & \Ktop[0](S_0, \alpha_0) \ar[r, "\iota_*"] & M \ar[r] & 0,
            \end{tikzcd}
        \end{equation}
        where $M \subset \Ktop[0](X, \alpha)$ is the subgroup of elements whose leading term is a multiple of $[S_0]$.
    \end{lemma}

    \begin{proof}
        The map $\iota_*:\Ktop[0](S_0, \alpha_0) \to M$ sends $\rF^i$ to $\rF^{i + 1}$, where $\rF$ denotes the twisted Atiyah--Hirzebruch filtration on both sides. The lemma follows from considering the induced morphisms of associated graded pieces:
        \begin{align*}
                \rH^0(S_0, \bZ) &\to \bZ \cdot [S_0] \\
                \rH^2(S_0, \bZ) &\to \rH^4(X, \bZ) \\
                \rH^4(S_0, \bZ) &\to \rH^6(X, \bZ)
        \end{align*}
        Note that the first line is an isomorphism since $\alpha_0$ is topologically trivial. 
    \end{proof}

    Let $w$ be a Hodge class with leading term $[S_0]$. Since the integral Hodge conjecture holds for twisted surfaces such as $(S_0, \alpha_0)$ \cite{hotchkiss-pi}, it suffices to show that there is a Hodge class in $\Ktop[0](S_0, \alpha_0)$ mapping to $w$, up to deforming $S_0$ in $|mH|$. We may choose: 
    \begin{itemize}
            \item A Hodge class $w' \in \Ktop[0](S_0, \alpha_0) \otimes \bQ$ mapping to $w$. 
            Indeed, the pushforward map $\iota_* \colon \Ktop[0](S_0, \alpha_0) \to \Ktop[0](X, \alpha)$ is a morphism of Hodge structures, hence (its complexification) is strict for the Hodge filtrations. 
            Since $S_0$ has minimal Picard rank among surfaces in $|mH|$ by assumption, any parallel transport of $w'$ is of Hodge type. 
            \item An integral class $w'' \in \Ktop[0](S_0, \alpha_0)$ mapping to $w$. Indeed, $w''$ exists by Lemma~\ref{lem:exact_sequence_ktop}. 
    \end{itemize}
    Then the leading term of $w' - w''$ lies in $\rH^2(S_0, \bQ)_{\mathrm{van}}$. By Lemma~\ref{lem:totaro}, there exists a surface $S_t$ in $|mH|$ such that a parallel transport of $w' - w''$ is of Hodge type. It follows that a parallel transport of $w''$ is a Hodge class in $\Ktop[0](S_t, \alpha_t)$ of rank $1$, as needed. \qed

\subsection{Computing the Hodge-theoretic index}

In the case of abelian varieties, it is not difficult to compute the Hodge-theoretic index explicitly. For simplicity, we stick to the case of abelian threefolds.

\begin{lemma}
\label{lem:computing_the_hodge_theoretic_index}
    Let $X$ be an abelian threefold, and let $B \in \rH^2(X, \bQ(1))$. Then the Hodge-theoretic index of $\alpha = \exp(B) \in \Br(X)$ is the least positive integer $N$ such that there exist $H_1 \in \NS(X)_{\bQ}$ and $H_2 \in \Hdg^4(X, \bQ)$ with
    \begin{align}
            N \cdot B - H_1 &\equiv 0 \mod \rH^2(X, \bZ) , \\
            N \cdot \frac{1}{2}B^2 - B \cdot H_1 + H_2 &\equiv 0 \mod \rH^4(X, \bZ).
    \end{align}
\end{lemma}

\begin{proof}
    According to Lemma~\ref{lem:twisted_mukai_abelian-families}, a rank $r$ Hodge class in $\Ktop[0](X, \alpha)$ corresponds to an element $(r, x, y, z) \in \rH^{\ev}(X, \bZ)$ with 
    \[
            \exp(B) \cdot (r, x, y, z) = (r, r \cdot B + x, r \cdot \frac{1}{2} B^2 + B \cdot x + y, \dots)
    \]
    of Hodge type in each degree. We set $H_1 = r \cdot B + x$ and 
    \begin{align*}
            H_2 &= r \cdot \frac{1}{2}B^2 + B \cdot x + y \\
                &= r \cdot \frac{1}{2}B^2 + B (H_1 - r\cdot B) + y \\
                &= - r \cdot \frac{1}{2} B^2 + B \cdot H_1 + y.
    \end{align*}
    We see that $r$ satisfies the assumptions of $N$ in the lemma. 
    This implies that $N \leq \ind_{\Hdg}(\alpha)$. 

    On the other hand, given $N, H_1$, and $H_2$ as in the statement, we set
    \[
            x = -N \cdot B + H_1, \quad y = N \cdot \frac{1}{2} B^2 - B \cdot H_1 + H_2.
    \]
    Then it is easy to see that $(N, x, y, 0) \in \rH^{\ev}(X, \bZ)$ has the property that $\exp(B) \cdot (N, x, y, 0)$ is of Hodge type in each degree, and hence corresponds to a rank $N$ Hodge class in $\Ktop[0](X, \alpha)$. 
    We conclude $\ind_{\Hdg}(\alpha) \leq N$.
\end{proof}

\subsection{A pathology} 
\label{section-pathology} 

In this section, we observe that there are Severi--Brauer varieties $P$ over the product of three elliptic curves such that the integral Hodge conjecture holds in topological $K$-theory, but fails in integral cohomology.

\begin{example}[Gabber]
\label{ex:gabber}
        Let $E_1, E_2, E_3$ be pairwise non-isogenous elliptic curves. Choose integral bases
        \begin{align*}
                \rH^1(E_1, \bZ) &= \langle x_1, x_2 \rangle \\
                \rH^1(E_2, \bZ) &= \langle y_1, y_2 \rangle\\
                \rH^1(E_3, \bZ) &= \langle z_1, z_2 \rangle.
        \end{align*}
        Let $X = E_1 \times E_2 \times E_3$, for any prime $\ell$ let 
        \[
                B = \frac{1}{\ell} \left( x_1 \wedge z_1 + y_1 \wedge z_2 \right) \in \rH^2(X, \bQ),  
        \]
        and let $\alpha = \exp(2 \pi i \cdot B) \in \Br(X)[\ell]$. 
        Then it is straightforward to compute using Lemma~\ref{lem:computing_the_hodge_theoretic_index} that $\ind_{\Hdg}(\alpha) = \ell^2$. Since $\ind_{\Hdg}(\alpha)$ divides $\ind(\alpha)$ and it is easy to see that $\ind(\alpha)$ divides $\ell^2$  (Lemma~\ref{lem:symbol_length_bounds}), we find that $\ind(\alpha) = \ell^2$, which recovers a result of Gabber \cite{CT-examples}. 
        Let us also note that there exists a Severi--Brauer variety $P \to X$ of class $\alpha$ and relative dimension $\ell^2 - 1$, since $\alpha$ is the class of the product of two degree $\ell$ cyclic algebras; cf. \cite[Proposition~5.19]{dJP-pi}, which also gives a slightly different argument for the equality $\ind(\alpha) = \ell^2$. 
\end{example}

\begin{corollary}        
\label{cor:existence}
        There exists a Severi--Brauer variety $P \to X$ of relative dimension $3$ over a product of elliptic curves $X = E_1 \times E_2 \times E_3$ such that the following hold:
        \begin{enumerate} 
                \item The integral Hodge conjecture fails in $\rH^6(P, \bZ)$.
                \item The integral Hodge conjecture holds in $\Ktop[0](P)$.
        \end{enumerate}
\end{corollary}

\begin{proof}
    We adopt the situation of Example~\ref{ex:gabber} for $\ell = 2$, and choose a Severi--Brauer variety $P \to X$ of class $\alpha$ and relative dimension $3$. 
    Since $\ind_{\Hdg}(\alpha) = \ell^2$, we conclude that the integral Hodge conjecture holds for $(X, \alpha)$ by Corollary~\ref{corollary-voisin-group-abelian3fold}. 
    Bernardara's decomposition \cite{bernardara-BS} has the shape
    \[
            \Dperf(P) = \langle \Dperf(X), \Dperf(X, \alpha), \Dperf(X), \Dperf(X, \alpha) \rangle,
    \]
    where we have omitted the functors giving the admissible embeddings. It follows that the integral Hodge conjecture holds for $\Dperf(P)$, since the integral Hodge conjecture for categories is compatible with semiorthogonal decompositions \cite[Lemma 5.20]{IHC-CY2}. On the other hand, by \cite[Theorem 6.1]{hotchkiss-pi}, the integral Hodge conjecture fails in $\rH^6(P, \bZ)$.
\end{proof}

\section{Symbol length}
\label{sec:symbol_length}

In this section, we describe symbol length bounds for Brauer classes on abelian varieties, and give an example (Example~\ref{ex:symbol_length_insufficient}) to show that they are not sufficient to prove Theorem~\ref{thm:period_index_conjecture}.

\begin{definition}
\label{def:symbol_length}
    Let $X$ be an abelian variety of dimension $g$.
    \begin{enumerate} 
            \item Given $\theta \in \rH^2(X, \bmu_n)$, we define $\ell(\theta)$ to be the least integer such that $\theta$ may be written
            \begin{equation} \label{eq:symbol_length}
                    \theta = x_1 \wedge y_1 + \cdots + x_{\ell(\theta)} \wedge y_{\ell(\theta)} 
            \end{equation}
            in $\rH^2(X, \bmu_n)$, with $x_i, y_i \in \rH^1(X, \bmu_n)$.
            \item Let $\alpha \in \Br(X)$ be a Brauer class of period $n$. We define the \emph{symbol length} of $\alpha$ to be 
    \[
            \length(\alpha) = \min\{\ell(\theta): \theta \in \rH^2(X, \bmu_n), \theta \mapsto \alpha\}.
    \]
    \end{enumerate}
\end{definition}

\begin{remark}
    Unlike the index, the symbol length of $\alpha \in \Br(X)$ defined above does not necessarily agree with the symbol length of the restriction of $\alpha$ to the function field $\bC(X)$. For instance, if $X$ is a general abelian surface, then by direct calculation as in Example~\ref{ex:symbol_length_insufficient} below, one may write down Brauer classes $\alpha \in \Br(X)$ with symbol length $2$ in the sense of Definition~\ref{def:symbol_length}. According to a folklore conjecture, however, any Brauer class on the function field of a surface over an algebraically closed field has symbol length $1$. 
\end{remark}

\begin{lemma}[Symbol length bound]
\label{lem:symbol_length_bounds}
    Let $X$ be an abelian variety of dimension $g$. For each $\alpha \in \Br(X)$ of period $n$, $\ind(\alpha) \mid n^{\ell(\alpha)}$. In particular, $\ind(\alpha) \mid n^g$.
\end{lemma}

\begin{proof}
    One can kill $\alpha$ on an isogeny $\pi:X' \to X$ of degree $n^{\ell(\alpha)}$ as follows: Take $\theta \in \rH^2(X, \bmu_n)$ with $\ell(\alpha) = \ell(\theta)$. Then take $X'$ to be the fiber product over $X$ of the $\bmu_n$-covers corresponding to $x_1, \dots, x_{\ell(\theta)}$ in \eqref{eq:symbol_length}. Since $\pi^* \theta = 0$, $\pi^* \alpha = 0$.
\end{proof}

The following example shows that symbol length bounds are not sufficient to prove Theorem~\ref{thm:period_index_conjecture}.

\begin{example}
\label{ex:symbol_length_insufficient}
        Let $(X, H)$ be a principally polarized abelian threefold of Picard rank $1$. We write
        \[
                H = x_1 \wedge y_1 + x_2 \wedge y_2 + x_3 \wedge y_3
        \]
        for $x_i, y_i \in \rH^1(X, \bZ)$. 
        Consider the Brauer class $\alpha$ of period $2$ corresponding to the reduction mod $2$ of
        \[
                b = x_1 \wedge (y_1 + y_2) + x_2 \wedge (y_1 + y_3) + x_3 \wedge (y_1 + y_2 + y_3).
        \]
        Note that
        \begin{align*}
                b + H = x_1 \wedge (2y_1+ y_2) + x_2 \wedge (y_1 + y_2 + y_3) + x_3 \wedge (y_1 + y_2 + 2y_3).
        \end{align*}
        Then, writing $\bar b$ and $\bar H$ for the reductions mod $2$, we have 
        \[
                \ell(\bar b) = \ell(\bar b + \bar H) = 3.
        \]
        Note that, since $X$ has Picard rank $1$, the exact sequence~\eqref{Br-kummer-seq} takes the form 
        \begin{equation*}
        0 \to \bZ/2 \to \rH^2(X, \bmu_2) \to \Br(X)[2] \to 0 , 
        \end{equation*}
        where $1 \in \bZ/2$ maps to $\bar{H} \in \rH^2(X, \bmu_2)$ and $\bar{b} \in \rH^2(X, \bmu_2)$ maps to $\alpha$. 
        It follows that $\ell(\alpha) = 3$.
\end{example}

\section{Twisted Fourier--Mukai partners}
\label{sec:fourier_mukai_partners}

In this section, we show that there are twisted abelian varieties such that $\Dperf(X, \alpha)$ is not equivalent to $\Dperf(Y)$ for a smooth proper variety $Y$.

To do so, we will consider a natural pairing on topological K-theory. 
In general, we recall from \cite[Lemma 5.2]{NCHPD} that for any proper $\bC$-linear category $\cC$, there is a canonical a bilinear form 
\begin{equation*}
\chi(-,-) \colon \Ktop[0](\cC) \otimes \Ktop[0](\cC) \to \bZ, 
\end{equation*}
called the (topological) \emph{Euler pairing}. 
On algebraic K-theory, there is also an Euler pairing $\chi(-,-) \colon \rK_0(\cC) \otimes \rK_0(\cC) \to \bZ$ given on objects $E, F \in \cC$ by 
\begin{equation*}
\chi(E,F) = \sum_i (-1)^i \dim \Ext^i_{\bC}(E,F), 
\end{equation*}
and the natural map $\rK_0(\cC) \to \Ktop[0](\cC)$ preserves the Euler pairings. 
When $\cC = \Dperf(X, \alpha)$ for a twisted variety with $\alpha$ topologically trivial, the Euler pairing can be described in terms of pairings of Chern classes in singular cohomology. 
In the case of an abelian variety, this result takes the following form.

\begin{lemma}
\label{lemma-abelian-variety-euler}
Let $X$ be an abelian variety 
and $\alpha = \exp(2 \pi i \cdot B) \in \Br(X)$ for $B \in \rH^2(X, \bQ)$. 
Then under the isomorphism 
$\varphi \colon \Ktop(X, \alpha) \to \muk(X, B; \bZ)$ of Lemma~\ref{lem:twisted_mukai_abelian-families}\eqref{item:identify_with_tw_muk}, we have 
\begin{equation*}
\chi(v,w) = \int_X \varphi(v)^{\vee} \varphi(w), 
\end{equation*}
where we have used the following notation: given $\gamma  = (\gamma_i) \in  \muk(X, B; \bZ)$ with $\gamma_i \in \rH^{2i}(X, \bZ)$ we write $\gamma^{\vee} = ((-1)^i\gamma_i) \in \muk(X, B; \bZ)$. 
\end{lemma}

\begin{proof}
By \cite[Corollary 4.14]{hotchkiss-pi}, this reduces to the case where $B = \alpha = 0$. In this case, the formula reads 
\begin{equation*}
\chi(v,w) = \int_X \ch(v)^{\vee}\ch(w), 
\end{equation*}
which holds by Hirzebruch--Riemann--Roch on $\Ktop[0](X)$ since the Todd class of $X$ is trivial.  
\end{proof}

\begin{lemma}
\label{lem:euler_pairing}
    Let $X$ be a very general principally polarized abelian threefold. There exists a class $b \in \rH^2(X, \bZ)$ and a prime $\ell$ such that for 
$\alpha = \exp\left(\frac{2 \pi i}{\ell} \cdot b\right) \in \Br(X)$, 
    the restriction of the Euler pairing $\chi(-, -)$ to the subgroup $\Hdg(X, \alpha, \bZ) \subset \Ktop[0](X, \alpha)$ is divisible by $\ell$.  
\end{lemma}

\begin{proof}
	We choose $(X, H)$, $b$, and $\ell$ such that the following hold: 
	\begin{enumerate}
	\item \label{XPic1} $(X,H)$ is a polarized abelian threefold with $\NS(X) = \bZ H$. 
	\item \label{HTIell} The class $\alpha = \exp\left(\frac{2 \pi i}{\ell} \cdot b\right)$ has Hodge-theoretic index $\ell^2$. 
	\item \label{bH-not-Hodge} $b \cdot H$ is not congruent to an integral Hodge class modulo $\ell$. 
	\end{enumerate} 
We can choose $(X,H)$ and $b$ so that~\eqref{XPic1} holds and~\eqref{HTIell} holds for any $\ell$ by using the construction of Gabber described in Example~\ref{ex:gabber} for a product of elliptic curves, and then passing to a general deformation of it to ensure that $X$ has Picard rank $1$; alternatively, one can bypass considering the product of elliptic curves and argue directly as in Example~\ref{ex:gabber} on a Picard rank $1$ abelian threefold. 
For condition~\eqref{bH-not-Hodge}, note that $b \cdot H$ is not a Hodge class, since by the hard Lefschetz theorem $\cdot H  \colon \rH^2(X, \bQ) \to \rH^4(X, \bQ)$ is an isomorphism of Hodge structures and $b$ is not Hodge since $\alpha$ is nontrivial. 
Thus $b \cdot H$ has nonzero image $b'$ in the torsion-free group $\rH^4(X, \bZ)/\Hdg^4(X, \bZ)$, and choosing any $\ell$ such that $b'$ is not $\ell$-divisible suffices. 

We will show that with the above choices, the restriction of the pairing $\chi(-, -)$ to the subgroup $\Hdg(X, \alpha, \bZ)$ is divisible by $\ell$. 
By Lemma~\ref{lemma-abelian-variety-euler}, it is equivalent to prove that the pairing $(\gamma, \delta) = \int_X \gamma^{\vee} \delta$ on $\Hdg(\muk(X, B; \bZ))$ is divisible by $\ell$, where $B = b/\ell$. 
Let $\gamma = (r, x, y, z) \in \Hdg(\muk(X, B; \bZ))$. 
By condition~\eqref{HTIell}, we have $\ell^2 \mid r$. 
We claim that furthermore $\ell$ divides $x$ in $\rH^2(X, \bZ)$, 
which directly implies the desired divisibility of the pairing on $\Hdg(\muk(X, B; \bZ))$. 

Note that $(\ell^2, \ell \cdot b, \frac{1}{2}b^2, 0) \in \Hdg(\muk(X, B; \bZ))$ is a Hodge class with the claimed property, so by subtracting a suitable multiple of it from $\gamma$, we may suppose that $r = 0$. 
Then the condition that $\gamma$ is Hodge shows that $x$ and $B \cdot x + y$ are Hodge. This implies $x = mH$ for an integer $m$ due to condition~\eqref{XPic1}, and 
\begin{equation*}
m b \cdot H + \ell y \in \Hdg^4(X, \bZ). 
\end{equation*} 
By condition~\eqref{bH-not-Hodge}, this implies that $\ell \mid m$, as claimed. 
\end{proof}

\begin{corollary}        
\label{cor:not_derived_equivalent}
    Let $(X, \alpha)$ be as in the statement of Lemma~\ref{lem:euler_pairing}. Then $\Dperf(X, \alpha)$ is not equivalent to $\Dperf(Y)$ for any smooth proper variety $Y$. 
\end{corollary}

\begin{proof}
    If $Y$ is a smooth proper variety, then the Euler form on the subgroup of Hodge classes $\Hdg(Y, \bZ) \subset \Ktop[0](Y)$ is primitive, since e.g. $\chi(\cO_Y, \kappa(y)) = 1$ for any skyscraper sheaf $\kappa(y)$. 
\end{proof}


\newpage
\addtocontents{toc}{\vspace{\normalbaselineskip}}

\newcommand{\etalchar}[1]{$^{#1}$}
\providecommand{\bysame}{\leavevmode\hbox to3em{\hrulefill}\thinspace}
\providecommand{\MR}{\relax\ifhmode\unskip\space\fi MR }
\providecommand{\MRhref}[2]{%
  \href{http://www.ams.org/mathscinet-getitem?mr=#1}{#2}
}
\providecommand{\href}[2]{#2}



\begin{thebibliography}{BLM{\etalchar{+}}21}

\bibitem[AB15]{arcara-bertram}
Daniele Arcara and Aaron Bertram, \emph{Bridgeland-stable moduli spaces for
  {$K$}-trivial surfaces}, J. Eur. Math. Soc. (JEMS) \textbf{2013} (15), no.~1,
  1--38, with an appendix by Max Lieblich.

\bibitem[ABGV11]{open-problems-Br}
Asher Auel, Eric Brussel, Skip Garibaldi, and Uzi Vishne, \emph{Open problems
  on central simple algebras}, Transform. Groups \textbf{16} (2011), no.~1,
  219--264.

\bibitem[ABM21]{HKR-charp}
Benjamin Antieau, Bhargav Bhatt, and Akhil Mathew, \emph{Counterexamples to
  {H}ochschild-{K}ostant-{R}osenberg in characteristic {$p$}}, Forum Math.
  Sigma \textbf{9} (2021), Paper No. e49, 26.

\bibitem[AG14]{antieau-gepner}
Benjamin Antieau and David Gepner, \emph{Brauer groups and \'{e}tale cohomology
  in derived algebraic geometry}, Geom. Topol. \textbf{18} (2014), no.~2,
  1149--1244.

\bibitem[AH61]{atiyah-hirzebruch}
M.~F. Atiyah and F.~Hirzebruch, \emph{Vector bundles and homogeneous spaces},
  Proc. {S}ympos. {P}ure {M}ath., {V}ol. {III}, Amer. Math. Soc., Providence,
  RI, 1961, pp.~7--38.

\bibitem[AP23]{normal-cone-artin-stacks}
Dhyan Aranha and Piotr Pstragowski, \emph{The intrinsic normal cone for {A}rtin
  stacks}, Annales de l'Institut Fourier (2023) (en), Online first.

\bibitem[Ara05]{leray_motivic}
Donu Arapura, \emph{The {L}eray spectral sequence is motivic}, Invent. Math.
  \textbf{160} (2005), no.~3, 567--589.

\bibitem[AS04]{atiyah_segal}
Michael Atiyah and Graeme Segal, \emph{Twisted {$K$}-theory}, Ukr. Mat. Visn.
  \textbf{1} (2004), no.~3, 287--330.

\bibitem[AS06]{atiyah-segal-2}
\bysame, \emph{Twisted {$K$}-theory and cohomology}, Inspired by {S}. {S}.
  {C}hern, Nankai Tracts Math., vol.~11, World Sci. Publ., Hackensack, NJ,
  2006, pp.~5--43.

\bibitem[AT09]{anel-toen}
Mathieu Anel and Bertrand To\"{e}n, \emph{D\'{e}nombrabilit\'{e} des classes
  d'\'{e}quivalences d\'{e}riv\'{e}es de vari\'{e}t\'{e}s alg\'{e}briques}, J.
  Algebraic Geom. \textbf{18} (2009), no.~2, 257--277.

\bibitem[AW14]{ant_will}
Benjamin Antieau and Ben Williams, \emph{The period-index problem for twisted
  topological {$K$}-theory}, Geom. Topol. \textbf{18} (2014), no.~2,
  1115--1148.

\bibitem[Beh09]{behrend}
Kai Behrend, \emph{Donaldson-{T}homas type invariants via microlocal geometry},
  Ann. of Math. (2) \textbf{170} (2009), no.~3, 1307--1338.

\bibitem[Ber09]{bernardara-BS}
Marcello Bernardara, \emph{A semiorthogonal decomposition for {B}rauer-{S}everi
  schemes}, Math. Nachr. \textbf{282} (2009), no.~10, 1406--1413.

\bibitem[BF97]{BF-normal-cone}
Kai Behrend and Barbara Fantechi, \emph{The intrinsic normal cone}, Invent.
  Math. \textbf{128} (1997), no.~1, 45--88.

\bibitem[BF08]{BF-symmetric}
\bysame, \emph{Symmetric obstruction theories and {H}ilbert schemes of points
  on threefolds}, Algebra Number Theory \textbf{2} (2008), no.~3, 313--345.

\bibitem[BJ17]{BJ}
Dennis Borisov and Dominic Joyce, \emph{Virtual fundamental classes for moduli
  spaces of sheaves on {C}alabi-{Y}au four-folds}, Geom. Topol. \textbf{21}
  (2017), no.~6, 3231--3311.

\bibitem[BKP22]{BKP}
Younghan Bae, Martijn Kool, and Hyeonjun Park, \emph{Counting surfaces on
  {C}alabi--{Y}au 4-folds {I}: {F}oundations}, arXiv:2208.09474 (2022).

\bibitem[BKRS22]{milnor-squares}
Tom Bachmann, Adeel~A. Khan, Charanya Ravi, and Vladimir Sosnilo,
  \emph{Categorical {M}ilnor squares and {K}-theory of algebraic stacks},
  Selecta Math. (N.S.) \textbf{28} (2022), no.~5, Paper No. 85, 72.

\bibitem[BL99]{bl_complex_tori}
Christina Birkenhake and Herbert Lange, \emph{Complex tori}, Progress in
  Mathematics, vol. 177, Birkh\"{a}user Boston, Inc., Boston, MA, 1999.

\bibitem[BL04]{birk_lange}
\bysame, \emph{Complex abelian varieties}, second ed., Grundlehren der
  mathematischen Wissenschaften [Fundamental Principles of Mathematical
  Sciences], vol. 302, Springer-Verlag, Berlin, 2004.

\bibitem[Bla16]{blanc}
Anthony Blanc, \emph{Topological {K}-theory of complex noncommutative spaces},
  Compos. Math. \textbf{152} (2016), no.~3, 489--555.

\bibitem[BLM{\etalchar{+}}21]{stability-families}
Arend Bayer, Mart\'i Lahoz, Emanuele Macr\`\i, Howard Nuer, Alexander Perry,
  and Paolo Stellari, \emph{Stability conditions in families}, Publ. Math.
  Inst. Hautes \'{E}tudes Sci. \textbf{133} (2021), 157--325.

\bibitem[BM23]{bayer-macri-ICM}
Arend Bayer and Emanuele Macr\`{i}, \emph{The unreasonable effectiveness of
  wall-crossing in algebraic geometry}, I{CM}---{I}nternational {C}ongress of
  {M}athematicians. {V}ol. 3. {S}ections 1--4, EMS Press, Berlin, 2023,
  pp.~2172--2195.

\bibitem[BMS16]{BMS}
Arend Bayer, Emanuele Macr\`\i, and Paolo Stellari, \emph{The space of
  stability conditions on abelian threefolds, and on some {C}alabi-{Y}au
  threefolds}, Invent. Math. \textbf{206} (2016), no.~3, 869--933.

\bibitem[BMSZ17]{Macri-et-al-Fano}
Marcello Bernardara, Emanuele Macr\`\i, Benjamin Schmidt, and Xiaolei Zhao,
  \emph{Bridgeland stability conditions on {F}ano threefolds}, \'{E}pijournal
  Geom. Alg\'{e}brique \textbf{1} (2017), Art. 2, 24.

\bibitem[BMT14]{BMT}
Arend Bayer, Emanuele Macr\`{i}, and Yukinobu Toda, \emph{Bridgeland stability
  conditions on threefolds {I}: {B}ogomolov-{G}ieseker type inequalities}, J.
  Algebraic Geom. \textbf{23} (2014), no.~1, 117--163.

\bibitem[BOPY18]{BOPR}
Jim Bryan, Georg Oberdieck, Rahul Pandharipande, and Qizheng Yin, \emph{Curve
  counting on abelian surfaces and threefolds}, Algebr. Geom. \textbf{5}
  (2018), no.~4, 398--463.

\bibitem[BP23]{Fano3fold}
Arend Bayer and Alexander Perry, \emph{Kuznetsov's {F}ano threefold conjecture
  via {K}3 categories and enhanced group actions}, J. Reine Angew. Math.
  \textbf{800} (2023), 107--153.

\bibitem[Bri07]{bridgeland-stability}
Tom Bridgeland, \emph{Stability conditions on triangulated categories}, Ann. of
  Math. (2) \textbf{166} (2007), no.~2, 317--345.

\bibitem[Bri08]{bridgeland-K3}
\bysame, \emph{Stability conditions on {$K3$} surfaces}, Duke Math. J.
  \textbf{141} (2008), no.~2, 241--291.

\bibitem[BS21]{bergh-BS}
Daniel Bergh and Olaf Schn\"{u}rer, \emph{Decompositions of derived categories
  of gerbes and of families of {B}rauer-{S}everi varieties}, Doc. Math.
  \textbf{26} (2021), 1465--1500.

\bibitem[BZFN10]{bzfn}
David Ben-Zvi, John Francis, and David Nadler, \emph{Integral transforms and
  {D}rinfeld centers in derived algebraic geometry}, J. Amer. Math. Soc.
  \textbf{23} (2010), no.~4, 909--966.

\bibitem[{\v{C}}S24]{KS-purity}
K\polhk{e}stutis {\v{C}}esnavi\v{c}ius and Peter Scholze, \emph{Purity for flat
  cohomology}, Ann. of Math. (2) \textbf{199} (2024), no.~1, 51--180.

\bibitem[CT01]{CT-PI}
Jean-Louis Colliot-Th\'{e}l\`ene, \emph{Die {B}rauersche {G}ruppe; ihre
  {V}erallgemeinerungen und {A}nwendungen in der {A}rithmetischen {G}eometrie},
  arXiv:2311.02437 (2001).

\bibitem[CT02]{CT-examples}
\bysame, \emph{Exposant et indice d'alg\`ebres simples centrales non
  ramifi\'{e}es}, Enseign. Math. (2) \textbf{48} (2002), no.~1-2, 127--146,
  With an appendix by Ofer Gabber.

\bibitem[CT06]{CT-bourbaki}
\bysame, \emph{Alg\`ebres simples centrales sur les corps de fonctions de deux
  variables (d'apr\`es {A}. {J}. de {J}ong)}, no. 307, 2006, S\'{e}minaire
  Bourbaki. Vol. 2004/2005, pp.~Exp. No. 949, ix, 379--413.

\bibitem[C\u00]{caldararu}
Andrei C\u{a}ld\u{a}raru, \emph{Derived categories of twisted sheaves on
  {C}alabi-{Y}au manifolds}, ProQuest LLC, Ann Arbor, MI, 2000, Thesis
  (Ph.D.)--Cornell University.

\bibitem[Del68]{degenerescence}
Pierre Deligne, \emph{Th\'{e}or\`eme de {L}efschetz et crit\`eres de
  d\'{e}g\'{e}n\'{e}rescence de suites spectrales}, Inst. Hautes \'{E}tudes
  Sci. Publ. Math. (1968), no.~35, 259--278.

\bibitem[Del71]{Deligne1971ThorieDH}
\bysame, \emph{Th\'{e}orie de {H}odge. {II}}, Inst. Hautes \'{E}tudes Sci.
  Publ. Math. (1971), no.~40, 5--57.

\bibitem[dJ]{dJ-gabber}
Aise~Johan de~Jong, \emph{A result of {G}abber}, available at
  \url{http://www.math.columbia.edu/~dejong/}.

\bibitem[dJ04]{dJ-period-index}
\bysame, \emph{The period-index problem for the {B}rauer group of an algebraic
  surface}, Duke Math. J. \textbf{123} (2004), no.~1, 71--94.

\bibitem[dJO22]{dejong2022point}
Aise~Johan de~Jong and Martin Olsson, \emph{Point objects on abelian
  varieties}, 2022.

\bibitem[dJP22]{dJP-pi}
Aise~Johan de~Jong and Alexander Perry, \emph{The period-index problem and
  {H}odge theory}, arXiv:2212.12971 (2022).

\bibitem[DL90]{debarre_laszlo}
Olivier Debarre and Yves Laszlo, \emph{Le lieu de {N}oether-{L}efschetz pour
  les vari\'{e}t\'{e}s ab\'{e}liennes}, C. R. Acad. Sci. Paris S\'{e}r. I Math.
  \textbf{311} (1990), no.~6, 337--340.

\bibitem[FH91]{fulton_harris}
William Fulton and Joe Harris, \emph{Representation theory}, Graduate Texts in
  Mathematics, vol. 129, Springer-Verlag, New York, 1991.

\bibitem[GR17]{DAG-gaitsgory-1}
Dennis Gaitsgory and Nick Rozenblyum, \emph{A study in derived algebraic
  geometry. {V}ol. {I}. {C}orrespondences and duality}, Mathematical Surveys
  and Monographs, vol. 221, American Mathematical Society, Providence, RI,
  2017.

\bibitem[Gra04]{grabowski-IHC}
Craig Grabowski, \emph{On the integral {H}odge conjecture for 3-folds},
  ProQuest LLC, Ann Arbor, MI, 2004, Thesis (Ph.D.)--Duke University.

\bibitem[Gro66]{EGAIV3}
Alexander Grothendieck, \emph{\'{E}l\'{e}ments de g\'{e}om\'{e}trie
  alg\'{e}brique. {IV}. \'{E}tude locale des sch\'{e}mas et des morphismes de
  sch\'{e}mas. {T}roisi\`{e}me partie}, Inst. Hautes \'{E}tudes Sci. Publ.
  Math. (1966), no.~28, 255.

\bibitem[Gro67]{EGAIV4}
\bysame, \emph{\'{E}l\'{e}ments de g\'{e}om\'{e}trie alg\'{e}brique. {IV}.
  \'{E}tude locale des sch\'{e}mas et des morphismes de sch\'{e}mas.
  {Q}uatri\`{e}me partie}, Inst. Hautes \'{E}tudes Sci. Publ. Math. (1967),
  no.~32, 361.

\bibitem[Gro68]{grothendieck-brauer}
\bysame, \emph{Le groupe de {B}rauer. {I--III}}, Dix expos\'{e}s sur la
  cohomologie des sch\'{e}mas, Adv. Stud. Pure Math., vol.~3, North-Holland,
  Amsterdam, 1968, pp.~46--188.

\bibitem[GS06]{gillet_sza}
Philippe Gille and Tam\'{a}s Szamuely, \emph{Central simple algebras and
  {G}alois cohomology}, Cambridge Studies in Advanced Mathematics, vol. 101,
  Cambridge University Press, Cambridge, 2006.

\bibitem[Gul13]{gulbrandsen}
Martin Gulbrandsen, \emph{Donaldson-{T}homas invariants for complexes on
  abelian threefolds}, Math. Z. \textbf{273} (2013), no.~1-2, 219--236.

\bibitem[Hat02]{hatcher}
Allen Hatcher, \emph{Algebraic topology}, Cambridge University Press,
  Cambridge, 2002.

\bibitem[HL10]{huybrechts-lehn}
Daniel Huybrechts and Manfred Lehn, \emph{The geometry of moduli spaces of
  sheaves}, second ed., Cambridge Mathematical Library, Cambridge University
  Press, Cambridge, 2010.

\bibitem[HM23]{huybrechts-mattei}
Daniel Huybrechts and Dominique Mattei, \emph{Splitting unramified {B}rauer
  classes by abelian torsors and the period-index problem}, arXiv:2310.04029
  (2023).

\bibitem[Hot22]{hotchkiss-pi}
James Hotchkiss, \emph{Hodge theory of twisted derived categories and the
  period-index problem}, arXiv:2212.10638 (2022).

\bibitem[Hot23]{hotchkiss-tori}
\bysame, \emph{The period-index problem for complex tori}, arXiv:2301.09293
  (2023).

\bibitem[HS05]{Huy_stell}
Daniel Huybrechts and Paolo Stellari, \emph{Equivalences of twisted {$K3$}
  surfaces}, Math. Ann. \textbf{332} (2005), no.~4, 901--936.

\bibitem[Huy06]{huybrechts_fm}
Daniel Huybrechts, \emph{Fourier-{M}ukai transforms in algebraic geometry},
  Oxford Mathematical Monographs, The Clarendon Press, Oxford University Press,
  Oxford, 2006.

\bibitem[Igu70]{igusa}
Jun-ichi Igusa, \emph{A classification of spinors up to dimension twelve},
  Amer. J. Math. \textbf{92} (1970), 997--1028.

\bibitem[Jan73]{janusz}
Gerald Janusz, \emph{Algebraic number fields}, Pure and Applied Mathematics,
  vol. Vol. 55, Academic Press [Harcourt Brace Jovanovich, Publishers], New
  York-London, 1973.

\bibitem[JS12]{joyce-song}
Dominic Joyce and Yinan Song, \emph{A theory of generalized
  {D}onaldson-{T}homas invariants}, Mem. Amer. Math. Soc. \textbf{217} (2012),
  no.~1020, iv+199.

\bibitem[Kal08]{kaledin1}
Dmitry Kaledin, \emph{Non-commutative {H}odge-to-de {R}ham degeneration via the
  method of {D}eligne-{I}llusie}, Pure Appl. Math. Q. \textbf{4} (2008), no.~3,
  Special Issue: In honor of Fedor Bogomolov. Part 2, 785--875.

\bibitem[Kal17]{kaledin2}
\bysame, \emph{Spectral sequences for cyclic homology}, Algebra, geometry, and
  physics in the 21st century, Progr. Math., vol. 324, Birkh\"{a}user/Springer,
  Cham, 2017, pp.~99--129.

\bibitem[Kle05]{kleiman-picard}
Steven~L. Kleiman, \emph{The {P}icard scheme}, Fundamental algebraic geometry,
  Math. Surveys Monogr., vol. 123, Amer. Math. Soc., Providence, RI, 2005,
  pp.~235--321.

\bibitem[KO18]{okawa-nonexistence}
Kotaro Kawatani and Shinnosuke Okawa, \emph{Nonexistence of semiorthogonal
  decompositions and sections of the canonical bundle}, arXiv:1508.00682
  (2018).

\bibitem[Kre99]{kresch-chow}
Andrew Kresch, \emph{Cycle groups for {A}rtin stacks}, Invent. Math.
  \textbf{138} (1999), no.~3, 495--536.

\bibitem[KS08]{kontsevich-soibelman}
Maxim Kontsevich and Yan Soibelman, \emph{Stability structures, motivic
  {D}onaldson-{T}homas invariants and cluster transformations}, arXiv:0811.2435
  (2008).

\bibitem[Kuz15]{kuznetsov-heights}
Alexander Kuznetsov, \emph{Height of exceptional collections and {H}ochschild
  cohomology of quasiphantom categories}, J. Reine Angew. Math. \textbf{708}
  (2015), 213--243.

\bibitem[Kuz19]{kuznetsov-CY}
\bysame, \emph{Calabi--{Y}au and fractional {C}alabi--{Y}au categories}, J.
  Reine Angew. Math. \textbf{753} (2019), 239--267.

\bibitem[Lan02]{Lang}
Serge Lang, \emph{Algebra}, third ed., Graduate Texts in Mathematics, vol. 211,
  Springer-Verlag, New York, 2002.

\bibitem[Li19a]{chunyi-quintic}
Chunyi Li, \emph{On stability conditions for the quintic threefold}, Invent.
  Math. \textbf{218} (2019), no.~1, 301--340.

\bibitem[Li19b]{Chunyi}
\bysame, \emph{Stability conditions on {F}ano threefolds of {P}icard number 1},
  J. Eur. Math. Soc. (JEMS) \textbf{21} (2019), no.~3, 709--726.

\bibitem[Lie06]{lieblich-moduli}
Max Lieblich, \emph{Moduli of complexes on a proper morphism}, J. Algebraic
  Geom. \textbf{15} (2006), no.~1, 175--206.

\bibitem[Lie07]{lieblich-moduli-twisted}
\bysame, \emph{Moduli of twisted sheaves}, Duke Math. J. \textbf{138} (2007),
  no.~1, 23--118.

\bibitem[Lie08]{lieblich-period-index}
\bysame, \emph{Twisted sheaves and the period-index problem}, Compos. Math.
  \textbf{144} (2008), no.~1, 1--31.

\bibitem[Lur04]{DAG}
Jacob Lurie, \emph{Derived algebraic geometry}, ProQuest LLC, Ann Arbor, MI,
  2004, Thesis (Ph.D.)--Massachusetts Institute of Technology.

\bibitem[Lur17]{HA}
\bysame, \emph{Higher algebra}, available at
  \url{https://www.math.ias.edu/~lurie/}, 2017.

\bibitem[Lur18]{SAG}
\bysame, \emph{Spectral algebraic geometry}, available at
  \url{https://www.math.ias.edu/~lurie/}, 2018.

\bibitem[Mat16]{matzri}
Eliyahu Matzri, \emph{Symbol length in the {B}rauer group of a field}, Trans.
  Amer. Math. Soc. \textbf{368} (2016), no.~1, 413--427.

\bibitem[Mat20]{akhil-degeneration}
Akhil Mathew, \emph{Kaledin's degeneration theorem and topological {H}ochschild
  homology}, Geom. Topol. \textbf{24} (2020), no.~6, 2675--2708.

\bibitem[Mou19]{moulinos}
Tasos Moulinos, \emph{Derived {A}zumaya algebras and twisted {$K$}-theory},
  Adv. Math. \textbf{351} (2019), 761--803.

\bibitem[MP15]{Piyaratne-abelian1}
Antony Maciocia and Dulip Piyaratne, \emph{Fourier-{M}ukai transforms and
  {B}ridgeland stability conditions on abelian threefolds}, Algebr. Geom.
  \textbf{2} (2015), no.~3, 270--297.

\bibitem[MP16]{Piyaratne-abelian2}
\bysame, \emph{Fourier-{M}ukai transforms and {B}ridgeland stability conditions
  on abelian threefolds {II}}, Internat. J. Math. \textbf{27} (2016), no.~1,
  1650007, 27.

\bibitem[MS17]{macri-schmidt}
Emanuele Macr\`{i} and Benjamin Schmidt, \emph{Lectures on {B}ridgeland
  stability}, Moduli of curves, Lect. Notes Unione Mat. Ital., vol.~21,
  Springer, Cham, 2017, pp.~139--211.

\bibitem[Muk78]{mukai_semihom}
Shigeru Mukai, \emph{Semi-homogeneous vector bundles on an {A}belian variety},
  J. Math. Kyoto Univ. \textbf{18} (1978), no.~2, 239--272.

\bibitem[Muk98]{mukai_spin}
\bysame, \emph{Abelian variety and spin representation}, available at
  \url{https://www.kurims.kyoto-u.ac.jp/~mukai/paper/warwick13.pdf}.

\bibitem[Nak35]{nakayama}
Tadasi Nakayama, \emph{\"{U}ber die direkte {Z}erlegung einer
  {D}ivisionsalgebra}, Jpn. J. Math. \textbf{12} (1935), 65--70.

\bibitem[NO80]{oort_norman}
Peter Norman and Frans Oort, \emph{Moduli of abelian varieties}, Annals of
  Mathematics \textbf{112} (1980), no.~2, 413--439.

\bibitem[Obe18]{obPT}
Georg Oberdieck, \emph{On reduced stable pair invariants}, Math. Z.
  \textbf{289} (2018), no.~1-2, 323--353.

\bibitem[OPT22]{OPT}
Georg Oberdieck, Dulip Piyaratne, and Yukinobu Toda, \emph{Donaldson-{T}homas
  invariants of abelian threefolds and {B}ridgeland stability conditions}, J.
  Algebraic Geom. \textbf{31} (2022), no.~1, 13--73.

\bibitem[Orl16]{orlov-gluing}
Dmitri Orlov, \emph{Smooth and proper noncommutative schemes and gluing of {DG}
  categories}, Adv. Math. \textbf{302} (2016), 59--105.

\bibitem[OS20]{obshen}
Georg Oberdieck and Junliang Shen, \emph{Curve counting on elliptic
  {C}alabi-{Y}au threefolds via derived categories}, J. Eur. Math. Soc. (JEMS)
  \textbf{22} (2020), no.~3, 967--1002.

\bibitem[OT23]{OT}
Jeongseok Oh and Richard~P. Thomas, \emph{Counting sheaves on {C}alabi-{Y}au
  4-folds, {I}}, Duke Math. J. \textbf{172} (2023), no.~7, 1333--1409.

\bibitem[Per19]{NCHPD}
Alexander Perry, \emph{Noncommutative homological projective duality}, Adv.
  Math. \textbf{350} (2019), 877--972.

\bibitem[Per22]{IHC-CY2}
\bysame, \emph{The integral {H}odge conjecture for two-dimensional
  {C}alabi--{Y}au categories}, Compos. Math. \textbf{158} (2022), no.~2,
  287--333.

\bibitem[Pol07]{polishchuk-t-structure}
Alexander Polishchuk, \emph{Constant families of {$t$}-structures on derived
  categories of coherent sheaves}, Mosc. Math. J. \textbf{7} (2007), no.~1,
  109--134, 167.

\bibitem[Pom15]{poma-vclass}
Flavia Poma, \emph{Virtual classes of {A}rtin stacks}, Manuscripta Math.
  \textbf{146} (2015), no.~1-2, 107--123.

\bibitem[Pri22]{pridham}
J.~P. Pridham, \emph{Semiregularity as a consequence of {G}oodwillie's
  theorem}, arXiv:1208.3111 (2022).

\bibitem[Rom11]{romagny}
Matthieu Romagny, \emph{Composantes connexes et irr\'{e}ductibles en familles},
  Manuscripta Math. \textbf{136} (2011), no.~1-2, 1--32.

\bibitem[Ros56]{algebraic-groups}
Maxwell Rosenlicht, \emph{Some basic theorems on algebraic groups}, Amer. J.
  Math. \textbf{78} (1956), 401--443.

\bibitem[Sch17]{schmidt-generalizedBG}
Benjamin Schmidt, \emph{Counterexample to the generalized
  {B}ogomolov-{G}ieseker inequality for threefolds}, Int. Math. Res. Not. IMRN
  (2017), no.~8, 2562--2566.

\bibitem[SdJ10]{dJ-starr}
Jason Starr and Johan de~Jong, \emph{Almost proper {GIT}-stacks and
  discriminant avoidance}, Doc. Math. \textbf{15} (2010), 957--972.

\bibitem[{Sta}24]{stacks-project}
The {Stacks Project Authors}, \emph{\textit{Stacks Project}},
  \url{https://stacks.math.columbia.edu}, 2024.

\bibitem[STV15]{toen-obstruction}
Timo Sch\"{u}rg, Bertrand To\"{e}n, and Gabriele Vezzosi, \emph{Derived
  algebraic geometry, determinants of perfect complexes, and applications to
  obstruction theories for maps and complexes}, J. Reine Angew. Math.
  \textbf{702} (2015), 1--40.

\bibitem[Tho00]{thomas-DT}
Richard Thomas, \emph{A holomorphic {C}asson invariant for {C}alabi-{Y}au
  3-folds, and bundles on {$K3$} fibrations}, J. Differential Geom. \textbf{54}
  (2000), no.~2, 367--438.

\bibitem[To{\"{e}}09]{toen-higher-derived}
Bertrand To{\"{e}}n, \emph{Higher and derived stacks: a global overview},
  Algebraic geometry---{S}eattle 2005. {P}art 1, Proc. Sympos. Pure Math.,
  vol.~80, Amer. Math. Soc., Providence, RI, 2009, pp.~435--487.

\bibitem[To{\"{e}}10]{toen-simplicial}
\bysame, \emph{Simplicial presheaves and derived algebraic geometry},
  Simplicial methods for operads and algebraic geometry, Adv. Courses Math. CRM
  Barcelona, Birkh\"{a}user/Springer Basel AG, Basel, 2010, pp.~119--186.

\bibitem[To{\"{e}}11]{toen-descent}
\bysame, \emph{Descente fid\`element plate pour les {$n$}-champs d'{A}rtin},
  Compos. Math. \textbf{147} (2011), no.~5, 1382--1412.

\bibitem[To{\"{e}}12]{toen-generator}
\bysame, \emph{Derived {A}zumaya algebras and generators for twisted derived
  categories}, Invent. Math. \textbf{189} (2012), no.~3, 581--652.

\bibitem[To{\"{e}}14]{toen-DAG}
\bysame, \emph{Derived algebraic geometry}, EMS Surv. Math. Sci. \textbf{1}
  (2014), no.~2, 153--240.

\bibitem[Tot21]{totaro}
Burt Totaro, \emph{The integral {H}odge conjecture for 3-folds of {K}odaira
  dimension zero}, J. Inst. Math. Jussieu \textbf{20} (2021), no.~5,
  1697--1717.

\bibitem[TV05]{HAG1}
Bertrand To\"{e}n and Gabriele Vezzosi, \emph{Homotopical algebraic geometry.
  {I}. {T}opos theory}, Adv. Math. \textbf{193} (2005), no.~2, 257--372.

\bibitem[TV07]{toen-moduli}
Bertrand To\"{e}n and Michel Vaqui\'{e}, \emph{Moduli of objects in
  dg-categories}, Ann. Sci. \'{E}cole Norm. Sup. (4) \textbf{40} (2007), no.~3,
  387--444.

\bibitem[TV08]{HAG2}
Bertrand To\"{e}n and Gabriele Vezzosi, \emph{Homotopical algebraic geometry.
  {II}. {G}eometric stacks and applications}, Mem. Amer. Math. Soc.
  \textbf{193} (2008), no.~902, x+224.

\bibitem[TVdB15]{HH-TVdB}
Gon\c{c}alo Tabuada and Michel Van~den Bergh, \emph{Noncommutative motives of
  {A}zumaya algebras}, J. Inst. Math. Jussieu \textbf{14} (2015), no.~2,
  379--403.

\bibitem[Vis89]{vistoli-chow}
Angelo Vistoli, \emph{Intersection theory on algebraic stacks and on their
  moduli spaces}, Invent. Math. \textbf{97} (1989), no.~3, 613--670.

\bibitem[Voi02]{voisin-hodge-theory-i}
Claire Voisin, \emph{Hodge theory and complex algebraic geometry. {I}},
  Cambridge Studies in Advanced Mathematics, vol.~76, Cambridge University
  Press, Cambridge, 2002.

\bibitem[Voi06]{voisin-IHC}
\bysame, \emph{On integral {H}odge classes on uniruled or {C}alabi-{Y}au
  threefolds}, Moduli spaces and arithmetic geometry, Adv. Stud. Pure Math.,
  vol.~45, Math. Soc. Japan, Tokyo, 2006, pp.~43--73.

\bibitem[Yos06]{yoshioka-twisted-sheaves}
Kota Yoshioka, \emph{Moduli spaces of twisted sheaves on a projective variety},
  Moduli spaces and arithmetic geometry, Adv. Stud. Pure Math., vol.~45, Math.
  Soc. Japan, Tokyo, 2006, pp.~1--30.

\end{thebibliography}
\end{document}